\documentclass{amsart}
\usepackage{amssymb}
\usepackage{amsfonts}
\usepackage{geometry}

\newtheorem{theorem}{Theorem}
\theoremstyle{plain}

\newtheorem{condition}{Condition}

\newtheorem{corollary}{Corollary}

\newtheorem{definition}{Definition}
\newtheorem{example}{Example}

\newtheorem{lemma}{Lemma}
\newtheorem{notation}{Notation}
\newtheorem{problem}{Problem}
\newtheorem{proposition}{Proposition}
\newtheorem{remark}{Remark}

\numberwithin{equation}{section}
\input{tcilatex}
\begin{document}
\title[Two weight boundedness]{A two weight theorem for $\alpha $-fractional
singular integrals with an energy side condition, quasicube testing and
common point masses}
\author[E.T. Sawyer]{Eric T. Sawyer}
\address{ Department of Mathematics \& Statistics, McMaster University, 1280
Main Street West, Hamilton, Ontario, Canada L8S 4K1 }
\email{sawyer@mcmaster.ca}
\thanks{Research supported in part by NSERC}
\author[C.-Y. Shen]{Chun-Yen Shen}
\address{ Department of Mathematics \\
National Central University \\
Chungli, 32054, Taiwan }
\email{chunyshen@gmail.com}
\thanks{C.-Y. Shen supported in part by the NSC, through grant
NSC102-2115-M-008-015-MY2}
\author[I. Uriarte-Tuero]{Ignacio Uriarte-Tuero}
\address{ Department of Mathematics \\
Michigan State University \\
East Lansing MI }
\email{ignacio@math.msu.edu}
\thanks{ I. Uriarte-Tuero has been partially supported by grants DMS-1056965
(US NSF), MTM2010-16232, MTM2015-65792-P (MINECO, Spain), and a Sloan
Foundation Fellowship. }
\date{April 16, 2016}

\begin{abstract}
Let $\sigma $ and $\omega $ be locally finite positive Borel measures on $%
\mathbb{R}^{n}$ (possibly having common point masses), and let $T^{\alpha }$%
\ be a standard $\alpha $-fractional Calder\'{o}n-Zygmund operator on $%
\mathbb{R}^{n}$ with $0\leq \alpha <n$. Suppose that $\Omega :\mathbb{R}%
^{n}\rightarrow \mathbb{R}^{n}$ is a globally biLipschitz map, and refer to
the images $\Omega Q$ of cubes $Q$ as \emph{quasicubes}. Furthermore, assume
as side conditions the $\mathcal{A}_{2}^{\alpha }$ conditions, punctured $%
A_{2}^{\alpha }$ conditions, and certain $\alpha $\emph{-energy conditions}
taken over quasicubes. Then we show that $T^{\alpha }$ is bounded from $%
L^{2}\left( \sigma \right) $ to $L^{2}\left( \omega \right) $ if the
quasicube testing conditions hold for $T^{\alpha }$\textbf{\ }and its dual,
and if the quasiweak boundedness property holds for $T^{\alpha }$.

Conversely, if $T^{\alpha }$ is bounded from $L^{2}\left( \sigma \right) $
to $L^{2}\left( \omega \right) $, then the quasitesting conditions hold, and
the quasiweak boundedness condition holds. If the vector of $\alpha $%
-fractional Riesz transforms $\mathbf{R}_{\sigma }^{\alpha }$ (or more
generally a strongly elliptic vector of transforms) is bounded from $%
L^{2}\left( \sigma \right) $ to $L^{2}\left( \omega \right) $, then both the 
$\mathcal{A}_{2}^{\alpha }$ conditions and the punctured $A_{2}^{\alpha }$
conditions hold. We do not know if our quasienergy conditions are necessary
when $n\geq 2$, except for certain situations in which one of the measures
is one-dimensional, or both measures are sufficiently dispersed.
\end{abstract}

\maketitle
\tableofcontents

\section{Introduction}

In this paper we prove a two weight inequality for standard $\alpha $%
-fractional Calder\'{o}n-Zygmund operators $T^{\alpha }$ in Euclidean space $%
\mathbb{R}^{n}$, where we assume $n$-dimensional $\mathcal{A}_{2}^{\alpha }$
conditions (with holes), punctured $A_{2}^{\alpha ,\limfunc{punct}}$
conditions, and certain $\alpha $\emph{-energy conditions} as side
conditions on the weights (in higher dimensions the Poisson kernels used in
these two conditions differ). We state and prove our theorem in the more
general setting of \emph{quasicubes} as in \cite{SaShUr5}, but here we now
permit the weights, or measures, to have common point masses, something not
permitted in \cite{SaShUr5}. As a consequence, we use $\mathcal{A}%
_{2}^{\alpha }$ conditions with holes as in the one-dimensional setting of
Hyt\"{o}nen \cite{Hyt2}, together with punctured $A_{2}^{\alpha ,\limfunc{%
punct}}$ conditions, as the usual $A_{2}^{\alpha }$ `\emph{without punctures}%
' fails whenever the measures have a common point mass. The extension to
permitting common point masses uses the two weight Poisson inequality in 
\cite{Saw3} to derive functional energy, together with a delicate adaptation
of arguments in \cite{SaShUr5}. The key point here is the use of the
(typically necessary) `punctured' Muckenhoupt $A_{2}^{\alpha ,\limfunc{punct}%
}$ conditions below. They turn out to be crucial in estimating the Poisson
testing conditions later in the paper. We remark that Hyt\"{o}nen's bilinear
dyadic Poisson operator and shifted dyadic grids (\cite{Hyt2}) in dimension $%
n=1$ can be extended to derive functional energy in higher dimensions, but
at a significant cost of increased complexity. See the earlier versions of
this paper on the \textit{arXiv} for this approach\footnote{%
Additional small arguments are needed to complete the shifted dyadic proof
given there, but we omit them in favour of the simpler approach here resting
on punctured Muckenhoupt conditions instead of holes. The authors can be
contacted regarding completion of the shifted dyadic proof.}, and also \cite%
{LaWi} where Lacey and Wick use this approach. Finally, we point out that
our use of punctured Muckenhoupt conditions provides a simpler alternative
to Hyt\"{o}nen's method of extending to common point masses the NTV
conjecture for the Hilbert transform \cite{Hyt2}.

On the other hand, the extension to quasicubes in the setting of common
point masses turns out to be, after checking all the details, mostly a
cosmetic modification of the proof in \cite{SaShUr5}, except for the
derivation of the $n$-dimensional $\mathcal{A}_{2}^{\alpha }$ conditions
with holes, which requires extensive modification of earlier arguments.

We begin by recalling the notion of quasicube used in \cite{SaShUr5} - a
special case of the classical notion used in quasiconformal theory.

\begin{definition}
We say that a homeomorphism $\Omega :\mathbb{R}^{n}\rightarrow \mathbb{R}%
^{n} $ is a globally biLipschitz map if%
\begin{equation}
\left\Vert \Omega \right\Vert _{Lip}\equiv \sup_{x,y\in \mathbb{R}^{n}}\frac{%
\left\Vert \Omega \left( x\right) -\Omega \left( y\right) \right\Vert }{%
\left\Vert x-y\right\Vert }<\infty ,  \label{rigid}
\end{equation}%
and $\left\Vert \Omega ^{-1}\right\Vert _{Lip}<\infty $.
\end{definition}

Note that a globally biLipschitz map $\Omega $ is differentiable almost
everywhere, and that there are constants $c,C>0$ such that%
\begin{equation*}
c\leq J_{\Omega }\left( x\right) \equiv \left\vert \det D\Omega \left(
x\right) \right\vert \leq C,\ \ \ \ \ x\in \mathbb{R}^{n}.
\end{equation*}

\begin{example}
\label{wild}Quasicubes can be wildly shaped, as illustrated by the standard
example of a logarithmic spiral in the plane $f_{\varepsilon }\left(
z\right) =z\left\vert z\right\vert ^{2\varepsilon i}=ze^{i\varepsilon \ln
\left( z\overline{z}\right) }$. Indeed, $f_{\varepsilon }:\mathbb{%
C\rightarrow C}$ is a globally biLipschitz map with Lipschitz constant $%
1+C\varepsilon $ since $f_{\varepsilon }^{-1}\left( w\right) =w\left\vert
w\right\vert ^{-2\varepsilon i}$ and%
\begin{equation*}
\nabla f_{\varepsilon }=\left( \frac{\partial f_{\varepsilon }}{\partial z},%
\frac{\partial f_{\varepsilon }}{\partial \overline{z}}\right) =\left(
\left\vert z\right\vert ^{2\varepsilon i}+i\varepsilon \left\vert
z\right\vert ^{2\varepsilon i},i\varepsilon \frac{z}{\overline{z}}\left\vert
z\right\vert ^{2\varepsilon i}\right) .
\end{equation*}%
On the other hand, $f_{\varepsilon }$ behaves wildly at the origin since the
image of the closed unit interval on the real line under $f_{\varepsilon }$
is an infinite logarithmic spiral.
\end{example}

\begin{notation}
We define $\mathcal{P}^{n}$ to be the collection of half open, half closed
cubes in $\mathbb{R}^{n}$ with sides parallel to the coordinate axes. A half
open, half closed cube $Q$ in $\mathbb{R}^{n}$ has the form $Q=Q\left(
c,\ell \right) \equiv \dprod\limits_{k=1}^{n}\left[ c_{k}-\frac{\ell }{2}%
,c_{k}+\frac{\ell }{2}\right) $ for some $\ell >0$ and $c=\left(
c_{1},...,c_{n}\right) \in \mathbb{R}^{n}$. The cube $Q\left( c,\ell \right) 
$ is described as having center $c$ and sidelength $\ell $.
\end{notation}

We typically use $\mathcal{D}$ to denote a dyadic grid of cubes from $%
\mathcal{P}^{n}$. We repeat the natural \emph{quasi} definitions from \cite%
{SaShUr5}.

\begin{definition}
Suppose that $\Omega :\mathbb{R}^{n}\rightarrow \mathbb{R}^{n}$ is a
globally biLipschitz map.

\begin{enumerate}
\item If $E$ is a measurable subset of $\mathbb{R}^{n}$, we define $\Omega
E\equiv \left\{ \Omega \left( x\right) :x\in E\right\} $ to be the image of $%
E$ under the homeomorphism $\Omega $.

\begin{enumerate}
\item In the special case that $E=Q$ is a cube in $\mathbb{R}^{n}$, we will
refer to $\Omega Q$ as a quasicube (or $\Omega $-quasicube if $\Omega $ is
not clear from the context).

\item We define the center $c_{\Omega Q}=c\left( \Omega Q\right) $ of the
quasicube $\Omega Q$ to be the point $\Omega c_{Q}$ where $c_{Q}=c\left(
Q\right) $ is the center of $Q$.

\item We define the side length $\ell \left( \Omega Q\right) $ of the
quasicube $\Omega Q$ to be the sidelength $\ell \left( Q\right) $ of the
cube $Q$.

\item For $r>0$ we define the `dilation' $r\Omega Q$ of a quasicube $\Omega
Q $ to be $\Omega rQ$ where $rQ$ is the usual `dilation' of a cube in $%
\mathbb{R}^{n}$ that is concentric with $Q$ and having side length $r\ell
\left( Q\right) $.
\end{enumerate}

\item If $\mathcal{K}$ is a collection of cubes in $\mathbb{R}^{n}$, we
define $\Omega \mathcal{K}\equiv \left\{ \Omega Q:Q\in \mathcal{K}\right\} $
to be the collection of quasicubes $\Omega Q$ as $Q$ ranges over $\mathcal{K}
$.

\item If $\mathcal{F}$ is a grid of cubes in $\mathbb{R}^{n}$, we define the
inherited quasigrid structure on $\Omega \mathcal{F}$ by declaring that $%
\Omega Q$ is a child of $\Omega Q^{\prime }$ in $\Omega \mathcal{F}$ if $Q$
is a child of $Q^{\prime }$ in the grid $\mathcal{F}$. We denote by $%
\mathfrak{C}\left( Q\right) $ the collection of children of $Q$.
\end{enumerate}
\end{definition}

Note that if $\Omega Q$ is a quasicube, then $\left\vert \Omega Q\right\vert
^{\frac{1}{n}}\approx \left\vert Q\right\vert ^{\frac{1}{n}}=\ell \left(
Q\right) =\ell \left( \Omega Q\right) $ shows that the measure of $\Omega Q$
is approximately its sidelength to the power $n$, more precisely there are
positive constants $c,C$ such that $c\left\vert J\right\vert ^{\frac{1}{n}%
}\leq \ell \left( J\right) \leq C\left\vert J\right\vert ^{\frac{1}{n}}$ for
any quasicube $J=\Omega Q$. We will generally use the expression $\left\vert
J\right\vert ^{\frac{1}{n}}$ in the various estimates arising in the proofs
below, but will often use $\ell \left( J\right) $ when defining collections
of quasicubes. Moreover, there are constants $R_{big}$ and $R_{small}$ such
that we have the comparability containments%
\begin{equation*}
Q+\Omega x_{Q}\subset R_{big}\Omega Q\text{ and }R_{small}\Omega Q\subset
Q+\Omega x_{Q}\ .
\end{equation*}

Given a fixed globally biLipschitz map $\Omega $ on $\mathbb{R}^{n}$, we
will define below the $n$-dimensional $\mathcal{A}_{2}^{\alpha }$ conditions
(with holes), punctured Muckenhoupt conditions $A_{2}^{\alpha ,\limfunc{punct%
}}$, testing conditions, and energy conditions using $\Omega $-quasicubes in
place of cubes, and we will refer to these new conditions as quasi$\mathcal{A%
}_{2}^{\alpha }$, quasitesting and quasienergy conditions. We will then
prove a $T1$ theorem with quasitesting and with quasi$\mathcal{A}%
_{2}^{\alpha }$ and quasienergy side conditions on the weights. We now
describe a particular case informally, and later explain the full theorem in
detail.

We show that for positive locally finite Borel measures $\sigma $ and $%
\omega $, possibly having common point masses, and \emph{assuming} the
quasienergy conditions in the Theorem below, a strongly elliptic collection
of standard $\alpha $-fractional Calder\'{o}n-Zygmund operators $\mathbf{T}%
^{\alpha }$ is bounded from $L^{2}\left( \sigma \right) $ to $L^{2}\left(
\omega \right) $,%
\begin{equation}
\left\Vert \mathbf{T}^{\alpha }\left( f\sigma \right) \right\Vert
_{L^{2}\left( \omega \right) }\leq \mathfrak{N}_{\mathbf{T}^{\alpha
}}\left\Vert f\right\Vert _{L^{2}\left( \sigma \right) },  \label{2 weight}
\end{equation}%
(with $0\leq \alpha <n$) \emph{if and only if} the $\mathcal{A}_{2}^{\alpha
} $ condition and its dual hold (we assume a mild additional condition on
the quasicubes for this), the punctured Muckenhoupt condition $A_{2}^{\alpha
,\limfunc{punct}}$ and its dual hold, the quasicube testing condition for $%
\mathbf{T}^{\alpha }$ and its dual hold, and the quasiweak boundedness
property holds.

Since the $\mathcal{A}_{2}^{\alpha }$ and punctured Muckenhoupt conditions
typically hold, this identifies the culprit in higher dimensions as the pair
of quasienergy conditions. We point out that these quasienergy conditions
are implied by higher dimensional analogues of essentially all the other
side conditions used previously in two weight theory, in particular doubling
conditions and the Energy Hypothesis (1.16) in \cite{LaSaUr2}, as well as
the uniformly full dimension hypothesis in \cite{LaWi} (uniformly full
dimension permits a reversal of energy, something not assumed in this paper,
and reversal of energy implies our energy conditions).

It turns out that in higher dimensions, there are two natural `Poisson
integrals' $\mathrm{P}^{\alpha }$ and $\mathcal{P}^{\alpha }$\ that arise,
the usual Poisson integral $\mathrm{P}^{\alpha }$ that emerges in connection
with energy considerations, and a different Poisson integral $\mathcal{P}%
^{\alpha }$ that emerges in connection with size considerations - in
dimension $n=1$ these two Poisson integrals coincide. The standard Poisson
integral $\mathrm{P}^{\alpha }$ appears in the energy conditions, and the
reproducing Poisson integral $\mathcal{P}^{\alpha }$ appears in the $%
\mathcal{A}_{2}^{\alpha }$ condition. These two kernels coincide in
dimension $n=1$ for the case $\alpha =0$ corresponding to the Hilbert
transform.

\section{Statements of results}

Now we turn to a precise description of our two weight theorem.

\begin{description}
\item[Assumption] We fix once and for all a globally biLipschitz map $\Omega
:\mathbb{R}^{n}\rightarrow \mathbb{R}^{n}$ for use in all of our
quasi-notions.
\end{description}

As already mentioned above we will prove a two weight inequality for
standard $\alpha $-fractional Calder\'{o}n-Zygmund operators $T^{\alpha }$
in Euclidean space $\mathbb{R}^{n}$, where we assume the $n$-dimensional $%
\mathcal{A}_{2}^{\alpha }$ conditions, new punctured $A_{2}^{\alpha }$
conditions, and certain $\alpha $\emph{-quasienergy conditions} as side
conditions on the weights. In order to state our theorem precisely, we need
to define standard fractional singular integrals, the two different Poisson
kernels, and a quasienergy condition sufficient for use in the proof of the
two weight theorem. These are introduced in the following subsections.

\begin{remark}
It is possible to collect our various Muckenhoupt and quasienergy
assumptions on the weight pair $\left( \sigma ,\omega \right) $ into just 
\emph{two} compact side conditions of Muckenhoupt and quasienergy type. We
prefer however, to keep the individual conditions separate so that the
interested reader can track their use through the arguments below.
\end{remark}

\subsection{Standard fractional singular integrals and the norm inequality}

Let $0\leq \alpha <n$. Consider a kernel function $K^{\alpha }(x,y)$ defined
on $\mathbb{R}^{n}\times \mathbb{R}^{n}$ satisfying the following fractional
size and smoothness conditions of order $1+\delta $ for some $\delta >0$,%
\begin{eqnarray}
\left\vert K^{\alpha }\left( x,y\right) \right\vert &\leq &C_{CZ}\left\vert
x-y\right\vert ^{\alpha -n},  \label{sizeandsmoothness'} \\
\left\vert \nabla K^{\alpha }\left( x,y\right) \right\vert &\leq
&C_{CZ}\left\vert x-y\right\vert ^{\alpha -n-1},  \notag \\
\left\vert \nabla K^{\alpha }\left( x,y\right) -\nabla K^{\alpha }\left(
x^{\prime },y\right) \right\vert &\leq &C_{CZ}\left( \frac{\left\vert
x-x^{\prime }\right\vert }{\left\vert x-y\right\vert }\right) ^{\delta
}\left\vert x-y\right\vert ^{\alpha -n-1},\ \ \ \ \ \frac{\left\vert
x-x^{\prime }\right\vert }{\left\vert x-y\right\vert }\leq \frac{1}{2}, 
\notag \\
\left\vert \nabla K^{\alpha }\left( x,y\right) -\nabla K^{\alpha }\left(
x,y^{\prime }\right) \right\vert &\leq &C_{CZ}\left( \frac{\left\vert
y-y^{\prime }\right\vert }{\left\vert x-y\right\vert }\right) ^{\delta
}\left\vert x-y\right\vert ^{\alpha -n-1},\ \ \ \ \ \frac{\left\vert
y-y^{\prime }\right\vert }{\left\vert x-y\right\vert }\leq \frac{1}{2}. 
\notag
\end{eqnarray}

We note that a more general definition of kernel has only order of
smoothness $\delta >0$, rather than $1+\delta $, but the use of the
Monotonicity and Energy Lemmas below, which involve first order Taylor
approximations to the kernel functions $K^{\alpha }\left( \cdot ,y\right) $,
requires order of smoothness more than $1$.

\subsubsection{Defining the operators and norm inequality}

We now turn to a precise definition of the weighted norm inequality%
\begin{equation}
\left\Vert T_{\sigma }^{\alpha }f\right\Vert _{L^{2}\left( \omega \right)
}\leq \mathfrak{N}_{T_{\sigma }^{\alpha }}\left\Vert f\right\Vert
_{L^{2}\left( \sigma \right) },\ \ \ \ \ f\in L^{2}\left( \sigma \right) .
\label{two weight'}
\end{equation}%
For this we introduce a family $\left\{ \eta _{\delta ,R}^{\alpha }\right\}
_{0<\delta <R<\infty }$ of nonnegative functions on $\left[ 0,\infty \right) 
$ so that the truncated kernels $K_{\delta ,R}^{\alpha }\left( x,y\right)
=\eta _{\delta ,R}^{\alpha }\left( \left\vert x-y\right\vert \right)
K^{\alpha }\left( x,y\right) $ are bounded with compact support for fixed $x$
or $y$. Then the truncated operators 
\begin{equation*}
T_{\sigma ,\delta ,R}^{\alpha }f\left( x\right) \equiv \int_{\mathbb{R}%
^{n}}K_{\delta ,R}^{\alpha }\left( x,y\right) f\left( y\right) d\sigma
\left( y\right) ,\ \ \ \ \ x\in \mathbb{R}^{n},
\end{equation*}%
are pointwise well-defined, and we will refer to the pair $\left( K^{\alpha
},\left\{ \eta _{\delta ,R}^{\alpha }\right\} _{0<\delta <R<\infty }\right) $
as an $\alpha $-fractional singular integral operator, which we typically
denote by $T^{\alpha }$, suppressing the dependence on the truncations.

\begin{definition}
We say that an $\alpha $-fractional singular integral operator $T^{\alpha
}=\left( K^{\alpha },\left\{ \eta _{\delta ,R}^{\alpha }\right\} _{0<\delta
<R<\infty }\right) $ satisfies the norm inequality (\ref{two weight'})
provided%
\begin{equation*}
\left\Vert T_{\sigma ,\delta ,R}^{\alpha }f\right\Vert _{L^{2}\left( \omega
\right) }\leq \mathfrak{N}_{T_{\sigma }^{\alpha }}\left\Vert f\right\Vert
_{L^{2}\left( \sigma \right) },\ \ \ \ \ f\in L^{2}\left( \sigma \right)
,0<\delta <R<\infty .
\end{equation*}
\end{definition}

It turns out that, in the presence of Muckenhoupt conditions, the norm
inequality (\ref{two weight'}) is independent of the choice of appropriate
truncations used, and we now explain this in some detail. A \emph{smooth
truncation} of $T^{\alpha }$ has kernel $\eta _{\delta ,R}\left( \left\vert
x-y\right\vert \right) K^{\alpha }\left( x,y\right) $ for a smooth function $%
\eta _{\delta ,R}$ compactly supported in $\left( \delta ,R\right) $, $%
0<\delta <R<\infty $, and satisfying standard CZ estimates. A typical
example of an $\alpha $-fractional transform is the $\alpha $-fractional 
\emph{Riesz} vector of operators%
\begin{equation*}
\mathbf{R}^{\alpha ,n}=\left\{ R_{\ell }^{\alpha ,n}:1\leq \ell \leq
n\right\} ,\ \ \ \ \ 0\leq \alpha <n.
\end{equation*}%
The Riesz transforms $R_{\ell }^{n,\alpha }$ are convolution fractional
singular integrals $R_{\ell }^{n,\alpha }f\equiv K_{\ell }^{n,\alpha }\ast f$
with odd kernel defined by%
\begin{equation*}
K_{\ell }^{\alpha ,n}\left( w\right) \equiv \frac{w^{\ell }}{\left\vert
w\right\vert ^{n+1-\alpha }}\equiv \frac{\Omega _{\ell }\left( w\right) }{%
\left\vert w\right\vert ^{n-\alpha }},\ \ \ \ \ w=\left(
w^{1},...,w^{n}\right) .
\end{equation*}

However, in dealing with energy considerations, and in particular in the
Monotonicity Lemma below where first order Taylor approximations are made on
the truncated kernels, it is necessary to use the \emph{tangent line
truncation} of\emph{\ }the Riesz transform $R_{\ell }^{\alpha ,n}$ whose
kernel is defined to be $\Omega _{\ell }\left( w\right) \psi _{\delta
,R}^{\alpha }\left( \left\vert w\right\vert \right) $ where $\psi _{\delta
,R}^{\alpha }$ is continuously differentiable on an interval $\left(
0,S\right) $ with $0<\delta <R<S$, and where $\psi _{\delta ,R}^{\alpha
}\left( r\right) =r^{\alpha -n}$ if $\delta \leq r\leq R$, and has constant
derivative on both $\left( 0,\delta \right) $ and $\left( R,S\right) $ where 
$\psi _{\delta ,R}^{\alpha }\left( S\right) =0$. Here $S$ is uniquely
determined by $R$ and $\alpha $. Finally we set $\psi _{\delta ,R}^{\alpha
}\left( 0\right) =0$ as well, so that the kernel vanishes on the diagonal
and common point masses do not `see' each other. Note also that the tangent
line extension of a $C^{1,\delta }$ function on the line is again $%
C^{1,\delta }$ with no increase in the $C^{1,\delta }$ norm.

It was shown in the one dimensional case with no common point masses in \cite%
{LaSaShUr3}, that boundedness of the Hilbert transform $H$ with one set of
appropriate truncations together with the $A_{2}^{\alpha }$ condition
without holes, is equivalent to boundedness of $H$ with any other set of
appropriate truncations. We need to extend this to $\mathbf{R}^{\alpha ,n}$
and to more general operators in higher dimensions, and also to permit
common point masses, so that we are free to use the tangent line truncations
throughout the proof of our theorem. For this purpose, we note that the
difference between the tangent line truncated kernel $\Omega _{\ell }\left(
w\right) \psi _{\delta ,R}^{\alpha }\left( \left\vert w\right\vert \right) $
and the corresponding cutoff kernel $\Omega _{\ell }\left( w\right) \mathbf{1%
}_{\left[ \delta ,R\right] }\left\vert w\right\vert ^{\alpha -n}$ satisfies
(since both kernels vanish at the origin)%
\begin{eqnarray*}
&&\left\vert \Omega _{\ell }\left( w\right) \psi _{\delta ,R}^{\alpha
}\left( \left\vert w\right\vert \right) -\Omega _{\ell }\left( w\right) 
\mathbf{1}_{\left[ \delta ,R\right] }\left\vert w\right\vert ^{\alpha
-n}\right\vert \\
&\lesssim &\sum_{k=0}^{\infty }2^{-k\left( n-\alpha \right) }\left\{ \left(
2^{-k}\delta \right) ^{\alpha -n}\mathbf{1}_{\left[ 2^{-k-1}\delta
,2^{-k}\delta \right] }\left( \left\vert w\right\vert \right) \right\}
+\sum_{k=1}^{\infty }2^{-k\left( n-\alpha \right) }\left\{ \left(
2^{k}R\right) ^{\alpha -n}\mathbf{1}_{\left[ 2^{k-1}R,2^{k}R\right] }\left(
\left\vert w\right\vert \right) \right\} \\
&\equiv &\sum_{k=0}^{\infty }2^{-k\left( n-\alpha \right) }K_{2^{-k}\delta
}\left( w\right) +\sum_{k=1}^{\infty }2^{-k\left( n-\alpha \right)
}K_{2^{k}R}\left( w\right) ,
\end{eqnarray*}%
where the kernels $K_{\rho }\left( w\right) \equiv \frac{1}{\rho ^{n-\alpha }%
}\mathbf{1}_{\left[ \rho ,2\rho \right] }\left( \left\vert w\right\vert
\right) $ are easily seen to satisfy, uniformly in $\rho $, the norm
inequality (\ref{two weight}) with constant controlled by the offset $%
A_{2}^{\alpha }$ condition (\ref{offset A2}) below. The equivalence of the
norm inequality for these two families of truncations now follows from the
summability of the series $\sum_{k=0}^{\infty }2^{-k\left( n-\alpha \right)
} $ for $0\leq \alpha <n$. The case of more general families of truncations
and operators is similar.

\subsection{Quasicube testing conditions}

The following `dual' quasicube testing conditions are necessary for the
boundedness of $T^{\alpha }$ from $L^{2}\left( \sigma \right) $ to $%
L^{2}\left( \omega \right) $, where $\Omega \mathcal{P}^{n}$ denotes the
collection of all quasicubes in $\mathbb{R}^{n}$ whose preimages under $%
\Omega $ are usual cubes with sides parallel to the coordinate axes:%
\begin{eqnarray*}
\mathfrak{T}_{T^{\alpha }}^{2} &\equiv &\sup_{Q\in \Omega \mathcal{P}^{n}}%
\frac{1}{\left\vert Q\right\vert _{\sigma }}\int_{Q}\left\vert T^{\alpha
}\left( \mathbf{1}_{Q}\sigma \right) \right\vert ^{2}\omega <\infty , \\
\left( \mathfrak{T}_{T^{\alpha }}^{\ast }\right) ^{2} &\equiv &\sup_{Q\in
\Omega \mathcal{P}^{n}}\frac{1}{\left\vert Q\right\vert _{\omega }}%
\int_{Q}\left\vert \left( T^{\alpha }\right) ^{\ast }\left( \mathbf{1}%
_{Q}\omega \right) \right\vert ^{2}\sigma <\infty ,
\end{eqnarray*}%
and where we interpret the right sides as holding uniformly over all tangent
line truncations of $T^{\alpha }$.

\begin{remark}
We alert the reader that the symbols $Q,I,J,K$ will all be used to denote
either cubes or quasicubes, and the context will make clear which is the
case. Throughout most of the proof of the main theorem only quasicubes are
considered.
\end{remark}

\subsection{Quasiweak boundedness property}

The quasiweak boundedness property for $T^{\alpha }$ with constant $C$ is
given by 
\begin{eqnarray*}
&&\left\vert \int_{Q}T^{\alpha }\left( 1_{Q^{\prime }}\sigma \right) d\omega
\right\vert \leq \mathcal{WBP}_{T^{\alpha }}\sqrt{\left\vert Q\right\vert
_{\omega }\left\vert Q^{\prime }\right\vert _{\sigma }}, \\
&&\ \ \ \ \ \text{for all quasicubes }Q,Q^{\prime }\text{ with }\frac{1}{C}%
\leq \frac{\left\vert Q\right\vert ^{\frac{1}{n}}}{\left\vert Q^{\prime
}\right\vert ^{\frac{1}{n}}}\leq C, \\
&&\ \ \ \ \ \text{and either }Q\subset 3Q^{\prime }\setminus Q^{\prime }%
\text{ or }Q^{\prime }\subset 3Q\setminus Q,
\end{eqnarray*}%
and where we interpret the left side above as holding uniformly over all
tangent line trucations of $T^{\alpha }$. Note that the quasiweak
boundedness property is implied by either the \emph{tripled} quasicube
testing condition,%
\begin{equation*}
\left\Vert \mathbf{1}_{3Q}\mathbf{T}^{\alpha }\left( \mathbf{1}_{Q}\sigma
\right) \right\Vert _{L^{2}\left( \omega \right) }\leq \mathfrak{T}_{\mathbf{%
T}^{\alpha }}^{\limfunc{triple}}\left\Vert \mathbf{1}_{Q}\right\Vert
_{L^{2}\left( \sigma \right) },\ \ \ \ \ \text{for all quasicubes }Q\text{
in }\mathbb{R}^{n},
\end{equation*}%
or the tripled dual quasicube testing condition defined with $\sigma $ and $%
\omega $ interchanged and the dual operator $\mathbf{T}^{\alpha ,\ast }$ in
place of $\mathbf{T}^{\alpha }$. In turn, the tripled quasicube testing
condition can be obtained from the quasicube testing condition for the
truncated weight pairs $\left( \omega ,\mathbf{1}_{Q}\sigma \right) $. See
also Remark \ref{surgery} below.

\subsection{Poisson integrals and $\mathcal{A}_{2}^{\protect\alpha }$}

Recall that we have fixed a globally biLipschitz map $\Omega :\mathbb{R}%
^{n}\rightarrow \mathbb{R}^{n}$. Now let $\mu $ be a locally finite positive
Borel measure on $\mathbb{R}^{n}$, and suppose $Q$ is an $\Omega $-quasicube
in $\mathbb{R}^{n}$. Recall that $\left\vert Q\right\vert ^{\frac{1}{n}%
}\approx \ell \left( Q\right) $ for a quasicube $Q$. The two $\alpha $%
-fractional Poisson integrals of $\mu $ on a quasicube $Q$ are given by:%
\begin{eqnarray*}
\mathrm{P}^{\alpha }\left( Q,\mu \right) &\equiv &\int_{\mathbb{R}^{n}}\frac{%
\left\vert Q\right\vert ^{\frac{1}{n}}}{\left( \left\vert Q\right\vert ^{%
\frac{1}{n}}+\left\vert x-x_{Q}\right\vert \right) ^{n+1-\alpha }}d\mu
\left( x\right) , \\
\mathcal{P}^{\alpha }\left( Q,\mu \right) &\equiv &\int_{\mathbb{R}%
^{n}}\left( \frac{\left\vert Q\right\vert ^{\frac{1}{n}}}{\left( \left\vert
Q\right\vert ^{\frac{1}{n}}+\left\vert x-x_{Q}\right\vert \right) ^{2}}%
\right) ^{n-\alpha }d\mu \left( x\right) ,
\end{eqnarray*}%
where we emphasize that $\left\vert x-x_{Q}\right\vert $ denotes Euclidean
distance between $x$ and $x_{Q}$ and $\left\vert Q\right\vert $ denotes the
Lebesgue measure of the quasicube $Q$. We refer to $\mathrm{P}^{\alpha }$ as
the \emph{standard} Poisson integral and to $\mathcal{P}^{\alpha }$ as the 
\emph{reproducing} Poisson integral.

We say that the pair $\left( K,K^{\prime }\right) $ in $\mathcal{P}%
^{n}\times \mathcal{P}^{n}$ are \emph{neighbours} if $K$ and $K^{\prime }$
live in a common dyadic grid and both $K\subset 3K^{\prime }\setminus
K^{\prime }$ and $K^{\prime }\subset 3K\setminus K$, and we denote by $%
\mathcal{N}^{n}$ the set of pairs $\left( K,K^{\prime }\right) $ in $%
\mathcal{P}^{n}\times \mathcal{P}^{n}$ that are neighbours. Let 
\begin{equation*}
\Omega \mathcal{N}^{n}=\left\{ \left( \Omega K,\Omega K^{\prime }\right)
:\left( K,K^{\prime }\right) \in \mathcal{N}^{n}\right\}
\end{equation*}%
be the corresponding collection of neighbour pairs of quasicubes. Let $%
\sigma $ and $\omega $ be locally finite positive Borel measures on $\mathbb{%
R}^{n}$, possibly having common point masses, and suppose $0\leq \alpha <n$.
Then we define the classical \emph{offset }$A_{2}^{\alpha }$\emph{\ constants%
} by 
\begin{equation}
A_{2}^{\alpha }\left( \sigma ,\omega \right) \equiv \sup_{\left( Q,Q^{\prime
}\right) \in \Omega \mathcal{N}^{n}}\frac{\left\vert Q\right\vert _{\sigma }%
}{\left\vert Q\right\vert ^{1-\frac{\alpha }{n}}}\frac{\left\vert Q^{\prime
}\right\vert _{\omega }}{\left\vert Q\right\vert ^{1-\frac{\alpha }{n}}}.
\label{offset A2}
\end{equation}%
Since the cubes in $\mathcal{P}^{n}$ are products of half open, half closed
intervals $\left[ a,b\right) $, the neighbouring quasicubes $\left(
Q,Q^{\prime }\right) \in \Omega \mathcal{N}^{n}$ are disjoint, and the
common point masses of $\sigma $ and $\omega $ do not simultaneously appear
in each factor.

We now define the \emph{one-tailed} $\mathcal{A}_{2}^{\alpha }$ constant
using $\mathcal{P}^{\alpha }$. The energy constants $\mathcal{E}_{\alpha }$
introduced in the next subsection will use the standard Poisson integral $%
\mathrm{P}^{\alpha }$.

\begin{definition}
The one-sided constants $\mathcal{A}_{2}^{\alpha }$ and $\mathcal{A}%
_{2}^{\alpha ,\ast }$ for the weight pair $\left( \sigma ,\omega \right) $
are given by%
\begin{eqnarray*}
\mathcal{A}_{2}^{\alpha } &\equiv &\sup_{Q\in \Omega \mathcal{P}^{n}}%
\mathcal{P}^{\alpha }\left( Q,\mathbf{1}_{Q^{c}}\sigma \right) \frac{%
\left\vert Q\right\vert _{\omega }}{\left\vert Q\right\vert ^{1-\frac{\alpha 
}{n}}}<\infty , \\
\mathcal{A}_{2}^{\alpha ,\ast } &\equiv &\sup_{Q\in \Omega \mathcal{P}^{n}}%
\mathcal{P}^{\alpha }\left( Q,\mathbf{1}_{Q^{c}}\omega \right) \frac{%
\left\vert Q\right\vert _{\sigma }}{\left\vert Q\right\vert ^{1-\frac{\alpha 
}{n}}}<\infty .
\end{eqnarray*}
\end{definition}

Note that these definitions are the analogues of the corresponding
conditions with `holes' introduced by Hyt\"{o}nen \cite{Hyt} in dimension $%
n=1$ - the supports of the measures $\mathbf{1}_{Q^{c}}\sigma $ and $\mathbf{%
1}_{Q}\omega $ in the definition of $\mathcal{A}_{2}^{\alpha }$ are
disjoint, and so the common point masses of $\sigma $ and $\omega $ do not
appear simultaneously in each factor. Note also that, unlike in \cite%
{SaShUr5}, where common point masses were not permitted, we can no longer
assert the equivalence of $\mathcal{A}_{2}^{\alpha }$ with holes taken over 
\emph{quasicubes} with $\mathcal{A}_{2}^{\alpha }$ with holes taken over 
\emph{cubes}.

\subsubsection{Punctured $A_{2}^{\protect\alpha }$ conditions}

As mentioned earlier, the classical $A_{2}^{\alpha }$ condition%
\begin{equation*}
A_{2}^{\alpha }\left( \sigma ,\omega \right) \equiv \sup_{Q\in \Omega 
\mathcal{P}^{n}}\frac{\left\vert Q\right\vert _{\omega }}{\left\vert
Q\right\vert ^{1-\frac{\alpha }{n}}}\frac{\left\vert Q\right\vert _{\sigma }%
}{\left\vert Q\right\vert ^{1-\frac{\alpha }{n}}}
\end{equation*}%
fails to be finite when the measures $\sigma $ and $\omega $ have a common
point mass - simply let $Q$ in the $\sup $ above shrink to a common mass
point. But there is a substitute that is quite similar in character that is
motivated by the fact that for large quasicubes $Q$, the $\sup $ above is
problematic only if just \emph{one} of the measures is \emph{mostly} a point
mass when restricted to $Q$. The one-dimensional version of the condition we
are about to describe arose in Conjecture 1.12 of Lacey \cite{Lac2}, and it
was pointed out in \cite{Hyt2} that its necessity on the line follows from
the proof of Proposition 2.1 in \cite{LaSaUr2}. We now extend this condition
to higher dimensions, where its necessity is more subtle.

Given an at most countable set $\mathfrak{P}=\left\{ p_{k}\right\}
_{k=1}^{\infty }$ in $\mathbb{R}^{n}$, a quasicube $Q\in \Omega \mathcal{Q}%
^{n}$, and a positive locally finite Borel measure $\mu $, define 
\begin{equation*}
\mu \left( Q,\mathfrak{P}\right) \equiv \left\vert Q\right\vert _{\mu }-\sup
\left\{ \mu \left( p_{k}\right) :p_{k}\in Q\cap \mathfrak{P}\right\} ,
\end{equation*}%
where the supremum is actually achieved since $\sum_{p_{k}\in Q\cap 
\mathfrak{P}}\mu \left( p_{k}\right) <\infty $ as $\mu $ is locally finite.
The quantity $\mu \left( Q,\mathfrak{P}\right) $ is simply the $\widetilde{%
\mu }$ measure of $Q$ where $\widetilde{\mu }$ is the measure $\mu $ with
its largest point mass from $\mathfrak{P}$ in $Q$ removed. Given a locally
finite measure pair $\left( \sigma ,\omega \right) $, let $\mathfrak{P}%
_{\left( \sigma ,\omega \right) }=\left\{ p_{k}\right\} _{k=1}^{\infty }$ be
the at most countable set of common point masses of $\sigma $ and $\omega $.
Then the weighted norm inequality (\ref{2 weight}) typically implies
finiteness of the following \emph{punctured} Muckenhoupt conditions:%
\begin{eqnarray*}
A_{2}^{\alpha ,\limfunc{punct}}\left( \sigma ,\omega \right) &\equiv
&\sup_{Q\in \Omega \mathcal{P}^{n}}\frac{\omega \left( Q,\mathfrak{P}%
_{\left( \sigma ,\omega \right) }\right) }{\left\vert Q\right\vert ^{1-\frac{%
\alpha }{n}}}\frac{\left\vert Q\right\vert _{\sigma }}{\left\vert
Q\right\vert ^{1-\frac{\alpha }{n}}}, \\
A_{2}^{\alpha ,\ast ,\limfunc{punct}}\left( \sigma ,\omega \right) &\equiv
&\sup_{Q\in \Omega \mathcal{P}^{n}}\frac{\left\vert Q\right\vert _{\omega }}{%
\left\vert Q\right\vert ^{1-\frac{\alpha }{n}}}\frac{\sigma \left( Q,%
\mathfrak{P}_{\left( \sigma ,\omega \right) }\right) }{\left\vert
Q\right\vert ^{1-\frac{\alpha }{n}}}.
\end{eqnarray*}

\begin{lemma}
\label{pointed A2}Let $\mathbf{T}^{\alpha }$ be an $\alpha $-fractional
singular integral operator as above, and suppose that there is a positive
constant $C_{0}$ such that%
\begin{equation*}
\sqrt{A_{2}^{\alpha }\left( \sigma ,\omega \right) }\leq C_{0}\mathfrak{N}_{%
\mathbf{T}^{\alpha }}\left( \sigma ,\omega \right) ,
\end{equation*}%
for all pairs of positive locally finite measures \textbf{having no common
point masses}. Now let $\sigma $ and $\omega $ be positive locally finite
Borel measures on $\mathbb{R}^{n}$ and let $\mathfrak{P}_{\left( \sigma
,\omega \right) }$ be the possibly nonempty set of common point masses. Then
we have%
\begin{equation*}
A_{2}^{\alpha ,\limfunc{punct}}\left( \sigma ,\omega \right) +A_{2}^{\alpha
,\ast ,\limfunc{punct}}\left( \sigma ,\omega \right) \leq 4C_{0}\mathfrak{N}%
_{\mathbf{T}^{\alpha }}^{2}\left( \sigma ,\omega \right) .
\end{equation*}
\end{lemma}

\begin{proof}
Fix a quasicube $Q\in \Omega \mathcal{P}^{n}$. Suppose first that $\mathfrak{%
P}_{\left( \sigma ,\omega \right) }\cap Q=\left\{ p_{k}\right\} _{k=1}^{2N}$
is finite. Choose $k_{1}\in \mathbb{N}_{2N}=\left\{ 1,2,...,2N\right\} $ so
that 
\begin{equation*}
\sigma \left( p_{k_{1}}\right) =\max_{k\in \mathbb{N}_{2N}}\sigma \left(
p_{k}\right) .
\end{equation*}%
Then choose $k_{2}\in \mathbb{N}_{2N}\setminus \left\{ k_{1}\right\} $ such
that 
\begin{equation*}
\omega \left( p_{k_{2}}\right) =\max_{k\in \mathbb{N}_{2N}\setminus \left\{
k_{1}\right\} }\omega \left( p_{k}\right) .
\end{equation*}%
Repeat this procedure so that%
\begin{eqnarray*}
\sigma \left( p_{k_{2m+1}}\right) &=&\max_{k\in \mathbb{N}_{2N}\setminus
\left\{ k_{1},...k_{2m}\right\} }\sigma \left( p_{k}\right) ,\ \ \ \ \
k_{2m+1}\in \mathbb{N}_{2N}\setminus \left\{ k_{1},...k_{2m}\right\} , \\
\omega \left( p_{k_{2m+2}}\right) &=&\max_{k\in \mathbb{N}_{2N}\setminus
\left\{ k_{1},...k_{2m+1}\right\} }\omega \left( p_{k}\right) ,\ \ \ \ \
k_{2m+2}\in \mathbb{N}_{2N}\setminus \left\{ k_{1},...k_{2m+1}\right\} ,
\end{eqnarray*}%
for each $m\leq N-1$. It is now clear that both%
\begin{eqnarray*}
\sum_{i=0}^{N-1}\sigma \left( p_{k_{2i+1}}\right) &\geq &\frac{1}{2}%
\sum_{k=1}^{2N}\sigma \left( p_{k}\right) =\frac{1}{2}\sigma \left( Q\cap 
\mathfrak{P}_{\left( \sigma ,\omega \right) }\right) , \\
\sum_{i=0}^{N-1}\omega \left( p_{k_{2i+2}}\right) &\geq &\frac{1}{2}%
\sum_{k=2}^{2N}\omega \left( p_{k}\right) =\frac{1}{2}\left[ \omega \left(
Q\cap \mathfrak{P}_{\left( \sigma ,\omega \right) }\right) -\omega \left(
p_{1}\right) \right] .
\end{eqnarray*}

Now define new measures $\widetilde{\sigma }$ and $\widetilde{\omega }$ by%
\begin{equation*}
\widetilde{\sigma }\equiv \mathbf{1}_{Q}\sigma -\sum_{i=0}^{N-1}\sigma
\left( p_{k_{2i+2}}\right) \delta _{p_{k_{2i+2}}}\text{ and }\widetilde{%
\omega }=\mathbf{1}_{Q}\omega -\sum_{i=0}^{N-1}\omega \left(
p_{k_{2i+1}}\right) \delta _{p_{k_{2i+1}}}
\end{equation*}%
so that%
\begin{eqnarray*}
\left\vert Q\right\vert _{\widetilde{\sigma }} &=&\left\vert Q\right\vert
_{\sigma }-\sum_{i=0}^{N-1}\sigma \left( p_{k_{2i+2}}\right) =\left\vert
Q\right\vert _{\sigma }-\sigma \left( Q\cap \mathfrak{P}_{\left( \sigma
,\omega \right) }\right) +\sum_{i=0}^{N-1}\sigma \left( p_{k_{2i+1}}\right)
\\
&\geq &\left\vert Q\right\vert _{\sigma }-\sigma \left( Q\cap \mathfrak{P}%
_{\left( \sigma ,\omega \right) }\right) +\frac{1}{2}\sigma \left( Q\cap 
\mathfrak{P}_{\left( \sigma ,\omega \right) }\right) \geq \frac{1}{2}%
\left\vert Q\right\vert _{\sigma }
\end{eqnarray*}%
and%
\begin{eqnarray*}
\left\vert Q\right\vert _{\widetilde{\omega }} &=&\left\vert Q\right\vert
_{\omega }-\sum_{i=0}^{N-1}\omega \left( p_{k_{2i+1}}\right) =\left\vert
Q\right\vert _{\omega }-\omega \left( Q\cap \mathfrak{P}_{\left( \sigma
,\omega \right) }\right) +\sum_{i=0}^{N-1}\omega \left( p_{k_{2i+2}}\right)
\\
&\geq &\left\vert Q\right\vert _{\omega }-\omega \left( Q\cap \mathfrak{P}%
_{\left( \sigma ,\omega \right) }\right) +\frac{1}{2}\left[ \omega \left(
Q\cap \mathfrak{P}_{\left( \sigma ,\omega \right) }\right) -\omega \left(
p_{1}\right) \right] \geq \frac{1}{2}\omega \left( Q,\mathfrak{P}_{\left(
\sigma ,\omega \right) }\right) .
\end{eqnarray*}%
Now $\widetilde{\sigma }$ and $\widetilde{\omega }$ have no common point
masses and $\mathfrak{N}_{\mathbf{T}^{\alpha }}\left( \sigma ,\omega \right) 
$ is monotone in each measure separately, so we have%
\begin{eqnarray*}
\frac{\omega \left( Q,\mathfrak{P}_{\left( \sigma ,\omega \right) }\right) }{%
\left\vert Q\right\vert ^{1-\frac{\alpha }{n}}}\frac{\left\vert Q\right\vert
_{\sigma }}{\left\vert Q\right\vert ^{1-\frac{\alpha }{n}}} &\leq &4\frac{%
\left\vert Q\right\vert _{\widetilde{\omega }}}{\left\vert Q\right\vert ^{1-%
\frac{\alpha }{n}}}\frac{\left\vert Q\right\vert _{\widetilde{\sigma }}}{%
\left\vert Q\right\vert ^{1-\frac{\alpha }{n}}} \\
&\leq &4A_{2}^{\alpha }\left( \widetilde{\sigma },\widetilde{\omega }\right)
\\
&\leq &4C_{0}\mathfrak{N}_{\mathbf{T}^{\alpha }}^{2}\left( \widetilde{\sigma 
},\widetilde{\omega }\right) \leq 4C_{0}\mathfrak{N}_{\mathbf{T}^{\alpha
}}^{2}\left( \sigma ,\omega \right) .
\end{eqnarray*}%
Now take the supremum over $Q\in \Omega \mathcal{Q}^{n}$ to conclude that $%
A_{2}^{\alpha ,\limfunc{punct}}\left( \sigma ,\omega \right) \leq 4C_{0}%
\mathfrak{N}_{\mathbf{T}^{\alpha }}^{2}\left( \sigma ,\omega \right) $ if
the number of common point masses in $Q$ is finite. A limiting argument
proves the general case. The dual inequality $A_{2}^{\alpha ,\ast ,\limfunc{%
punct}}\left( \sigma ,\omega \right) \leq 4C_{0}\mathfrak{N}_{\mathbf{T}%
^{\alpha }}^{2}\left( \sigma ,\omega \right) $ now follows upon
interchanging the measures $\sigma $ and $\omega $.
\end{proof}

Now we turn to the definition of a quasiHaar basis of $L^{2}\left( \mu
\right) $.

\subsection{A weighted quasiHaar basis}

Recall we have fixed a globally biLipschitz map $\Omega :\mathbb{R}%
^{n}\rightarrow \mathbb{R}^{n}$. We will use a construction of a quasiHaar
basis in $\mathbb{R}^{n}$ that is adapted to a measure $\mu $ (c.f. \cite%
{NTV2} for the nonquasi case). Given a dyadic quasicube $Q\in \Omega 
\mathcal{D}$, where $\mathcal{D}$ is a dyadic grid from $\mathcal{P}^{n}$,
let $\bigtriangleup _{Q}^{\mu }$ denote orthogonal projection onto the
finite dimensional subspace $L_{Q}^{2}\left( \mu \right) $ of $L^{2}\left(
\mu \right) $ that consists of linear combinations of the indicators of\ the
children $\mathfrak{C}\left( Q\right) $ of $Q$ that have $\mu $-mean zero
over $Q$:%
\begin{equation*}
L_{Q}^{2}\left( \mu \right) \equiv \left\{ f=\dsum\limits_{Q^{\prime }\in 
\mathfrak{C}\left( Q\right) }a_{Q^{\prime }}\mathbf{1}_{Q^{\prime
}}:a_{Q^{\prime }}\in \mathbb{R},\int_{Q}fd\mu =0\right\} .
\end{equation*}%
Then we have the important telescoping property for dyadic quasicubes $%
Q_{1}\subset Q_{2}$:%
\begin{equation}
\mathbf{1}_{Q_{0}}\left( x\right) \left( \dsum\limits_{Q\in \left[
Q_{1},Q_{2}\right] }\bigtriangleup _{Q}^{\mu }f\left( x\right) \right) =%
\mathbf{1}_{Q_{0}}\left( x\right) \left( \mathbb{E}_{Q_{0}}^{\mu }f-\mathbb{E%
}_{Q_{2}}^{\mu }f\right) ,\ \ \ \ \ Q_{0}\in \mathfrak{C}\left( Q_{1}\right)
,\ f\in L^{2}\left( \mu \right) .  \label{telescope}
\end{equation}%
We will at times find it convenient to use a fixed orthonormal basis $%
\left\{ h_{Q}^{\mu ,a}\right\} _{a\in \Gamma _{n}}$ of $L_{Q}^{2}\left( \mu
\right) $ where $\Gamma _{n}\equiv \left\{ 0,1\right\} ^{n}\setminus \left\{ 
\mathbf{1}\right\} $ is a convenient index set with $\mathbf{1}=\left(
1,1,...,1\right) $. Then $\left\{ h_{Q}^{\mu ,a}\right\} _{a\in \Gamma _{n}%
\text{ and }Q\in \Omega \mathcal{D}}$ is an orthonormal basis for $%
L^{2}\left( \mu \right) $, with the understanding that we add the constant
function $\mathbf{1}$ if $\mu $ is a finite measure. In particular we have%
\begin{equation*}
\left\Vert f\right\Vert _{L^{2}\left( \mu \right) }^{2}=\sum_{Q\in \Omega 
\mathcal{D}}\left\Vert \bigtriangleup _{Q}^{\mu }f\right\Vert _{L^{2}\left(
\mu \right) }^{2}=\sum_{Q\in \Omega \mathcal{D}}\sum_{a\in \Gamma
_{n}}\left\vert \widehat{f}\left( Q\right) \right\vert ^{2},
\end{equation*}%
where%
\begin{equation*}
\left\vert \widehat{f}\left( Q\right) \right\vert ^{2}\equiv \sum_{a\in
\Gamma _{n}}\left\vert \left\langle f,h_{Q}^{\mu ,a}\right\rangle _{\mu
}\right\vert ^{2},
\end{equation*}%
and the measure is suppressed in the notation. Indeed, this follows from (%
\ref{telescope}) and Lebesgue's differentiation theorem for quasicubes. We
also record the following useful estimate. If $I^{\prime }$ is any of the $%
2^{n}$ $\Omega \mathcal{D}$-children of $I$, and $a\in \Gamma _{n}$, then 
\begin{equation}
\left\vert \mathbb{E}_{I^{\prime }}^{\mu }h_{I}^{\mu ,a}\right\vert \leq 
\sqrt{\mathbb{E}_{I^{\prime }}^{\mu }\left( h_{I}^{\mu ,a}\right) ^{2}}\leq 
\frac{1}{\sqrt{\left\vert I^{\prime }\right\vert _{\mu }}}.
\label{useful Haar}
\end{equation}

\subsection{The strong quasienergy conditions}

Given a dyadic quasicube $K\in \Omega \mathcal{D}$ and a positive measure $%
\mu $ we define the quasiHaar projection $\mathsf{P}_{K}^{\mu }\equiv
\sum_{_{J\in \Omega \mathcal{D}:\ J\subset K}}\bigtriangleup _{J}^{\mu }$ on 
$K$ by 
\begin{equation*}
\mathsf{P}_{K}^{\mu }f=\sum_{_{J\in \Omega \mathcal{D}:\ J\subset
K}}\dsum\limits_{a\in \Gamma _{n}}\left\langle f,h_{J}^{\mu ,a}\right\rangle
_{\mu }h_{J}^{\mu ,a}\text{ and }\left\Vert \mathsf{P}_{K}^{\mu
}f\right\Vert _{L^{2}\left( \mu \right) }^{2}=\sum_{_{J\in \Omega \mathcal{D}%
:\ J\subset K}}\dsum\limits_{a\in \Gamma _{n}}\left\vert \left\langle
f,h_{J}^{\mu ,a}\right\rangle _{\mu }\right\vert ^{2},
\end{equation*}%
and where a quasiHaar basis $\left\{ h_{J}^{\mu ,a}\right\} _{a\in \Gamma
_{n}\text{ and }J\in \Omega \mathcal{D}\Omega }$ adapted to the measure $\mu 
$ was defined in the subsubsection on a weighted quasiHaar basis above.

Now we define various notions for quasicubes which are inherited from the
same notions for cubes. The main objective here is to use the familiar
notation that one uses for cubes, but now extended to $\Omega $-quasicubes.
We have already introduced the notions of quasigrids $\Omega \mathcal{D}$,
and center, sidelength and dyadic associated to quasicubes $Q\in \Omega 
\mathcal{D}$, as well as quasiHaar functions, and we will continue to extend
to quasicubes the additional familiar notions related to cubes as we come
across them. We begin with the notion of \emph{deeply embedded}. Fix a
quasigrid $\Omega \mathcal{D}$, $\mathbf{r}\in \mathbb{N}$ and $%
0<\varepsilon <1$. We say that a dyadic quasicube $J$ is $\left( \mathbf{r}%
,\varepsilon \right) $-\emph{deeply embedded} in a (not necessarily dyadic)
quasicube $K$, which we write as $J\Subset _{\mathbf{r},\varepsilon }K$,
when $J\subset K$ and both 
\begin{eqnarray}
\ell \left( J\right) &\leq &2^{-\mathbf{r}}\ell \left( K\right) ,
\label{def deep embed} \\
\limfunc{qdist}\left( J,\partial K\right) &\geq &\frac{1}{2}\ell \left(
J\right) ^{\varepsilon }\ell \left( K\right) ^{1-\varepsilon },  \notag
\end{eqnarray}%
where we define the quasidistance $\limfunc{qdist}\left( E,F\right) $
between two sets $E$ and $F$ to be the Euclidean distance $\limfunc{dist}%
\left( \Omega ^{-1}E,\Omega ^{-1}F\right) $ between the preimages $\Omega
^{-1}E$ and $\Omega ^{-1}F$ of $E$ and $F$ under the map $\Omega $, and
where we recall that $\ell \left( J\right) \approx \left\vert J\right\vert ^{%
\frac{1}{n}}$. For the most part we will consider $J\Subset _{\mathbf{r}%
,\varepsilon }K$ when $J$ and $K$ belong to a common quasigrid $\Omega 
\mathcal{D}$, but an exception is made when defining the strong energy
constants below.

Recall that in dimension $n=1$, and for $\alpha =0$, the energy condition
constant was defined by%
\begin{equation*}
\mathcal{E}^{2}\equiv \sup_{I=\dot{\cup}I_{r}}\frac{1}{\left\vert
I\right\vert _{\sigma }}\sum_{r=1}^{\infty }\left( \frac{\mathrm{P}^{\alpha
}\left( I_{r},\mathbf{1}_{I}\sigma \right) }{\left\vert I_{r}\right\vert }%
\right) ^{2}\left\Vert \mathsf{P}_{I_{r}}^{\omega }\mathbf{x}\right\Vert
_{L^{2}\left( \omega \right) }^{2}\ ,
\end{equation*}%
where $I$, $I_{r}$ and $J$ are intervals in the real line. The extension to
higher dimensions we use here is that of `strong quasienergy condition'
below. Later on, in the proof of the theorem, we will break down this strong
quasienergy condition into various smaller quasienergy conditions, which are
then used in different ways in the proof.

We define a quasicube $K$ (not necessarily in $\Omega \mathcal{D}$) to be an 
\emph{alternate} $\Omega \mathcal{D}$-quasicube if it is a union of $2^{n}$ $%
\Omega \mathcal{D}$-quasicubes $K^{\prime }$ with side length $\ell \left(
K^{\prime }\right) =\frac{1}{2}\ell \left( K\right) $ (such quasicubes were
called shifted in \cite{SaShUr5}, but that terminology conflicts with the
more familiar notion of shifted quasigrid). Thus for any $\Omega \mathcal{D}$%
-quasicube $L$ there are exactly $2^{n}$ alternate $\Omega \mathcal{D}$%
-quasicubes of twice the side length that contain $L$, and one of them is of
course the $\Omega \mathcal{D}$-parent of $L$. We denote the collection of
alternate $\Omega \mathcal{D}$-quasicubes by $\mathcal{A}\Omega \mathcal{D}$.

The extension of the energy conditions to higher dimensions in \cite{SaShUr5}
used the collection 
\begin{equation*}
\mathcal{M}_{\mathbf{r},\varepsilon -\limfunc{deep}}\left( K\right) \equiv
\left\{ \text{maximal }J\Subset _{\mathbf{r},\varepsilon }K\right\}
\end{equation*}%
of \emph{maximal} $\left( \mathbf{r},\varepsilon \right) $-deeply embedded
dyadic subquasicubes of a quasicube $K$ (a subquasicube $J$ of $K$ is a 
\emph{dyadic} subquasicube of $K$ if $J\in \Omega \mathcal{D}$ when $\Omega 
\mathcal{D}$ is a dyadic quasigrid containing $K$). This collection of
dyadic subquasicubes of $K$ is of course a pairwise disjoint decomposition
of $K$. We also defined there a refinement and extension of the collection $%
\mathcal{M}_{\left( \mathbf{r},\varepsilon \right) -\limfunc{deep}}\left(
K\right) $ for certain $K$ and each $\ell \geq 1$. For an alternate
quasicube $K\in \mathcal{A}\Omega \mathcal{D}$, define $\mathcal{M}_{\left( 
\mathbf{r},\varepsilon \right) -\limfunc{deep}}\left( K\right) $ to consist
of the \emph{maximal} $\mathbf{r}$-deeply embedded $\Omega \mathcal{D}$%
-dyadic subquasicubes $J$ of $K$. (In the special case that $K$ itself
belongs to $\Omega \mathcal{D}$, then these definitions coincide.) Then in 
\cite{SaShUr5} for $\ell \geq 1$ we defined for $K\in \mathcal{A}\Omega 
\mathcal{D}$ the refinement%
\begin{eqnarray*}
\mathcal{M}_{\left( \mathbf{r},\varepsilon \right) -\limfunc{deep},\Omega 
\mathcal{D}}^{\ell }\left( K\right) &\equiv &\left\{ J\in \mathcal{M}%
_{\left( \mathbf{r},\varepsilon \right) -\limfunc{deep}}\left( \pi ^{\ell
}K^{\prime }\right) \text{ for some }K^{\prime }\in \mathfrak{C}\left(
K\right) :\right. \\
&&\ \ \ \ \ \ \ \ \ \ \ \ \ \ \ \ \ \ \ \ \ \ \ \ \ \ \ \ \ \ \left.
J\subset L\text{ for some }L\in \mathcal{M}_{\left( \mathbf{r},\varepsilon
\right) -\limfunc{deep}}\left( K\right) \right\} ,
\end{eqnarray*}%
where $\mathfrak{C}\left( K\right) $ is the obvious extension to alternate
quasicubes $K$ of the set of $\Omega \mathcal{D}$-dyadic children of a
dyadic quasicube. Thus $\mathcal{M}_{\left( \mathbf{r},\varepsilon \right) -%
\limfunc{deep}}^{\ell }\left( K\right) $ is the union, over all
quasichildren $K^{\prime }$ of $K$, of those quasicubes in $\mathcal{M}%
_{\left( \mathbf{r},\varepsilon \right) -\limfunc{deep}}\left( \pi ^{\ell
}K^{\prime }\right) $ that happen to be contained in some $L\in \mathcal{M}%
_{\left( \mathbf{r},\varepsilon \right) -\limfunc{deep}}\left( K\right) $.
Note that $\mathcal{M}_{\left( \mathbf{r},\varepsilon \right) -\limfunc{deep}%
}^{0}\left( K\right) $ is in general different than $\mathcal{M}_{\left( 
\mathbf{r},\varepsilon \right) -\limfunc{deep}}\left( K\right) $. We then
define the \emph{strong} quasienergy condition as follows.

\begin{definition}
Let $0\leq \alpha <n$ and fix parameters $\left( \mathbf{r},\varepsilon
\right) $. Suppose $\sigma $ and $\omega $ are positive Borel measures on $%
\mathbb{R}^{n}$ possibly with common point masses. Then the \emph{strong}
quasienergy constant $\mathcal{E}_{\alpha }^{\limfunc{strong}}$ is defined
by 
\begin{eqnarray*}
\left( \mathcal{E}_{\alpha }^{\limfunc{strong}}\right) ^{2} &\equiv &\sup_{I=%
\dot{\cup}I_{r}}\frac{1}{\left\vert I\right\vert _{\sigma }}%
\sum_{r=1}^{\infty }\sum_{J\in \mathcal{M}_{\mathbf{r},\varepsilon -\limfunc{%
deep}}\left( I_{r}\right) }\left( \frac{\mathrm{P}^{\alpha }\left( J,\mathbf{%
1}_{I}\sigma \right) }{\left\vert J\right\vert ^{\frac{1}{n}}}\right)
^{2}\left\Vert \mathsf{P}_{J}^{\omega }\mathbf{x}\right\Vert _{L^{2}\left(
\omega \right) }^{2} \\
&&+\sup_{\Omega \mathcal{D}}\sup_{I\in \mathcal{A}\Omega \mathcal{D}%
}\sup_{\ell \geq 0}\frac{1}{\left\vert I\right\vert _{\sigma }}\sum_{J\in 
\mathcal{M}_{\left( \mathbf{r},\varepsilon \right) -\limfunc{deep},\Omega 
\mathcal{D}}^{\ell }\left( I\right) }\left( \frac{\mathrm{P}^{\alpha }\left(
J,\mathbf{1}_{I}\sigma \right) }{\left\vert J\right\vert ^{\frac{1}{n}}}%
\right) ^{2}\left\Vert \mathsf{P}_{J}^{\omega }\mathbf{x}\right\Vert
_{L^{2}\left( \omega \right) }^{2}\ .
\end{eqnarray*}
\end{definition}

Similarly we have a dual version of $\mathcal{E}_{\alpha }^{\limfunc{strong}%
} $ denoted $\mathcal{E}_{\alpha }^{\limfunc{strong},\ast }$, and both
depend on $\mathbf{r}$ and $\varepsilon $ as well as on $n$ and $\alpha $.
An important point in this definition is that the quasicube $I$ in the
second line is permitted to lie \emph{outside} the quasigrid $\Omega 
\mathcal{D}$, but only as an alternate dyadic quasicube $I\in \mathcal{A}%
\Omega \mathcal{D} $. In the setting of quasicubes we continue to use the
linear function $\mathbf{x}$ in the final factor $\left\Vert \mathsf{P}%
_{J}^{\omega }\mathbf{x}\right\Vert _{L^{2}\left( \omega \right) }^{2}$ of
each line, and not the pushforward of $\mathbf{x}$ by $\Omega $. The reason
of course is that this condition is used to capture the first order
information in the Taylor expansion of a singular kernel. There is a
logically weaker form of the quasienergy conditions that we discuss after
stating our main theorem, but these refined quasienergy conditions are more
complicated to state, and have as yet found no application - the strong
energy conditions above suffice for use when one measure is compactly
supported on a $C^{1,\delta }$ curve as in \cite{SaShUr8}.

\subsection{Statement of the Theorem}

We can now state our main quasicube two weight theorem for general measures
allowing common point masses. Recall that $\Omega :\mathbb{R}^{n}\rightarrow 
\mathbb{R}^{n}$ is a globally biLipschitz map, and that $\Omega \mathcal{P}%
^{n}$ denotes the collection of all quasicubes in $\mathbb{R}^{n}$ whose
preimages under $\Omega $ are usual cubes with sides parallel to the
coordinate axes. Denote by $\Omega \mathcal{D}\subset \Omega \mathcal{P}^{n}$
a dyadic quasigrid in $\mathbb{R}^{n}$. For the purpose of obtaining
necessity of $\mathcal{A}_{2}^{\alpha }$ in the range $\frac{n}{2}\leq
\alpha <n$, we adapt to the setting of quasicubes the notion of strong
ellipticity from \cite{SaShUr5}.

\begin{definition}
\label{strong ell}Fix a globally biLipschitz map $\Omega $. Let $\mathbf{T}%
^{\alpha }=\left\{ T_{j}^{\alpha }\right\} _{j=1}^{J}$ be a vector of Calder%
\'{o}n-Zygmund operators with standard kernels $\left\{ K_{j}^{\alpha
}\right\} _{j=1}^{J}$. We say that $\mathbf{T}^{\alpha }$ is \emph{strongly
elliptic} with respect to $\Omega $ if for each $m\in \left\{ 1,-1\right\}
^{n}$, there is a sequence of coefficients $\left\{ \lambda _{j}^{m}\right\}
_{j=1}^{J}$ such that%
\begin{equation}
\left\vert \sum_{j=1}^{J}\lambda _{j}^{m}K_{j}^{\alpha }\left( x,x+t\mathbf{u%
}\right) \right\vert \geq ct^{\alpha -n},\ \ \ \ \ t\in \mathbb{R},
\label{Ktalpha strong}
\end{equation}%
holds for \emph{all} unit vectors $\mathbf{u}$ in the quasi-$n$-ant $\Omega
V_{m}$ where%
\begin{equation*}
V_{m}=\left\{ x\in \mathbb{R}^{n}:m_{i}x_{i}>0\text{ for }1\leq i\leq
n\right\} ,\ \ \ \ \ m\in \left\{ 1,-1\right\} ^{n}.
\end{equation*}
\end{definition}

\begin{theorem}
\label{T1 theorem}Suppose that $T^{\alpha }$ is a standard $\alpha $%
-fractional Calder\'{o}n-Zygmund operator on $\mathbb{R}^{n}$, and that $%
\omega $ and $\sigma $ are positive Borel measures on $\mathbb{R}^{n}$
(possibly having common point masses). Set $T_{\sigma }^{\alpha }f=T^{\alpha
}\left( f\sigma \right) $ for any smooth truncation of $T_{\sigma }^{\alpha
} $. Let $\Omega :\mathbb{R}^{n}\rightarrow \mathbb{R}^{n}$ be a globally
biLipschitz map.

\begin{enumerate}
\item Suppose $0\leq \alpha <n$. Then the operator $T_{\sigma }^{\alpha }$
is bounded from $L^{2}\left( \sigma \right) $ to $L^{2}\left( \omega \right) 
$, i.e. 
\begin{equation}
\left\Vert T_{\sigma }^{\alpha }f\right\Vert _{L^{2}\left( \omega \right)
}\leq \mathfrak{N}_{T_{\sigma }^{\alpha }}\left\Vert f\right\Vert
_{L^{2}\left( \sigma \right) },  \label{two weight}
\end{equation}%
uniformly in smooth truncations of $T^{\alpha }$, and moreover%
\begin{equation*}
\mathfrak{N}_{T_{\sigma }^{\alpha }}\leq C_{\alpha }\left( \sqrt{\mathcal{A}%
_{2}^{\alpha }+\mathcal{A}_{2}^{\alpha ,\ast }+A_{2}^{\alpha ,\limfunc{punct}%
}+A_{2}^{\alpha ,\ast ,\limfunc{punct}}}+\mathfrak{T}_{T^{\alpha }}+%
\mathfrak{T}_{T^{\alpha }}^{\ast }+\mathcal{E}_{\alpha }^{\limfunc{strong}}+%
\mathcal{E}_{\alpha }^{\limfunc{strong},\ast }+\mathcal{WBP}_{T^{\alpha
}}\right) ,
\end{equation*}%
provided that the two dual $\mathcal{A}_{2}^{\alpha }$ conditions and the
two dual punctured Muckenhoupt conditions all hold, and the two dual
quasitesting conditions for $T^{\alpha }$ hold, the quasiweak boundedness
property for $T^{\alpha }$ holds for a sufficiently large constant $C$
depending on the goodness parameter $\mathbf{r}$, and provided that the two
dual strong quasienergy conditions hold uniformly over all dyadic quasigrids 
$\Omega \mathcal{D}\subset \Omega \mathcal{P}^{n}$, i.e. $\mathcal{E}%
_{\alpha }^{\limfunc{strong}}+\mathcal{E}_{\alpha }^{\limfunc{strong},\ast
}<\infty $, and where the goodness parameters $\mathbf{r}$ and $\varepsilon $
implicit in the definition of the collections $\mathcal{M}_{\left( \mathbf{r}%
,\varepsilon \right) -\limfunc{deep}}\left( K\right) $ and $\mathcal{M}%
_{\left( \mathbf{r},\varepsilon \right) -\limfunc{deep},\Omega \mathcal{D}%
}^{\ell }\left( K\right) $ appearing in the strong energy conditions, are
fixed sufficiently large and small respectively depending only on $n$ and $%
\alpha $.

\item Conversely, suppose $0\leq \alpha <n$ and that $\mathbf{T}^{\alpha
}=\left\{ T_{j}^{\alpha }\right\} _{j=1}^{J}$ is a vector of Calder\'{o}%
n-Zygmund operators with standard kernels $\left\{ K_{j}^{\alpha }\right\}
_{j=1}^{J}$. In the range $0\leq \alpha <\frac{n}{2}$, we assume the \emph{%
ellipticity} condition from (\cite{SaShUr5}): there is $c>0$ such that for 
\emph{each} unit vector $\mathbf{u}$ there is $j$ satisfying%
\begin{equation}
\left\vert K_{j}^{\alpha }\left( x,x+t\mathbf{u}\right) \right\vert \geq
ct^{\alpha -n},\ \ \ \ \ t\in \mathbb{R}.  \label{Ktalpha}
\end{equation}%
For the range $\frac{n}{2}\leq \alpha <n$, we assume the strong ellipticity
condition in Definition \ref{strong ell} above. Furthermore, assume that
each operator $T_{j}^{\alpha }$ is bounded from $L^{2}\left( \sigma \right) $
to $L^{2}\left( \omega \right) $, 
\begin{equation*}
\left\Vert \left( T_{j}^{\alpha }\right) _{\sigma }f\right\Vert
_{L^{2}\left( \omega \right) }\leq \mathfrak{N}_{T_{j}^{\alpha }}\left\Vert
f\right\Vert _{L^{2}\left( \sigma \right) }.
\end{equation*}%
Then the fractional $\mathcal{A}_{2}^{\alpha }$ conditions (with `holes')
hold as well as the punctured Muckenhoupt conditions, and moreover,%
\begin{equation*}
\sqrt{\mathcal{A}_{2}^{\alpha }+\mathcal{A}_{2}^{\alpha ,\ast
}+A_{2}^{\alpha ,\limfunc{punct}}+A_{2}^{\alpha ,\ast ,\limfunc{punct}}}\leq
C\mathfrak{N}_{\mathbf{T}^{\alpha }}.
\end{equation*}
\end{enumerate}
\end{theorem}

\begin{problem}
Given any strongly elliptic vector $\mathbf{T}^{\alpha }$ of classical $%
\alpha $-fractional Calder\'{o}n-Zygmund operators, it is an open question
whether or not the usual (quasi or not) energy conditions are necessary for
boundedness of $\mathbf{T}^{\alpha }$. See \cite{SaShUr4} for a failure of 
\emph{energy reversal} in higher dimensions - such an energy reversal was
used in dimension $n=1$ to prove the necessity of the energy condition for
the Hilbert transform, and also in \cite{SaShUr3} and \cite{LaSaShUrWi} for
the Riesz transforms and Cauchy transforms respectively when one of the
measures is supported on a line, and in \cite{SaShUr8} when one of the
measures is supported on a $C^{1,\delta }$ curve.
\end{problem}

\begin{remark}
If Definition \ref{strong ell} holds for some $\mathbf{T}^{\alpha }$ and $%
\Omega $, then $\Omega $ must be fairly tame, in particular the logarithmic
spirals in Example \ref{wild} are ruled out! On the other hand, the vector
of Riesz transforms $\mathbf{R}^{\alpha ,n}$ is easily seen to be strongly
elliptic with respect to $\Omega $ if $\Omega $ satisfies the following 
\emph{sector separation property}. Given a hyperplane $H$ and a
perpendicular line $L$ intersecting at point $P$, there exist spherical
cones $S_{H}$ and $S_{L}$ intersecting only at the point $P^{\prime }=\Omega
\left( P\right) $, such that $H^{\prime }\equiv \Omega H\subset S_{H}$ and $%
L^{\prime }\equiv \Omega L\subset S_{L}$ and%
\begin{eqnarray*}
\limfunc{dist}\left( x,\partial S_{H}\right) &\approx &\left\vert
x\right\vert ,\ \ \ \ \ x\in H\ , \\
\limfunc{dist}\left( x,\partial S_{L}\right) &\approx &\left\vert
x\right\vert ,\ \ \ \ \ x\in L\ .
\end{eqnarray*}%
Examples of globally biLipshcitz maps $\Omega $ that satisfy the sector
separation property include finite compositions of maps of the form%
\begin{equation*}
\Omega \left( x_{1},x^{\prime }\right) =\left( x_{1},x^{\prime }+\psi \left(
x_{1}\right) \right) ,\ \ \ \ \ \left( x_{1},x^{\prime }\right) \in \mathbb{R%
}^{n},
\end{equation*}%
where $\psi :\mathbb{R}\rightarrow \mathbb{R}^{n-1}$ is a Lipschitz map with
sufficiently small Lipschitz constant.
\end{remark}

\begin{remark}
\label{surgery}In \cite{LaWi}, in the setting of usual (nonquasi) cubes and
measures having no common point masses, M. Lacey and B. Wick use the NTV
technique of surgery to show that the weak boundedness property for the
Riesz transform vector $\mathbf{R}^{\alpha ,n}$ is implied by the $\mathcal{A%
}_{2}^{\alpha }$ and testing conditions, and this has the consequence of
eliminating the weak boundedness property as a condition. Their proof of
this implication extends to the more general operators $T^{\alpha }$ and
quasicubes considered here, and so the quasiweak boundedness property can be
dropped from the statement of Theorem \ref{T1 theorem}. In any event, the
weak boundedness property is necessary for the norm inequality, and as such
can be viewed as a weak close cousin of the testing conditions.
\end{remark}

\section{Proof of Theorem \protect\ref{T1 theorem}}

We now give the proof of Theorem \ref{T1 theorem} in the following sections.
Sections 5, 7 and 10 are largely taken verbatim from the corresponding
sections of \cite{SaShUr5}, but are included here since their omission here
would hinder the readability of an already complicated argument.

\subsection{Good quasicubes and quasienergy conditions}

Here we extend the notion of goodness to quasicubes and define various
refinements of the strong quasienergy conditions appearing in the main
theorem above. These refinements represent the `weakest' energy side
conditions that suffice for use in our proof. We begin with goodness.

\begin{definition}
Let $\mathbf{r}\in \mathbb{N}$ and $0<\varepsilon <1$. Fix a quasigrid $%
\Omega \mathcal{D}$. A dyadic quasicube $J$ is $\left( \mathbf{r}%
,\varepsilon \right) $\emph{-good}, or simply \emph{good}, if for \emph{every%
} dyadic superquasicube $I$, it is the case that \textbf{either} $J$ has
side length greater than $2^{-\mathbf{r}}$ times that of $I$, \textbf{or} $%
J\Subset _{\mathbf{r},\varepsilon }I$ is $\left( \mathbf{r},\varepsilon
\right) $-deeply embedded in $I$.
\end{definition}

Note that this definition simply asserts that a dyadic quasicube $J=\Omega
J^{\prime }$ is $\left( \mathbf{r},\varepsilon \right) $-good if and only if
the cube $J^{\prime }$ is $\left( \mathbf{r},\varepsilon \right) $-good in
the usual sense. Finally, we say that $J$ is $\mathbf{r}$\emph{-nearby} in $%
K $ when $J\subset K$ and%
\begin{equation*}
\ell \left( J\right) >2^{-\mathbf{r}}\ell \left( K\right) .
\end{equation*}%
The parameters $\mathbf{r},\varepsilon $ will be fixed sufficiently large
and small respectively later in the proof, and we denote the set of such
good dyadic quasicubes by $\Omega \mathcal{D}_{\left( \mathbf{r},\varepsilon
\right) -\limfunc{good}}$, or simply $\Omega \mathcal{D}_{\limfunc{good}}$
when the goodness parameters $\left( \mathbf{r},\varepsilon \right) $ are
understood. Note that if $J^{\prime }\in \Omega \mathcal{D}_{\left( \mathbf{r%
},\varepsilon \right) -\limfunc{good}}$ and if $J^{\prime }\subset K\in
\Omega \mathcal{D}$, then \textbf{either} $J^{\prime }$ is $\mathbf{r}$\emph{%
-nearby} in $K$ \textbf{or} $J^{\prime }\subset J\Subset _{\mathbf{r}%
,\varepsilon }K$.

Throughout the proof, it will be convenient to also consider pairs of
quasicubes $J,K$ where $J$ is $\left( \mathbf{\rho },\varepsilon \right) $%
\emph{-deeply embedded} in $K$, written $J\Subset _{\mathbf{\rho }%
,\varepsilon }K$ and meaning (\ref{def deep embed}) holds with the same $%
\varepsilon >0$ but with $\mathbf{\rho }$ in place of $\mathbf{r}$; as well
as pairs of quasicubes $J,K$ where $J$ is $\mathbf{\rho }$\emph{-nearby} in%
\emph{\ }$K$, $\ell \left( J\right) >2^{-\mathbf{\rho }}\ell \left( K\right) 
$, for a parameter $\mathbf{\rho }\gg \mathbf{r}$ that will be fixed later.
We define the smaller `good' quasiHaar projection $\mathsf{P}_{K}^{\limfunc{%
good},\omega }=\mathsf{P}_{K}^{\left( \mathbf{r},\varepsilon \right) -%
\limfunc{good},\omega }$ by%
\begin{equation*}
\mathsf{P}_{K}^{\limfunc{good},\mu }f\equiv \sum_{_{J\in \mathcal{G}\left(
K\right) }}\bigtriangleup _{J}^{\mu }f=\sum_{_{J\in \mathcal{G}\left(
K\right) }}\dsum\limits_{a\in \Gamma _{n}}\left\langle f,h_{J}^{\mu
,a}\right\rangle _{\mu }h_{J}^{\mu ,a},
\end{equation*}%
where $\mathcal{G}\left( K\right) $ consists of the good subcubes of $K$:%
\begin{equation*}
\mathcal{G}\left( K\right) \equiv \left\{ J\in \Omega \mathcal{D}_{\limfunc{%
good}}:J\subset K\right\} ,
\end{equation*}%
and also the larger `subgood' quasiHaar projection $\mathsf{P}_{K}^{\limfunc{%
subgood},\mu }$ by%
\begin{equation*}
\mathsf{P}_{K}^{\limfunc{subgood},\mu }f\equiv \sum_{_{J\in \mathcal{M}_{%
\limfunc{good}}\left( K\right) }}\sum_{J^{\prime }\subset J}\bigtriangleup
_{J^{\prime }}^{\mu }f=\sum_{_{J\in \mathcal{M}_{\limfunc{good}}\left(
K\right) }}\sum_{J^{\prime }\subset J}\dsum\limits_{a\in \Gamma
_{n}}\left\langle f,h_{J^{\prime }}^{\mu ,a}\right\rangle _{\mu
}h_{J^{\prime }}^{\mu ,a},
\end{equation*}%
where $\mathcal{M}_{\limfunc{good}}\left( K\right) $ consists of the \emph{%
maximal} good subcubes of $K$. We thus have 
\begin{eqnarray*}
\left\Vert \mathsf{P}_{K}^{\limfunc{good},\mu }\mathbf{x}\right\Vert
_{L^{2}\left( \mu \right) }^{2} &\leq &\left\Vert \mathsf{P}_{K}^{\limfunc{%
subgood},\mu }\mathbf{x}\right\Vert _{L^{2}\left( \mu \right) }^{2} \\
&\leq &\left\Vert \mathsf{P}_{I}^{\mu }\mathbf{x}\right\Vert _{L^{2}\left(
\mu \right) }^{2}=\int_{I}\left\vert \mathbf{x}-\left( \frac{1}{\left\vert
I\right\vert _{\mu }}\int_{I}\mathbf{x}d\mu \right) \right\vert ^{2}d\mu
\left( x\right) ,\ \ \ \ \ \mathbf{x}=\left( x_{1},...,x_{n}\right) ,
\end{eqnarray*}%
where $\mathsf{P}_{I}^{\mu }\mathbf{x}$ is the orthogonal projection of the
identity function $\mathbf{x}:\mathbb{R}^{n}\rightarrow \mathbb{R}^{n}$ onto
the vector-valued subspace of $\oplus _{k=1}^{n}L^{2}\left( \mu \right) $
consisting of functions supported in $I$ with $\mu $-mean value zero.

Recall that the extension of the energy conditions to higher dimensions in 
\cite{SaShUr5} used the collection 
\begin{equation*}
\mathcal{M}_{\left( \mathbf{r},\varepsilon \right) -\limfunc{deep}}\left(
K\right) \equiv \left\{ \text{maximal }J\Subset _{\mathbf{r},\varepsilon
}K\right\}
\end{equation*}%
of \emph{maximal} $\mathbf{r}$-deeply embedded dyadic subquasicubes of a
quasicube $K$ (a subquasicube $J$ of $K$ is a \emph{dyadic} subquasicube of $%
K$ if $J\in \Omega \mathcal{D}$ when $\Omega \mathcal{D}$ is a dyadic
quasigrid containing $K$). We let $J^{\ast }=\gamma J$ where $\gamma \geq 2$%
. Then the following bounded overlap property holds.

\begin{lemma}
\begin{equation}
\sum_{J\in \mathcal{M}_{\left( \mathbf{r},\varepsilon \right) -\limfunc{deep}%
}\left( K\right) }\mathbf{1}_{J^{\ast }}\leq \beta \mathbf{1}%
_{\dbigcup\limits_{J\in \mathcal{M}_{\left( \mathbf{r},\varepsilon \right) -%
\limfunc{deep}}\left( K\right) }J^{\ast }}  \label{bounded overlap}
\end{equation}%
holds for some positive constant $\beta $ depending only on $n,\gamma ,%
\mathbf{r}$ and $\varepsilon $. If in addition we have $\gamma \leq 2^{%
\mathbf{r}\left( 1-\varepsilon \right) }$, then $\gamma J\subset K$ for all $%
J\in \mathcal{M}_{\left( \mathbf{r},\varepsilon \right) -\limfunc{deep}%
}\left( K\right) $, and consequently%
\begin{equation}
\sum_{J\in \mathcal{M}_{\left( \mathbf{r},\varepsilon \right) -\limfunc{deep}%
}\left( K\right) }\mathbf{1}_{J^{\ast }}\leq \beta \mathbf{1}_{K}\ .
\label{bounded overlap in K}
\end{equation}
\end{lemma}

\begin{proof}
To prove (\ref{bounded overlap}), we first note that there are at most $2^{n%
\mathbf{r}+1}$ quasicubes $J$ for which $\ell \left( J\right) >2^{-\mathbf{r}%
}\ell \left( K\right) $. On the other hand, the maximal $\mathbf{r}$-deeply
embedded subquasicubes $J$ of $K$ also satisfy the comparability condition%
\begin{equation*}
\frac{1}{2}\ell \left( J\right) ^{\varepsilon }\ell \left( K\right)
^{1-\varepsilon }\leq \limfunc{qdist}\left( J,K^{c}\right) \leq C_{n}\left(
2^{\mathbf{r}}\ell \left( J\right) \right) ^{\varepsilon }\ell \left(
K\right) ^{1-\varepsilon }.
\end{equation*}%
Now with $0<\varepsilon <1$ and $\gamma \geq 2$ fixed, let $y\in K$. Then if 
$y\in \gamma J$, we have%
\begin{eqnarray*}
\frac{1}{2}\ell \left( J\right) ^{\varepsilon }\ell \left( K\right)
^{1-\varepsilon } &\leq &\limfunc{qdist}\left( J,K^{c}\right) \leq \gamma
\ell \left( J\right) +\limfunc{qdist}\left( \gamma J,K^{c}\right) \\
&\leq &\gamma \ell \left( J\right) +\limfunc{qdist}\left( y,K^{c}\right) .
\end{eqnarray*}%
Now assume that $\frac{\ell \left( J\right) }{\ell \left( K\right) }\leq
\left( \frac{1}{4\gamma }\right) ^{\frac{1}{1-\varepsilon }}$. Then we have $%
\gamma \ell \left( J\right) \leq \frac{1}{4}\ell \left( J\right)
^{\varepsilon }\ell \left( K\right) ^{1-\varepsilon }$ and so 
\begin{equation*}
\frac{1}{4}\ell \left( J\right) ^{\varepsilon }\ell \left( K\right)
^{1-\varepsilon }\leq \limfunc{qdist}\left( y,K^{c}\right) .
\end{equation*}%
But we also have 
\begin{equation*}
\limfunc{qdist}\left( y,K^{c}\right) \leq \ell \left( J\right) +dist\left(
J,K^{c}\right) \leq \ell \left( J\right) +C_{n}2^{\mathbf{r}\varepsilon
}\ell \left( J\right) ^{\varepsilon }\ell \left( K\right) ^{1-\varepsilon
}\leq \left( \frac{1}{4\gamma }+C_{n}2^{\mathbf{r}\varepsilon }\right) \ell
\left( J\right) ^{\varepsilon }\ell \left( K\right) ^{1-\varepsilon },
\end{equation*}%
and so altogether, under the assumption that $\frac{\ell \left( J\right) }{%
\ell \left( K\right) }\leq \left( \frac{1}{4\gamma }\right) ^{\frac{1}{%
1-\varepsilon }}$, we have%
\begin{eqnarray*}
\frac{1}{\frac{1}{4\gamma }+C_{n}2^{\mathbf{r}\varepsilon }}\limfunc{qdist}%
\left( y,K^{c}\right) &\leq &\ell \left( J\right) ^{\varepsilon }\ell \left(
K\right) ^{1-\varepsilon }\leq 2\limfunc{qdist}\left( y,K^{c}\right) , \\
\text{i.e. }\left( \frac{1}{\frac{1}{4\gamma }+C_{n}2^{\mathbf{r}\varepsilon
}}\frac{\limfunc{qdist}\left( y,K^{c}\right) }{\ell \left( K\right)
^{1-\varepsilon }}\right) ^{\frac{1}{\varepsilon }} &\leq &\ell \left(
J\right) \leq \left( 2\frac{\limfunc{qdist}\left( y,K^{c}\right) }{\ell
\left( K\right) ^{1-\varepsilon }}\right) ^{\frac{1}{\varepsilon }},
\end{eqnarray*}%
which shows that the number of $J^{\prime }s$ satisfying $y\in \gamma J$ and 
$\frac{\ell \left( J\right) }{\ell \left( K\right) }\leq \left( \frac{1}{%
4\gamma }\right) ^{\frac{1}{1-\varepsilon }}$ is at most $C_{n}^{\prime
}\gamma ^{n}\frac{1}{\varepsilon }\log _{2}\left( \frac{1}{2\gamma }%
+2C_{n}2^{\mathbf{r}\varepsilon }\right) $. The number of $J^{\prime }s$
satisfying $y\in \gamma J$ and $\frac{\ell \left( J\right) }{\ell \left(
K\right) }>\left( \frac{1}{4\gamma }\right) ^{\frac{1}{1-\varepsilon }}$ is
at most $C_{n}^{\prime }\gamma ^{n}\frac{1}{1-\varepsilon }\log _{2}\left(
4\gamma \right) $. This proves (\ref{bounded overlap}) with 
\begin{eqnarray*}
\beta &=&2^{n\mathbf{r}+1}+C_{n}^{\prime }\gamma ^{n}\frac{1}{\varepsilon }%
\log _{2}\left( \frac{1}{2\gamma }+2C_{n}2^{\mathbf{r}\varepsilon }\right)
+C_{n}^{\prime }\gamma ^{n}\frac{1}{1-\varepsilon }\log _{2}\left( 4\gamma
\right) \\
&\leq &2^{n\mathbf{r}+1}+C_{n}^{\prime \prime }\frac{1}{\varepsilon }\left(
1+\mathbf{r}\varepsilon \right) \gamma ^{n}+C_{n}^{\prime \prime }\frac{1}{%
1-\varepsilon }\gamma ^{n}\left( 1+\log _{2}\gamma \right) .
\end{eqnarray*}

In order to prove (\ref{bounded overlap in K}) it suffices, by (\ref{bounded
overlap}), to prove $\gamma J\subset K$ for all $J\in \mathcal{M}_{\left( 
\mathbf{r},\varepsilon \right) -\limfunc{deep}}\left( K\right) $. But $J\in 
\mathcal{M}_{\left( \mathbf{r},\varepsilon \right) -\limfunc{deep}}\left(
K\right) $ implies%
\begin{equation*}
\frac{1}{2}\ell \left( J\right) ^{\varepsilon }\ell \left( K\right)
^{1-\varepsilon }\leq \limfunc{qdist}\left( J,K^{c}\right) =\limfunc{qdist}%
\left( c_{J},K^{c}\right) +\frac{1}{2}\ell \left( J\right) .
\end{equation*}%
We wish to show $\gamma J\subset K$, which is equivalent to 
\begin{equation*}
\gamma \frac{1}{2}\ell \left( J\right) \leq \limfunc{qdist}\left(
c_{J},K^{c}\right) =\limfunc{qdist}\left( J,K^{c}\right) -\frac{1}{2}\ell
\left( J\right) .
\end{equation*}%
But we have%
\begin{equation*}
\limfunc{qdist}\left( J,K^{c}\right) -\frac{1}{2}\ell \left( J\right) \geq 
\frac{1}{2}\ell \left( J\right) ^{\varepsilon }\ell \left( K\right)
^{1-\varepsilon }-\frac{1}{2}\ell \left( J\right) ,
\end{equation*}%
and so it suffices to show that%
\begin{equation*}
\frac{1}{2}\ell \left( J\right) ^{\varepsilon }\ell \left( K\right)
^{1-\varepsilon }-\frac{1}{2}\ell \left( J\right) \geq \gamma \frac{1}{2}%
\ell \left( J\right) ,
\end{equation*}%
which is equivalent to%
\begin{equation*}
\gamma +1\leq \ell \left( J\right) ^{\varepsilon -1}\ell \left( K\right)
^{1-\varepsilon }.
\end{equation*}%
But the smallest that $\ell \left( J\right) ^{\varepsilon -1}\ell \left(
K\right) ^{1-\varepsilon }$ can get for $J\in \mathcal{M}_{\left( \mathbf{r}%
,\varepsilon \right) -\limfunc{deep}}\left( K\right) $ is $2^{\mathbf{r}%
\left( 1-\varepsilon \right) }$, and this completes the proof.
\end{proof}

\begin{remark}
The parameter $\gamma $ only plays a role in defining various refinements of
the strong energy condition in our main theorem. We will use these
refinements throughout our proof, both to highlight the nature of what
precisely is required in different situations, and also to provide the
`weakest' energy conditions under which our main theorem holds, namely under
the assumptions $\mathcal{E}_{\alpha }^{\limfunc{deep}},\mathcal{E}_{\alpha
}^{\limfunc{refined}}<\infty $ with $2\leq \gamma \leq 2^{\mathbf{r}\left(
1-\varepsilon \right) }$, and the corresponding dual conditions.
\end{remark}

Recall the collection $\mathcal{M}_{\left( \mathbf{r},\varepsilon \right) -%
\limfunc{deep},\Omega \mathcal{D}}^{\ell }\left( K\right) $, which is the
union, over all quasichildren $K^{\prime }$ of $K$, of those quasicubes in $%
\mathcal{M}_{\left( \mathbf{r},\varepsilon \right) -\limfunc{deep}}\left(
\pi ^{\ell }K^{\prime }\right) $ that happen to be contained in some $L\in 
\mathcal{M}_{\left( \mathbf{r},\varepsilon \right) -\limfunc{deep},\Omega 
\mathcal{D}}\left( K\right) $. Since $J\in \mathcal{M}_{\left( \mathbf{r}%
,\varepsilon \right) -\limfunc{deep},\Omega \mathcal{D}}^{\ell }\left(
K\right) $ implies $\gamma J\subset K$, we also have from (\ref{bounded
overlap}) and (\ref{bounded overlap in K}) that%
\begin{equation}
\sum_{J\in \mathcal{M}_{\left( \mathbf{r},\varepsilon \right) -\limfunc{deep}%
,\Omega \mathcal{D}}^{\left( \ell \right) }\left( K\right) }\mathbf{1}%
_{J^{\ast }}\leq \beta \mathbf{1}_{\dbigcup\limits_{J\in \mathcal{M}_{\left( 
\mathbf{r},\varepsilon \right) -\limfunc{deep},\Omega \mathcal{D}}^{\left(
\ell \right) }\left( K\right) }J^{\ast }}\text{ and }\sum_{J\in \mathcal{M}%
_{\left( \mathbf{r},\varepsilon \right) -\limfunc{deep},\Omega \mathcal{D}%
}^{\left( \ell \right) }\left( K\right) }\mathbf{1}_{J^{\ast }\cap K}\leq
\beta \mathbf{1}_{K}\ ,\ \ \ \ \ \text{for each }\ell \geq 0,
\label{bounded overlap'}
\end{equation}%
and if $\gamma \leq 2^{\left( 1-\varepsilon \right) \mathbf{r}}$, then%
\begin{equation}
\sum_{J\in \mathcal{M}_{\left( \mathbf{r},\varepsilon \right) -\limfunc{deep}%
,\Omega \mathcal{D}}^{\left( \ell \right) }\left( K\right) }\mathbf{1}%
_{J^{\ast }}\leq \beta \mathbf{1}_{K}\ ,\ \ \ \ \ \text{for each }\ell \geq
0.  \label{bounded overlap in K'}
\end{equation}%
Of course $\mathcal{M}_{\left( \mathbf{r},\varepsilon \right) -\limfunc{deep}%
,\Omega \mathcal{D}}^{1}\left( K\right) \supset \mathcal{M}_{\left( \mathbf{r%
},\varepsilon \right) -\limfunc{deep}}\left( K\right) $, but $\mathcal{M}%
_{\left( \mathbf{r},\varepsilon \right) -\limfunc{deep},\Omega \mathcal{D}%
}^{\ell }\left( K\right) $ is in general a finer subdecomposition of $K$ the
larger $\ell $ is, and may in fact be empty.

Now we proceed to recall the definition of various quasienergy conditions
from \cite{SaShUr5}, but with the additional natural restriction that $%
\gamma \leq 2^{\mathbf{r}\left( 1-\varepsilon \right) }$. The point of this
restriction is that the bounded overlap conditions (\ref{bounded overlap in
K}) and (\ref{bounded overlap in K'}) can be used effectively later on.

\begin{definition}
\label{energy condition}Suppose $\sigma $ and $\omega $ are positive Borel
measures on $\mathbb{R}^{n}$ (possibly with common point masses) and fix $%
\gamma \geq 2$. Then the\ deep quasienergy condition constant $\mathcal{E}%
_{\alpha }^{\limfunc{deep}}$ and the refined quasienergy condition constant $%
\mathcal{E}_{\alpha }^{\limfunc{refined}}$ are given by%
\begin{eqnarray*}
\left( \mathcal{E}_{\alpha }^{\limfunc{deep}}\right) ^{2} &\equiv &\sup_{I=%
\dot{\cup}I_{r}}\frac{1}{\left\vert I\right\vert _{\sigma }}%
\sum_{r=1}^{\infty }\sum_{J\in \mathcal{M}_{\left( \mathbf{r},\varepsilon
\right) -\limfunc{deep}}\left( I_{r}\right) }\left( \frac{\mathrm{P}^{\alpha
}\left( J,\mathbf{1}_{I\setminus \gamma J}\sigma \right) }{\left\vert
J\right\vert ^{\frac{1}{n}}}\right) ^{2}\left\Vert \mathsf{P}_{J}^{\limfunc{%
subgood},\omega }\mathbf{x}\right\Vert _{L^{2}\left( \omega \right) }^{2}, \\
\left( \mathcal{E}_{\alpha }^{\limfunc{refined}}\right) ^{2} &\equiv
&\sup_{\Omega \mathcal{D}}\sup_{I\in \mathcal{A}\Omega \mathcal{D}%
}\sup_{\ell \geq 0}\frac{1}{\left\vert I\right\vert _{\sigma }}\sum_{J\in 
\mathcal{M}_{\left( \mathbf{r},\varepsilon \right) -\limfunc{deep},\Omega 
\mathcal{D}}^{\ell }\left( I\right) }\left( \frac{\mathrm{P}^{\alpha }\left(
J,\mathbf{1}_{I\setminus \gamma J}\sigma \right) }{\left\vert J\right\vert ^{%
\frac{1}{n}}}\right) ^{2}\left\Vert \mathsf{P}_{J}^{\limfunc{subgood},\omega
}\mathbf{x}\right\Vert _{L^{2}\left( \omega \right) }^{2}\ ,
\end{eqnarray*}%
where $\sup_{\Omega \mathcal{D}}\sup_{I}$ in the second line is taken over
all quasigrids $\Omega \mathcal{D}$ and alternate quasicubes $I\in \mathcal{A%
}\Omega \mathcal{D}$, and $\sup_{I=\dot{\cup}I_{r}}$ in the first line is
taken over

\begin{enumerate}
\item all dyadic quasigrids $\Omega \mathcal{D}$,

\item all $\Omega \mathcal{D}$-dyadic quasicubes $I$,

\item and all subpartitions $\left\{ I_{r}\right\} _{r=1}^{N\text{ or }%
\infty }$ of the quasicube $I$ into $\Omega \mathcal{D}$-dyadic
subquasicubes $I_{r}$.
\end{enumerate}
\end{definition}

Note that in the refined quasienergy conditions there is no outer
decomposition $I=\dot{\cup}I_{r}$. There are similar definitions for the
dual (backward) quasienergy conditions that simply interchange $\sigma $ and 
$\omega $ everywhere. These definitions of\ the quasienergy conditions
depend on the choice of $\gamma $ and the goodness parameters $\mathbf{r}$
and $\varepsilon $. Note that we can `partially' plug the $\gamma $-hole in
the Poisson integral $\mathrm{P}^{\alpha }\left( J,\mathbf{1}_{I\setminus
\gamma J}\sigma \right) $ for both $\mathcal{E}_{\alpha }^{\limfunc{deep}}$
and $\mathcal{E}_{\alpha }^{\limfunc{refined}}$ using the offset $%
A_{2}^{\alpha }$ condition and the bounded overlap property (\ref{bounded
overlap in K'}). Indeed, define 
\begin{eqnarray}
&&  \label{plug} \\
&&\left( \mathcal{E}_{\alpha }^{\limfunc{deep}\limfunc{partial}}\right)
^{2}\equiv \sup_{I=\dot{\cup}I_{r}}\frac{1}{\left\vert I\right\vert _{\sigma
}}\sum_{r=1}^{\infty }\sum_{J\in \mathcal{M}_{\left( \mathbf{r},\varepsilon
\right) -\limfunc{deep}}\left( I_{r}\right) }\left( \frac{\mathrm{P}^{\alpha
}\left( J,\mathbf{1}_{I\setminus J}\sigma \right) }{\left\vert J\right\vert
^{\frac{1}{n}}}\right) ^{2}\left\Vert \mathsf{P}_{J}^{\limfunc{subgood}%
,\omega }\mathbf{x}\right\Vert _{L^{2}\left( \omega \right) }^{2}\ ,  \notag
\\
&&\left( \mathcal{E}_{\alpha }^{\limfunc{refined}\limfunc{partial}}\right)
^{2}\equiv \sup_{\Omega \mathcal{D}}\sup_{I}\sup_{\ell \geq 0}\frac{1}{%
\left\vert I\right\vert _{\sigma }}\sum_{J\in \mathcal{M}_{\left( \mathbf{r}%
,\varepsilon \right) -\limfunc{deep},\Omega \mathcal{D}}^{\ell }\left(
I\right) }\left( \frac{\mathrm{P}^{\alpha }\left( J,\mathbf{1}_{I\setminus
J}\sigma \right) }{\left\vert J\right\vert ^{\frac{1}{n}}}\right)
^{2}\left\Vert \mathsf{P}_{J}^{\limfunc{subgood},\omega }\mathbf{x}%
\right\Vert _{L^{2}\left( \omega \right) }^{2}\ .  \notag
\end{eqnarray}%
An easy calculation shows that%
\begin{equation}
\gamma J\subset I_{r}\text{ for all }J\in \mathcal{M}_{\left( \mathbf{r}%
,\varepsilon \right) -\limfunc{deep}}\left( I_{r}\right) \text{ provided }%
\gamma \leq c_{n}2^{\left( 1-\varepsilon \right) \mathbf{r}}.
\label{gamma contained}
\end{equation}%
Thus if $\gamma \leq 2^{\left( 1-\varepsilon \right) \mathbf{r}}$, we have
both%
\begin{eqnarray}
&&\left( \mathcal{E}_{\alpha }^{\limfunc{deep}\limfunc{partial}}\right) ^{2}
\label{plug the hole deep} \\
&\lesssim &\sup_{I=\dot{\cup}I_{r}}\frac{1}{\left\vert I\right\vert _{\sigma
}}\sum_{r=1}^{\infty }\sum_{J\in \mathcal{M}_{\left( \mathbf{r},\varepsilon
\right) -\limfunc{deep}}\left( I_{r}\right) }\left( \frac{\mathrm{P}^{\alpha
}\left( J,\mathbf{1}_{I\setminus \gamma J}\sigma \right) }{\left\vert
J\right\vert ^{\frac{1}{n}}}\right) ^{2}\left\Vert \mathsf{P}_{J}^{\limfunc{%
subgood},\omega }\mathbf{x}\right\Vert _{L^{2}\left( \omega \right) }^{2} 
\notag \\
&&+\sup_{I=\dot{\cup}I_{r}}\frac{1}{\left\vert I\right\vert _{\sigma }}%
\sum_{r=1}^{\infty }\sum_{J\in \mathcal{M}_{\left( \mathbf{r},\varepsilon
\right) -\limfunc{deep}}\left( I_{r}\right) }\left( \frac{\mathrm{P}^{\alpha
}\left( J,\mathbf{1}_{\gamma J\setminus J}\sigma \right) }{\left\vert
J\right\vert ^{\frac{1}{n}}}\right) ^{2}\left\Vert \mathsf{P}_{J}^{\limfunc{%
subgood},\omega }\mathbf{x}\right\Vert _{L^{2}\left( \omega \right) }^{2} 
\notag \\
&\lesssim &\left( \mathcal{E}_{\alpha }^{\limfunc{deep}}\right) ^{2}+\sup_{I=%
\dot{\cup}I_{r}}\frac{1}{\left\vert I\right\vert _{\sigma }}%
\sum_{r=1}^{\infty }\sum_{J\in \mathcal{M}_{\left( \mathbf{r},\varepsilon
\right) -\limfunc{deep}}\left( I_{r}\right) }\left( \frac{\left\vert \gamma
J\setminus J\right\vert _{\sigma }}{\left\vert J\right\vert ^{1+\frac{1}{n}-%
\frac{\alpha }{n}}}\right) ^{2}\left\vert J\right\vert ^{\frac{2}{n}%
}\left\vert J\right\vert _{\omega }  \notag \\
&\lesssim &\left( \mathcal{E}_{\alpha }^{\limfunc{deep}}\right)
^{2}+A_{2}^{\alpha }\sup_{I=\dot{\cup}I_{r}}\frac{1}{\left\vert I\right\vert
_{\sigma }}\sum_{r=1}^{\infty }\sum_{J\in \mathcal{M}_{\left( \mathbf{r}%
,\varepsilon \right) -\limfunc{deep}}\left( I_{r}\right) }\left\vert \gamma
J\right\vert _{\sigma }\lesssim \left( \mathcal{E}_{\alpha }^{\limfunc{deep}%
}\right) ^{2}+\beta A_{2}^{\alpha }\ ,  \notag
\end{eqnarray}%
and similarly 
\begin{equation}
\left( \mathcal{E}_{\alpha }^{\limfunc{refined}\limfunc{partial}}\right)
^{2}\lesssim \left( \mathcal{E}_{\alpha }^{\limfunc{refined}}\right)
^{2}+\beta A_{2}^{\alpha }\text{ }.  \label{plug the hole refined}
\end{equation}%
by (\ref{bounded overlap in K}) and (\ref{bounded overlap in K'})
respectively.

\begin{notation}
As above, we will typically use the side length $\ell \left( J\right) $ of a 
$\Omega $-quasicube when we are describing collections of quasicubes, and
when we want $\ell \left( J\right) $ to be a dyadic or related number; while
in estimates we will typically use $\left\vert J\right\vert ^{\frac{1}{n}%
}\approx \ell \left( J\right) $, and when we want to compare powers of
volumes of quasicubes. We will continue to use the prefix `quasi' when
discussing quasicubes, quasiHaar, quasienergy and quasidistance in the text,
but will not use the prefix `quasi' when discussing other notions. In
particular, since $\limfunc{quasi}\mathcal{A}_{2}^{\alpha }+\limfunc{quasi}%
A_{2}^{\alpha ,\limfunc{punct}}\approx \mathcal{A}_{2}^{\alpha
}+A_{2}^{\alpha ,\limfunc{punct}}$ (see e.g. \cite{SaShUr8} for a proof) we
do not use $\limfunc{quasi}$ as a prefix for the Muckenhoupt conditions,
even though $\limfunc{quasi}\mathcal{A}_{2}^{\alpha }$ alone is not
comparable to$\mathcal{A}_{2}^{\alpha }$. Finally, we will not modify any
mathematical symbols to reflect quasinotions, except for using $\Omega 
\mathcal{D}$ to denote a quasigrid, and $\limfunc{qdist}\left( E,F\right)
\equiv \limfunc{dist}\left( \Omega ^{-1}E,\Omega ^{-1}F\right) $ to denote
quasidistance between sets $E$ and $F$, and using $\left\vert x-y\right\vert
_{\limfunc{qdist}}\equiv \left\vert \Omega ^{-1}x-\Omega ^{-1}y\right\vert $
to denote quasidistance between points $x$ and $y$. This limited use of
quasi in the text serves mainly to remind the reader we are working entirely
in the `quasiworld'.
\end{notation}

\subsubsection{Plugged quasienergy conditions}

We now use the punctured conditions $A_{2}^{\alpha ,\limfunc{punct}}$ and $%
A_{2}^{\alpha ,\ast ,\limfunc{punct}}$ to control the \emph{plugged }%
quasienergy conditions, where the hole in the argument of the Poisson term $%
\mathrm{P}^{\alpha }\left( J,\mathbf{1}_{I\setminus J}\sigma \right) $ in
the partially plugged quasienergy conditions above, is replaced with the
`plugged' term $\mathrm{P}^{\alpha }\left( J,\mathbf{1}_{I}\sigma \right) $.
The resulting plugged quasienergy condition constants will be denoted by%
\begin{equation*}
\mathcal{E}_{\alpha }^{\limfunc{deep}\limfunc{plug}},\mathcal{E}_{\alpha }^{%
\limfunc{refined}\limfunc{plug}}\text{ and }\mathcal{E}_{\alpha }\equiv 
\mathcal{E}_{\alpha }^{\limfunc{deep}}+\mathcal{E}_{\alpha }^{\limfunc{%
refined}},
\end{equation*}%
for example%
\begin{equation}
\left( \mathcal{E}_{\alpha }^{\limfunc{deep}\limfunc{plug}}\right)
^{2}\equiv \sup_{I=\dot{\cup}I_{r}}\frac{1}{\left\vert I\right\vert _{\sigma
}}\sum_{r=1}^{\infty }\sum_{J\in \mathcal{M}_{\left( \mathbf{r},\varepsilon
\right) -\limfunc{deep}}\left( I_{r}\right) }\left( \frac{\mathrm{P}^{\alpha
}\left( J,\mathbf{1}_{I}\sigma \right) }{\left\vert J\right\vert ^{\frac{1}{n%
}}}\right) ^{2}\left\Vert \mathsf{P}_{J}^{\limfunc{subgood},\omega }\mathbf{x%
}\right\Vert _{L^{2}\left( \omega \right) }^{2}\ .  \label{def deep plug}
\end{equation}

We first show that the punctured Muckenhoupt conditions $A_{2}^{\alpha ,%
\limfunc{punct}}$ and $A_{2}^{\alpha ,\ast ,\limfunc{punct}}$ control
respectively the `energy $A_{2}^{\alpha }$ conditions', denoted $%
A_{2}^{\alpha ,\limfunc{energy}}$ and $A_{2}^{\alpha ,\ast ,\limfunc{energy}%
} $ where%
\begin{eqnarray}
A_{2}^{\alpha ,\limfunc{energy}}\left( \sigma ,\omega \right) &\equiv
&\sup_{Q\in \Omega \mathcal{P}^{n}}\frac{\left\Vert \mathsf{P}_{Q}^{\omega }%
\frac{\mathbf{x}}{\ell \left( Q\right) }\right\Vert _{L^{2}\left( \omega
\right) }^{2}}{\left\vert Q\right\vert ^{1-\frac{\alpha }{n}}}\frac{%
\left\vert Q\right\vert _{\sigma }}{\left\vert Q\right\vert ^{1-\frac{\alpha 
}{n}}},  \label{def energy A2} \\
A_{2}^{\alpha ,\ast ,\limfunc{energy}}\left( \sigma ,\omega \right) &\equiv
&\sup_{Q\in \Omega \mathcal{P}^{n}}\frac{\left\vert Q\right\vert _{\omega }}{%
\left\vert Q\right\vert ^{1-\frac{\alpha }{n}}}\frac{\left\Vert \mathsf{P}%
_{Q}^{\sigma }\frac{\mathbf{x}}{\ell \left( Q\right) }\right\Vert
_{L^{2}\left( \sigma \right) }^{2}}{\left\vert Q\right\vert ^{1-\frac{\alpha 
}{n}}}.  \notag
\end{eqnarray}

\begin{lemma}
\label{energy A2}For any positive locally finite Borel measures $\sigma
,\omega $ we have%
\begin{eqnarray*}
A_{2}^{\alpha ,\limfunc{energy}}\left( \sigma ,\omega \right) &\leq &\max
\left\{ n,3\right\} A_{2}^{\alpha ,\limfunc{punct}}\left( \sigma ,\omega
\right) , \\
A_{2}^{\alpha ,\ast ,\limfunc{energy}}\left( \sigma ,\omega \right) &\leq
&\max \left\{ n,3\right\} A_{2}^{\alpha ,\ast ,\limfunc{punct}}\left( \sigma
,\omega \right) .
\end{eqnarray*}
\end{lemma}

\begin{proof}
Fix a quasicube $Q\in \Omega \mathcal{D}$. If $\omega \left( Q,\mathfrak{P}%
_{\left( \sigma ,\omega \right) }\right) \geq \frac{1}{2}\left\vert
Q\right\vert _{\omega }$, then we trivially have%
\begin{eqnarray*}
\frac{\left\Vert \mathsf{P}_{Q}^{\omega }\frac{\mathbf{x}}{\ell \left(
Q\right) }\right\Vert _{L^{2}\left( \omega \right) }^{2}}{\left\vert
Q\right\vert ^{1-\frac{\alpha }{n}}}\frac{\left\vert Q\right\vert _{\sigma }%
}{\left\vert Q\right\vert ^{1-\frac{\alpha }{n}}} &\leq &n\frac{\left\vert
Q\right\vert _{\omega }}{\left\vert Q\right\vert ^{1-\frac{\alpha }{n}}}%
\frac{\left\vert Q\right\vert _{\sigma }}{\left\vert Q\right\vert ^{1-\frac{%
\alpha }{n}}} \\
&\leq &2n\frac{\omega \left( Q,\mathfrak{P}_{\left( \sigma ,\omega \right)
}\right) }{\left\vert Q\right\vert ^{1-\frac{\alpha }{n}}}\frac{\left\vert
Q\right\vert _{\sigma }}{\left\vert Q\right\vert ^{1-\frac{\alpha }{n}}}\leq
2nA_{2}^{\alpha ,\limfunc{punct}}\left( \sigma ,\omega \right) .
\end{eqnarray*}%
On the other hand, if $\omega \left( Q,\mathfrak{P}_{\left( \sigma ,\omega
\right) }\right) <\frac{1}{2}\left\vert Q\right\vert _{\omega }$ then there
is a point $p\in Q\cap \mathfrak{P}_{\left( \sigma ,\omega \right) }$ such
that%
\begin{equation*}
\omega \left( \left\{ p\right\} \right) >\frac{1}{2}\left\vert Q\right\vert
_{\omega }\ ,
\end{equation*}%
and consequently, $p$ is the largest $\omega $-point mass in $Q$. Thus if we
define $\widetilde{\omega }=\omega -\omega \left( \left\{ p\right\} \right)
\delta _{p}$, then we have%
\begin{equation*}
\omega \left( Q,\mathfrak{P}_{\left( \sigma ,\omega \right) }\right)
=\left\vert Q\right\vert _{\widetilde{\omega }}\ .
\end{equation*}%
Now we observe from the construction of Haar projections that%
\begin{equation*}
\bigtriangleup _{J}^{\widetilde{\omega }}=\bigtriangleup _{J}^{\omega },\ \
\ \ \ \text{for all }J\in \Omega \mathcal{D}\text{ with }p\notin J.
\end{equation*}%
So for each $s\geq 0$ there is a unique quasicube $J_{s}\in \Omega \mathcal{D%
}$ with $\ell \left( J_{s}\right) =2^{-s}\ell \left( Q\right) $ that
contains the point $p$. For this quasicube we have, if $\left\{
h_{J}^{\omega ,a}\right\} _{J\in \Omega \mathcal{D},\ a\in \Gamma _{n}}$ is
a basis for $L^{2}\left( \omega \right) $,%
\begin{eqnarray*}
\left\Vert \bigtriangleup _{J_{s}}^{\omega }\mathbf{x}\right\Vert
_{L^{2}\left( \omega \right) }^{2} &=&\sum_{a\in \Gamma _{n}}\left\vert
\left\langle h_{J_{s}}^{\omega ,a},x\right\rangle _{\omega }\right\vert
^{2}=\sum_{a\in \Gamma _{n}}\left\vert \left\langle h_{J_{s}}^{\omega
,a},x-p\right\rangle _{\omega }\right\vert ^{2} \\
&=&\sum_{a\in \Gamma _{n}}\left\vert \int_{J_{s}}h_{J_{s}}^{\omega ,a}\left(
x\right) \left( x-p\right) d\omega \left( x\right) \right\vert
^{2}=\sum_{a\in \Gamma _{n}}\left\vert \int_{J_{s}}h_{J_{s}}^{\omega
,a}\left( x\right) \left( x-p\right) d\widetilde{\omega }\left( x\right)
\right\vert ^{2} \\
&\leq &\sum_{a\in \Gamma _{n}}\left\Vert h_{J_{s}}^{\omega ,a}\right\Vert
_{L^{2}\left( \widetilde{\omega }\right) }^{2}\left\Vert \mathbf{1}%
_{J_{s}}\left( x-p\right) \right\Vert _{L^{2}\left( \widetilde{\omega }%
\right) }^{2}\leq \sum_{a\in \Gamma _{n}}\left\Vert h_{J_{s}}^{\omega
,a}\right\Vert _{L^{2}\left( \omega \right) }^{2}\left\Vert \mathbf{1}%
_{J_{s}}\left( x-p\right) \right\Vert _{L^{2}\left( \widetilde{\omega }%
\right) }^{2} \\
&\leq &n2^{n}\ell \left( J_{s}\right) ^{2}\left\vert J_{s}\right\vert _{%
\widetilde{\omega }}\leq 2^{-2s}\ell \left( Q\right) ^{2}\left\vert
Q\right\vert _{\widetilde{\omega }}.
\end{eqnarray*}%
Thus we can estimate%
\begin{eqnarray*}
\left\Vert \mathsf{P}_{Q}^{\omega }\frac{\mathbf{x}}{\ell \left( Q\right) }%
\right\Vert _{L^{2}\left( \omega \right) }^{2} &=&\frac{1}{\ell \left(
Q\right) ^{2}}\sum_{J\in \Omega \mathcal{D}:\ J\subset Q}\left\Vert
\bigtriangleup _{J}^{\omega }\mathbf{x}\right\Vert _{L^{2}\left( \omega
\right) }^{2} \\
&=&\frac{1}{\ell \left( Q\right) ^{2}}\left( \sum_{J\in \Omega \mathcal{D}:\
p\notin J\subset Q}\left\Vert \bigtriangleup _{J}^{\widetilde{\omega }}%
\mathbf{x}\right\Vert _{L^{2}\left( \widetilde{\omega }\right)
}^{2}+\sum_{s=0}^{\infty }\left\Vert \bigtriangleup _{J_{s}}^{\omega }%
\mathbf{x}\right\Vert _{L^{2}\left( \omega \right) }^{2}\right) \\
&\leq &\frac{1}{\ell \left( Q\right) ^{2}}\left( \left\Vert \mathsf{P}_{Q}^{%
\widetilde{\omega }}\mathbf{x}\right\Vert _{L^{2}\left( \widetilde{\omega }%
\right) }^{2}+\sum_{s=0}^{\infty }2^{-2s}\ell \left( Q\right) ^{2}\left\vert
Q\right\vert _{\widetilde{\omega }}\right) \\
&\leq &\frac{1}{\ell \left( Q\right) ^{2}}\left( \ell \left( Q\right)
^{2}\left\vert Q\right\vert _{\widetilde{\omega }}+\sum_{s=0}^{\infty
}2^{-2s}\ell \left( Q\right) ^{2}\left\vert Q\right\vert _{\widetilde{\omega 
}}\right) \\
&\leq &3\left\vert Q\right\vert _{\widetilde{\omega }}\leq 3\omega \left( Q,%
\mathfrak{P}_{\left( \sigma ,\omega \right) }\right) ,
\end{eqnarray*}%
and so 
\begin{equation*}
\frac{\left\Vert \mathsf{P}_{Q}^{\omega }\frac{\mathbf{x}}{\ell \left(
Q\right) }\right\Vert _{L^{2}\left( \omega \right) }^{2}}{\left\vert
Q\right\vert ^{1-\frac{\alpha }{n}}}\frac{\left\vert Q\right\vert _{\sigma }%
}{\left\vert Q\right\vert ^{1-\frac{\alpha }{n}}}\leq \frac{3\omega \left( Q,%
\mathfrak{P}_{\left( \sigma ,\omega \right) }\right) }{\left\vert
Q\right\vert ^{1-\frac{\alpha }{n}}}\frac{\left\vert Q\right\vert _{\sigma }%
}{\left\vert Q\right\vert ^{1-\frac{\alpha }{n}}}\leq 3A_{2}^{\alpha ,%
\limfunc{punct}}\left( \sigma ,\omega \right) .
\end{equation*}%
Now take the supremum over $Q\in \Omega \mathcal{D}$ to obtain $%
A_{2}^{\alpha ,\limfunc{energy}}\left( \sigma ,\omega \right) \leq \max
\left\{ n,3\right\} A_{2}^{\alpha ,\limfunc{punct}}\left( \sigma ,\omega
\right) $. The dual inequality follows upon interchanging the measures $%
\sigma $ and $\omega $.
\end{proof}

The next corollary follows immediately from Lemma \ref{energy A2} and the
argument used above in (\ref{plug the hole deep}) with (\ref{gamma contained}%
). Define%
\begin{equation*}
\left( \mathcal{E}_{\alpha }^{\limfunc{plug}}\right) ^{2}=\left( \mathcal{E}%
_{\alpha }^{\limfunc{deep}\limfunc{plug}}\right) ^{2}+\left( \mathcal{E}%
_{\alpha }^{\limfunc{refined}\limfunc{plug}}\right) ^{2}.
\end{equation*}

\begin{corollary}
\label{all plugged}Provided $\gamma \leq c_{n}2^{\left( 1-\varepsilon
\right) \mathbf{r}}$, 
\begin{eqnarray*}
\mathcal{E}_{\alpha }^{\limfunc{deep}\limfunc{plug}} &\lesssim &\mathcal{E}%
_{\alpha }^{\limfunc{deep}\limfunc{partial}}+A_{2}^{\alpha ,\limfunc{punct}%
}\lesssim \mathcal{E}_{\alpha }^{\limfunc{deep}}+A_{2}^{\alpha
}+A_{2}^{\alpha ,\limfunc{punct}}, \\
\mathcal{E}_{\alpha }^{\limfunc{refined}\limfunc{plug}} &\lesssim &\mathcal{E%
}_{\alpha }^{\limfunc{refined}\limfunc{partial}}+A_{2}^{\alpha ,\limfunc{%
punct}}\lesssim \mathcal{E}_{\alpha }^{\limfunc{refined}}+A_{2}^{\alpha
}+A_{2}^{\alpha ,\limfunc{punct}}, \\
\mathcal{E}_{\alpha }^{\limfunc{plug}} &\lesssim &\mathcal{E}_{\alpha
}+A_{2}^{\alpha }+A_{2}^{\alpha ,\limfunc{punct}}\ .
\end{eqnarray*}%
and similarly for the dual plugged quasienergy conditions.
\end{corollary}

The two constants $\mathcal{E}_{\alpha }^{\limfunc{deep}\limfunc{plug}}$ and 
$\mathcal{E}_{\alpha }^{\limfunc{refined}\limfunc{plug}}$ , but with a
larger projection $\mathsf{P}_{J}^{\omega }$, are what we bundled together
as the strong energy constant $\mathcal{E}_{\alpha }^{\limfunc{strong}}$ in
the statement of our main result, Theorem \ref{T1 theorem}.

\subsubsection{Plugged $\mathcal{A}_{2}^{\protect\alpha ,\limfunc{energy}%
\limfunc{plug}}$ conditions}

Using Lemma \ref{energy A2} we can control the `plugged' energy $\mathcal{A}%
_{2}^{\alpha }$ conditions:%
\begin{eqnarray*}
\mathcal{A}_{2}^{\alpha ,\limfunc{energy}\limfunc{plug}}\left( \sigma
,\omega \right) &\equiv &\sup_{Q\in \Omega \mathcal{P}^{n}}\frac{\left\Vert 
\mathsf{P}_{Q}^{\omega }\frac{\mathbf{x}}{\ell \left( Q\right) }\right\Vert
_{L^{2}\left( \omega \right) }^{2}}{\left\vert Q\right\vert ^{1-\frac{\alpha 
}{n}}}\mathcal{P}^{\alpha }\left( Q,\sigma \right) , \\
\mathcal{A}_{2}^{\alpha ,\ast ,\limfunc{energy}\limfunc{plug}}\left( \sigma
,\omega \right) &\equiv &\sup_{Q\in \Omega \mathcal{P}^{n}}\mathcal{P}%
^{\alpha }\left( Q,\omega \right) \frac{\left\Vert \mathsf{P}_{Q}^{\sigma }%
\frac{\mathbf{x}}{\ell \left( Q\right) }\right\Vert _{L^{2}\left( \sigma
\right) }^{2}}{\left\vert Q\right\vert ^{1-\frac{\alpha }{n}}}.
\end{eqnarray*}

\begin{lemma}
\label{energy A2 plugged}We have
\end{lemma}

\begin{eqnarray*}
\mathcal{A}_{2}^{\alpha ,\limfunc{energy}\limfunc{plug}}\left( \sigma
,\omega \right) &\mathcal{\lesssim }&\mathcal{A}_{2}^{\alpha }\left( \sigma
,\omega \right) +A_{2}^{\alpha ,\limfunc{energy}}\left( \sigma ,\omega
\right) , \\
\mathcal{A}_{2}^{\alpha ,\ast ,\limfunc{energy}\limfunc{plug}}\left( \sigma
,\omega \right) &\mathcal{\lesssim }&\mathcal{A}_{2}^{\alpha ,\ast }\left(
\sigma ,\omega \right) +A_{2}^{\alpha ,\ast ,\limfunc{energy}}\left( \sigma
,\omega \right) .
\end{eqnarray*}

\begin{proof}
We have%
\begin{eqnarray*}
\frac{\left\Vert \mathsf{P}_{Q}^{\omega }\frac{\mathbf{x}}{\ell \left(
Q\right) }\right\Vert _{L^{2}\left( \omega \right) }^{2}}{\left\vert
Q\right\vert ^{1-\frac{\alpha }{n}}}\mathcal{P}^{\alpha }\left( Q,\sigma
\right) &=&\frac{\left\Vert \mathsf{P}_{Q}^{\omega }\frac{\mathbf{x}}{\ell
\left( Q\right) }\right\Vert _{L^{2}\left( \omega \right) }^{2}}{\left\vert
Q\right\vert ^{1-\frac{\alpha }{n}}}\mathcal{P}^{\alpha }\left( Q,\mathbf{1}%
_{Q^{c}}\sigma \right) +\frac{\left\Vert \mathsf{P}_{Q}^{\omega }\frac{%
\mathbf{x}}{\ell \left( Q\right) }\right\Vert _{L^{2}\left( \omega \right)
}^{2}}{\left\vert Q\right\vert ^{1-\frac{\alpha }{n}}}\mathcal{P}^{\alpha
}\left( Q,\mathbf{1}_{Q}\sigma \right) \\
&\lesssim &\frac{\left\vert Q\right\vert _{\omega }}{\left\vert Q\right\vert
^{1-\frac{\alpha }{n}}}\mathcal{P}^{\alpha }\left( Q,\mathbf{1}%
_{Q^{c}}\sigma \right) +\frac{\left\Vert \mathsf{P}_{Q}^{\omega }\frac{%
\mathbf{x}}{\ell \left( Q\right) }\right\Vert _{L^{2}\left( \omega \right)
}^{2}}{\left\vert Q\right\vert ^{1-\frac{\alpha }{n}}}\frac{\left\vert
Q\right\vert _{\sigma }}{\left\vert Q\right\vert ^{1-\frac{\alpha }{n}}} \\
&\lesssim &\mathcal{A}_{2}^{\alpha }\left( \sigma ,\omega \right)
+A_{2}^{\alpha ,\limfunc{energy}}\left( \sigma ,\omega \right) .
\end{eqnarray*}
\end{proof}

\subsection{Random grids}

Using the analogue for dyadic quasigrids of the good random grids of
Nazarov, Treil and Volberg, a standard argument of NTV, see e.g. \cite{Vol},
reduces the two weight inequality (\ref{2 weight}) for $T^{\alpha }$ to
proving boundedness of a bilinear form $\mathcal{T}^{\alpha }\left(
f,g\right) $ with uniform constants over dyadic quasigrids, and where the
quasiHaar supports $\limfunc{supp}\widehat{f}$ and $\limfunc{supp}\widehat{g}
$ of the functions $f$ and $g$ are contained in the collection $\Omega 
\mathcal{D}^{\limfunc{good}}$ of good quasicubes, whose children are all
good as well, with goodness parameters $\mathbf{r<\infty }$ and $\varepsilon
>0$ chosen sufficiently large and small respectively depending only on $n$
and $\alpha $. Here the quasiHaar support of $f$ is $\limfunc{supp}\widehat{f%
}\equiv \left\{ I\in \Omega \mathcal{D}:\bigtriangleup _{I}^{\sigma }f\neq
0\right\} $, and similarly for $g$. In fact we can assume even more, namely
that the quasiHaar supports $\limfunc{supp}\widehat{f}$ and $\limfunc{supp}%
\widehat{g}$ of $f$ and $g$ are contained in the collection of $\mathbf{\tau 
}$\emph{-good} quasicubes%
\begin{equation}
\Omega \mathcal{D}_{\left( \mathbf{r},\varepsilon \right) -\limfunc{good}}^{%
\mathbf{\tau }}\equiv \left\{ K\in \Omega \mathcal{D}:\mathfrak{C}%
_{K}\subset \Omega \mathcal{D}_{\left( \mathbf{r},\varepsilon \right) -%
\limfunc{good}}\text{ and }\pi _{\Omega \mathcal{D}}^{\ell }K\in \Omega 
\mathcal{D}_{\left( \mathbf{r},\varepsilon \right) -\limfunc{good}}\text{
for all }0\leq \ell \leq \mathbf{\tau }\right\} ,  \label{extended good grid}
\end{equation}%
that are $\left( \mathbf{r},\varepsilon \right) $-$\limfunc{good}$ and whose
children are also $\left( \mathbf{r},\varepsilon \right) $-$\limfunc{good}$,
and whose $\ell $-parents up to level $\mathbf{\tau }$\ are also $\left( 
\mathbf{r},\varepsilon \right) $-$\limfunc{good}$. Here $\mathbf{\tau }>%
\mathbf{r}$ is a parameter to be fixed in Definition \ref{def parameters}
below. We may assume this restriction on the quasiHaar supports of $f$ and $%
g $ by the following lemma.

\begin{lemma}
\label{better good}Given $\mathbf{r}\geq 3$, $\mathbf{\tau }\geq 1$ and $%
\frac{1}{\mathbf{r}}<\varepsilon <1-\frac{1}{\mathbf{r}}$, we have 
\begin{equation*}
\Omega \mathcal{D}_{\left( \mathbf{r}-1,\delta \right) -\limfunc{good}%
}\subset \Omega \mathcal{D}_{\left( \mathbf{r},\varepsilon \right) -\limfunc{%
good}}^{\mathbf{\tau }}\ ,
\end{equation*}%
provided%
\begin{equation}
0<\delta \leq \frac{\mathbf{r}\varepsilon -1}{\mathbf{r}+\mathbf{\tau }}.
\label{choice of delta}
\end{equation}
\end{lemma}

\begin{proof}
Suppose that $I\in \Omega \mathcal{D}_{\left( \mathbf{r}-1,\delta \right) -%
\limfunc{good}}$ where $\delta $ is as in (\ref{choice of delta}). If $J$ is
a child of $I$, then $J\in \Omega \mathcal{D}_{\left( \mathbf{r},\delta
\right) -\limfunc{good}}$, and since $\delta <\varepsilon $, we also have $%
\Omega \mathcal{D}_{\left( \mathbf{r},\delta \right) -\limfunc{good}}\subset
\Omega \mathcal{D}_{\left( \mathbf{r},\varepsilon \right) -\limfunc{good}}$.
It remains to show that $\pi _{\Omega \mathcal{D}}^{\left( m\right) }I\in
\Omega \mathcal{D}_{\left( \mathbf{r},\varepsilon \right) -\limfunc{good}}$
for $0\leq m\leq \mathbf{\tau }$. For this it suffices to show that if $K\in
\Omega \mathcal{D}$ satisfies\thinspace $\pi _{\Omega \mathcal{D}}^{\left(
m\right) }I\subset K$ and $\ell \left( \pi _{\Omega \mathcal{D}}^{\left(
m\right) }I\right) \leq 2^{-\mathbf{r}}\ell \left( K\right) $ , then%
\begin{equation}
\frac{1}{2}\left( \frac{\ell \left( \pi _{\Omega \mathcal{D}}^{\left(
m\right) }I\right) }{\ell \left( K\right) }\right) ^{\varepsilon }\ell
\left( K\right) \leq \limfunc{qdist}\left( \pi _{\Omega \mathcal{D}}^{\left(
m\right) }I,K^{c}\right) .  \label{want eps}
\end{equation}%
Now $\ell \left( I\right) =2^{-m}\ell \left( \pi _{\Omega \mathcal{D}%
}^{\left( m\right) }I\right) \leq 2^{-\left( m+\mathbf{r}\right) }\ell
\left( K\right) \leq 2^{-\left( \mathbf{r}-1\right) }\ell \left( K\right) $
and $I\in \Omega \mathcal{D}_{\left( \mathbf{r}-1,\delta \right) -\limfunc{%
good}}$ imply that%
\begin{equation*}
\frac{1}{2}\left( \frac{\ell \left( I\right) }{\ell \left( K\right) }\right)
^{\delta }\ell \left( K\right) \leq \limfunc{qdist}\left( I,K^{c}\right) ,
\end{equation*}%
and since the triangle inequality gives%
\begin{equation*}
\limfunc{qdist}\left( I,K^{c}\right) \leq \limfunc{qdist}\left( \pi _{\Omega 
\mathcal{D}}^{\left( m\right) }I,K^{c}\right) +2^{m}\ell \left( I\right) ,
\end{equation*}

we see that it suffices to show%
\begin{equation}
\frac{1}{2}\left( \frac{\ell \left( \pi _{\Omega \mathcal{D}}^{\left(
m\right) }I\right) }{\ell \left( K\right) }\right) ^{\varepsilon }\ell
\left( K\right) +2^{m}\ell \left( I\right) \leq \frac{1}{2}\left( \frac{\ell
\left( I\right) }{\ell \left( K\right) }\right) ^{\delta }\ell \left(
K\right) ,\ \ \ \ \ 0\leq m\leq \mathbf{\tau }.  \label{suff delta}
\end{equation}
This is equivalent to successively,%
\begin{eqnarray*}
\frac{1}{2}\left( \frac{2^{m}\ell \left( I\right) }{\ell \left( K\right) }%
\right) ^{\varepsilon }\ell \left( K\right) +2^{m}\ell \left( I\right) &\leq
&\frac{1}{2}\left( \frac{\ell \left( I\right) }{\ell \left( K\right) }%
\right) ^{\delta }\ell \left( K\right) ; \\
\left( \frac{2^{m}\ell \left( I\right) }{\ell \left( K\right) }\right)
^{\varepsilon }+2^{m+1}\frac{\ell \left( I\right) }{\ell \left( K\right) }
&\leq &\left( \frac{\ell \left( I\right) }{\ell \left( K\right) }\right)
^{\delta }; \\
2^{m\varepsilon }\left( \frac{\ell \left( I\right) }{\ell \left( K\right) }%
\right) ^{\varepsilon -\delta }+2^{m+1}\left( \frac{\ell \left( I\right) }{%
\ell \left( K\right) }\right) ^{1-\delta } &\leq &1,\ \ \ \ \ 0\leq m\leq 
\mathbf{\tau }.
\end{eqnarray*}%
Since $0<\delta <\varepsilon <1$ by our restriction on $\varepsilon $ and
our choice of $\delta $ in (\ref{choice of delta}), and since $\frac{\ell
\left( I\right) }{\ell \left( K\right) }\leq 2^{-\left( m+\mathbf{r}\right)
} $, it thus suffices to show that%
\begin{eqnarray*}
2^{m\varepsilon }\left( 2^{-\left( m+\mathbf{r}\right) }\right)
^{\varepsilon -\delta }+2^{m+1}\left( 2^{-\left( m+\mathbf{r}\right)
}\right) ^{1-\delta } &\leq &1, \\
\text{i.e. }2^{m\varepsilon -\left( m+\mathbf{r}\right) \left( \varepsilon
-\delta \right) }+2^{m+1-\left( m+\mathbf{r}\right) \left( 1-\delta \right)
} &\leq &1,
\end{eqnarray*}%
for $0\leq m\leq \mathbf{\tau }$. In particular then it suffices to show both%
\begin{eqnarray*}
m\varepsilon -\left( m+\mathbf{r}\right) \left( \varepsilon -\delta \right)
&\leq &-1, \\
m+1-\left( m+\mathbf{r}\right) \left( 1-\delta \right) &\leq &-1,
\end{eqnarray*}%
equivalently both%
\begin{eqnarray*}
\left( \mathbf{r}+m\right) \delta &\leq &\mathbf{r}\varepsilon -1, \\
\left( \mathbf{r}+m\right) \delta &\leq &\mathbf{r}-2,
\end{eqnarray*}%
for $0\leq m\leq \mathbf{\tau }$. Finally then it suffices to show both%
\begin{equation*}
\delta \leq \frac{\mathbf{r}\varepsilon -1}{\mathbf{r}+\mathbf{\tau }}\text{
and }\delta \leq \frac{\mathbf{r}-2}{\mathbf{r}+\mathbf{\tau }}.
\end{equation*}%
Because of the restriction that $\frac{1}{\mathbf{r}}<\varepsilon <1-\frac{1%
}{\mathbf{r}}$, we see that $0<\mathbf{r}\varepsilon -1<\mathbf{r}-2$, and
it is now clear that the above display holds for our choice of $\delta $ in (%
\ref{choice of delta}).
\end{proof}

For convenience in notation we will sometimes suppress the dependence on $%
\alpha $ in our nonlinear forms, but will retain it in the operators,
Poisson integrals and constants. More precisely, let $\Omega \mathcal{D}%
^{\sigma }=\Omega \mathcal{D}^{\omega }$ be an $\left( \mathbf{r}%
,\varepsilon \right) $-good quasigrid on $\mathbb{R}^{n}$, and let $\left\{
h_{I}^{\sigma ,a}\right\} _{I\in \Omega \mathcal{D}^{\sigma },\ a\in \Gamma
_{n}}$ and $\left\{ h_{J}^{\omega ,b}\right\} _{J\in \Omega \mathcal{D}%
^{\omega },\ b\in \Gamma _{n}}$ be corresponding quasiHaar bases as
described above, so that%
\begin{eqnarray*}
f &=&\sum_{I\in \Omega \mathcal{D}^{\sigma }}\bigtriangleup _{I}^{\sigma
}f=\sum_{I\in \Omega \mathcal{D}^{\sigma },\text{\ }a\in \Gamma _{n}\text{ }%
}\left\langle f,h_{I}^{\sigma ,a}\right\rangle \ h_{I}^{\sigma
,a}=\sum_{I\in \Omega \mathcal{D}^{\sigma },\text{\ }a\in \Gamma _{n}}%
\widehat{f}\left( I;a\right) \ h_{I}^{\sigma ,a}, \\
g &=&\sum_{J\in \Omega \mathcal{D}^{\omega }\text{ }}\bigtriangleup
_{J}^{\omega }g=\sum_{J\in \Omega \mathcal{D}^{\omega },\text{\ }b\in \Gamma
_{n}}\left\langle g,h_{J}^{\omega ,b}\right\rangle \ h_{J}^{\omega
,b}=\sum_{J\in \Omega \mathcal{D}^{\omega },\text{\ }b\in \Gamma _{n}}%
\widehat{g}\left( J;b\right) \ h_{J}^{\omega ,b},
\end{eqnarray*}%
where the appropriate measure is understood in the notation $\widehat{f}%
\left( I;a\right) $ and $\widehat{g}\left( J;b\right) $, and where these
quasiHaar coefficients $\widehat{f}\left( I;a\right) $ and $\widehat{g}%
\left( J;b\right) $ vanish if the quasicubes $I$ and $J$ are not good.
Inequality (\ref{two weight}) is equivalent to boundedness of the bilinear
form%
\begin{equation*}
\mathcal{T}^{\alpha }\left( f,g\right) \equiv \left\langle T_{\sigma
}^{\alpha }\left( f\right) ,g\right\rangle _{\omega }=\sum_{I\in \Omega 
\mathcal{D}^{\sigma }\text{ and }J\in \Omega \mathcal{D}^{\omega
}}\left\langle T_{\sigma }^{\alpha }\left( \bigtriangleup _{I}^{\sigma
}f\right) ,\bigtriangleup _{J}^{\omega }g\right\rangle _{\omega }
\end{equation*}%
on $L^{2}\left( \sigma \right) \times L^{2}\left( \omega \right) $, i.e.%
\begin{equation*}
\left\vert \mathcal{T}^{\alpha }\left( f,g\right) \right\vert \leq \mathfrak{%
N}_{T^{\alpha }}\left\Vert f\right\Vert _{L^{2}\left( \sigma \right)
}\left\Vert g\right\Vert _{L^{2}\left( \omega \right) },
\end{equation*}%
uniformly over all quasigrids and appropriate truncations. We may assume the
two quasigrids $\Omega \mathcal{D}^{\sigma }$ and $\Omega \mathcal{D}%
^{\omega }$ are equal here, and this we will do throughout the paper,
although we sometimes continue to use the measure as a superscript on $%
\Omega \mathcal{D}$ for clarity of exposition. Roughly speaking, we analyze
the form $\mathcal{T}^{\alpha }\left( f,g\right) $ by splitting it in a
nonlinear way into three main pieces, following in part the approach in \cite%
{LaSaShUr2} and \cite{LaSaShUr3}. The first piece consists of quasicubes $I$
and $J$ that are either disjoint or of comparable side length, and this
piece is handled using the section on preliminaries of NTV type. The second
piece consists of quasicubes $I$ and $J$ that overlap, but are `far apart'
in a nonlinear way, and this piece is handled using the sections on the
Intertwining Proposition and the control of the functional quasienergy
condition by the quasienergy condition. Finally, the remaining local piece
where the overlapping quasicubes are `close' is handled by generalizing
methods of NTV as in \cite{LaSaShUr}, and then splitting the stopping form
into two sublinear stopping forms, one of which is handled using techniques
of \cite{LaSaUr2}, and the other using the stopping time and recursion of M.
Lacey \cite{Lac}. See the schematic diagram in Subsection 7.4 below.

We summarize our assumptions on the Haar supports of $f$ and $g$, and on the
dyadic quasigrids $\Omega \mathcal{D}$.

\begin{condition}[on Haar supports and quasigrids]
\label{assume shift}We suppose the quasiHaar supports of the functions $f$
and $g$ satisfy $\limfunc{supp}\widehat{f},\limfunc{supp}\widehat{g}\subset
\Omega \mathcal{D}_{\left( \mathbf{r},\varepsilon \right) -\limfunc{good}}^{%
\mathbf{\tau }}$. We also assume that $\left\vert \partial Q\right\vert
_{\sigma +\omega }=0$ for all dyadic quasicubes $Q$ in the grids $\Omega 
\mathcal{D}$ (since this property holds with probability $1$ for random
grids $\Omega \mathcal{D}$).
\end{condition}

\section{Necessity of the $\mathcal{A}_{2}^{\protect\alpha }$ conditions}

Here we prove in particular the necessity of the fractional $\mathcal{A}%
_{2}^{\alpha }$ condition (with holes) when $0\leq \alpha <n$, for the
boundedness from $L^{2}\left( \sigma \right) $ to $L^{2}\left( \omega
\right) $ (where $\sigma $ and $\omega $ may have common point masses) of
the $\alpha $-fractional Riesz vector transform $\mathbf{R}^{\alpha }$
defined by%
\begin{equation*}
\mathbf{R}^{\alpha }\left( f\sigma \right) \left( x\right) =\int_{\mathbb{R}%
^{n}}K_{j}^{\alpha }(x,y)f\left( y\right) d\sigma \left( y\right) ,\ \ \ \ \
K_{j}^{\alpha }\left( x,y\right) =\frac{x^{j}-y^{j}}{\left\vert
x-y\right\vert ^{n+1-\alpha }},
\end{equation*}%
whose kernel $K_{j}^{\alpha }\left( x,y\right) $ satisfies (\ref%
{sizeandsmoothness'}) for $0\leq \alpha <n$. See \cite{SaShUr5} for the case
without holes.

\begin{lemma}
\label{necc frac A2}Suppose $0\leq \alpha <n$. Let $T^{\alpha }$ be any
collection of operators with $\alpha $-standard fractional kernel
satisfying\ the ellipticity condition (\ref{Ktalpha}), and in the case $%
\frac{n}{2}\leq \alpha <n$, we also assume the more restrictive condition (%
\ref{Ktalpha strong}). Then for $0\leq \alpha <n$ we have%
\begin{equation*}
\sqrt{\mathcal{A}_{2}^{\alpha }}\lesssim \mathfrak{N}_{\alpha }\left(
T^{\alpha }\right) .
\end{equation*}
\end{lemma}

\begin{remark}
Cancellation properties of $T^{\alpha }$ play no role in the proof below.
Indeed the proof shows that $\mathcal{A}_{2}^{\alpha }$ is dominated by the
best constant $C$ in the restricted inequality%
\begin{equation*}
\left\Vert \chi _{E}T^{\alpha }(f\sigma )\right\Vert _{L^{2,\infty }\left(
\omega \right) }\leq C\left\Vert f\right\Vert _{L^{2}\left( \sigma \right)
},\ \ \ \ \ E=\mathbb{R}^{n}\setminus supp\ f.
\end{equation*}
\end{remark}

\begin{proof}
First we give the proof for the case when $T^{\alpha }$ is the $\alpha $%
-fractional Riesz transform $\mathbf{R}^{\alpha }$, whose kernel is $\mathbf{%
K}^{\alpha }\left( x,y\right) =\frac{x-y}{\left\vert x-y\right\vert
^{n+1-\alpha }}$ . Define the $2^{n}$ generalized $n$-ants $\mathcal{Q}_{m}$
for $m\in \left\{ -1,1\right\} ^{n}$, and their translates $\mathcal{Q}%
_{m}\left( w\right) $ for $w\in \mathbb{R}^{n}$ by 
\begin{eqnarray*}
\mathcal{Q}_{m} &=&\left\{ \left( x_{1},...,x_{n}\right)
:m_{k}x_{k}>0\right\} , \\
\mathcal{Q}_{m}\left( w\right) &=&\left\{ z:z-w\in \mathcal{Q}_{m}\right\}
,\ \ \ \ \ w\in \mathbb{R}^{n}.
\end{eqnarray*}%
Fix $m\in \left\{ -1,1\right\} ^{n}$ and a quasicube $I$. For $a\in \mathbb{R%
}^{n}$ and $r>0$ let 
\begin{eqnarray*}
s_{I}\left( x\right) &=&\frac{\ell \left( I\right) }{\ell \left( I\right)
+\left\vert x-\zeta _{I}\right\vert }, \\
f_{a,r}\left( y\right) &=&\mathbf{1}_{\mathcal{Q}_{-m}\left( a\right) \cap
B\left( 0,r\right) }\left( y\right) s_{I}\left( y\right) ^{n-\alpha },
\end{eqnarray*}%
where $\zeta _{I}$ is the center of the cube $I$. Now%
\begin{eqnarray*}
\ell \left( I\right) \left\vert x-y\right\vert &\leq &\ell \left( I\right)
\left\vert x-\zeta _{I}\right\vert +\ell \left( I\right) \left\vert \zeta
_{I}-y\right\vert \\
&\leq &\left[ \ell \left( I\right) +\left\vert x-\zeta _{I}\right\vert %
\right] \left[ \ell \left( I\right) +\left\vert \zeta _{I}-y\right\vert %
\right]
\end{eqnarray*}%
implies%
\begin{equation*}
\frac{1}{\left\vert x-y\right\vert }\geq \frac{1}{\ell \left( I\right) }%
s_{I}\left( x\right) s_{I}\left( y\right) ,\ \ \ \ \ x,y\in \mathbb{R}^{n}.
\end{equation*}%
Now the key observation is that with $L\zeta \equiv m\cdot \zeta $, we have%
\begin{equation*}
L\left( x-y\right) =m\cdot \left( x-y\right) \geq \left\vert x-y\right\vert
,\ \ \ \ \ x\in \mathcal{Q}_{m}\left( y\right) ,
\end{equation*}%
which yields%
\begin{eqnarray}
L\left( \mathbf{K}^{\alpha }\left( x,y\right) \right) &=&\frac{L\left(
x-y\right) }{\left\vert x-y\right\vert ^{n+1-\alpha }}
\label{key separation} \\
&\geq &\frac{1}{\left\vert x-y\right\vert ^{n-\alpha }}\geq \ell \left(
I\right) ^{\alpha -n}s_{I}\left( x\right) ^{n-\alpha }s_{I}\left( y\right)
^{n-\alpha },  \notag
\end{eqnarray}%
provided $x\in \mathcal{Q}_{m}\left( y\right) $. Now we note that $x\in 
\mathcal{Q}_{m}\left( y\right) $ when $x\in \mathcal{Q}_{m}\left( a\right) $
and $y\in \mathcal{Q}_{-m}\left( a\right) $ to obtain that for $x\in 
\mathcal{Q}_{m}\left( a\right) $, 
\begin{eqnarray*}
L\left( T^{\alpha }\left( f_{a,r}\sigma \right) \left( x\right) \right)
&=&\int_{\mathcal{Q}_{-m}\left( a\right) \cap B\left( 0,r\right) }\frac{%
L\left( x-y\right) }{\left\vert x-y\right\vert ^{n+1-\alpha }}s_{I}\left(
y\right) ^{n-\alpha }d\sigma \left( y\right) \\
&\geq &\ell \left( I\right) ^{\alpha -n}s_{I}\left( x\right) ^{n-\alpha
}\int_{\mathcal{Q}_{-m}\left( a\right) \cap B\left( 0,r\right) }s_{I}\left(
y\right) ^{2n-2\alpha }d\sigma \left( y\right) .
\end{eqnarray*}

Applying $\left\vert L\zeta \right\vert \leq \sqrt{n}\left\vert \zeta
\right\vert $ and our assumed two weight inequality for the fractional Riesz
transform, we see that for $r>0$ large, 
\begin{align*}
& \ell \left( I\right) ^{2\alpha -2n}\int_{\mathcal{Q}_{m}\left( a\right)
}s_{I}\left( x\right) ^{2n-2\alpha }\left( \int_{\mathcal{Q}_{-m}\left(
a\right) \cap B\left( 0,r\right) }s_{I}\left( y\right) ^{2n-2\alpha }d\sigma
\left( y\right) \right) ^{2}d\omega \left( x\right) \\
& \leq \left\Vert LT(\sigma f_{a,r})\right\Vert _{L^{2}(\omega
)}^{2}\lesssim \mathfrak{N}_{\alpha }\left( \mathbf{R}^{\alpha }\right)
^{2}\left\Vert f_{a,r}\right\Vert _{L^{2}(\sigma )}^{2}=\mathfrak{N}_{\alpha
}\left( \mathbf{R}^{\alpha }\right) ^{2}\int_{\mathcal{Q}_{-m}\left(
a\right) \cap B\left( 0,r\right) }s_{I}\left( y\right) ^{2n-2\alpha }d\sigma
\left( y\right) .
\end{align*}

Rearranging the last inequality, we obtain%
\begin{equation*}
\ell \left( I\right) ^{2\alpha -2n}\int_{\mathcal{Q}_{m}\left( a\right)
}s_{I}\left( x\right) ^{2n-2\alpha }d\omega \left( x\right) \int_{\mathcal{Q}%
_{-m}\left( a\right) \cap B\left( 0,r\right) }s_{I}\left( y\right)
^{2n-2\alpha }d\sigma \left( y\right) \lesssim \mathfrak{N}_{\alpha }\left( 
\mathbf{R}^{\alpha }\right) ^{2},
\end{equation*}%
and upon letting $r\rightarrow \infty $, 
\begin{equation*}
\int_{\mathcal{Q}_{m}\left( a\right) }\frac{\ell \left( I\right) ^{n-\alpha }%
}{\left( \ell \left( I\right) +\left\vert x-\zeta _{I}\right\vert \right)
^{2n-2\alpha }}d\omega \left( x\right) \int_{\mathcal{Q}_{-m}\left( a\right)
}\frac{\ell \left( I\right) ^{n-\alpha }}{\left( \ell \left( I\right)
+\left\vert y-\zeta _{I}\right\vert \right) ^{2n-2\alpha }}d\sigma \left(
y\right) \lesssim \mathfrak{N}_{\alpha }\left( \mathbf{R}^{\alpha }\right)
^{2}.
\end{equation*}%
Note that the ranges of integration above are pairs of opposing $n$-ants.

Fix a quasicube $Q$, which without loss of generality can be taken to be
centered at the origin, $\zeta _{Q}=0$. Then choose $a=\left( 2\ell \left(
Q\right) ,2\ell \left( Q\right) \right) $ and $I=Q$ so that we have%
\begin{eqnarray*}
&&\left( \int_{\mathcal{Q}_{m}\left( a\right) }\frac{\ell \left( Q\right)
^{n-\alpha }}{\left( \ell \left( Q\right) +\left\vert x\right\vert \right)
^{2n-2\alpha }}d\omega \left( x\right) \right) \left( \ell \left( Q\right)
^{\alpha -n}\int_{Q}d\sigma \right) \\
&\leq &C_{\alpha }\int_{\mathcal{Q}_{m}\left( a\right) }\frac{\ell \left(
Q\right) ^{n-\alpha }}{\left( \ell \left( Q\right) +\left\vert x\right\vert
\right) ^{2n-2\alpha }}d\omega \left( x\right) \int_{\mathcal{Q}_{-m}\left(
a\right) }\frac{\ell \left( Q\right) ^{n-\alpha }}{\left( \ell \left(
Q\right) +\left\vert y\right\vert \right) ^{2n-2\alpha }}d\sigma \left(
y\right) \lesssim \mathfrak{N}_{\alpha }\left( \mathbf{R}^{\alpha }\right)
^{2}.
\end{eqnarray*}%
Now fix $m=\left( 1,1,...,1\right) $ and note that there is a fixed $N$
(independent of $\ell \left( Q\right) $) and a fixed collection of rotations 
$\left\{ \rho _{k}\right\} _{k=1}^{N}$, such that the rotates $\rho _{k}%
\mathcal{Q}_{m}\left( a\right) $, $1\leq k\leq N$, of the $n$-ant $\mathcal{Q%
}_{m}\left( a\right) $ cover the complement of the ball $B\left( 0,4\sqrt{n}%
\ell \left( Q\right) \right) $: 
\begin{equation*}
B\left( 0,4\sqrt{n}\ell \left( Q\right) \right) ^{c}\subset
\bigcup_{k=1}^{N}\rho _{k}\mathcal{Q}_{m}\left( a\right) .
\end{equation*}%
Then we obtain, upon applying the same argument to these rotated pairs of $n$%
-ants, 
\begin{equation}
\left( \int_{B\left( 0,4\sqrt{n}\ell \left( Q\right) \right) ^{c}}\frac{\ell
\left( Q\right) ^{n-\alpha }}{\left( \ell \left( Q\right) +\left\vert
x\right\vert \right) ^{2n-2\alpha }}d\omega \left( x\right) \right) \left(
\ell \left( Q\right) ^{\alpha -n}\int_{Q}d\sigma \right) \lesssim \mathfrak{N%
}_{\alpha }\left( \mathbf{R}^{\alpha }\right) ^{2}.  \label{prelim A2}
\end{equation}

Now we assume for the moment the offset $A_{2}^{\alpha }$ condition 
\begin{equation*}
\ell \left( Q\right) ^{2\left( \alpha -n\right) }\left( \int_{Q^{\prime
}}d\omega \right) \left( \int_{Q}d\sigma \right) \leq A_{2}^{\alpha },
\end{equation*}%
where $Q^{\prime }$ and $Q$ are neighbouring quasicubes, i.e. $\left(
Q^{\prime },Q\right) \in \Omega \mathcal{N}^{n}$. If we use this offset
inequality with $Q^{\prime }$ ranging over $3Q\setminus Q$, and then use the
separation of $B\left( 0,4\sqrt{n}\ell \left( Q\right) \right) \setminus 3Q$
and $Q$ to obtain the inequality%
\begin{equation*}
\ell \left( Q\right) ^{2\left( \alpha -n\right) }\left( \int_{B\left( 0,4%
\sqrt{n}\ell \left( Q\right) \right) \setminus 3Q}d\omega \right) \left(
\int_{Q}d\sigma \right) \lesssim A_{2}^{\alpha }\ ,
\end{equation*}%
together with (\ref{prelim A2}), we obtain%
\begin{equation*}
\left( \int_{\mathbb{R}^{n}\setminus Q}\frac{\ell \left( Q\right) ^{n-\alpha
}}{\left( \ell \left( Q\right) +\left\vert x\right\vert \right) ^{2n-2\alpha
}}d\omega \left( x\right) \right) ^{\frac{1}{2}}\left( \ell \left( Q\right)
^{\alpha -n}\int_{Q}d\sigma \right) ^{\frac{1}{2}}\lesssim \mathfrak{N}%
_{\alpha }\left( \mathbf{R}^{\alpha }\right) +\sqrt{A_{2}^{\alpha }},
\end{equation*}%
or%
\begin{equation*}
\ell \left( Q\right) ^{\alpha }\left( \frac{1}{\left\vert Q\right\vert }%
\int_{\mathbb{R}^{n}\setminus Q}\frac{1}{\left( 1+\frac{\left\vert x-\zeta
_{Q}\right\vert }{\ell \left( Q\right) }\right) ^{2n-2\alpha }}d\omega
\left( x\right) \right) ^{\frac{1}{2}}\left( \frac{1}{\left\vert
Q\right\vert }\int_{Q}d\sigma \right) ^{\frac{1}{2}}\lesssim \mathfrak{N}%
_{\alpha }\left( \mathbf{R}^{\alpha }\right) +\sqrt{A_{2}^{\alpha }}.
\end{equation*}%
Clearly we can reverse the roles of the measures $\omega $ and $\sigma $ and
obtain 
\begin{equation*}
\sqrt{\mathcal{A}_{2}^{\alpha ,\ast }}\lesssim \mathfrak{N}_{\alpha }\left( 
\mathbf{R}^{\alpha }\right) +\sqrt{A_{2}^{\alpha }}
\end{equation*}%
for the kernels $\mathbf{K}^{\alpha }$, $0\leq \alpha <n$.

More generally, to obtain the case when $T^{\alpha }$ is elliptic and the
offset $A_{2}^{\alpha }$ condition holds, we note that the key estimate (\ref%
{key separation}) above extends to the kernel $\sum_{j=1}^{J}\lambda
_{j}^{m}K_{j}^{\alpha }$ of $\sum_{j=1}^{J}\lambda _{j}^{m}T_{j}^{\alpha }$
in (\ref{Ktalpha strong}) if the $n$-ants above are replaced by thin cones
of sufficently small aperture, and there is in addition sufficient
separation between opposing cones, which in turn may require a larger
constant than $4\sqrt{n}$ in the choice of $Q^{\prime }$ above.

Finally, we turn to showing that the offset $A_{2}^{\alpha }$ condition is
implied by the norm inequality, i.e.%
\begin{eqnarray*}
&&\sqrt{A_{2}^{\alpha }}\equiv \sup_{\left( Q^{\prime },Q\right) \in \Omega 
\mathcal{N}^{n}}\ell \left( Q\right) ^{\alpha }\left( \frac{1}{\left\vert
Q^{\prime }\right\vert }\int_{Q^{\prime }}d\omega \right) ^{\frac{1}{2}%
}\left( \frac{1}{\left\vert Q\right\vert }\int_{Q}d\sigma \right) ^{\frac{1}{%
2}}\lesssim \mathfrak{N}_{\alpha }\left( \mathbf{R}^{\alpha }\right) ; \\
&&\text{i.e. }\left( \int_{Q^{\prime }}d\omega \right) \left(
\int_{Q}d\sigma \right) \lesssim \mathfrak{N}_{\alpha }\left( \mathbf{R}%
^{\alpha }\right) ^{2}\left\vert Q\right\vert ^{2-\frac{2\alpha }{n}},\ \ \
\ \ \left( Q^{\prime },Q\right) \in \Omega \mathcal{N}^{n}.
\end{eqnarray*}%
In the range $0\leq \alpha <\frac{n}{2}$ where we only assume (\ref{Ktalpha}%
), we adapt a corresponding argument from \cite{LaSaUr1}.

The `one weight' argument on page 211 of Stein \cite{Ste} yields the \emph{%
asymmetric} two weight $A_{2}^{\alpha }$ condition%
\begin{equation}
\left\vert Q^{\prime }\right\vert _{\omega }\left\vert Q\right\vert _{\sigma
}\leq C\mathfrak{N}_{\alpha }\left( \mathbf{R}^{\alpha }\right)
^{2}\left\vert Q\right\vert ^{2\left( 1-\frac{\alpha }{n}\right) },
\label{asym}
\end{equation}%
where $Q$ and $Q^{\prime }$ are quasicubes of equal side length $r$ and
distance $C_{0}r$ apart for some (fixed large) positive constant $C_{0}$
(for this argument we choose the unit vector $\mathbf{u}$ in (\ref{Ktalpha})
to point in the direction from $Q$ to $Q^{\prime }$). In the one weight case
treated in \cite{Ste} it is easy to obtain from this (even for a \emph{single%
} direction $\mathbf{u}$) the usual (symmetric) $A_{2}$ condition. Here we
will have to employ a different approach.

Now recall (see {Sec 2 of} \cite{Saw3} for the case of usual cubes, and the
case of half open, half closed quasicubes here is no different) that given
an open subset $\Phi $ of $\mathbb{R}^{n}$, we can choose $R\geq 3$
sufficiently large, depending only on the dimension, such that if $\left\{
Q_{j}^{k}\right\} _{j}$ are the dyadic quasicubes maximal among those dyadic
quasicubes $Q$ satisfying $RQ\subset \Phi $, then the following properties
hold:%
\begin{equation}
\left\{ 
\begin{array}{ll}
\text{(disjoint cover)} & \Phi =\bigcup_{j}Q_{j}\text{ and }Q_{j}\cap
Q_{i}=\emptyset \text{ if }i\neq j \\ 
\text{(Whitney condition)} & RQ_{j}\subset \Phi \text{ and }3RQ_{j}\cap \Phi
^{c}\neq \emptyset \text{ for all }j \\ 
\text{(finite overlap)} & \sum_{j}\chi _{3Q_{j}}\leq C\chi _{\Phi }%
\end{array}%
\right. .  \label{Whitney}
\end{equation}

So fix a pair of neighbouring quasicubes $\left( Q_{0}^{\prime
},Q_{0}\right) \in \Omega \mathcal{N}^{n}$,\ and let $\left\{ \mathsf{Q}%
_{i}\right\} _{i}$ be a Whitney decomposition into quasicubes of the set $%
\Phi \equiv \left( Q_{0}^{\prime }\times Q_{0}\right) \setminus \mathfrak{D}$
relative to the diagonal $\mathfrak{D}$ in $\mathbb{R}^{n}\times \mathbb{R}%
^{n}$. Of course, there are no common point masses of $\omega $ in $%
Q_{0}^{\prime }$ and $\sigma $ in $Q_{0}$ since the quasicubes $%
Q_{0}^{\prime }$ and $Q_{0}$ are disjoint. Note that if $\mathsf{Q}%
_{i}=Q_{i}^{\prime }\times Q_{i}$, then (\ref{asym}) can be written%
\begin{equation}
\left\vert \mathsf{Q}_{i}\right\vert _{\omega \times \sigma }\leq C\mathfrak{%
N}_{\alpha }\left( \mathbf{R}^{\alpha }\right) ^{2}\left\vert \mathsf{Q}%
_{i}\right\vert ^{1-\frac{\alpha }{n}},  \label{asym'}
\end{equation}%
where $\omega \times \sigma $ denotes product measure on $\mathbb{R}%
^{n}\times \mathbb{R}^{n}$. We choose $R$ sufficiently large in the Whitney
decomposition (\ref{Whitney}), depending on $C_{0}$, such that (\ref{asym'})
holds for all the Whitney quasicubes $\mathsf{Q}_{i}$. We have $%
\sum_{i}\left\vert \mathsf{Q}_{i}\right\vert =\left\vert Q^{\prime }\times
Q\right\vert =\left\vert Q\right\vert ^{2}$.

Moreover, if $\mathsf{R}=Q^{\prime }\times Q$ is a rectangle in $\mathbb{R}%
^{n}\times \mathbb{R}^{n}$ (i.e. $Q^{\prime },Q$ are quasicubes in $\mathbb{R%
}^{n}$), and if $\mathsf{R}=\overset{\cdot }{\cup }_{i}\mathsf{R}_{i}$ is a
finite disjoint union of rectangles $\mathsf{R}_{\alpha }$, then by
additivity of the product measure $\omega \times \sigma $, 
\begin{equation*}
\left\vert \mathsf{R}\right\vert _{\omega \times \sigma }=\sum_{i}\left\vert 
\mathsf{R}_{i}\right\vert _{\omega \times \sigma }.
\end{equation*}

Let $\mathsf{Q}_{0}=Q_{0}^{\prime }\times Q_{0}$ and set 
\begin{equation*}
\Lambda \equiv \left\{ \mathsf{Q}=Q^{\prime }\times Q:\mathsf{Q}\subset 
\mathsf{Q}_{0},\ell \left( Q\right) =\ell \left( Q^{\prime }\right) \approx
C_{0}^{-1}\limfunc{qdist}\left( Q,Q^{\prime }\right) \text{ and (\ref{asym})
holds}\right\} .
\end{equation*}%
Divide $Q_{0}$ into $2n\times 2n=4n^{2}$ congruent subquasicubes $%
Q_{0}^{1},...,Q_{0}^{4^{n}}$ of side length $\frac{1}{2}$, and set aside
those $Q_{0}^{j}\in \Lambda $ (those for which (\ref{asym}) holds) into a
collection of stopping cubes $\Gamma $. Continue to divide the remaining $%
Q_{0}^{j}\in \Lambda $ of side length $\frac{1}{4}$, and again, set aside
those $Q_{0}^{j,i}\in \Phi $ into $\Gamma $, and continue subdividing those
that remain. We continue with such subdivisions for $N$ generations so that
all the cubes \emph{not} set aside into $\Gamma $ have side length $2^{-N}$
. The important property these cubes have is that they all lie within
distance $Cr2^{-N}$ of the diagonal $\mathfrak{D}=\left\{ \left( x,x\right)
:\left( x,x\right) \in Q_{0}^{\prime }\times Q_{0}\right\} $ in $\mathsf{Q}%
_{0}=Q_{0}^{\prime }\times Q_{0}$ since (\ref{asym}) holds for all pairs of
cubes $Q^{\prime }$ and $Q$ of equal side length $r$ having distance
approximately $C_{0}r$ apart. Enumerate the cubes in $\Gamma $ as $\left\{ 
\mathsf{Q}_{i}\right\} _{i}$ and those remaining that are not in $\Gamma $
as $\left\{ \mathsf{P}_{j}\right\} _{j}$. Thus we have the pairwise disjoint
decomposition%
\begin{equation*}
\mathsf{Q}_{0}=\left( \dbigcup\limits_{i}\mathsf{Q}_{i}\right) \dbigcup
\left( \dbigcup\limits_{j}\mathsf{P}_{j}\right) .
\end{equation*}%
The countable additivity of the product measure $\omega \times \sigma $
shows that%
\begin{equation*}
\left\vert \mathsf{Q}_{0}\right\vert _{\omega \times \sigma
}=\sum_{i}\left\vert \mathsf{Q}_{i}\right\vert _{\omega \times \sigma
}+\sum_{j}\left\vert \mathsf{P}_{j}\right\vert _{\omega \times \sigma }\ .
\end{equation*}

Now we have%
\begin{equation*}
\sum_{i}\left\vert \mathsf{Q}_{i}\right\vert _{\omega \times \sigma
}\lesssim \sum_{i}\mathfrak{N}_{\alpha }\left( \mathbf{R}^{\alpha }\right)
^{2}\left\vert \mathsf{Q}_{i}\right\vert ^{1-\frac{\alpha }{n}},
\end{equation*}%
and 
\begin{eqnarray*}
\sum_{i}\left\vert \mathsf{Q}_{i}\right\vert ^{1-\frac{\alpha }{n}}
&=&\sum_{k\in \mathbb{Z}:\ 2^{k}\leq \ell \left( Q_{0}\right) }\sum_{i:\
\ell \left( Q_{i}\right) =2^{k}}\left( 2^{2nk}\right) ^{1-\frac{\alpha }{n}}
\\
&\approx &\sum_{k\in \mathbb{Z}:\ 2^{k}\leq \ell \left( Q_{0}\right) }\left( 
\frac{2^{k}}{\ell \left( Q_{0}\right) }\right) ^{-n}\left( 2^{2nk}\right)
^{1-\frac{\alpha }{n}}\ \ \ \text{(Whitney)} \\
&=&\ell \left( Q_{0}\right) ^{n}\sum_{k\in \mathbb{Z}:\ 2^{k}\leq \ell
\left( Q_{0}\right) }2^{nk\left( -1+2-\frac{2\alpha }{n}\right) } \\
&\leq &C_{\alpha }\ell \left( Q_{0}\right) ^{n}\ell \left( Q_{0}\right)
^{n\left( 1-\frac{2\alpha }{n}\right) }=C_{\alpha }\left\vert Q_{0}\times
Q_{0}\right\vert ^{2-\frac{2\alpha }{n}}=C_{\alpha }\left\vert \mathsf{Q}%
_{0}\right\vert ^{1-\frac{\alpha }{n}},
\end{eqnarray*}%
provided $0\leq \alpha <\frac{n}{2}$. Using that the side length of $\mathsf{%
P}_{j}=P_{j}\times P_{j}^{\prime }$ is $2^{-N}$ and $dist\left( \mathsf{P}%
_{j},\mathfrak{D}\right) \leq C_{r}2^{-N}$, we have the following limit,%
\begin{equation*}
\sum_{j}\left\vert \mathsf{P}_{j}\right\vert _{\omega \times \sigma
}=\left\vert \dbigcup\limits_{j}\mathsf{P}_{j}\right\vert _{\omega \times
\sigma }\rightarrow 0\text{ as }N\rightarrow \infty ,
\end{equation*}%
since $\dbigcup\limits_{j}\mathsf{P}_{j}$ shrinks to the empty set as $%
N\rightarrow \infty $, and since locally finite measures such as $\omega
\times \sigma $ are regular in Euclidean space. This completes the proof
that $\sqrt{A_{2}^{\alpha }}\lesssim \mathfrak{N}_{\alpha }\left( \mathbf{R}%
^{\alpha }\right) $ for the range $0\leq \alpha <\frac{n}{2}$.

Now we turn to proving $\sqrt{A_{2}^{\alpha }}\lesssim \mathfrak{N}_{\alpha
}\left( \mathbf{R}^{\alpha }\right) $ for the range $\frac{n}{2}\leq \alpha
<n$, where we assume the stronger ellipticity condition (\ref{Ktalpha strong}%
). So fix a pair of neighbouring quasicubes $\left( K^{\prime },K\right) \in
\Omega \mathcal{N}^{n}$, and assume that $\sigma +\omega $ doesn't charge
the intersection $\overline{K^{\prime }}\cap \overline{K}$ of the closures
of $K^{\prime }$ and $K$. It will be convenient to add another dimension by
replacing $n$ by $n+1$ and working with the preimages $Q^{\prime }=\Omega
^{-1}K^{\prime }$ and $Q=\Omega ^{-1}K$ that are usual cubes, and with the
corresponding pullbacks $\widetilde{\omega }=\mathcal{L}_{1}\times \Omega
^{\ast }\omega $ and $\widetilde{\sigma }=\mathcal{L}_{1}\times \Omega
^{\ast }\sigma $ of the measures $\omega $ and $\sigma $; here $\mathcal{L}%
_{1}$ is one-dimensional Lebesgue measure. We may also assume that 
\begin{equation*}
Q^{\prime }=\left[ -1,0\right) \times \dprod\limits_{i=1}^{n}Q_{i},\ \ \ \ \
Q=\left[ 0,1\right) \times \dprod\limits_{i=1}^{n}Q_{i}.
\end{equation*}%
where $Q_{i}=\left[ a_{i},b_{i}\right] $ for $1\leq i\leq n$ (since the
other cases are handled in similar fashion). It is important to note that we
are considering the intervals $Q_{i}$ here to be closed, and we will track
this difference as we proceed.

Choose $\theta _{1}\in \left[ a_{1},b_{1}\right] $ so that both 
\begin{equation*}
\left\vert \left[ -1,0\right) \times \left[ a_{1},\theta _{1}\right] \times
\dprod\limits_{i=2}^{n}Q_{i}\right\vert _{\widetilde{\omega }},\ \ \
\left\vert \left[ -1,0\right) \times \left[ \theta _{1},b_{1}\right] \times
\dprod\limits_{i=2}^{n}Q_{i}\right\vert _{\widetilde{\omega }}\geq \frac{1}{2%
}\left\vert Q^{\prime }\right\vert _{\widetilde{\omega }}.
\end{equation*}%
Now denote the two intervals $\left[ a_{1},\theta _{1}\right] $ and $\left[
\theta _{1},b_{1}\right] $ by $\left[ a_{1}^{\ast },b_{1}^{\ast }\right] $
and $\left[ a_{1}^{\ast \ast },b_{1}^{\ast \ast }\right] $ where the order
is chosen so that 
\begin{equation*}
\left\vert \left[ 0,1\right) \times \left[ a_{1}^{\ast },b_{1}^{\ast }\right]
\times \dprod\limits_{i=2}^{n}Q_{i}\right\vert _{\widetilde{\sigma }}\leq
\left\vert \left[ 0,1\right) \times \left[ a_{1}^{\ast \ast },b_{1}^{\ast
\ast }\right] \times \dprod\limits_{i=2}^{n}Q_{i}\right\vert _{\widetilde{%
\sigma }}.
\end{equation*}%
Then we have both%
\begin{eqnarray*}
\left\vert \left[ -1,0\right) \times \left[ a_{1}^{\ast },b_{1}^{\ast }%
\right] \times \dprod\limits_{i=2}^{n}Q_{i}\right\vert _{\widetilde{\omega }%
} &\geq &\frac{1}{2}\left\vert Q\right\vert _{\widetilde{\omega }}\ , \\
\left\vert \left[ 0,1\right) \times \left[ a_{1}^{\ast \ast },b_{1}^{\ast
\ast }\right] \times \dprod\limits_{i=2}^{n}Q_{i}\right\vert _{\widetilde{%
\sigma }} &\geq &\frac{1}{2}\left\vert Q\right\vert _{\widetilde{\sigma }}\ .
\end{eqnarray*}%
Now choose $\theta _{2}\in \left[ a_{2},b_{2}\right] $ so that both%
\begin{equation*}
\left\vert \left[ -1,0\right) \times \left[ a_{1}^{\ast },b_{1}^{\ast }%
\right] \times \left[ a_{2},\theta _{2}\right] \times
\dprod\limits_{i=3}^{n}Q_{i}\right\vert _{\widetilde{\omega }},\ \ \
\left\vert \left[ -1,0\right) \times \left[ a_{1}^{\ast },b_{1}^{\ast }%
\right] \times \left[ \theta _{2},b_{2}\right] \times
\dprod\limits_{i=3}^{n}Q_{i}\right\vert _{\widetilde{\omega }}\geq \frac{1}{4%
}\left\vert Q\right\vert _{\widetilde{\omega }},
\end{equation*}%
and denote the two intervals $\left[ a_{2},\theta _{2}\right] $ and $\left[
\theta _{2},b_{2}\right] $ by $\left[ a_{2}^{\ast },b_{2}^{\ast }\right] $
and $\left[ a_{2}^{\ast \ast },b_{2}^{\ast \ast }\right] $ where the order
is chosen so that%
\begin{equation*}
\left[ 0,1\right) \times \left\vert \left[ a_{1}^{\ast \ast },b_{1}^{\ast
\ast }\right] \times \left[ a_{2}^{\ast },b_{2}^{\ast }\right] \times
\dprod\limits_{i=2}^{n}Q_{i}\right\vert _{\widetilde{\sigma }}\leq
\left\vert \left[ 0,1\right) \times \left[ a_{1}^{\ast \ast },b_{1}^{\ast
\ast }\right] \times \left[ a_{2}^{\ast \ast },b_{2}^{\ast \ast }\right]
\times \dprod\limits_{i=2}^{n}Q_{i}\right\vert _{\widetilde{\sigma }}.
\end{equation*}%
Then we have both%
\begin{eqnarray*}
\left\vert \left[ -1,0\right) \times \left[ a_{1}^{\ast },b_{1}^{\ast }%
\right] \times \left[ a_{2}^{\ast },b_{2}^{\ast }\right] \times
\dprod\limits_{i=3}^{n}Q_{i}\right\vert _{\widetilde{\omega }} &\geq &\frac{1%
}{4}\left\vert Q\right\vert _{\widetilde{\omega }}\ , \\
\left\vert \left[ 0,1\right) \times \left[ a_{1}^{\ast \ast },b_{1}^{\ast
\ast }\right] \times \left[ a_{2}^{\ast \ast },b_{2}^{\ast \ast }\right]
\times \dprod\limits_{i=3}^{n}Q_{i}\right\vert _{\widetilde{\sigma }} &\geq &%
\frac{1}{4}\left\vert Q\right\vert _{\widetilde{\sigma }}\ .
\end{eqnarray*}%
Then we choose $\theta _{3}\in \left[ a_{3},b_{3}\right] $ so that both%
\begin{eqnarray*}
\left\vert \left[ -1,0\right) \times \left[ a_{1}^{\ast },b_{1}^{\ast }%
\right] \times \left[ a_{2}^{\ast },b_{2}^{\ast }\right] \times \left[
a_{3},\theta _{3}\right] \times \dprod\limits_{i=4}^{n}Q_{i}\right\vert _{%
\widetilde{\omega }} &\geq &\frac{1}{8}\left\vert Q\right\vert _{\widetilde{%
\omega }}, \\
\left\vert \left[ -1,0\right) \times \left[ a_{1}^{\ast },b_{1}^{\ast }%
\right] \times \left[ a_{2}^{\ast },b_{2}^{\ast }\right] \times \left[
\theta _{3},b_{3}\right] \times \dprod\limits_{i=4}^{n}Q_{i}\right\vert _{%
\widetilde{\omega }} &\geq &\frac{1}{8}\left\vert Q\right\vert _{\widetilde{%
\omega }},
\end{eqnarray*}%
and continuing in this way we end up with two rectangles,%
\begin{eqnarray*}
G &\equiv &\left[ -1,0\right) \times \left[ a_{1}^{\ast },b_{1}^{\ast }%
\right] \times \left[ a_{2}^{\ast },b_{2}^{\ast }\right] \times ...\left[
a_{n}^{\ast },b_{n}^{\ast }\right] , \\
H &\equiv &\left[ 0,1\right) \times \left[ a_{1}^{\ast \ast },b_{1}^{\ast
\ast }\right] \times \left[ a_{2}^{\ast \ast },b_{2}^{\ast \ast }\right]
\times ...\left[ a_{n}^{\ast \ast },b_{n}^{\ast \ast }\right] ,
\end{eqnarray*}%
that satisfy%
\begin{eqnarray*}
\left\vert G\right\vert _{\widetilde{\omega }} &=&\left\vert \left[
-1,0\right) \times \left[ a_{1}^{\ast },b_{1}^{\ast }\right] \times \left[
a_{2}^{\ast },b_{2}^{\ast }\right] \times ...\left[ a_{n}^{\ast
},b_{n}^{\ast }\right] \right\vert _{\widetilde{\omega }}\geq \frac{1}{2^{n}}%
\left\vert Q\right\vert _{\widetilde{\omega }}, \\
\left\vert H\right\vert _{\widetilde{\sigma }} &=&\left\vert \left[
0,1\right) \times \left[ a_{1}^{\ast \ast },b_{1}^{\ast \ast }\right] \times %
\left[ a_{2}^{\ast \ast },b_{2}^{\ast \ast }\right] \times ...\left[
a_{n}^{\ast \ast },b_{n}^{\ast \ast }\right] \right\vert _{\widetilde{\sigma 
}}\geq \frac{1}{2^{n}}\left\vert Q\right\vert _{\widetilde{\sigma }}.
\end{eqnarray*}

However, the quasirectangles $\Omega G$ and $\Omega H$ lie in opposing quasi-%
$n$-ants at the vertex $\Omega \theta =\Omega \left( \theta _{1},\theta
_{2},...,\theta _{n}\right) $, and so we can apply (\ref{Ktalpha strong}) to
obtain that for $x\in \Omega G$,%
\begin{eqnarray*}
\left\vert \sum_{j=1}^{J}\lambda _{j}^{m}T_{j}^{\alpha }\left( \mathbf{1}%
_{\Omega H}\sigma \right) \left( x\right) \right\vert &=&\left\vert
\int_{\Omega H}\sum_{j=1}^{J}\lambda _{j}^{m}K_{j}^{\alpha }\left(
x,y\right) d\sigma \left( y\right) \right\vert \\
&\gtrsim &\int_{\Omega H}\left\vert x-y\right\vert ^{\alpha -n}d\sigma
\left( y\right) \gtrsim \left\vert \Omega Q\right\vert ^{\frac{\alpha }{n}%
-1}\left\vert \Omega H\right\vert _{\sigma }.
\end{eqnarray*}%
For the inequality above, we need to know that the distinguished point $%
\Omega \theta $ is not a common point mass of $\sigma $ and $\omega $, but
this follows from our assumption that $\sigma +\omega $ doesn't charge the
intersection $\overline{K^{\prime }}\cap \overline{K}$ of the closures of $%
K^{\prime }$ and $K$. Then from the norm inequality we get%
\begin{eqnarray*}
\left\vert \Omega G\right\vert _{\omega }\left( \left\vert \Omega
Q\right\vert ^{\frac{\alpha }{n}-1}\left\vert \Omega H\right\vert _{\sigma
}\right) ^{2} &\lesssim &\int_{G}\left\vert \sum_{j=1}^{J}\lambda
_{j}^{m}T_{j}^{\alpha }\left( \mathbf{1}_{\Omega H}\sigma \right)
\right\vert ^{2}d\omega \\
&\lesssim &\mathfrak{N}_{\sum_{j=1}^{J}\lambda _{j}^{m}T_{j}^{\alpha
}}^{2}\int \mathbf{1}_{\Omega H}^{2}d\sigma =\mathfrak{N}_{\sum_{j=1}^{J}%
\lambda _{j}^{m}T_{j}^{\alpha }}^{2}\left\vert \Omega H\right\vert _{\sigma
},
\end{eqnarray*}%
from which we deduce that%
\begin{eqnarray*}
\left\vert \Omega Q\right\vert ^{2\left( \frac{\alpha }{n}-1\right)
}\left\vert \Omega Q^{\prime }\right\vert _{\omega }\left\vert \Omega
Q\right\vert _{\sigma } &\lesssim &2^{2n}\left\vert \Omega Q\right\vert
^{2\left( \frac{\alpha }{n}-1\right) }\left\vert \Omega G\right\vert
_{\omega }\left\vert \Omega H\right\vert _{\sigma }\lesssim 2^{2n}\mathfrak{N%
}_{\sum_{j=1}^{J}\lambda _{j}^{m}T_{j}^{\alpha }}^{2}; \\
\left\vert K\right\vert ^{2\left( \frac{\alpha }{n}-1\right) }\left\vert
K^{\prime }\right\vert _{\omega }\left\vert K\right\vert _{\sigma }
&\lesssim &2^{2n}\mathfrak{N}_{\sum_{j=1}^{J}\lambda _{j}^{m}T_{j}^{\alpha
}}^{2}\ ,
\end{eqnarray*}%
and hence%
\begin{equation*}
A_{2}^{\alpha }\lesssim 2^{2n}\mathfrak{N}_{\sum_{j=1}^{J}\lambda
_{j}^{m}T_{j}^{\alpha }}^{2}\ .
\end{equation*}

Thus we have obtained the offset $A_{2}^{\alpha }$ condition for pairs $%
\left( K^{\prime },K\right) \in \Omega \mathcal{N}^{n}$ such that $\sigma
+\omega $ doesn't charge the intersection $\overline{K^{\prime }}\cap 
\overline{K}$ of the closures of $K^{\prime }$ and $K$. From this and the
argument at the beginning of this proof, we obtain the one-tailed $\mathcal{A%
}_{2}^{\alpha }$ conditions. Indeed, we note that $\left\vert \partial
\left( rQ\right) \right\vert _{\sigma +\omega }>0$ for only a countable
number of dilates $r>1$, and so a limiting argument applies. This completes
the proof of Lemma \ref{necc frac A2}.
\end{proof}

\section{Monotonicity Lemma and Energy lemma}

The Monotonicity Lemma below will be used to prove the Energy Lemma, which
is then used in several places in the proof of Theorem \ref{T1 theorem}. The
formulation of the Monotonicity Lemma with $m=2$ for cubes is due to M.
Lacey and B. Wick \cite{LaWi}, and corrects that used in early versions of
our paper \cite{SaShUr5}.

\subsection{The Monotonicity Lemma}

For $0\leq \alpha <n$ and $m\in \mathbb{R}_{+}$, we recall the $m$-weighted
fractional Poisson integral%
\begin{equation*}
\mathrm{P}_{m}^{\alpha }\left( J,\mu \right) \equiv \int_{\mathbb{R}^{n}}%
\frac{\left\vert J\right\vert ^{\frac{m}{n}}}{\left( \left\vert J\right\vert
^{\frac{1}{n}}+\left\vert y-c_{J}\right\vert \right) ^{n+m-\alpha }}d\mu
\left( y\right) ,
\end{equation*}%
where $\mathrm{P}_{1}^{\alpha }\left( J,\mu \right) =\mathrm{P}^{\alpha
}\left( J,\mu \right) $ is the standard Poisson integral.

\begin{lemma}[Monotonicity]
\label{mono}Suppose that$\ I$ and $J$ are quasicubes in $\mathbb{R}^{n}$
such that $J\subset 2J\subset I$, and that $\mu $ is a signed measure on $%
\mathbb{R}^{n}$ supported outside $I$. Finally suppose that $T^{\alpha }$ is
a standard fractional singular integral on $\mathbb{R}^{n}$ with $0<\alpha
<n $. Then we have the estimate%
\begin{equation}
\left\Vert \bigtriangleup _{J}^{\omega }T^{\alpha }\mu \right\Vert
_{L^{2}\left( \omega \right) }\lesssim \Phi ^{\alpha }\left( J,\left\vert
\mu \right\vert \right) ,  \label{estimate}
\end{equation}%
where for a positive measure $\nu $,%
\begin{eqnarray*}
\Phi ^{\alpha }\left( J,\nu \right) ^{2} &\equiv &\left( \frac{\mathrm{P}%
^{\alpha }\left( J,\nu \right) }{\left\vert J\right\vert ^{\frac{1}{n}}}%
\right) ^{2}\left\Vert \bigtriangleup _{J}^{\omega }\mathbf{x}\right\Vert
_{L^{2}\left( \omega \right) }^{2}+\left( \frac{\mathrm{P}_{1+\delta
}^{\alpha }\left( J,\nu \right) }{\left\vert J\right\vert ^{\frac{1}{n}}}%
\right) ^{2}\left\Vert \mathbf{x}-\mathbf{m}_{J}\right\Vert _{L^{2}\left( 
\mathbf{1}_{J}\omega \right) }^{2}\ , \\
\mathbf{m}_{J} &\equiv &\mathbb{E}_{J}^{\omega }\mathbf{x}=\frac{1}{%
\left\vert J\right\vert _{\omega }}\int_{J}\mathbf{x}d\omega .
\end{eqnarray*}
\end{lemma}

\begin{proof}
Let $\left\{ h_{J}^{\omega ,a}\right\} _{a\in \Gamma }$ be an orthonormal
basis of $L_{J}^{2}\left( \mu \right) $ as in a previous subsection. Now we
use the smoothness estimate (\ref{sizeandsmoothness'}), together with
Taylor's formula and the vanishing mean of the quasiHaar functions $%
h_{J}^{\omega ,a}$ and $\mathbf{m}_{J}\equiv \frac{1}{\left\vert
J\right\vert _{\mu }}\int_{J}\mathbf{x}d\mu \left( x\right) \in J$, to
obtain 
\begin{eqnarray*}
\left\vert \left\langle T^{\alpha }\mu ,h_{J}^{\omega ,a}\right\rangle
_{\omega }\right\vert &=&\left\vert \int \left\{ \int K^{\alpha }\left(
x,y\right) h_{J}^{\omega ,a}\left( x\right) d\omega \left( x\right) \right\}
d\mu \left( y\right) \right\vert =\left\vert \int \left\langle K_{y}^{\alpha
},h_{J}^{\omega ,a}\right\rangle _{\omega }d\mu \left( y\right) \right\vert
\\
&=&\left\vert \int \left\langle K_{y}^{\alpha }\left( x\right)
-K_{y}^{\alpha }\left( \mathbf{m}_{J}\right) ,h_{J}^{\omega ,a}\right\rangle
_{\omega }d\mu \left( y\right) \right\vert \\
&\leq &\left\vert \left\langle \left[ \int \nabla K_{y}^{\alpha }\left( 
\mathbf{m}_{J}\right) d\mu \left( y\right) \right] \left( \mathbf{x}-\mathbf{%
m}_{J}\right) ,h_{J}^{\omega ,a}\right\rangle _{\omega }\right\vert \\
&&+\left\langle \left[ \int \sup_{\mathbf{\theta }_{J}\in J}\left\vert
\nabla K_{y}^{\alpha }\left( \mathbf{\theta }_{J}\right) -\nabla
K_{y}^{\alpha }\left( \mathbf{m}_{J}\right) \right\vert d\mu \left( y\right) %
\right] \left\vert \mathbf{x}-\mathbf{m}_{J}\right\vert ,\left\vert
h_{J}^{\omega ,a}\right\vert \right\rangle _{\omega } \\
&\lesssim &C_{CZ}\frac{\mathrm{P}^{\alpha }\left( J,\left\vert \mu
\right\vert \right) }{\left\vert J\right\vert ^{\frac{1}{n}}}\left\Vert
\bigtriangleup _{J}^{\omega }\mathbf{x}\right\Vert _{L^{2}\left( \omega
\right) }+C_{CZ}\frac{\mathrm{P}_{1+\delta }^{\alpha }\left( J,\left\vert
\mu \right\vert \right) }{\left\vert J\right\vert ^{\frac{1}{n}}}\left\Vert 
\mathbf{x}-\mathbf{m}_{J}\right\Vert _{L^{2}\left( \mathbf{1}_{J}\omega
\right) }.
\end{eqnarray*}
\end{proof}

\subsection{The Energy Lemma}

Suppose now we are given a subset $\mathcal{H}$ of the dyadic quasigrid $%
\Omega \mathcal{D}^{\omega }$. Let $\mathsf{P}_{\mathcal{H}}^{\omega
}=\sum_{J\in \mathcal{H}}\bigtriangleup _{J}^{\omega }$ be the corresponding 
$\omega $-quasiHaar projection. We define $\mathcal{H}^{\ast }\equiv
\dbigcup\limits_{J\in \mathcal{H}}\left\{ J^{\prime }\in \Omega \mathcal{D}%
^{\omega }:J^{\prime }\subset J\right\} $.

\begin{lemma}[\textbf{Energy Lemma}]
\label{ener}Let $J\ $be a quasicube in $\Omega \mathcal{D}^{\omega }$. Let $%
\Psi _{J}$ be an $L^{2}\left( \omega \right) $ function supported in $J$ and
with $\omega $-integral zero, and denote its quasiHaar support by $\mathcal{H%
}=\limfunc{supp}\widehat{\Psi _{J}}$. Let $\nu $ be a positive measure
supported in $\mathbb{R}^{n}\setminus \gamma J$ with $\gamma \geq 2$, and
for each $J^{\prime }\in \mathcal{H}$, let $d\nu _{J^{\prime }}=\varphi
_{J^{\prime }}d\nu $ with $\left\vert \varphi _{J^{\prime }}\right\vert \leq
1$. Let $T^{\alpha }$ be a standard $\alpha $-fractional singular integral
operator with $0\leq \alpha <n$. Then with $\delta ^{\prime }=\frac{\delta }{%
2}$ we have%
\begin{eqnarray*}
\left\vert \sum_{J^{\prime }\in \mathcal{H}}\left\langle T^{\alpha }\left(
\nu _{J^{\prime }}\right) ,\bigtriangleup _{J^{\prime }}^{\omega }\Psi
_{J}\right\rangle _{\omega }\right\vert &\lesssim &\left\Vert \Psi
_{J}\right\Vert _{L^{2}\left( \omega \right) }\left( \frac{\mathrm{P}%
^{\alpha }\left( J,\nu \right) }{\left\vert J\right\vert ^{\frac{1}{n}}}%
\right) \left\Vert \mathsf{P}_{\mathcal{H}}^{\omega }\mathbf{x}\right\Vert
_{L^{2}\left( \omega \right) } \\
&&+\left\Vert \Psi _{J}\right\Vert _{L^{2}\left( \omega \right) }\frac{1}{%
\gamma ^{\delta ^{\prime }}}\left( \frac{\mathrm{P}_{1+\delta ^{\prime
}}^{\alpha }\left( J,\nu \right) }{\left\vert J\right\vert ^{\frac{1}{n}}}%
\right) \left\Vert \mathsf{P}_{\mathcal{H}^{\ast }}^{\omega }\mathbf{x}%
\right\Vert _{L^{2}\left( \omega \right) } \\
&\lesssim &\left\Vert \Psi _{J}\right\Vert _{L^{2}\left( \omega \right)
}\left( \frac{\mathrm{P}^{\alpha }\left( J,\nu \right) }{\left\vert
J\right\vert ^{\frac{1}{n}}}\right) \left\Vert \mathsf{P}_{\mathcal{H}^{\ast
}}^{\omega }\mathbf{x}\right\Vert _{L^{2}\left( \omega \right) },
\end{eqnarray*}%
and in particular the `pivotal' bound%
\begin{equation*}
\left\vert \left\langle T^{\alpha }\left( \nu \right) ,\Psi
_{J}\right\rangle _{\omega }\right\vert \leq C\left\Vert \Psi
_{J}\right\Vert _{L^{2}\left( \omega \right) }\mathrm{P}^{\alpha }\left(
J,\left\vert \nu \right\vert \right) \sqrt{\left\vert J\right\vert _{\omega }%
}\ .
\end{equation*}
\end{lemma}

\begin{remark}
The first term on the right side of the energy inequality above is the `big'
Poisson integral $\mathrm{P}^{\alpha }$ times the `small' energy term $%
\left\Vert \mathsf{P}_{\mathcal{H}}^{\omega }\mathbf{x}\right\Vert
_{L^{2}\left( \omega \right) }^{2}$ that is additive in $\mathcal{H}$, while
the second term on the right is the `small' Poisson integral $\mathrm{P}%
_{1+\delta ^{\prime }}^{\alpha }$ times the `big' energy term $\left\Vert 
\mathsf{P}_{\mathcal{H}^{\ast }}^{\omega }\mathbf{x}\right\Vert
_{L^{2}\left( \omega \right) }$ that is no longer additive in $\mathcal{H}$.
The first term presents no problems in subsequent analysis due solely to the
additivity of\ the `small' energy term. It is the second term that must be
handled by special methods. For example, in the Intertwining Proposition
below, the interaction of the singular integral occurs with a pair of
quasicubes $J\subset I$ at \emph{highly separated} levels, where the
goodness of $J$ can exploit the decay $\delta ^{\prime }$ in the kernel of
the `small' Poisson integral $\mathrm{P}_{1+\delta ^{\prime }}^{\alpha }$
relative to the `big' Poisson integral $\mathrm{P}^{\alpha }$, and results
in a bound directly by the quasienergy condition. On the other hand, in the
local recursion of M. Lacey at the end of the \ paper, the separation of
levels in the pairs $J\subset I$ can be as \emph{little} as a fixed
parameter $\mathbf{\rho }$, and here we must first separate the stopping
form into two sublinear forms that involve the two estimates respectively.
The form corresponding to the smaller Poisson integral $\mathrm{P}_{1+\delta
^{\prime }}^{\alpha }$ is again handled using goodness and the decay $\delta
^{\prime }$ in the kernel, while the form corresponding to the larger
Poisson integral $\mathrm{P}^{\alpha }$ requires the stopping time and
recursion argument of M. Lacey.
\end{remark}

\begin{proof}
Using the Monotonicity Lemma \ref{mono}, followed by $\left\vert \nu
_{J^{\prime }}\right\vert \leq \nu $ and the Poisson equivalence 
\begin{equation}
\frac{\mathrm{P}_{m}^{\alpha }\left( J^{\prime },\nu \right) }{\left\vert
J^{\prime }\right\vert ^{\frac{m}{n}}}\approx \frac{\mathrm{P}_{m}^{\alpha
}\left( J,\nu \right) }{\left\vert J\right\vert ^{\frac{m}{n}}},\ \ \ \ \
J^{\prime }\subset J\subset 2J,\ \ \ \limfunc{supp}\nu \cap 2J=\emptyset ,
\label{Poisson equiv}
\end{equation}%
we have%
\begin{eqnarray*}
&&\left\vert \sum_{J^{\prime }\in \mathcal{H}}\left\langle T^{\alpha }\left(
\nu _{J^{\prime }}\right) ,\bigtriangleup _{J^{\prime }}^{\omega }\Psi
_{J}\right\rangle _{\omega }\right\vert =\left\vert \sum_{J^{\prime }\in 
\mathcal{H}}\left\langle \bigtriangleup _{J^{\prime }}^{\omega }T^{\alpha
}\left( \nu _{J^{\prime }}\right) ,\bigtriangleup _{J^{\prime }}^{\omega
}\Psi _{J}\right\rangle _{\omega }\right\vert \\
&\lesssim &\sum_{J^{\prime }\in \mathcal{H}}\Phi ^{\alpha }\left( J^{\prime
},\left\vert \nu _{J^{\prime }}\right\vert \right) \left\Vert \bigtriangleup
_{J^{\prime }}^{\omega }\Psi _{J}\right\Vert _{L^{2}\left( \omega \right) }
\\
&\lesssim &\left( \sum_{J^{\prime }\in \mathcal{H}}\left( \frac{\mathrm{P}%
^{\alpha }\left( J^{\prime },\nu \right) }{\left\vert J^{\prime }\right\vert
^{\frac{1}{n}}}\right) ^{2}\left\Vert \bigtriangleup _{J^{\prime }}^{\omega }%
\mathbf{x}\right\Vert _{L^{2}\left( \omega \right) }^{2}\right) ^{\frac{1}{2}%
}\left( \sum_{J^{\prime }\in \mathcal{H}}\left\Vert \bigtriangleup
_{J^{\prime }}^{\omega }\Psi _{J}\right\Vert _{L^{2}\left( \omega \right)
}^{2}\right) ^{\frac{1}{2}} \\
&&+\left( \sum_{J^{\prime }\in \mathcal{H}}\left( \frac{\mathrm{P}_{1+\delta
}^{\alpha }\left( J^{\prime },\nu \right) }{\left\vert J^{\prime
}\right\vert ^{\frac{1}{n}}}\right) ^{2}\sum_{J^{\prime \prime }\subset
J^{\prime }}\left\Vert \bigtriangleup _{J^{\prime \prime }}^{\omega }\mathbf{%
x}\right\Vert _{L^{2}\left( \omega \right) }^{2}\right) ^{\frac{1}{2}}\left(
\sum_{J^{\prime }\in \mathcal{H}}\left\Vert \bigtriangleup _{J^{\prime
}}^{\omega }\Psi _{J}\right\Vert _{L^{2}\left( \omega \right) }^{2}\right) ^{%
\frac{1}{2}} \\
&\lesssim &\left( \frac{\mathrm{P}^{\alpha }\left( J,\nu \right) }{%
\left\vert J\right\vert ^{\frac{1}{n}}}\right) \left\Vert \mathsf{P}_{%
\mathcal{H}}^{\omega }\mathbf{x}\right\Vert _{L^{2}\left( \omega \right)
}\left\Vert \Psi _{J}\right\Vert _{L^{2}\left( \omega \right) }+\frac{1}{%
\gamma ^{\delta ^{\prime }}}\left( \frac{\mathrm{P}_{1+\delta ^{\prime
}}^{\alpha }\left( J,\nu \right) }{\left\vert J\right\vert ^{\frac{1}{n}}}%
\right) \left\Vert \mathsf{P}_{\mathcal{H}^{\ast }}^{\omega }\mathbf{x}%
\right\Vert _{L^{2}\left( \omega \right) }\left\Vert \Psi _{J}\right\Vert
_{L^{2}\left( \omega \right) }\ .
\end{eqnarray*}%
The last inequality follows from

\begin{eqnarray*}
&&\sum_{J^{\prime }\in \mathcal{H}}\left( \frac{\mathrm{P}_{1+\delta
}^{\alpha }\left( J^{\prime },\nu \right) }{\left\vert J^{\prime
}\right\vert ^{\frac{1}{n}}}\right) ^{2}\sum_{J^{\prime \prime }\subset
J^{\prime }}\left\Vert \bigtriangleup _{J^{\prime \prime }}^{\omega }\mathbf{%
x}\right\Vert _{L^{2}\left( \omega \right) }^{2} \\
&=&\sum_{J^{\prime \prime }\subset J}\left\{ \sum_{J^{\prime }:\ J^{\prime
\prime }\subset J^{\prime }\subset J}\left( \frac{\mathrm{P}_{1+\delta
}^{\alpha }\left( J^{\prime },\nu \right) }{\left\vert J^{\prime
}\right\vert ^{\frac{1}{n}}}\right) ^{2}\right\} \left\Vert \bigtriangleup
_{J^{\prime \prime }}^{\omega }\mathbf{x}\right\Vert _{L^{2}\left( \omega
\right) }^{2} \\
&\lesssim &\frac{1}{\gamma ^{2\delta ^{\prime }}}\sum_{J^{\prime \prime }\in 
\mathcal{H}^{\ast }}\left( \frac{\mathrm{P}_{1+\delta ^{\prime }}^{\alpha
}\left( J^{\prime \prime },\nu \right) }{\left\vert J^{\prime \prime
}\right\vert ^{\frac{1}{n}}}\right) ^{2}\left\Vert \bigtriangleup
_{J^{\prime \prime }}^{\omega }\mathbf{x}\right\Vert _{L^{2}\left( \omega
\right) }^{2}\ ,
\end{eqnarray*}

which in turn follows from (recalling $\delta =2\delta ^{\prime }$ and using 
$\left\vert J^{\prime }\right\vert ^{\frac{1}{n}}+\left\vert y-c_{J^{\prime
}}\right\vert \approx \left\vert J\right\vert ^{\frac{1}{n}}+\left\vert
y-c_{J}\right\vert $ and $\frac{\left\vert J\right\vert ^{\frac{1}{n}}}{%
\left\vert J\right\vert ^{\frac{1}{n}}+\left\vert y-c_{J}\right\vert }\leq 
\frac{1}{\gamma }$ for $y\in \mathbb{R}^{n}\setminus \gamma J$)%
\begin{eqnarray*}
&&\sum_{J^{\prime }:\ J^{\prime \prime }\subset J^{\prime }\subset J}\left( 
\frac{\mathrm{P}_{1+\delta }^{\alpha }\left( J^{\prime },\nu \right) }{%
\left\vert J^{\prime }\right\vert ^{\frac{1}{n}}}\right) ^{2} \\
&=&\sum_{J^{\prime }:\ J^{\prime \prime }\subset J^{\prime }\subset
J}\left\vert J^{\prime }\right\vert ^{\frac{2\delta }{n}}\left( \int_{%
\mathbb{R}^{n}\setminus \gamma J}\frac{1}{\left( \left\vert J^{\prime
}\right\vert ^{\frac{1}{n}}+\left\vert y-c_{J^{\prime }}\right\vert \right)
^{n+1+\delta -\alpha }}d\nu \left( y\right) \right) ^{2} \\
&\lesssim &\sum_{J^{\prime }:\ J^{\prime \prime }\subset J^{\prime }\subset
J}\frac{1}{\gamma ^{2\delta ^{\prime }}}\frac{\left\vert J^{\prime
}\right\vert ^{\frac{2\delta }{n}}}{\left\vert J\right\vert ^{\frac{2\delta 
}{n}}}\left( \int_{\mathbb{R}^{n}\setminus \gamma J}\frac{\left\vert
J\right\vert ^{\frac{\delta ^{\prime }}{n}}}{\left( \left\vert J\right\vert
^{\frac{1}{n}}+\left\vert y-c_{J}\right\vert \right) ^{n+1+\delta ^{\prime
}-\alpha }}d\nu \left( y\right) \right) ^{2} \\
&=&\frac{1}{\gamma ^{2\delta ^{\prime }}}\left( \sum_{J^{\prime }:\
J^{\prime \prime }\subset J^{\prime }\subset J}\frac{\left\vert J^{\prime
}\right\vert ^{\frac{2\delta }{n}}}{\left\vert J\right\vert ^{\frac{2\delta 
}{n}}}\right) \left( \frac{\mathrm{P}_{1+\delta ^{\prime }}^{\alpha }\left(
J,\nu \right) }{\left\vert J\right\vert ^{\frac{1}{n}}}\right) ^{2}\lesssim 
\frac{1}{\gamma ^{2\delta ^{\prime }}}\left( \frac{\mathrm{P}_{1+\delta
^{\prime }}^{\alpha }\left( J,\nu \right) }{\left\vert J\right\vert ^{\frac{1%
}{n}}}\right) ^{2}.
\end{eqnarray*}%
Finally we have the `pivotal' bound from (\ref{Poisson equiv}) and%
\begin{equation*}
\sum_{J^{\prime \prime }\subset J}\left\Vert \bigtriangleup _{J^{\prime
\prime }}^{\omega }\mathbf{x}\right\Vert _{L^{2}\left( \omega \right)
}^{2}=\left\Vert \mathbf{x}-\mathbf{m}_{J}\right\Vert _{L^{2}\left( \mathbf{1%
}_{J}\omega \right) }^{2}\leq \left\vert J\right\vert ^{\frac{2}{n}%
}\left\vert J\right\vert _{\omega }\ .
\end{equation*}
\end{proof}

\section{Preliminaries of NTV type}

An important reduction of our theorem is delivered by the following two
lemmas, that in the case of one dimension are due to Nazarov, Treil and
Volberg (see \cite{NTV3} and \cite{Vol}). The proofs given there do not
extend in standard ways to higher dimensions with common point masses, and
we use the quasiweak boundedness property to handle the case of touching
quasicubes, and an application of Schur's Lemma to handle the case of
separated quasicubes. The first lemma below is Lemmas 8.1 and 8.7 in \cite%
{LaWi} but with the larger constant $\mathcal{A}_{2}^{\alpha }$ there in
place of the smaller constant $A_{2}^{\alpha }$ here. We emphasize that only
the offset $A_{2}^{\alpha }$ condition is needed with testing and weak
boundedness in these preliminary estimates.

\begin{lemma}
\label{standard delta}Suppose $T^{\alpha }$ is a standard fractional
singular integral with $0\leq \alpha <n$, and that all of the quasicubes $%
I\in \Omega \mathcal{D}^{\sigma },J\in \Omega \mathcal{D}^{\omega }$ below
are good with goodness parameters $\varepsilon $ and $\mathbf{r}$. Fix a
positive integer $\mathbf{\rho }>\mathbf{r}$. For $f\in L^{2}\left( \sigma
\right) $ and $g\in L^{2}\left( \omega \right) $ we have%
\begin{equation}
\sum_{\substack{ \left( I,J\right) \in \Omega \mathcal{D}^{\sigma }\times
\Omega \mathcal{D}^{\omega }  \\ 2^{-\mathbf{\rho }}\ell \left( I\right)
\leq \ell \left( J\right) \leq 2^{\mathbf{\rho }}\ell \left( I\right) }}%
\left\vert \left\langle T_{\sigma }^{\alpha }\left( \bigtriangleup
_{I}^{\sigma }f\right) ,\bigtriangleup _{J}^{\omega }g\right\rangle _{\omega
}\right\vert \lesssim \left( \mathfrak{T}_{\alpha }+\mathfrak{T}_{\alpha
}^{\ast }+\mathcal{WBP}_{T^{\alpha }}+\sqrt{A_{2}^{\alpha }}\right)
\left\Vert f\right\Vert _{L^{2}\left( \sigma \right) }\left\Vert
g\right\Vert _{L^{2}\left( \omega \right) }  \label{delta near}
\end{equation}%
and 
\begin{equation}
\sum_{\substack{ \left( I,J\right) \in \Omega \mathcal{D}^{\sigma }\times
\Omega \mathcal{D}^{\omega }  \\ I\cap J=\emptyset \text{ and }\frac{\ell
\left( J\right) }{\ell \left( I\right) }\notin \left[ 2^{-\mathbf{\rho }},2^{%
\mathbf{\rho }}\right] }}\left\vert \left\langle T_{\sigma }^{\alpha }\left(
\bigtriangleup _{I}^{\sigma }f\right) ,\bigtriangleup _{J}^{\omega
}g\right\rangle _{\omega }\right\vert \lesssim \sqrt{A_{2}^{\alpha }}%
\left\Vert f\right\Vert _{L^{2}\left( \sigma \right) }\left\Vert
g\right\Vert _{L^{2}\left( \omega \right) }.  \label{delta far}
\end{equation}
\end{lemma}

\begin{lemma}
\label{standard indicator}Suppose $T^{\alpha }$ is a standard fractional
singular integral with $0\leq \alpha <n$, that all of the quasicubes $I\in
\Omega \mathcal{D}^{\sigma },J\in \Omega \mathcal{D}^{\omega }$ below are
good, that $\mathbf{\rho }>\mathbf{r}$, that $f\in L^{2}\left( \sigma
\right) $ and $g\in L^{2}\left( \omega \right) $, that $\mathcal{F}\subset
\Omega \mathcal{D}^{\sigma }$ and $\mathcal{G}\subset \Omega \mathcal{D}%
^{\omega }$ are $\sigma $-Carleson and $\omega $-Carleson collections
respectively, i.e.,%
\begin{equation*}
\sum_{F^{\prime }\in \mathcal{F}:\ F^{\prime }\subset F}\left\vert F^{\prime
}\right\vert _{\sigma }\lesssim \left\vert F\right\vert _{\sigma },\ \ \ \ \
F\in \mathcal{F},\text{ and }\sum_{G^{\prime }\in \mathcal{G}:\ G^{\prime
}\subset G}\left\vert G^{\prime }\right\vert _{\omega }\lesssim \left\vert
G\right\vert _{\omega },\ \ \ \ \ G\in \mathcal{G},
\end{equation*}%
that there are numerical sequences $\left\{ \alpha _{\mathcal{F}}\left(
F\right) \right\} _{F\in \mathcal{F}}$ and $\left\{ \beta _{\mathcal{G}%
}\left( G\right) \right\} _{G\in \mathcal{G}}$ such that%
\begin{equation}
\sum_{F\in \mathcal{F}}\alpha _{\mathcal{F}}\left( F\right) ^{2}\left\vert
F\right\vert _{\sigma }\leq \left\Vert f\right\Vert _{L^{2}\left( \sigma
\right) }^{2}\text{ and }\sum_{G\in \mathcal{G}}\beta _{\mathcal{G}}\left(
G\right) ^{2}\left\vert G\right\vert _{\sigma }\leq \left\Vert g\right\Vert
_{L^{2}\left( \sigma \right) }^{2}\ ,  \label{qo}
\end{equation}%
and finally that for each pair of quasicubes $\left( I,J\right) \in \Omega 
\mathcal{D}^{\sigma }\times \Omega \mathcal{D}^{\omega }$, there are bounded
functions $\beta _{I,J}$ and $\gamma _{I,J}$ supported in $I\setminus 2J$
and $J\setminus 2I$ respectively, satisfying%
\begin{equation*}
\left\Vert \beta _{I,J}\right\Vert _{\infty },\left\Vert \gamma
_{I,J}\right\Vert _{\infty }\leq 1.
\end{equation*}%
Then%
\begin{eqnarray}
&&\sum_{\substack{ \left( F,J\right) \in \mathcal{F}\times \Omega \mathcal{D}%
^{\omega }  \\ F\cap J=\emptyset \text{ and }\ell \left( J\right) \leq 2^{-%
\mathbf{\rho }}\ell \left( F\right) }}\left\vert \left\langle T_{\sigma
}^{\alpha }\left( \beta _{F,J}\mathbf{1}_{F}\alpha _{\mathcal{F}}\left(
F\right) \right) ,\bigtriangleup _{J}^{\omega }g\right\rangle _{\omega
}\right\vert  \label{indicator far} \\
&&+\sum_{\substack{ \left( I,G\right) \in \Omega \mathcal{D}^{\sigma }\times 
\mathcal{G}  \\ I\cap G=\emptyset \text{ and }\ell \left( I\right) \leq 2^{-%
\mathbf{\rho }}\ell \left( G\right) }}\left\vert \left\langle T_{\sigma
}^{\alpha }\left( \bigtriangleup _{I}^{\sigma }f\right) ,\gamma _{I,G}%
\mathbf{1}_{G}\beta _{\mathcal{G}}\left( G\right) \right\rangle _{\omega
}\right\vert  \notag \\
&\lesssim &\sqrt{A_{2}^{\alpha }}\left\Vert f\right\Vert _{L^{2}\left(
\sigma \right) }\left\Vert g\right\Vert _{L^{2}\left( \omega \right) }. 
\notag
\end{eqnarray}
\end{lemma}

\begin{remark}
If $\mathcal{F}$ and $\mathcal{G}$ are $\sigma $-Carleson and $\omega $%
-Carleson collections respectively, and if $\alpha _{\mathcal{F}}\left(
F\right) =\mathbb{E}_{F}^{\sigma }\left\vert f\right\vert $ and $\beta _{%
\mathcal{G}}\left( G\right) =\mathbb{E}_{G}^{\omega }\left\vert g\right\vert 
$, then the `quasi' orthogonality condition (\ref{qo}) holds (here `quasi'
has a different meaning than quasi), and this special case of Lemma \ref%
{standard indicator} serves as a basic example.
\end{remark}

\begin{remark}
Lemmas \ref{standard delta} and \ref{standard indicator} differ mainly in
that an orthogonal collection of quasiHaar projections is replaced by a
`quasi' orthogonal collection of indicators $\left\{ \mathbf{1}_{F}\alpha _{%
\mathcal{F}}\left( F\right) \right\} _{F\in \mathcal{F}}$. More precisely,
the main difference between (\ref{delta far}) and (\ref{indicator far}) is
that a quasiHaar projection $\bigtriangleup _{I}^{\sigma }f$ or $%
\bigtriangleup _{J}^{\omega }g$ has been replaced with a constant multiple
of an indicator $\mathbf{1}_{F}\alpha _{\mathcal{F}}\left( F\right) $ or $%
\mathbf{1}_{G}\beta _{\mathcal{G}}\left( G\right) $, and in addition, a
bounded function is permitted to multiply the indicator of the quasicube
having larger sidelength.
\end{remark}

\begin{proof}
Note that in (\ref{delta near}) we have used the parameter $\mathbf{\rho }$
in the exponent rather than $\mathbf{r}$, and this is possible because the
arguments we use here only require that there are finitely many levels of
scale separating $I$ and $J$. To handle this term we first decompose it into%
\begin{eqnarray*}
&&\left\{ \sum_{\substack{ \left( I,J\right) \in \Omega \mathcal{D}^{\sigma
}\times \Omega \mathcal{D}^{\omega }:\ J\subset 3I  \\ 2^{-\mathbf{\rho }%
}\ell \left( I\right) \leq \ell \left( J\right) \leq 2^{\mathbf{\rho }}\ell
\left( I\right) }}+\sum_{\substack{ \left( I,J\right) \in \Omega \mathcal{D}%
^{\sigma }\times \Omega \mathcal{D}^{\omega }:\ I\subset 3J  \\ 2^{-\mathbf{%
\rho }}\ell \left( I\right) \leq \ell \left( J\right) \leq 2^{\mathbf{\rho }%
}\ell \left( I\right) }}+\sum_{\substack{ \left( I,J\right) \in \Omega 
\mathcal{D}^{\sigma }\times \Omega \mathcal{D}^{\omega }  \\ 2^{-\mathbf{%
\rho }}\ell \left( I\right) \leq \ell \left( J\right) \leq 2^{\mathbf{\rho }%
}\ell \left( I\right)  \\ J\not\subset 3I\text{ and }I\not\subset 3J}}%
\right\} \left\vert \left\langle T_{\sigma }^{\alpha }\left( \bigtriangleup
_{I}^{\sigma }f\right) ,\bigtriangleup _{J}^{\omega }g\right\rangle _{\omega
}\right\vert \\
&&\ \ \ \ \ \ \ \ \ \ \equiv A_{1}+A_{2}+A_{3}.
\end{eqnarray*}%
The proof of the bound for term $A_{3}$ is similar to that of the bound for
the left side of (\ref{delta far}), and so we will defer the bound for $%
A_{3} $ until after (\ref{delta far}) has been proved.

We now consider term $A_{1}$ as term $A_{2}$ is symmetric. To handle this
term we will write the quasiHaar functions $h_{I}^{\sigma }$ and $%
h_{J}^{\omega }$ as linear combinations of the indicators of the children of
their supporting quasicubes, denoted $I_{\theta }$ and $J_{\theta ^{\prime
}} $ respectively. Then we use the quasitesting condition on $I_{\theta }$
and $J_{\theta ^{\prime }}$ when they \emph{overlap}, i.e. their interiors
intersect; we use the quasiweak boundedness property on $I_{\theta }$ and $%
J_{\theta ^{\prime }}$ when they \emph{touch}, i.e. their interiors are
disjoint but their closures intersect (even in just a point); and finally we
use the $A_{2}^{\alpha }$ condition when $I_{\theta }$ and $J_{\theta
^{\prime }}$ are \emph{separated}, i.e. their closures are disjoint. We will
suppose initially that the side length of $J$ is at most the side length $I$%
, i.e. $\ell \left( J\right) \leq \ell \left( I\right) $, the proof for $%
J=\pi I$ being similar but for one point mentioned below. So suppose that $%
I_{\theta }$ is a child of $I$ and that $J_{\theta ^{\prime }}$ is a child
of $J$. If $J_{\theta ^{\prime }}\subset I_{\theta }$ we have from (\ref%
{useful Haar}) that, 
\begin{eqnarray*}
\left\vert \left\langle T_{\sigma }^{\alpha }\left( \mathbf{1}_{I_{\theta
}}\bigtriangleup _{I}^{\sigma }f\right) ,\mathbf{1}_{J_{\theta ^{\prime
}}}\bigtriangleup _{J}^{\omega }g\right\rangle _{\omega }\right\vert
&\lesssim &\sup_{a,a^{\prime }\in \Gamma _{n}}\frac{\left\vert \left\langle
f,h_{I}^{\sigma ,a}\right\rangle _{\sigma }\right\vert }{\sqrt{\left\vert
I_{\theta }\right\vert _{\sigma }}}\left\vert \left\langle T_{\sigma
}^{\alpha }\left( \mathbf{1}_{I_{\theta }}\right) ,\mathbf{1}_{J_{\theta
^{\prime }}}\right\rangle _{\omega }\right\vert \frac{\left\vert
\left\langle g,h_{J}^{\omega ,a^{\prime }}\right\rangle _{\omega
}\right\vert }{\sqrt{\left\vert J_{\theta ^{\prime }}\right\vert _{\omega }}}
\\
&\lesssim &\sup_{a,a^{\prime }\in \Gamma _{n}}\frac{\left\vert \left\langle
f,h_{I}^{\sigma ,a}\right\rangle _{\sigma }\right\vert }{\sqrt{\left\vert
I_{\theta }\right\vert _{\sigma }}}\left( \int_{J_{\theta ^{\prime
}}}\left\vert T_{\sigma }^{\alpha }\left( \mathbf{1}_{I_{\theta }}\right)
\right\vert ^{2}d\omega \right) ^{\frac{1}{2}}\left\vert \left\langle
g,h_{J}^{\omega ,a^{\prime }}\right\rangle _{\omega }\right\vert \\
&\lesssim &\sup_{a,a^{\prime }\in \Gamma _{n}}\frac{\left\vert \left\langle
f,h_{I}^{\sigma ,a}\right\rangle _{\sigma }\right\vert }{\sqrt{\left\vert
I_{\theta }\right\vert _{\sigma }}}\mathfrak{T}_{T_{\alpha }}\left\vert
I_{\theta }\right\vert _{\sigma }^{\frac{1}{2}}\left\vert \left\langle
g,h_{J}^{\omega ,a^{\prime }}\right\rangle _{\omega }\right\vert \\
&\lesssim &\sup_{a,a^{\prime }\in \Gamma _{n}}\mathfrak{T}_{T_{\alpha
}}\left\vert \left\langle f,h_{I}^{\sigma ,a}\right\rangle _{\sigma
}\right\vert \left\vert \left\langle g,h_{J}^{\omega ,a^{\prime
}}\right\rangle _{\omega }\right\vert .
\end{eqnarray*}%
The point referred to above is that when $J=\pi I$ we write $\left\langle
T_{\sigma }^{\alpha }\left( \mathbf{1}_{I_{\theta }}\right) ,\mathbf{1}%
_{J_{\theta ^{\prime }}}\right\rangle _{\omega }=\left\langle \mathbf{1}%
_{I_{\theta }},T_{\omega }^{\alpha ,\ast }\left( \mathbf{1}_{J_{\theta
^{\prime }}}\right) \right\rangle _{\sigma }$ and get the dual quasitesting
constant $\mathfrak{T}_{T_{\alpha }}^{\ast }$. If $J_{\theta ^{\prime }}$
and $I_{\theta }$ touch, then $\ell \left( J_{\theta ^{\prime }}\right) \leq
\ell \left( I_{\theta }\right) $ and we have $J_{\theta ^{\prime }}\subset
3I_{\theta }\setminus I_{\theta }$, and so%
\begin{eqnarray*}
\left\vert \left\langle T_{\sigma }^{\alpha }\left( \mathbf{1}_{I_{\theta
}}\bigtriangleup _{I}^{\sigma }f\right) ,\mathbf{1}_{J_{\theta ^{\prime
}}}\bigtriangleup _{J}^{\omega }g\right\rangle _{\omega }\right\vert
&\lesssim &\sup_{a,a^{\prime }\in \Gamma _{n}}\frac{\left\vert \left\langle
f,h_{I}^{\sigma ,a}\right\rangle _{\sigma }\right\vert }{\sqrt{\left\vert
I_{\theta }\right\vert _{\sigma }}}\left\vert \left\langle T_{\sigma
}^{\alpha }\left( \mathbf{1}_{I_{\theta }}\right) ,\mathbf{1}_{J_{\theta
^{\prime }}}\right\rangle _{\omega }\right\vert \frac{\left\vert
\left\langle g,h_{J}^{\omega ,a^{\prime }}\right\rangle _{\omega
}\right\vert }{\sqrt{\left\vert J_{\theta ^{\prime }}\right\vert _{\omega }}}
\\
&\lesssim &\sup_{a,a^{\prime }\in \Gamma _{n}}\frac{\left\vert \left\langle
f,h_{I}^{\sigma ,a}\right\rangle _{\sigma }\right\vert }{\sqrt{\left\vert
I_{\theta }\right\vert _{\sigma }}}\mathcal{WBP}_{T^{\alpha }}\sqrt{%
\left\vert I_{\theta }\right\vert _{\sigma }\left\vert J_{\theta ^{\prime
}}\right\vert _{\omega }}\frac{\left\vert \left\langle g,h_{J}^{\omega
,a^{\prime }}\right\rangle _{\omega }\right\vert }{\sqrt{\left\vert
J_{\theta ^{\prime }}\right\vert _{\omega }}} \\
&=&\sup_{a,a^{\prime }\in \Gamma _{n}}\mathcal{WBP}_{T^{\alpha }}\left\vert
\left\langle f,h_{I}^{\sigma ,a}\right\rangle _{\sigma }\right\vert
\left\vert \left\langle g,h_{J}^{\omega ,a^{\prime }}\right\rangle _{\omega
}\right\vert .
\end{eqnarray*}%
Finally, if $J_{\theta ^{\prime }}$ and $I_{\theta }$ are separated, and if $%
K$ is the smallest (not necessarily dyadic) quasicube containing both $%
J_{\theta ^{\prime }}$ and $I_{\theta }$, then $\limfunc{dist}\left(
I_{\theta },J_{\theta ^{\prime }}\right) \approx \left\vert K\right\vert ^{%
\frac{1}{n}}$ and we have%
\begin{eqnarray*}
\left\vert \left\langle T_{\sigma }^{\alpha }\left( \mathbf{1}_{I_{\theta
}}\bigtriangleup _{I}^{\sigma }f\right) ,\mathbf{1}_{J_{\theta ^{\prime
}}}\bigtriangleup _{J}^{\omega }g\right\rangle _{\omega }\right\vert
&\lesssim &\sup_{a,a^{\prime }\in \Gamma _{n}}\frac{\left\vert \left\langle
f,h_{I}^{\sigma ,a}\right\rangle _{\sigma }\right\vert }{\sqrt{\left\vert
I_{\theta }\right\vert _{\sigma }}}\left\vert \left\langle T_{\sigma
}^{\alpha }\left( \mathbf{1}_{I_{\theta }}\right) ,\mathbf{1}_{J_{\theta
^{\prime }}}\right\rangle _{\omega }\right\vert \frac{\left\vert
\left\langle g,h_{J}^{\omega ,a^{\prime }}\right\rangle _{\omega
}\right\vert }{\sqrt{\left\vert J_{\theta ^{\prime }}\right\vert _{\omega }}}
\\
&\lesssim &\sup_{a,a^{\prime }\in \Gamma _{n}}\frac{\left\vert \left\langle
f,h_{I}^{\sigma ,a}\right\rangle _{\sigma }\right\vert }{\sqrt{\left\vert
I_{\theta }\right\vert _{\sigma }}}\frac{1}{\limfunc{dist}\left( I_{\theta
},J_{\theta ^{\prime }}\right) ^{n-\alpha }}\left\vert I_{\theta
}\right\vert _{\sigma }\left\vert J_{\theta ^{\prime }}\right\vert _{\omega }%
\frac{\left\vert \left\langle g,h_{J}^{\omega ,a^{\prime }}\right\rangle
_{\omega }\right\vert }{\sqrt{\left\vert J_{\theta ^{\prime }}\right\vert
_{\omega }}} \\
&=&\sup_{a,a^{\prime }\in \Gamma _{n}}\frac{\sqrt{\left\vert I_{\theta
}\right\vert _{\sigma }\left\vert J_{\theta ^{\prime }}\right\vert _{\omega }%
}}{\limfunc{dist}\left( I_{\theta },J_{\theta ^{\prime }}\right) ^{n-\alpha }%
}\left\vert \left\langle f,h_{I}^{\sigma ,a}\right\rangle _{\sigma
}\right\vert \left\vert \left\langle g,h_{J}^{\omega ,a^{\prime
}}\right\rangle _{\omega }\right\vert \\
&\lesssim &\sqrt{A_{2}^{\alpha }}\sup_{a,a^{\prime }\in \Gamma
_{n}}\left\vert \left\langle f,h_{I}^{\sigma ,a}\right\rangle _{\sigma
}\right\vert \left\vert \left\langle g,h_{J}^{\omega ,a^{\prime
}}\right\rangle _{\omega }\right\vert .
\end{eqnarray*}%
Now we sum over all the children of $J$ and $I$ satisfying $2^{-\mathbf{\rho 
}}\ell \left( I\right) \leq \ell \left( J\right) \leq 2^{\mathbf{\rho }}\ell
\left( I\right) $ for which $J\subset 3I$ to obtain that%
\begin{equation*}
A_{1}\lesssim \left( \mathfrak{T}_{T_{\alpha }}+\mathfrak{T}_{T_{\alpha
}}^{\ast }+\mathcal{WBP}_{T^{\alpha }}+\sqrt{A_{2}^{\alpha }}\right)
\sup_{a,a^{\prime }\in \Gamma _{n}}\sum_{\substack{ \left( I,J\right) \in
\Omega \mathcal{D}^{\sigma }\times \Omega \mathcal{D}^{\omega }:\ J\subset
3I  \\ 2^{-\mathbf{\rho }}\ell \left( I\right) \leq \ell \left( J\right)
\leq 2^{\mathbf{\rho }}\ell \left( I\right) }}\left\vert \left\langle
f,h_{I}^{\sigma ,a}\right\rangle _{\sigma }\right\vert \left\vert
\left\langle g,h_{J}^{\omega ,a^{\prime }}\right\rangle _{\omega
}\right\vert \ .
\end{equation*}%
Now Cauchy-Schwarz gives the estimate%
\begin{eqnarray*}
&&\sup_{a,a^{\prime }\in \Gamma _{n}}\sum_{\substack{ \left( I,J\right) \in
\Omega \mathcal{D}^{\sigma }\times \Omega \mathcal{D}^{\omega }:\ J\subset
3I  \\ 2^{-\mathbf{\rho }}\ell \left( I\right) \leq \ell \left( J\right)
\leq 2^{\mathbf{\rho }}\ell \left( I\right) }}\left\vert \left\langle
f,h_{I}^{\sigma ,a}\right\rangle _{\sigma }\right\vert \left\vert
\left\langle g,h_{J}^{\omega ,a^{\prime }}\right\rangle _{\omega }\right\vert
\\
&\leq &\sup_{a,a^{\prime }\in \Gamma _{n}}\left( \sum_{\substack{ \left(
I,J\right) \in \Omega \mathcal{D}^{\sigma }\times \Omega \mathcal{D}^{\omega
}:\ J\subset 3I  \\ 2^{-\mathbf{\rho }}\ell \left( I\right) \leq \ell \left(
J\right) \leq 2^{\mathbf{\rho }}\ell \left( I\right) }}\left\vert
\left\langle f,h_{I}^{\sigma }\right\rangle _{\sigma }\right\vert
^{2}\right) ^{\frac{1}{2}}\left( \sum_{\substack{ \left( I,J\right) \in
\Omega \mathcal{D}^{\sigma }\times \Omega \mathcal{D}^{\omega }:\ J\subset
3I  \\ 2^{-\mathbf{\rho }}\ell \left( I\right) \leq \ell \left( J\right)
\leq 2^{\mathbf{\rho }}\ell \left( I\right) }}\left\vert \left\langle
g,h_{J}^{\omega }\right\rangle _{\omega }\right\vert ^{2}\right) ^{\frac{1}{2%
}} \\
&\lesssim &\left\Vert f\right\Vert _{L^{2}\left( \sigma \right) }\left\Vert
g\right\Vert _{L^{2}\left( \omega \right) }\ ,
\end{eqnarray*}%
This completes our proof of (\ref{delta near}) save for the deferral of term 
$A_{3}$, which we bound below.

\bigskip

Now we turn to the sum of separated cubes in (\ref{delta far}) and (\ref%
{indicator far}). In each of these inequalities we have either orthogonality
or `quasi' orthogonality, due either to the presence of a quasiHaar
projection such as $\bigtriangleup _{I}^{\sigma }f$, or the presence of an
appropriate Carleson indicator such as $\beta _{F,J}\mathbf{1}_{F}\alpha _{%
\mathcal{F}}\left( F\right) $. We will prove below the estimate for the
separated sum corresponding to (\ref{delta far}). The corresponding
estimates for (\ref{indicator far}) are handled in a similar way, the only
difference being that the `quasi' orthogonality of Carleson indicators such
as $\beta _{F,J}\mathbf{1}_{F}\alpha _{\mathcal{F}}\left( F\right) $ is used
in place of the orthogonality of quasiHaar functions such as $\bigtriangleup
_{I}^{\sigma }f$. The bounded functions $\beta _{F,J}$ are replaced with
constants after an application of the energy lemma, and then the arguments
proceed as below.

We split the pairs $\left( I,J\right) \in \Omega \mathcal{D}^{\sigma }\times
\Omega \mathcal{D}^{\omega }$ occurring in (\ref{delta far}) into two
groups, those with side length of $J$ smaller than side length of $I$, and
those with side length of $I$ smaller than side length of $J$, treating only
the former case, the latter being symmetric. Thus we prove the following
bound:%
\begin{eqnarray*}
\mathcal{A}\left( f,g\right) &\equiv &\sum_{\substack{ \left( I,J\right) \in
\Omega \mathcal{D}^{\sigma }\times \Omega \mathcal{D}^{\omega }  \\ I\cap
J=\emptyset \text{ and }\ell \left( J\right) \leq 2^{-\mathbf{\rho }}\ell
\left( I\right) }}\left\vert \left\langle T_{\sigma }^{\alpha }\left(
\bigtriangleup _{I}^{\sigma }f\right) ,\bigtriangleup _{J}^{\omega
}g\right\rangle _{\omega }\right\vert \\
&\lesssim &\sqrt{A_{2}^{\alpha }}\left\Vert f\right\Vert _{L^{2}\left(
\sigma \right) }\left\Vert g\right\Vert _{L^{2}\left( \omega \right) }.
\end{eqnarray*}%
We apply the `pivotal' bound from the Energy Lemma \ref{ener} to estimate
the inner product $\left\langle T_{\sigma }^{\alpha }\left( \bigtriangleup
_{I}^{\sigma }f\right) ,\bigtriangleup _{J}^{\omega }g\right\rangle _{\omega
}$ and obtain,%
\begin{equation*}
\left\vert \left\langle T_{\sigma }^{\alpha }\left( \bigtriangleup
_{I}^{\sigma }f\right) ,\bigtriangleup _{J}^{\omega }g\right\rangle _{\omega
}\right\vert \lesssim \left\Vert \bigtriangleup _{J}^{\omega }g\right\Vert
_{L^{2}\left( \omega \right) }\mathrm{P}^{\alpha }\left( J,\left\vert
\bigtriangleup _{I}^{\sigma }f\right\vert \sigma \right) \sqrt{\left\vert
J\right\vert _{\omega }}\,,
\end{equation*}%
Denote by $\limfunc{dist}$ the $\ell ^{\infty }$ distance in $\mathbb{R}^{n}$%
: $\limfunc{dist}\left( x,y\right) =\max_{1\leq j\leq n}\left\vert
x_{j}-y_{j}\right\vert $, and denote the corresponding quasidistance by $%
\limfunc{qdist}\left( x,y\right) =\limfunc{dist}\left( \Omega ^{-1}x,\Omega
^{-1}y\right) $. We now estimate separately the long-range and mid-range
cases where $\limfunc{qdist}\left( J,I\right) \geq \ell \left( I\right) $
holds or not, and we decompose $\mathcal{A}$ accordingly:%
\begin{equation*}
\mathcal{A}\left( f,g\right) \equiv \mathcal{A}^{\limfunc{long}}\left(
f,g\right) +\mathcal{A}^{\limfunc{mid}}\left( f,g\right) .
\end{equation*}

\bigskip

\textbf{The long-range case}: We begin with the case where $\limfunc{qdist}%
\left( J,I\right) $ is at least $\ell \left( I\right) $, i.e. $J\cap
3I=\emptyset $. Since $J$ and $I$ are separated by at least $\max \left\{
\ell \left( J\right) ,\ell \left( I\right) \right\} $, we have the inequality%
\begin{equation*}
\mathrm{P}^{\alpha }\left( J,\left\vert \bigtriangleup _{I}^{\sigma
}f\right\vert \sigma \right) \approx \int_{I}\frac{\ell \left( J\right) }{%
\left\vert y-c_{J}\right\vert ^{n+1-\alpha }}\left\vert \bigtriangleup
_{I}^{\sigma }f\left( y\right) \right\vert d\sigma \left( y\right) \lesssim
\left\Vert \bigtriangleup _{I}^{\sigma }f\right\Vert _{L^{2}\left( \sigma
\right) }\frac{\ell \left( J\right) \sqrt{\left\vert I\right\vert _{\sigma }}%
}{\limfunc{qdist}\left( I,J\right) ^{n+1-\alpha }},
\end{equation*}%
since $\int_{I}\left\vert \bigtriangleup _{I}^{\sigma }f\left( y\right)
\right\vert d\sigma \left( y\right) \leq \left\Vert \bigtriangleup
_{I}^{\sigma }f\right\Vert _{L^{2}\left( \sigma \right) }\sqrt{\left\vert
I\right\vert _{\sigma }}$. Thus with $A\left( f,g\right) =\mathcal{A}^{%
\limfunc{long}}\left( f,g\right) $ we have%
\begin{eqnarray*}
A\left( f,g\right) &\lesssim &\sum_{I\in \Omega \mathcal{D}}\sum_{J\;:\;\ell
\left( J\right) \leq \ell \left( I\right) :\ \limfunc{qdist}\left(
I,J\right) \geq \ell \left( I\right) }\left\Vert \bigtriangleup _{I}^{\sigma
}f\right\Vert _{L^{2}\left( \sigma \right) }\left\Vert \bigtriangleup
_{J}^{\omega }g\right\Vert _{L^{2}\left( \omega \right) } \\
&&\ \ \ \ \ \ \ \ \ \ \ \ \ \ \ \times \frac{\ell \left( J\right) }{\limfunc{%
qdist}\left( I,J\right) ^{n+1-\alpha }}\sqrt{\left\vert I\right\vert
_{\sigma }}\sqrt{\left\vert J\right\vert _{\omega }} \\
&\equiv &\sum_{\left( I,J\right) \in \mathcal{P}}\left\Vert \bigtriangleup
_{I}^{\sigma }f\right\Vert _{L^{2}\left( \sigma \right) }\left\Vert
\bigtriangleup _{J}^{\omega }g\right\Vert _{L^{2}\left( \omega \right)
}A\left( I,J\right) ; \\
\text{with }A\left( I,J\right) &\equiv &\frac{\ell \left( J\right) }{%
\limfunc{qdist}\left( I,J\right) ^{n+1-\alpha }}\sqrt{\left\vert
I\right\vert _{\sigma }}\sqrt{\left\vert J\right\vert _{\omega }}; \\
\text{ and }\mathcal{P} &\equiv &\left\{ \left( I,J\right) \in \Omega 
\mathcal{D}\times \Omega \mathcal{D}:\ell \left( J\right) \leq \ell \left(
I\right) \text{ and }\limfunc{qdist}\left( I,J\right) \geq \ell \left(
I\right) \right\} .
\end{eqnarray*}%
Now let $\Omega \mathcal{D}_{N}\equiv \left\{ K\in \Omega \mathcal{D}:\ell
\left( K\right) =2^{N}\right\} $ for each $N\in \mathbb{Z}$. For $N\in 
\mathbb{Z}$ and $s\in \mathbb{Z}_{+}$, we further decompose $A\left(
f,g\right) $ by pigeonholing the sidelengths of $I$ and $J$ by $2^{N}$ and $%
2^{N-s}$ respectively: 
\begin{eqnarray*}
A\left( f,g\right) &=&\sum_{s=0}^{\infty }\sum_{N\in \mathbb{Z}%
}A_{N}^{s}\left( f,g\right) ; \\
A_{N}^{s}\left( f,g\right) &\equiv &\sum_{\left( I,J\right) \in \mathcal{P}%
_{N}^{s}}\left\Vert \bigtriangleup _{I}^{\sigma }f\right\Vert _{L^{2}\left(
\sigma \right) }\left\Vert \bigtriangleup _{J}^{\omega }g\right\Vert
_{L^{2}\left( \omega \right) }A\left( I,J\right) \\
\text{where }\mathcal{P}_{N}^{s} &\equiv &\left\{ \left( I,J\right) \in
\Omega \mathcal{D}_{N}\times \Omega \mathcal{D}_{N-s}:\limfunc{qdist}\left(
I,J\right) \geq \ell \left( I\right) \right\} .
\end{eqnarray*}%
Now $A_{N}^{s}\left( f,g\right) =A_{N}^{s}\left( \mathsf{P}_{N}^{\sigma }f,%
\mathsf{P}_{N-s}^{\omega }g\right) $ where $\mathsf{P}_{M}^{\mu
}=\dsum\limits_{K\in \Omega \mathcal{D}_{M}}\bigtriangleup _{K}^{\mu }$
denotes quasiHaar projection onto $\limfunc{Span}\left\{ h_{K}^{\mu
,a}\right\} _{K\in \Omega \mathcal{D}_{M},a\in \Gamma _{n}}$, and so by
orthogonality of the projections $\left\{ \mathsf{P}_{M}^{\mu }\right\}
_{M\in \mathbb{Z}}$ we have%
\begin{eqnarray*}
\left\vert \sum_{N\in \mathbb{Z}}A_{N}^{s}\left( f,g\right) \right\vert
&=&\sum_{N\in \mathbb{Z}}\left\vert A_{N}^{s}\left( \mathsf{P}_{N}^{\sigma
}f,\mathsf{P}_{N-s}^{\omega }g\right) \right\vert \leq \sum_{N\in \mathbb{Z}%
}\left\Vert A_{N}^{s}\right\Vert \left\Vert \mathsf{P}_{N}^{\sigma
}f\right\Vert _{L^{2}\left( \sigma \right) }\left\Vert \mathsf{P}%
_{N-s}^{\omega }g\right\Vert _{L^{2}\left( \omega \right) } \\
&\leq &\left\{ \sup_{N\in \mathbb{Z}}\left\Vert A_{N}^{s}\right\Vert
\right\} \left( \sum_{N\in \mathbb{Z}}\left\Vert \mathsf{P}_{N}^{\sigma
}f\right\Vert _{L^{2}\left( \sigma \right) }^{2}\right) ^{\frac{1}{2}}\left(
\sum_{N\in \mathbb{Z}}\left\Vert \mathsf{P}_{N-s}^{\omega }g\right\Vert
_{L^{2}\left( \omega \right) }^{2}\right) ^{\frac{1}{2}} \\
&\leq &\left\{ \sup_{N\in \mathbb{Z}}\left\Vert A_{N}^{s}\right\Vert
\right\} \left\Vert f\right\Vert _{L^{2}\left( \sigma \right) }\left\Vert
g\right\Vert _{L^{2}\left( \omega \right) }.
\end{eqnarray*}%
Thus it suffices to show an estimate uniform in $N$ with geometric decay in $%
s$, and we will show%
\begin{equation}
\left\vert A_{N}^{s}\left( f,g\right) \right\vert \leq C2^{-s}\sqrt{%
A_{2}^{\alpha }}\left\Vert f\right\Vert _{L^{2}\left( \sigma \right)
}\left\Vert g\right\Vert _{L^{2}\left( \omega \right) },\ \ \ \ \ \text{for }%
s\geq 0\text{ and }N\in \mathbb{Z}.  \label{AsN}
\end{equation}

We now pigeonhole the distance between $I$ and $J$:%
\begin{eqnarray*}
A_{N}^{s}\left( f,g\right) &=&\dsum\limits_{\ell =0}^{\infty }A_{N,\ell
}^{s}\left( f,g\right) ; \\
A_{N,\ell }^{s}\left( f,g\right) &\equiv &\sum_{\left( I,J\right) \in 
\mathcal{P}_{N,\ell }^{s}}\left\Vert \bigtriangleup _{I}^{\sigma
}f\right\Vert _{L^{2}\left( \sigma \right) }\left\Vert \bigtriangleup
_{J}^{\omega }g\right\Vert _{L^{2}\left( \omega \right) }A\left( I,J\right)
\\
\text{where }\mathcal{P}_{N,\ell }^{s} &\equiv &\left\{ \left( I,J\right)
\in \Omega \mathcal{D}_{N}\times \Omega \mathcal{D}_{N-s}:\limfunc{qdist}%
\left( I,J\right) \approx 2^{N+\ell }\right\} .
\end{eqnarray*}%
If we define $\mathcal{H}\left( A_{N,\ell }^{s}\right) $ to be the bilinear
form on $\ell ^{2}\times \ell ^{2}$ with matrix $\left[ A\left( I,J\right) %
\right] _{\left( I,J\right) \in \mathcal{P}_{N,\ell }^{s}}$, then it remains
to show that the norm $\left\Vert \mathcal{H}\left( A_{N,\ell }^{s}\right)
\right\Vert _{\ell ^{2}\rightarrow \ell ^{2}}$ of $\mathcal{H}\left(
A_{N,\ell }^{s}\right) $ on the sequence space $\ell ^{2}$ is bounded by $%
C2^{-s-\ell }\sqrt{A_{2}^{\alpha }}$. In turn, this is equivalent to showing
that the norm $\left\Vert \mathcal{H}\left( B_{N,\ell }^{s}\right)
\right\Vert _{\ell ^{2}\rightarrow \ell ^{2}}$ of the bilinear form $%
\mathcal{H}\left( B_{N,\ell }^{s}\right) \equiv \mathcal{H}\left( A_{N,\ell
}^{s}\right) ^{\limfunc{tr}}\mathcal{H}\left( A_{N,\ell }^{s}\right) $ on
the sequence space $\ell ^{2}$ is bounded by $C^{2}2^{-2s-2\ell
}A_{2}^{\alpha }$. Here $\mathcal{H}\left( B_{N,\ell }^{s}\right) $ is the
quadratic form with matrix kernel $\left[ B_{N,\ell }^{s}\left( J,J^{\prime
}\right) \right] _{J,J^{\prime }\in \Omega \mathcal{D}_{N-s}}$ having
entries:%
\begin{equation*}
B_{N,\ell }^{s}\left( J,J^{\prime }\right) \equiv \sum_{I\in \Omega \mathcal{%
D}_{N}:\ \limfunc{qdist}\left( I,J\right) \approx \limfunc{qdist}\left(
I,J^{\prime }\right) \approx 2^{N+\ell }}A\left( I,J\right) A\left(
I,J^{\prime }\right) ,\ \ \ \ \ \text{for }J,J^{\prime }\in \Omega \mathcal{D%
}_{N-s}.
\end{equation*}

We are reduced to showing,%
\begin{equation*}
\left\Vert \mathcal{H}\left( B_{N,\ell }^{s}\right) \right\Vert _{\ell
^{2}\rightarrow \ell ^{2}}\leq C2^{-2s-2\ell }A_{2}^{\alpha }\ \ \ \ \text{%
for }s\geq 0\text{, }\ell \geq 0\text{ and }N\in \mathbb{Z}.
\end{equation*}%
We begin by computing $B_{N,\ell }^{s}\left( J,J^{\prime }\right) $:%
\begin{eqnarray*}
B_{N,\ell }^{s}\left( J,J^{\prime }\right) &=&\sum_{\substack{ I\in \Omega 
\mathcal{D}_{N}  \\ \limfunc{qdist}\left( I,J\right) \approx \limfunc{qdist}%
\left( I,J^{\prime }\right) \approx 2^{N+\ell }}}\frac{\ell \left( J\right) 
}{\limfunc{qdist}\left( I,J\right) ^{n+1-\alpha }}\sqrt{\left\vert
I\right\vert _{\sigma }}\sqrt{\left\vert J\right\vert _{\omega }} \\
&&\ \ \ \ \ \ \ \ \ \ \times \frac{\ell \left( J^{\prime }\right) }{\limfunc{%
qdist}\left( I,J^{\prime }\right) ^{n+1-\alpha }}\sqrt{\left\vert
I\right\vert _{\sigma }}\sqrt{\left\vert J^{\prime }\right\vert _{\omega }}
\\
&=&\left\{ \sum_{\substack{ I\in \Omega \mathcal{D}_{N}  \\ \limfunc{qdist}%
\left( I,J\right) \approx \limfunc{qdist}\left( I,J^{\prime }\right) \approx
2^{N+\ell }}}\left\vert I\right\vert _{\sigma }\frac{1}{\limfunc{qdist}%
\left( I,J\right) ^{n+1-\alpha }\limfunc{qdist}\left( I,J^{\prime }\right)
^{n+1-\alpha }}\right\} \\
&&\ \ \ \ \ \ \ \ \ \ \times \ell \left( J\right) \ell \left( J^{\prime
}\right) \sqrt{\left\vert J\right\vert _{\omega }}\sqrt{\left\vert J^{\prime
}\right\vert _{\omega }}.
\end{eqnarray*}%
Now we show that%
\begin{equation}
\left\Vert B_{N,\ell }^{s}\right\Vert _{\ell ^{2}\rightarrow \ell
^{2}}\lesssim 2^{-2s-2\ell }A_{2}^{\alpha }\ ,  \label{Schur s}
\end{equation}%
by applying the proof of Schur's lemma. Fix $\ell \geq 0$ and $s\geq 0$.
Choose the Schur function $\beta \left( K\right) =\frac{1}{\sqrt{\left\vert
K\right\vert _{\omega }}}$. Fix $J\in \Omega \mathcal{D}_{N-s}$. We now
group those $I\in \Omega \mathcal{D}_{N}$ with $\limfunc{qdist}\left(
I,J\right) \approx 2^{N+\ell }$ into finitely many groups $%
G_{1},...G_{C_{n}} $ for which the union of the $I$ in each group is
contained in quasicube of side length roughly $\frac{1}{100}2^{N+\ell }$ ,
and we set $I_{k}^{\ast }\equiv \dbigcup\limits_{I\in G_{k}}I$ for $1\leq
k\leq C_{n}$. We then have%
\begin{eqnarray*}
&&\sum_{J^{\prime }\in \Omega \mathcal{D}_{N-s}}\frac{\beta \left( J\right) 
}{\beta \left( J^{\prime }\right) }B_{N,\ell }^{s}\left( J,J^{\prime }\right)
\\
&=&\sum_{\substack{ J^{\prime }\in \Omega \mathcal{D}_{N-s}  \\ \limfunc{%
qdist}\left( J^{\prime },J\right) \leq \frac{1}{100}2^{N+\ell +2}}}\frac{%
\beta \left( J\right) }{\beta \left( J^{\prime }\right) }B_{N,\ell
}^{s}\left( J,J^{\prime }\right) +\sum_{\substack{ J^{\prime }\in \Omega 
\mathcal{D}_{N-s}  \\ \limfunc{qdist}\left( J^{\prime },J\right) >\frac{1}{%
100}2^{N+\ell +2}}}\frac{\beta \left( J\right) }{\beta \left( J^{\prime
}\right) }B_{N,\ell }^{s}\left( J,J^{\prime }\right) \\
&=&A+B,
\end{eqnarray*}%
where%
\begin{eqnarray*}
&&A\lesssim \sum_{\substack{ J^{\prime }\in \Omega \mathcal{D}_{N-s}  \\ 
\limfunc{qdist}\left( J,J^{\prime }\right) \leq \frac{1}{100}2^{N+\ell +2}}}%
\left\{ \sum_{\substack{ I\in \Omega \mathcal{D}_{N}  \\ \limfunc{qdist}%
\left( I,J\right) \approx 2^{N+\ell }}}\left\vert I\right\vert _{\sigma
}\right\} \ \frac{2^{2\left( N-s\right) }}{2^{2\left( \ell +N\right) \left(
n+1-\alpha \right) }}\left\vert J^{\prime }\right\vert _{\omega } \\
&=&\sum_{\substack{ J^{\prime }\in \Omega \mathcal{D}_{N-s}  \\ \limfunc{%
qdist}\left( J,J^{\prime }\right) \leq \frac{1}{100}2^{N+\ell +2}}}\left\{
\sum_{k=1}^{C_{n}}\left\vert I_{k}^{\ast }\right\vert _{\sigma }\right\} \ 
\frac{2^{2\left( N-s\right) }}{2^{2\left( \ell +N\right) \left( n+1-\alpha
\right) }}\left\vert J^{\prime }\right\vert _{\omega } \\
&=&\frac{2^{2\left( N-s\right) }}{2^{2\left( \ell +N\right) \left(
n+1-\alpha \right) }}\sum_{k=1}^{C_{n}}\sum_{\substack{ J^{\prime }\in
\Omega \mathcal{D}_{N-s}  \\ \limfunc{qdist}\left( J,J^{\prime }\right) \leq 
\frac{1}{100}2^{N+\ell +2}}}\left\vert I_{k}^{\ast }\right\vert _{\sigma }\
\left\vert J^{\prime }\right\vert _{\omega } \\
&\lesssim &2^{-2s-2\ell }\sum_{k=1}^{C_{n}}\frac{\left\vert I_{k}^{\ast
}\right\vert _{\sigma }}{2^{\left( \ell +N\right) \left( n-\alpha \right) }}%
\frac{\left\vert \frac{1}{100}2^{N+\ell +2}J\right\vert _{\omega }}{%
2^{\left( \ell +N\right) \left( n-\alpha \right) }}\lesssim 2^{-2s-2\ell
}A_{2}^{\alpha },
\end{eqnarray*}%
since the quasicubes $I_{k}^{\ast }$ and $\frac{1}{100}2^{N+\ell +2}J$ are
well separated. If we let $Q_{k}$ be the smallest quasicube containing the
set%
\begin{equation*}
E_{k}\equiv \dbigcup\limits_{\substack{ J^{\prime }\in \Omega \mathcal{D}%
_{N-s}:\ \limfunc{qdist}\left( I_{k}^{\ast },J^{\prime }\right) \approx
2^{N+\ell }  \\ \limfunc{qdist}\left( J,J^{\prime }\right) >\frac{1}{100}%
2^{N+\ell +2}}}J^{\prime }\ ,
\end{equation*}%
we also have%
\begin{eqnarray*}
B &\lesssim &\sum_{\substack{ J^{\prime }\in \Omega \mathcal{D}_{N-s}  \\ 
\limfunc{qdist}\left( J,J^{\prime }\right) >\frac{1}{100}2^{N+\ell +2}}}%
\left\{ \sum_{\substack{ I\in \Omega \mathcal{D}_{N}  \\ \limfunc{qdist}%
\left( I,J^{\prime }\right) \approx \limfunc{qdist}\left( I,J\right) \approx
2^{N+\ell }}}\left\vert I\right\vert _{\sigma }\right\} \ \frac{2^{2\left(
N-s\right) }}{2^{2\left( \ell +N\right) \left( n+1-\alpha \right) }}%
\left\vert J^{\prime }\right\vert _{\omega } \\
&\lesssim &\sum_{\substack{ J^{\prime }\in \Omega \mathcal{D}_{N-s}  \\ 
\limfunc{qdist}\left( J,J^{\prime }\right) >\frac{1}{100}2^{N+\ell +2}}}%
\left\{ \sum_{k:\ \limfunc{qdist}\left( I_{k}^{\ast },J^{\prime }\right)
\approx 2^{N+\ell }}\left\vert I_{k}^{\ast }\right\vert _{\sigma }\right\} \ 
\frac{2^{2\left( N-s\right) }}{2^{2\left( \ell +N\right) \left( n+1-\alpha
\right) }}\left\vert J^{\prime }\right\vert _{\omega } \\
&\lesssim &\frac{2^{2\left( N-s\right) }}{2^{2\left( \ell +N\right) \left(
n+1-\alpha \right) }}\sum_{k=1}^{C_{n}}\left\vert I_{k}^{\ast }\right\vert
_{\sigma }\left\vert E_{k}\right\vert _{\omega } \\
&\lesssim &2^{-2s-2\ell }\sum_{k=1}^{C_{n}}\frac{\left\vert I_{k}^{\ast
}\right\vert _{\sigma }}{2^{\left( \ell +N\right) \left( n-\alpha \right) }}%
\frac{\left\vert Q_{k}\right\vert _{\omega }}{2^{\left( \ell +N\right)
\left( n-\alpha \right) }}\lesssim 2^{-2s-2\ell }A_{2}^{\alpha },
\end{eqnarray*}%
since the quasicube $I_{k}^{\ast }$ is well separated from the quasicube $%
Q_{k}$.

Thus we can now apply Schur's argument with $\sum_{J}\left( a_{J}\right)
^{2}=\sum_{J^{\prime }}\left( b_{J^{\prime }}\right) ^{2}=1$ to obtain%
\begin{eqnarray*}
&&\sum_{J,J^{\prime }\in \Omega \mathcal{D}_{N-s}}a_{J}b_{J^{\prime
}}B_{N,\ell }^{s}\left( J,J^{\prime }\right) \\
&=&\sum_{J,J^{\prime }\in \Omega \mathcal{D}_{N-s}}a_{J}\beta \left(
J\right) b_{J^{\prime }}\beta \left( J^{\prime }\right) \frac{B_{N,\ell
}^{s}\left( J,J^{\prime }\right) }{\beta \left( J\right) \beta \left(
J^{\prime }\right) } \\
&\leq &\sum_{J}\left( a_{J}\beta \left( J\right) \right) ^{2}\sum_{J^{\prime
}}\frac{B_{N,\ell }^{s}\left( J,J^{\prime }\right) }{\beta \left( J\right)
\beta \left( J^{\prime }\right) }+\sum_{J^{\prime }}\left( b_{J^{\prime
}}\beta \left( J^{\prime }\right) \right) ^{2}\sum_{J}\frac{B_{N,\ell
}^{s}\left( J,J^{\prime }\right) }{\beta \left( J\right) \beta \left(
J^{\prime }\right) } \\
&=&\sum_{J}\left( a_{J}\right) ^{2}\left\{ \sum_{J^{\prime }}\frac{\beta
\left( J\right) }{\beta \left( J^{\prime }\right) }B_{N,\ell }^{s}\left(
J,J^{\prime }\right) \right\} +\sum_{J^{\prime }}\left( b_{J^{\prime
}}\right) ^{2}\left\{ \sum_{J}\frac{\beta \left( J^{\prime }\right) }{\beta
\left( J\right) }B_{N,\ell }^{s}\left( J,J^{\prime }\right) \right\} \\
&\lesssim &2^{-2s-2\ell }A_{2}^{\alpha }\left( \sum_{J}\left( a_{J}\right)
^{2}+\sum_{J^{\prime }}\left( b_{J^{\prime }}\right) ^{2}\right)
=2^{1-2s-2\ell }A_{2}^{\alpha }.
\end{eqnarray*}%
This completes the proof of (\ref{Schur s}). We can now sum in $\ell $ to
get (\ref{AsN}) and we are done. This completes our proof of the long-range
estimate%
\begin{equation*}
\mathcal{A}^{\limfunc{long}}\left( f,g\right) \lesssim \sqrt{A_{2}^{\alpha }}%
\left\Vert f\right\Vert _{L^{2}\left( \sigma \right) }\left\Vert
g\right\Vert _{L^{2}\left( \omega \right) }\ .
\end{equation*}

\bigskip

At this point we pause to complete the proof of (\ref{delta near}). Indeed,
the deferred term $A_{3}$ can be handled using the above argument since $%
3J\cap I=\emptyset =J\cap 3I$ implies that we can use the Energy Lemma \ref%
{ener} as we did above.

\bigskip

\textbf{The mid range case}: Let%
\begin{equation*}
\mathcal{P}\equiv \left\{ \left( I,J\right) \in \Omega \mathcal{D}\times
\Omega \mathcal{D}:J\text{ is good},\ \ell \left( J\right) \leq 2^{-\mathbf{%
\rho }}\ell \left( I\right) ,\text{ }J\subset 3I\setminus I\right\} .
\end{equation*}%
For $\left( I,J\right) \in \mathcal{P}$, the `pivotal' estimate from the
Energy Lemma \ref{ener} gives%
\begin{equation*}
\left\vert \left\langle T_{\sigma }^{\alpha }\left( \bigtriangleup
_{I}^{\sigma }f\right) ,\bigtriangleup _{J}^{\omega }g\right\rangle _{\omega
}\right\vert \lesssim \left\Vert \bigtriangleup _{J}^{\omega }g\right\Vert
_{L^{2}\left( \omega \right) }\mathrm{P}^{\alpha }\left( J,\left\vert
\bigtriangleup _{I}^{\sigma }f\right\vert \sigma \right) \sqrt{\left\vert
J\right\vert _{\omega }}\,.
\end{equation*}%
Now we pigeonhole the lengths of $I$ and $J$ and the distance between them
by defining%
\begin{equation*}
\mathcal{P}_{N,d}^{s}\equiv \left\{ \left( I,J\right) \in \Omega \mathcal{D}%
\times \Omega \mathcal{D}:J\text{ is good},\ \ell \left( I\right) =2^{N},\
\ell \left( J\right) =2^{N-s},\text{ }J\subset 3I\setminus I,\ 2^{d-1}\leq 
\limfunc{qdist}\left( I,J\right) \leq 2^{d}\right\} .
\end{equation*}%
Note that the closest a good quasicube $J$ can come to $I$ is determined by
the goodness inequality, which gives this bound for $2^{d}\geq \limfunc{qdist%
}\left( I,J\right) $: 
\begin{eqnarray*}
&&2^{d}\geq \frac{1}{2}\ell \left( I\right) ^{1-\varepsilon }\ell \left(
J\right) ^{\varepsilon }=\frac{1}{2}2^{N\left( 1-\varepsilon \right)
}2^{\left( N-s\right) \varepsilon }=\frac{1}{2}2^{N-\varepsilon s}; \\
&&\text{which implies }N-\varepsilon s-1\leq d\leq N,
\end{eqnarray*}%
where the last inequality holds because we are in the case of the mid-range
term. Thus we have%
\begin{eqnarray*}
&&\dsum\limits_{\left( I,J\right) \in \mathcal{P}}\left\vert \left\langle
T_{\sigma }^{\alpha }\left( \bigtriangleup _{I}^{\sigma }f\right)
,\bigtriangleup _{J}^{\omega }g\right\rangle _{\omega }\right\vert \lesssim
\dsum\limits_{\left( I,J\right) \in \mathcal{P}}\left\Vert \bigtriangleup
_{J}^{\omega }g\right\Vert _{L^{2}\left( \omega \right) }\mathrm{P}^{\alpha
}\left( J,\left\vert \bigtriangleup _{I}^{\sigma }f\right\vert \sigma
\right) \sqrt{\left\vert J\right\vert _{\omega }} \\
&&\ \ \ \ \ =\dsum\limits_{s=\mathbf{\rho }}^{\infty }\ \sum_{N\in \mathbb{Z}%
}\ \sum_{d=N-\varepsilon s-1}^{N}\ \sum_{\left( I,J\right) \in \mathcal{P}%
_{N,d}^{s}}\ \left\Vert \bigtriangleup _{J}^{\omega }g\right\Vert
_{L^{2}\left( \omega \right) }\mathrm{P}^{\alpha }\left( J,\left\vert
\bigtriangleup _{I}^{\sigma }f\right\vert \sigma \right) \sqrt{\left\vert
J\right\vert _{\omega }}.
\end{eqnarray*}%
Now we use%
\begin{eqnarray*}
\mathrm{P}^{\alpha }\left( J,\left\vert \bigtriangleup _{I}^{\sigma
}f\right\vert \sigma \right) &=&\int_{I}\frac{\ell \left( J\right) }{\left(
\ell \left( J\right) +\left\vert y-c_{J}\right\vert \right) ^{n+1-\alpha }}%
\left\vert \bigtriangleup _{I}^{\sigma }f\left( y\right) \right\vert d\sigma
\left( y\right) \\
&\lesssim &\frac{2^{N-s}}{2^{d\left( n+1-\alpha \right) }}\left\Vert
\bigtriangleup _{I}^{\sigma }f\right\Vert _{L^{2}\left( \sigma \right) }%
\sqrt{\left\vert I\right\vert _{\sigma }}
\end{eqnarray*}%
and apply Cauchy-Schwarz in $J$ and use $J\subset 3I\setminus I$ to get%
\begin{eqnarray*}
&&\dsum\limits_{\left( I,J\right) \in \mathcal{P}}\left\vert \left\langle
T_{\sigma }^{\alpha }\left( \bigtriangleup _{I}^{\sigma }f\right)
,\bigtriangleup _{J}^{\omega }g\right\rangle _{\omega }\right\vert \\
&\lesssim &\dsum\limits_{s=\mathbf{\rho }}^{\infty }\ \sum_{N\in \mathbb{Z}%
}\ \sum_{d=N-\varepsilon s-1}^{N}\ \sum_{I\in \Omega \mathcal{D}_{N}}\frac{%
2^{N-s}2^{N\left( n-\alpha \right) }}{2^{d\left( n+1-\alpha \right) }}%
\left\Vert \bigtriangleup _{I}^{\sigma }f\right\Vert _{L^{2}\left( \sigma
\right) }\frac{\sqrt{\left\vert I\right\vert _{\sigma }}\sqrt{\left\vert
3I\setminus I\right\vert _{\omega }}}{2^{N\left( n-\alpha \right) }} \\
&&\ \ \ \ \ \ \ \ \ \ \ \ \ \ \ \ \ \ \ \ \ \ \ \ \ \ \ \ \ \ \times \sqrt{%
\sum_{\substack{ J\in \Omega \mathcal{D}_{N-s}  \\ J\subset 3I\setminus I%
\text{ and }\limfunc{qdist}\left( I,J\right) \approx 2^{d}}}\left\Vert
\bigtriangleup _{J}^{\omega }g\right\Vert _{L^{2}\left( \omega \right) }^{2}}
\\
&\lesssim &\left( 1+\varepsilon s\right) \dsum\limits_{s=\mathbf{\rho }%
}^{\infty }\ \sum_{N\in \mathbb{Z}}\frac{2^{N-s}2^{N\left( n-\alpha \right) }%
}{2^{\left( N-\varepsilon s\right) \left( n+1-\alpha \right) }}\sqrt{%
A_{2}^{\alpha }}\sum_{I\in \Omega \mathcal{D}_{N}}\left\Vert \bigtriangleup
_{I}^{\sigma }f\right\Vert _{L^{2}\left( \sigma \right) }\sqrt{\sum 
_{\substack{ J\in \Omega \mathcal{D}_{N-s}  \\ J\subset 3I\setminus I}}%
\left\Vert \bigtriangleup _{J}^{\omega }g\right\Vert _{L^{2}\left( \omega
\right) }^{2}} \\
&\lesssim &\left( 1+\varepsilon s\right) \dsum\limits_{s=\mathbf{\rho }%
}^{\infty }2^{-s\left[ 1-\varepsilon \left( n+1-\alpha \right) \right] }%
\sqrt{A_{2}^{\alpha }}\left\Vert f\right\Vert _{L^{2}\left( \sigma \right)
}\left\Vert g\right\Vert _{L^{2}\left( \omega \right) }\lesssim \sqrt{%
A_{2}^{\alpha }}\left\Vert f\right\Vert _{L^{2}\left( \sigma \right)
}\left\Vert g\right\Vert _{L^{2}\left( \omega \right) },
\end{eqnarray*}%
where in the third line above we have used $\sum_{d=N-\varepsilon
s-1}^{N}1\lesssim 1+\varepsilon s$, and in the last line $\frac{%
2^{N-s}2^{N\left( n-\alpha \right) }}{2^{\left( N-\varepsilon s\right)
\left( n+1-\alpha \right) }}=2^{-s\left[ 1-\varepsilon \left( n+1-\alpha
\right) \right] }$ followed by Cauchy-Schwarz in $I$ and $N$, using that we
have bounded overlap in the triples of $I$ for $I\in \Omega \mathcal{D}_{N}$%
. More precisely, if we define $f_{k}\equiv \sum_{I\in \Omega \mathcal{D}%
_{k}}\bigtriangleup _{I}^{\sigma }fh_{I}^{\sigma }$ and $g_{k}\equiv
\sum_{I\in \Omega \mathcal{D}_{k}}\bigtriangleup _{J}^{\omega
}gh_{J}^{\omega }$, then we have the orthogonality inequality 
\begin{eqnarray*}
\sum_{N\in \mathbb{Z}}\left\Vert f_{N}\right\Vert _{L^{2}\left( \sigma
\right) }\left\Vert g_{N-s}\right\Vert _{L^{2}\left( \omega \right) } &\leq
&\left( \sum_{N\in \mathbb{Z}}\left\Vert f_{N}\right\Vert _{L^{2}\left(
\sigma \right) }^{2}\right) ^{\frac{1}{2}}\left( \sum_{N\in \mathbb{Z}%
}\left\Vert g_{N-s}\right\Vert _{L^{2}\left( \omega \right) }^{2}\right) ^{%
\frac{1}{2}} \\
&=&\left\Vert f\right\Vert _{L^{2}\left( \sigma \right) }\left\Vert
g\right\Vert _{L^{2}\left( \omega \right) }.
\end{eqnarray*}%
We have assumed that $0<\varepsilon <\frac{1}{n+1-\alpha }$ in the
calculations above, and this completes the proof of Lemma \ref{standard
delta}.
\end{proof}

\section{Corona Decompositions and splittings}

We will use two different corona constructions, namely a Calder\'{o}%
n-Zygmund decomposition and an energy decomposition of NTV type, to reduce
matters to the stopping form, the main part of which is handled by Lacey's
recursion argument. We will then iterate these coronas into a double corona.
We first recall our basic setup. For convenience in notation we will
sometimes suppress the dependence on $\alpha $ in our nonlinear forms, but
will retain it in the operators, Poisson integrals and constants. We will
assume that the good/bad quasicube machinery of Nazarov, Treil and Volberg 
\cite{Vol} is in force here. Let $\Omega \mathcal{D}^{\sigma }=\Omega 
\mathcal{D}^{\omega }$ be an $\left( \mathbf{r},\varepsilon \right) $-good
quasigrid on $\mathbb{R}^{n}$, and let $\left\{ h_{I}^{\sigma ,a}\right\}
_{I\in \Omega \mathcal{D}^{\sigma },\ a\in \Gamma _{n}}$ and $\left\{
h_{J}^{\omega ,b}\right\} _{J\in \Omega \mathcal{D}^{\omega },\ b\in \Gamma
_{n}}$ be corresponding quasiHaar bases as described above, so that%
\begin{equation*}
f=\sum_{I\in \Omega \mathcal{D}^{\sigma }}\bigtriangleup _{I}^{\sigma }f%
\text{ and }g=\sum_{J\in \Omega \mathcal{D}^{\omega }\text{ }}\bigtriangleup
_{J}^{\omega }g\ ,
\end{equation*}%
where the quasiHaar projections $\bigtriangleup _{I}^{\sigma }f$ and $%
\bigtriangleup _{J}^{\omega }g$ vanish if the quasicubes $I$ and $J$ are not
good. Inequality (\ref{two weight}) is equivalent to boundedness of the
bilinear form%
\begin{equation*}
\mathcal{T}^{\alpha }\left( f,g\right) \equiv \left\langle T_{\sigma
}^{\alpha }\left( f\right) ,g\right\rangle _{\omega }=\sum_{I\in \Omega 
\mathcal{D}^{\sigma }\text{ and }J\in \Omega \mathcal{D}^{\omega
}}\left\langle T_{\sigma }^{\alpha }\left( \bigtriangleup _{I}^{\sigma
}f\right) ,\bigtriangleup _{J}^{\omega }g\right\rangle _{\omega }
\end{equation*}%
on $L^{2}\left( \sigma \right) \times L^{2}\left( \omega \right) $, i.e.%
\begin{equation*}
\left\vert \mathcal{T}^{\alpha }\left( f,g\right) \right\vert \leq \mathfrak{%
N}_{T^{\alpha }}\left\Vert f\right\Vert _{L^{2}\left( \sigma \right)
}\left\Vert g\right\Vert _{L^{2}\left( \omega \right) }.
\end{equation*}

\subsection{The Calder\'{o}n-Zygmund corona}

We now introduce a stopping tree $\mathcal{F}$ for the function $f\in
L^{2}\left( \sigma \right) $. Let $\mathcal{F}$ be a collection of Calder%
\'{o}n-Zygmund stopping quasicubes for $f$, and let $\Omega \mathcal{D}%
^{\sigma }=\dbigcup\limits_{F\in \mathcal{F}}\mathcal{C}_{F}$ be the
associated corona decomposition of the dyadic quasigrid $\Omega \mathcal{D}%
^{\sigma }$.

For a quasicube $I\in \Omega \mathcal{D}^{\sigma }$ let $\pi _{\Omega 
\mathcal{D}^{\sigma }}I$ be the $\Omega \mathcal{D}^{\sigma }$-parent of $I$
in the quasigrid $\Omega \mathcal{D}^{\sigma }$, and let $\pi _{\mathcal{F}%
}I $ be the smallest member of $\mathcal{F}$ that contains $I$. For $%
F,F^{\prime }\in \mathcal{F}$, we say that $F^{\prime }$ is an $\mathcal{F}$%
-child of $F$ if $\pi _{\mathcal{F}}\left( \pi _{\Omega \mathcal{D}^{\sigma
}}F^{\prime }\right) =F$ (it could be that $F=\pi _{\Omega \mathcal{D}%
^{\sigma }}F^{\prime }$), and we denote by $\mathfrak{C}_{\mathcal{F}}\left(
F\right) $ the set of $\mathcal{F}$-children of $F$. For $F\in \mathcal{F}$,
define the projection $\mathsf{P}_{\mathcal{C}_{F}}^{\sigma }$ onto the
linear span of the quasiHaar functions $\left\{ h_{I}^{\sigma ,a}\right\}
_{I\in \mathcal{C}_{F},\ a\in \Gamma _{n}}$ by%
\begin{equation*}
\mathsf{P}_{\mathcal{C}_{F}}^{\sigma }f=\sum_{I\in \mathcal{C}%
_{F}}\bigtriangleup _{I}^{\sigma }f=\sum_{I\in \mathcal{C}_{F},\ a\in \Gamma
_{n}}\left\langle f,h_{I}^{\sigma ,a}\right\rangle _{\sigma }h_{I}^{\sigma
,a}.
\end{equation*}%
The standard properties of these projections are%
\begin{equation*}
f=\sum_{F\in \mathcal{F}}\mathsf{P}_{\mathcal{C}_{F}}^{\sigma }f,\ \ \ \ \
\int \left( \mathsf{P}_{\mathcal{C}_{F}}^{\sigma }f\right) \sigma =0,\ \ \ \
\ \left\Vert f\right\Vert _{L^{2}\left( \sigma \right) }^{2}=\sum_{F\in 
\mathcal{F}}\left\Vert \mathsf{P}_{\mathcal{C}_{F}}^{\sigma }f\right\Vert
_{L^{2}\left( \sigma \right) }^{2}.
\end{equation*}

\subsection{The energy corona}

We must also impose a quasienergy corona decomposition as in \cite{NTV3} and 
\cite{LaSaUr2}.

\begin{definition}
\label{def energy corona 3}Given a quasicube $S_{0}$, define $\mathcal{S}%
\left( S_{0}\right) $ to be the maximal subquasicubes $I\subset S_{0}$ such
that%
\begin{equation}
\sum_{J\in \mathcal{M}_{\mathbf{\tau }-\limfunc{deep}}\left( I\right)
}\left( \frac{\mathrm{P}^{\alpha }\left( J,\mathbf{1}_{S_{0}\setminus \gamma
J}\sigma \right) }{\left\vert J\right\vert ^{\frac{1}{n}}}\right)
^{2}\left\Vert \mathsf{P}_{J}^{\limfunc{subgood},\omega }\mathbf{x}%
\right\Vert _{L^{2}\left( \omega \right) }^{2}\geq C_{\limfunc{energy}}\left[
\left( \mathcal{E}_{\alpha }^{\limfunc{deep}}\right) ^{2}+A_{2}^{\alpha
}+A_{2}^{\alpha ,\limfunc{punct}}\right] \ \left\vert I\right\vert _{\sigma
},  \label{def stop 3}
\end{equation}%
where $\mathcal{E}_{\alpha }^{\limfunc{deep}}$ is the constant in the deep
quasienergy condition defined in Definition \ref{energy condition}, and $C_{%
\limfunc{energy}}$ is a sufficiently large positive constant depending only
on $\mathbf{\tau }\geq \mathbf{r},n$ and $\alpha $. Then define the $\sigma $%
-energy stopping quasicubes of $S_{0}$ to be the collection 
\begin{equation*}
\mathcal{S}=\left\{ S_{0}\right\} \cup \dbigcup\limits_{n=0}^{\infty }%
\mathcal{S}_{n}
\end{equation*}%
where $\mathcal{S}_{0}=\mathcal{S}\left( S_{0}\right) $ and $\mathcal{S}%
_{n+1}=\dbigcup\limits_{S\in \mathcal{S}_{n}}\mathcal{S}\left( S\right) $
for $n\geq 0$.
\end{definition}

From the quasienergy condition in Definition \ref{energy condition} we
obtain the $\sigma $-Carleson estimate%
\begin{equation}
\sum_{S\in \mathcal{S}:\ S\subset I}\left\vert S\right\vert _{\sigma }\leq
2\left\vert I\right\vert _{\sigma },\ \ \ \ \ I\in \Omega \mathcal{D}%
^{\sigma }.  \label{sigma Carleson 3}
\end{equation}%
Indeed, using the deep quasienergy condition, the first generation satisfies%
\begin{eqnarray}
&&\ \ \ \ \ \ \ \ \ \ \ \ \ \ \ \ \ \ \ \ \sum_{S\in \mathcal{S}%
_{1}}\left\vert S\right\vert _{\sigma }  \label{first gen} \\
&\leq &\frac{1}{C_{\limfunc{energy}}\left[ \left( \mathcal{E}_{\alpha }^{%
\limfunc{deep}}\right) ^{2}+A_{2}^{\alpha }\right] }\sum_{S\in \mathcal{S}%
_{1}}\sum_{J\in \mathcal{M}_{\mathbf{\tau }-\limfunc{deep}}\left( S\right)
}\left( \frac{\mathrm{P}^{\alpha }\left( J,\mathbf{1}_{S_{0}\setminus \gamma
J}\sigma \right) }{\left\vert J\right\vert ^{\frac{1}{n}}}\right)
^{2}\left\Vert \mathsf{P}_{J}^{\limfunc{subgood},\omega }\mathbf{x}%
\right\Vert _{L^{2}\left( \omega \right) }^{2}  \notag \\
&\leq &\frac{1}{C_{\limfunc{energy}}\left[ \left( \mathcal{E}_{\alpha }^{%
\limfunc{deep}}\right) ^{2}+A_{2}^{\alpha }\right] }\sum_{S\in \mathcal{S}%
_{1}}\sum_{J\in \mathcal{M}_{\mathbf{\tau }-\limfunc{deep}}\left( S\right)
}\left( \frac{\mathrm{P}^{\alpha }\left( J,\mathbf{1}_{S_{0}}\sigma \right) 
}{\left\vert J\right\vert ^{\frac{1}{n}}}\right) ^{2}\left\Vert \mathsf{P}%
_{J}^{\limfunc{subgood},\omega }\mathbf{x}\right\Vert _{L^{2}\left( \omega
\right) }^{2}  \notag \\
&\leq &\frac{C_{\mathbf{\tau },\mathbf{r},n,\alpha }}{C_{\limfunc{energy}}%
\left[ \left( \mathcal{E}_{\alpha }^{\limfunc{deep}}\right)
^{2}+A_{2}^{\alpha }\right] }\sum_{S\in \mathcal{S}_{1}}\sum_{J\in \mathcal{M%
}_{\left( \mathbf{r},\varepsilon \right) -\limfunc{deep}}\left( S\right)
}\left( \frac{\mathrm{P}^{\alpha }\left( J,\mathbf{1}_{S_{0}}\sigma \right) 
}{\left\vert J\right\vert ^{\frac{1}{n}}}\right) ^{2}\left\Vert \mathsf{P}%
_{J}^{\limfunc{subgood},\omega }\mathbf{x}\right\Vert _{L^{2}\left( \omega
\right) }^{2}  \notag \\
&\leq &\frac{C_{\mathbf{\tau },\mathbf{r},n,\alpha }}{C_{\limfunc{energy}}%
\left[ \left( \mathcal{E}_{\alpha }^{\limfunc{deep}}\right)
^{2}+A_{2}^{\alpha }+A_{2}^{\alpha ,\limfunc{punct}}\right] }\left( \mathcal{%
E}_{\alpha }^{\limfunc{deep}\limfunc{plug}}\right) ^{2}\ \left\vert
S_{0}\right\vert _{\sigma }=\frac{1}{2}\left\vert S_{0}\right\vert _{\sigma
}\ ,  \notag
\end{eqnarray}%
provided we take $C_{\limfunc{energy}}=2C_{\mathbf{\tau },\mathbf{r}%
,n,\alpha }\frac{\left( \mathcal{E}_{\alpha }^{\limfunc{deep}\limfunc{plug}%
}\right) ^{2}}{\left( \mathcal{E}_{\alpha }^{\limfunc{deep}}\right)
^{2}+A_{2}^{\alpha }}$ and where from Corollary \ref{all plugged} we have $%
\mathcal{E}_{\alpha }^{\limfunc{deep}\limfunc{plug}}\lesssim \mathcal{E}%
_{\alpha }^{\limfunc{deep}}+\sqrt{A_{2}^{\alpha }}+\sqrt{A_{2}^{\alpha ,%
\limfunc{punct}}}$. The third inequality above, in which $\mathbf{\tau }$ is
replaced by $\mathbf{r}$ (but the goodness parameter $\varepsilon >0$ is
unchanged), follows because if $J_{1}\in \mathcal{M}_{\mathbf{\tau }-%
\limfunc{deep}}\left( S\right) $, then $J_{1}\subset J_{2}$ for a unique $%
J_{2}\in \mathcal{M}_{\left( \mathbf{r},\varepsilon \right) -\limfunc{deep}%
}\left( S\right) $ and we have $\ell \left( J_{2}\right) \leq 2^{\mathbf{%
\tau }-\mathbf{r}}\ell \left( J_{1}\right) $ from the definitions of $%
\mathcal{M}_{\mathbf{\tau }-\limfunc{deep}}\left( S\right) $ and $\mathcal{M}%
_{\left( \mathbf{r},\varepsilon \right) -\limfunc{deep}}\left( S\right) $,
hence $\frac{\mathrm{P}^{\alpha }\left( J_{1},\mathbf{1}_{S_{0}}\sigma
\right) }{\left\vert J_{1}\right\vert ^{\frac{1}{n}}}\leq C_{\mathbf{\tau },%
\mathbf{r},n,\alpha }\frac{\mathrm{P}^{\alpha }\left( J_{2},\mathbf{1}%
_{S_{0}}\sigma \right) }{\left\vert J_{2}\right\vert ^{\frac{1}{n}}}$.
Subsequent generations satisfy a similar estimate, which then easily gives (%
\ref{sigma Carleson 3}). We emphasize that this collection of stopping times
depends only on $S_{0}$ and the weight pair $\left( \sigma ,\omega \right) $%
, and not on any functions at hand.

Finally, we record the reason for introducing quasienergy stopping times. If 
\begin{equation}
X_{\alpha }\left( \mathcal{C}_{S}\right) ^{2}\equiv \sup_{I\in \mathcal{C}%
_{S}}\frac{1}{\left\vert I\right\vert _{\sigma }}\sum_{J\in \mathcal{M}_{%
\mathbf{\tau }-\limfunc{deep}}\left( I\right) }\left( \frac{\mathrm{P}%
^{\alpha }\left( J,\mathbf{1}_{S\setminus \gamma J}\sigma \right) }{%
\left\vert J\right\vert ^{\frac{1}{n}}}\right) ^{2}\left\Vert \mathsf{P}%
_{J}^{\limfunc{subgood},\omega }\mathbf{x}\right\Vert _{L^{2}\left( \omega
\right) }^{2}  \label{def stopping energy 3}
\end{equation}%
is (the square of) the $\alpha $\emph{-stopping quasienergy} of the weight
pair $\left( \sigma ,\omega \right) $ with respect to the corona $\mathcal{C}%
_{S}$, then we have the \emph{stopping quasienergy bounds}%
\begin{equation}
X_{\alpha }\left( \mathcal{C}_{S}\right) \leq \sqrt{C_{\limfunc{energy}}}%
\sqrt{\left( \mathcal{E}_{\alpha }^{\limfunc{deep}}\right)
^{2}+A_{2}^{\alpha }+A_{2}^{\alpha ,\limfunc{punct}}},\ \ \ \ \ S\in 
\mathcal{S},  \label{def stopping bounds 3}
\end{equation}%
where $A_{2}^{\alpha }$, $A_{2}^{\alpha ,\limfunc{punct}}$ and the the deep
quasienergy constant $\mathcal{E}_{\alpha }^{\limfunc{deep}}$ are controlled
by assumption.

\subsection{General stopping data}

It is useful to extend our notion of corona decomposition to more general
stopping data. Our general definition of stopping data will use a positive
constant $C_{0}\geq 4$.

\begin{definition}
\label{general stopping data}Suppose we are given a positive constant $%
C_{0}\geq 4$, a subset $\mathcal{F}$ of the dyadic quasigrid $\Omega 
\mathcal{D}^{\sigma }$ (called the stopping times), and a corresponding
sequence $\alpha _{\mathcal{F}}\equiv \left\{ \alpha _{\mathcal{F}}\left(
F\right) \right\} _{F\in \mathcal{F}}$ of nonnegative numbers $\alpha _{%
\mathcal{F}}\left( F\right) \geq 0$ (called the stopping data). Let $\left( 
\mathcal{F},\prec ,\pi _{\mathcal{F}}\right) $ be the tree structure on $%
\mathcal{F}$ inherited from $\Omega \mathcal{D}^{\sigma }$, and for each $%
F\in \mathcal{F}$ denote by $\mathcal{C}_{F}=\left\{ I\in \Omega \mathcal{D}%
^{\sigma }:\pi _{\mathcal{F}}I=F\right\} $ the corona associated with $F$: 
\begin{equation*}
\mathcal{C}_{F}=\left\{ I\in \Omega \mathcal{D}^{\sigma }:I\subset F\text{
and }I\not\subset F^{\prime }\text{ for any }F^{\prime }\prec F\right\} .
\end{equation*}%
We say the triple $\left( C_{0},\mathcal{F},\alpha _{\mathcal{F}}\right) $
constitutes \emph{stopping data} for a function $f\in L_{loc}^{1}\left(
\sigma \right) $ if

\begin{enumerate}
\item $\mathbb{E}_{I}^{\sigma }\left\vert f\right\vert \leq \alpha _{%
\mathcal{F}}\left( F\right) $ for all $I\in \mathcal{C}_{F}$ and $F\in 
\mathcal{F}$,

\item $\sum_{F^{\prime }\preceq F}\left\vert F^{\prime }\right\vert _{\sigma
}\leq C_{0}\left\vert F\right\vert _{\sigma }$ for all $F\in \mathcal{F}$,

\item $\sum_{F\in \mathcal{F}}\alpha _{\mathcal{F}}\left( F\right)
^{2}\left\vert F\right\vert _{\sigma }\mathbf{\leq }C_{0}^{2}\left\Vert
f\right\Vert _{L^{2}\left( \sigma \right) }^{2}$,

\item $\alpha _{\mathcal{F}}\left( F\right) \leq \alpha _{\mathcal{F}}\left(
F^{\prime }\right) $ whenever $F^{\prime },F\in \mathcal{F}$ with $F^{\prime
}\subset F$.
\end{enumerate}
\end{definition}

\begin{definition}
If $\left( C_{0},\mathcal{F},\alpha _{\mathcal{F}}\right) $ constitutes
(general) \emph{stopping data} for a function $f\in L_{loc}^{1}\left( \sigma
\right) $, we refer to the othogonal decomposition%
\begin{equation*}
f=\sum_{F\in \mathcal{F}}\mathsf{P}_{\mathcal{C}_{F}}^{\sigma }f;\ \ \ \ \ 
\mathsf{P}_{\mathcal{C}_{F}}^{\sigma }f\equiv \sum_{I\in \mathcal{C}%
_{F}}\bigtriangleup _{I}^{\sigma }f,
\end{equation*}%
as the (general) \emph{corona decomposition} of $f$ associated with the
stopping times $\mathcal{F}$.
\end{definition}

Property (1) says that $\alpha _{\mathcal{F}}\left( F\right) $ bounds the
quasiaverages of $f$ in the corona $\mathcal{C}_{F}$, and property (2) says
that the quasicubes at the tops of the coronas satisfy a Carleson condition
relative to the weight $\sigma $. Note that a standard `maximal quasicube'
argument extends the Carleson condition in property (2) to the inequality%
\begin{equation*}
\sum_{F^{\prime }\in \mathcal{F}:\ F^{\prime }\subset A}\left\vert F^{\prime
}\right\vert _{\sigma }\leq C_{0}\left\vert A\right\vert _{\sigma }\text{
for all open sets }A\subset \mathbb{R}^{n}.
\end{equation*}%
Property (3) is the `quasi' orthogonality condition that says the sequence
of functions $\left\{ \alpha _{\mathcal{F}}\left( F\right) \mathbf{1}%
_{F}\right\} _{F\in \mathcal{F}}$ is in the vector-valued space $L^{2}\left(
\ell ^{2};\sigma \right) $, and property (4) says that the control on
stopping data is nondecreasing on the stopping tree $\mathcal{F}$. We
emphasize that we are \emph{not} assuming in this definition the stronger
property that there is $C>1$ such that $\alpha _{\mathcal{F}}\left(
F^{\prime }\right) >C\alpha _{\mathcal{F}}\left( F\right) $ whenever $%
F^{\prime },F\in \mathcal{F}$ with $F^{\prime }\subsetneqq F$. Instead, the
properties (2) and (3) substitute for this lack. Of course the stronger
property \emph{does} hold for the familiar \emph{Calder\'{o}n-Zygmund}
stopping data determined by the following requirements for $C>1$,%
\begin{eqnarray*}
\mathbb{E}_{F^{\prime }}^{\sigma }\left\vert f\right\vert &>&C\mathbb{E}%
_{F}^{\sigma }\left\vert f\right\vert \text{ whenever }F^{\prime },F\in 
\mathcal{F}\text{ with }F^{\prime }\subsetneqq F, \\
\mathbb{E}_{I}^{\sigma }\left\vert f\right\vert &\leq &C\mathbb{E}%
_{F}^{\sigma }\left\vert f\right\vert \text{ for }I\in \mathcal{C}_{F},
\end{eqnarray*}%
which are themselves sufficiently strong to automatically force properties
(2) and (3) with $\alpha _{\mathcal{F}}\left( F\right) =\mathbb{E}%
_{F}^{\sigma }\left\vert f\right\vert $.

We have the following useful consequence of (2) and (3) that says the
sequence $\left\{ \alpha _{\mathcal{F}}\left( F\right) \mathbf{1}%
_{F}\right\} _{F\in \mathcal{F}}$ has a \emph{`quasi' orthogonal} property
relative to $f$ with a constant $C_{0}^{\prime }$ depending only on $C_{0}$:%
\begin{equation}
\left\Vert \sum_{F\in \mathcal{F}}\alpha _{\mathcal{F}}\left( F\right) 
\mathbf{1}_{F}\right\Vert _{L^{2}\left( \sigma \right) }^{2}\leq
C_{0}^{\prime }\left\Vert f\right\Vert _{L^{2}\left( \sigma \right) }^{2}.
\label{q orth}
\end{equation}%
Indeed, the Carleson condition (2) implies a geometric decay in levels of
the tree $\mathcal{F}$, namely that there are positive constants $C_{1}$ and 
$\varepsilon $, depending on $C_{0}$, such that if $\mathfrak{C}_{\mathcal{F}%
}^{\left( n\right) }\left( F\right) $ denotes the set of $n^{th}$ generation
children of $F$ in $\mathcal{F}$,%
\begin{equation*}
\sum_{F^{\prime }\in \mathfrak{C}_{\mathcal{F}}^{\left( n\right) }\left(
F\right) :\ }\left\vert F^{\prime }\right\vert _{\sigma }\leq \left(
C_{1}2^{-\varepsilon n}\right) ^{2}\left\vert F\right\vert _{\sigma },\ \ \
\ \ \text{for all }n\geq 0\text{ and }F\in \mathcal{F}.
\end{equation*}%
From this we obtain that%
\begin{eqnarray*}
\sum_{n=0}^{\infty }\sum_{F^{\prime }\in \mathfrak{C}_{\mathcal{F}}^{\left(
n\right) }\left( F\right) :\ }\alpha _{\mathcal{F}}\left( F^{\prime }\right)
\left\vert F^{\prime }\right\vert _{\sigma } &\leq &\sum_{n=0}^{\infty }%
\sqrt{\sum_{F^{\prime }\in \mathfrak{C}_{\mathcal{F}}^{\left( n\right)
}\left( F\right) }\alpha _{\mathcal{F}}\left( F^{\prime }\right)
^{2}\left\vert F^{\prime }\right\vert _{\sigma }}C_{1}2^{-\varepsilon n}%
\sqrt{\left\vert F\right\vert _{\sigma }} \\
&\leq &C_{1}\sqrt{\left\vert F\right\vert _{\sigma }}C_{\varepsilon }\sqrt{%
\sum_{n=0}^{\infty }2^{-\varepsilon n}\sum_{F^{\prime }\in \mathfrak{C}_{%
\mathcal{F}}^{\left( n\right) }\left( F\right) }\alpha _{\mathcal{F}}\left(
F^{\prime }\right) ^{2}\left\vert F^{\prime }\right\vert _{\sigma }},
\end{eqnarray*}%
and hence that$\ $%
\begin{eqnarray*}
&&\sum_{F\in \mathcal{F}}\alpha _{\mathcal{F}}\left( F\right) \left\{
\sum_{n=0}^{\infty }\sum_{F^{\prime }\in \mathfrak{C}_{\mathcal{F}}^{\left(
n\right) }\left( F\right) }\alpha _{\mathcal{F}}\left( F^{\prime }\right)
\left\vert F^{\prime }\right\vert _{\sigma }\right\} \\
&\lesssim &\sum_{F\in \mathcal{F}}\alpha _{\mathcal{F}}\left( F\right) \sqrt{%
\left\vert F\right\vert _{\sigma }}\sqrt{\sum_{n=0}^{\infty }2^{-\varepsilon
n}\sum_{F^{\prime }\in \mathfrak{C}_{\mathcal{F}}^{\left( n\right) }\left(
F\right) }\alpha _{\mathcal{F}}\left( F^{\prime }\right) ^{2}\left\vert
F^{\prime }\right\vert _{\sigma }} \\
&\lesssim &\left( \sum_{F\in \mathcal{F}}\alpha _{\mathcal{F}}\left(
F\right) ^{2}\left\vert F\right\vert _{\sigma }\right) ^{\frac{1}{2}}\left(
\sum_{n=0}^{\infty }2^{-\varepsilon n}\sum_{F\in \mathcal{F}}\sum_{F^{\prime
}\in \mathfrak{C}_{\mathcal{F}}^{\left( n\right) }\left( F\right) }\alpha _{%
\mathcal{F}}\left( F^{\prime }\right) ^{2}\left\vert F^{\prime }\right\vert
_{\sigma }\right) ^{\frac{1}{2}} \\
&\lesssim &\left\Vert f\right\Vert _{L^{2}\left( \sigma \right) }\left(
\sum_{F^{\prime }\in \mathcal{F}}\alpha _{\mathcal{F}}\left( F^{\prime
}\right) ^{2}\left\vert F^{\prime }\right\vert _{\sigma }\right) ^{\frac{1}{2%
}}\lesssim \left\Vert f\right\Vert _{L^{2}\left( \sigma \right) }^{2}.
\end{eqnarray*}%
This proves (\ref{q orth}) since $\left\Vert \sum_{F\in \mathcal{F}}\alpha _{%
\mathcal{F}}\left( F\right) \mathbf{1}_{F}\right\Vert _{L^{2}\left( \sigma
\right) }^{2}$ is dominated by twice the left hand side above.

We will use a construction that permits \emph{iteration} of general corona
decompositions.

\begin{lemma}
\label{iterating coronas}Suppose that $\left( C_{0},\mathcal{F},\alpha _{%
\mathcal{F}}\right) $ constitutes \emph{stopping data} for a function $f\in
L_{loc}^{1}\left( \sigma \right) $, and that for each $F\in \mathcal{F}$, $%
\left( C_{0},\mathcal{K}\left( F\right) ,\alpha _{\mathcal{K}\left( F\right)
}\right) $ constitutes \emph{stopping data} for the corona projection $%
\mathsf{P}_{\mathcal{C}_{F}}^{\sigma }f$, where in addition $F\in \mathcal{K}%
\left( F\right) $. There is a positive constant $C_{1}$, depending only on $%
C_{0}$, such that if%
\begin{eqnarray*}
\mathcal{K}^{\ast }\left( F\right) &\equiv &\left\{ K\in \mathcal{K}\left(
F\right) \cap \mathcal{C}_{F}:\alpha _{\mathcal{K}\left( F\right) }\left(
K\right) \geq \alpha _{\mathcal{F}}\left( F\right) \right\} \\
\mathcal{K} &\equiv &\mathop{\displaystyle \bigcup }\limits_{F\in \mathcal{F}%
}\mathcal{K}^{\ast }\left( F\right) \cup \left\{ F\right\} , \\
\alpha _{\mathcal{K}}\left( K\right) &\equiv &%
\begin{array}{ccc}
\alpha _{\mathcal{K}\left( F\right) }\left( K\right) & \text{ for } & K\in 
\mathcal{K}^{\ast }\left( F\right) \setminus \left\{ F\right\} \\ 
\max \left\{ \alpha _{\mathcal{F}}\left( F\right) ,\alpha _{\mathcal{K}%
\left( F\right) }\left( F\right) \right\} & \text{ for } & K=F%
\end{array}%
,\ \ \ \ \ \text{for }F\in \mathcal{F},
\end{eqnarray*}%
the triple $\left( C_{1},\mathcal{K},\alpha _{\mathcal{K}}\right) $
constitutes \emph{stopping data} for $f$. We refer to the collection of
quasicubes $\mathcal{K}$ as the \emph{iterated} stopping times, and to the
orthogonal decomposition $f=\sum_{K\in \mathcal{K}}P_{\mathcal{C}_{K}^{%
\mathcal{K}}}f$ as the \emph{iterated} corona decomposition of $f$, where 
\begin{equation*}
\mathcal{C}_{K}^{\mathcal{K}}\equiv \left\{ I\in \Omega \mathcal{D}:I\subset
K\text{ and }I\not\subset K^{\prime }\text{ for }K^{\prime }\prec _{\mathcal{%
K}}K\right\} .
\end{equation*}
\end{lemma}

Note that in our definition of $\left( C_{1},\mathcal{K},\alpha _{\mathcal{K}%
}\right) $ we have `discarded' from $\mathcal{K}\left( F\right) $ all of
those $K\in \mathcal{K}\left( F\right) $ that are not in the corona $%
\mathcal{C}_{F}$, and also all of those $K\in \mathcal{K}\left( F\right) $
for which $\alpha _{\mathcal{K}\left( F\right) }\left( K\right) $ is
strictly less than $\alpha _{\mathcal{F}}\left( F\right) $. Then the union
over $F$ of what remains is our new collection of stopping times. We then
define stopping data $\alpha _{\mathcal{K}}\left( K\right) $ according to
whether or not $K\in \mathcal{F}$: if $K\notin \mathcal{F}$ but $K\in 
\mathcal{C}_{F}$ then $\alpha _{\mathcal{K}}\left( K\right) $ equals $\alpha
_{\mathcal{K}\left( F\right) }\left( K\right) $, while if $K\in \mathcal{F}$%
, then $\alpha _{\mathcal{K}}\left( K\right) $ is the larger of $\alpha _{%
\mathcal{K}\left( F\right) }\left( F\right) $ and $\alpha _{\mathcal{F}%
}\left( K\right) $.

\begin{proof}
The monotonicity property (4) for the triple $\left( C_{1},\mathcal{K}%
,\alpha _{\mathcal{K}}\right) $ is obvious from the construction of $%
\mathcal{K}$ and $\alpha _{\mathcal{K}}\left( K\right) $. To establish
property (1), we must distinguish between the various coronas $\mathcal{C}%
_{K}^{\mathcal{K}}$, $\mathcal{C}_{K}^{\mathcal{K}\left( F\right) }$ and $%
\mathcal{C}_{K}^{\mathcal{F}}$ that could be associated with $K\in \mathcal{K%
}$, when $K$ belongs to any of the stopping trees $\mathcal{K}$, $\mathcal{K}%
\left( F\right) $ or $\mathcal{F}$. Suppose now that $I\in \mathcal{C}_{K}^{%
\mathcal{K}}$ for some $K\in \mathcal{K}$. Then there is a unique $F\in 
\mathcal{F}$ such that $\mathcal{C}_{K}^{\mathcal{K}}\subset \mathcal{C}%
_{K}^{\mathcal{K}\left( F\right) }\subset C_{F}^{\mathcal{F}}$, and so $%
\mathbb{E}_{I}^{\sigma }\left\vert f\right\vert \leq \alpha _{\mathcal{F}%
}\left( F\right) $ by property (1) for the triple $\left( C_{0},\mathcal{F}%
,\alpha _{\mathcal{F}}\right) $. Then $\alpha _{\mathcal{F}}\left( F\right)
\leq \alpha _{\mathcal{K}}\left( K\right) $ follows from the definition of $%
\alpha _{\mathcal{K}}\left( K\right) $, and we have property (1) for the
triple $\left( C_{1},\mathcal{K},\alpha _{\mathcal{K}}\right) $. Property
(2) holds for the triple $\left( C_{1},\mathcal{K},\alpha _{\mathcal{K}%
}\right) $ since if $K\in \mathcal{C}_{F}^{\mathcal{F}}$, then 
\begin{eqnarray*}
\sum_{K^{\prime }\preceq _{\mathcal{K}}K}\left\vert K^{\prime }\right\vert
_{\sigma } &=&\sum_{K^{\prime }\in \mathcal{K}\left( F\right) :\ K^{\prime
}\subset K}\left\vert K^{\prime }\right\vert _{\sigma }+\sum_{F^{\prime
}\prec _{\mathcal{F}}F:\ F^{\prime }\subset K}\sum_{K^{\prime }\in \mathcal{K%
}\left( F^{\prime }\right) }\left\vert K^{\prime }\right\vert _{\sigma } \\
&\leq &C_{0}\left\vert K\right\vert _{\sigma }+\sum_{F^{\prime }\prec _{%
\mathcal{F}}F:\ F^{\prime }\subset K}C_{0}\left\vert F^{\prime }\right\vert
_{\sigma }\leq 2C_{0}^{2}\left\vert K\right\vert _{\sigma }.
\end{eqnarray*}%
Finally, property (3) holds for the triple $\left( C_{1},\mathcal{K},\alpha
_{\mathcal{K}}\right) $ since 
\begin{eqnarray*}
\sum_{K\in \mathcal{K}}\alpha _{\mathcal{K}}\left( K\right) ^{2}\left\vert
K\right\vert _{\sigma } &\leq &\sum_{F\in \mathcal{F}}\sum_{K\in \mathcal{K}%
\left( F\right) }\alpha _{\mathcal{K}\left( F\right) }\left( K\right)
^{2}\left\vert K\right\vert _{\sigma }+\sum_{F\in \mathcal{F}}\alpha _{%
\mathcal{F}}\left( F\right) ^{2}\left\vert F\right\vert _{\sigma } \\
&\leq &\sum_{F\in \mathcal{F}}C_{0}^{2}\left\Vert \mathsf{P}_{\mathcal{C}%
_{F}}^{\sigma }f\right\Vert _{L^{2}\left( \sigma \right)
}^{2}+C_{0}^{2}\left\Vert f\right\Vert _{L^{2}\left( \sigma \right)
}^{2}\leq 2C_{0}^{2}\left\Vert f\right\Vert _{L^{2}\left( \sigma \right)
}^{2}.
\end{eqnarray*}
\end{proof}

\subsection{Doubly iterated coronas and the NTV quasicube size splitting}

Here is a brief schematic diagram of the decompositions, with bounds in $%
\fbox{}$, used in this subsection:%
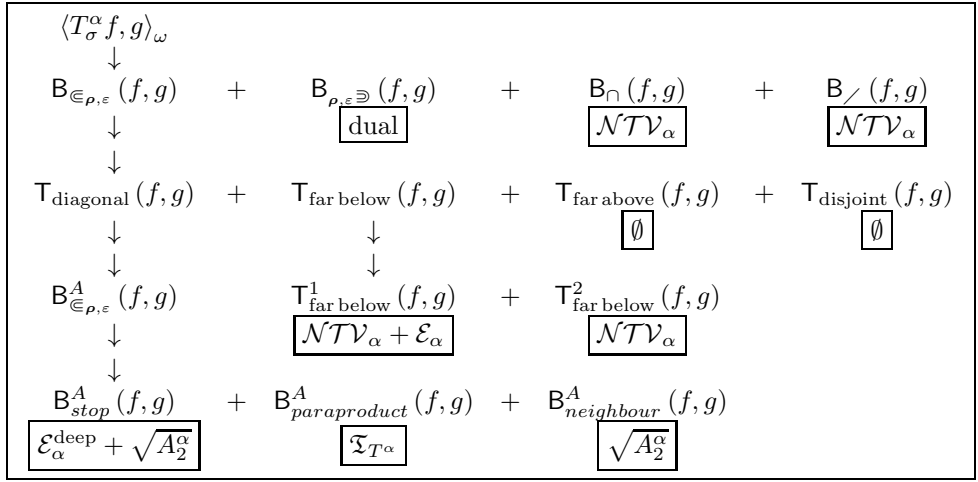
\begin{equation*}
\fbox{$%
\begin{array}{ccccccc}
\left\langle T_{\sigma }^{\alpha }f,g\right\rangle _{\omega } &  &  &  &  & 
&  \\ 
\downarrow &  &  &  &  &  &  \\ 
\mathsf{B}_{\Subset _{\mathbf{\rho },\varepsilon }}\left( f,g\right) & + & 
\mathsf{B}_{_{\mathbf{\rho },\varepsilon }\Supset }\left( f,g\right) & + & 
\mathsf{B}_{\cap }\left( f,g\right) & + & \mathsf{B}_{\diagup }\left(
f,g\right) \\ 
\downarrow &  & \fbox{dual} &  & \fbox{$\mathcal{NTV}_{\alpha }$} &  & \fbox{%
$\mathcal{NTV}_{\alpha }$} \\ 
\downarrow &  &  &  &  &  &  \\ 
\mathsf{T}_{\limfunc{diagonal}}\left( f,g\right) & + & \mathsf{T}_{\limfunc{%
far}\limfunc{below}}\left( f,g\right) & + & \mathsf{T}_{\limfunc{far}%
\limfunc{above}}\left( f,g\right) & + & \mathsf{T}_{\limfunc{disjoint}%
}\left( f,g\right) \\ 
\downarrow &  & \downarrow &  & \fbox{$\emptyset $} &  & \fbox{$\emptyset $}
\\ 
\downarrow &  & \downarrow &  &  &  &  \\ 
\mathsf{B}_{\Subset _{\mathbf{\rho },\varepsilon }}^{A}\left( f,g\right) & 
& \mathsf{T}_{\limfunc{far}\limfunc{below}}^{1}\left( f,g\right) & + & 
\mathsf{T}_{\limfunc{far}\limfunc{below}}^{2}\left( f,g\right) &  &  \\ 
\downarrow &  & \fbox{$\mathcal{NTV}_{\alpha }+\mathcal{E}_{\alpha }$} &  & 
\fbox{$\mathcal{NTV}_{\alpha }$} &  &  \\ 
\downarrow &  &  &  &  &  &  \\ 
\mathsf{B}_{stop}^{A}\left( f,g\right) & + & \mathsf{B}_{paraproduct}^{A}%
\left( f,g\right) & + & \mathsf{B}_{neighbour}^{A}\left( f,g\right) &  &  \\ 
\fbox{$\mathcal{E}_{\alpha }^{\limfunc{deep}}+\sqrt{A_{2}^{\alpha }}$} &  & 
\fbox{$\mathfrak{T}_{T^{\alpha }}$} &  & \fbox{$\sqrt{A_{2}^{\alpha }}$} & 
& 
\end{array}%
$}
\end{equation*}

We begin with the NTV \emph{quasicube size splitting} of the inner product $%
\left\langle T_{\sigma }^{\alpha }f,g\right\rangle _{\omega }$ - and later
apply the iterated corona construction - that splits the pairs of quasicubes 
$\left( I,J\right) $ in a simultaneous quasiHaar decomposition of $f$ and $g$
into four groups, namely those pairs that:

\begin{enumerate}
\item are below the size diagonal and $\mathbf{\rho }$-deeply embedded,

\item are above the size diagonal and $\mathbf{\rho }$-deeply embedded,

\item are disjoint, and

\item are of $\mathbf{\rho }$-comparable size.
\end{enumerate}

More precisely we have%
\begin{eqnarray*}
\left\langle T_{\sigma }^{\alpha }f,g\right\rangle _{\omega }
&=&\dsum\limits_{I\in \Omega \mathcal{D}^{\sigma },\ J\in \Omega \mathcal{D}%
^{\omega }}\left\langle T_{\sigma }^{\alpha }\left( \bigtriangleup
_{I}^{\sigma }f\right) ,\left( \bigtriangleup _{I}^{\omega }g\right)
\right\rangle _{\omega } \\
&=&\dsum\limits_{\substack{ I\in \Omega \mathcal{D}^{\sigma },\ J\in \Omega 
\mathcal{D}^{\omega }  \\ J\Subset _{\mathbf{\rho },\varepsilon }I}}%
\left\langle T_{\sigma }^{\alpha }\left( \bigtriangleup _{I}^{\sigma
}f\right) ,\left( \bigtriangleup _{J}^{\omega }g\right) \right\rangle
_{\omega }+\dsum\limits_{\substack{ I\in \Omega \mathcal{D}^{\sigma },\ J\in
\Omega \mathcal{D}^{\omega }  \\ J_{\mathbf{\rho },\varepsilon }\Supset I}}%
\left\langle T_{\sigma }^{\alpha }\left( \bigtriangleup _{I}^{\sigma
}f\right) ,\left( \bigtriangleup _{J}^{\omega }g\right) \right\rangle
_{\omega } \\
&&+\dsum\limits_{\substack{ I\in \Omega \mathcal{D}^{\sigma },\ J\in \Omega 
\mathcal{D}^{\omega }  \\ J\cap I=\emptyset }}\left\langle T_{\sigma
}^{\alpha }\left( \bigtriangleup _{I}^{\sigma }f\right) ,\left(
\bigtriangleup _{J}^{\omega }g\right) \right\rangle _{\omega }+\dsum\limits 
_{\substack{ I\in \Omega \mathcal{D}^{\sigma },\ J\in \Omega \mathcal{D}%
^{\omega }  \\ 2^{-\mathbf{\rho }}\leq \ell \left( J\right) \diagup \ell
\left( I\right) \leq 2^{\mathbf{\rho }}}}\left\langle T_{\sigma }^{\alpha
}\left( \bigtriangleup _{I}^{\sigma }f\right) ,\left( \bigtriangleup
_{J}^{\omega }g\right) \right\rangle _{\omega } \\
&=&\mathsf{B}_{\Subset _{\mathbf{\rho },\varepsilon }}\left( f,g\right) +%
\mathsf{B}_{_{\mathbf{\rho },\varepsilon }\Supset }\left( f,g\right) +%
\mathsf{B}_{\cap }\left( f,g\right) +\mathsf{B}_{\diagup }\left( f,g\right) .
\end{eqnarray*}%
Lemma \ref{standard delta} in the section on NTV peliminaries show that the 
\emph{disjoint} and \emph{comparable} forms $\mathsf{B}_{\cap }\left(
f,g\right) $ and $\mathsf{B}_{\diagup }\left( f,g\right) $ are both bounded
by the $\mathcal{A}_{2}^{\alpha }$, $A_{2}^{\alpha ,\limfunc{punct}}$,
quasitesting and quasiweak boundedness property constants. The \emph{below}
and \emph{above} forms are clearly symmetric, so we need only consider the
form $\mathsf{B}_{\Subset _{\mathbf{\rho },\varepsilon }}\left( f,g\right) $%
, to which we turn for the remainder of the proof.

In order to bound the below form $\mathsf{B}_{\Subset _{\mathbf{\rho }%
,\varepsilon }}\left( f,g\right) $, we will apply two different corona
decompositions in succession to the function $f\in L^{2}\left( \sigma
\right) $, gaining structure with each application; first to a boundedness
property for $f$, and then to a regularizing property of the weight $\sigma $%
. We first apply the Calder\'{o}n-Zygmund corona decomposition to the
function $f\in L^{2}\left( \sigma \right) $ obtain%
\begin{equation*}
f=\sum_{F\in \mathcal{F}}\mathsf{P}_{\mathcal{C}_{F}^{\sigma }}^{\sigma }f.
\end{equation*}%
Then for each fixed $F\in \mathcal{F}$, construct the \emph{quasienergy}
corona decomposition $\left\{ \mathcal{C}_{S}^{\sigma }\right\} _{S\in 
\mathcal{S}\left( F\right) }$\ corresponding to the weight pair $\left(
\sigma ,\omega \right) $ with top quasicube $S_{0}=F$, as given in
Definition \ref{def energy corona 3}. At this point we apply Lemma \ref%
{iterating coronas} to obtain iterated stopping times $\mathcal{S}$ and
iterated stopping data $\left\{ \alpha _{\mathcal{S}}\left( S\right)
\right\} _{S\in \mathcal{S}}$. This gives us the following \emph{double
corona decomposition} of $f$,%
\begin{equation}
f=\sum_{F\in \mathcal{F}}\mathsf{P}_{\mathcal{C}_{F}^{\sigma }}^{\sigma
}f=\sum_{F\in \mathcal{F}}\sum_{S\in \mathcal{S}^{\ast }\left( F\right) \cup
\left\{ F\right\} }\mathsf{P}_{\mathcal{C}_{S}^{\sigma }}^{\sigma }\mathsf{P}%
_{\mathcal{C}_{F}^{\sigma }}^{\sigma }f=\sum_{S\in \mathcal{S}}\mathsf{P}_{%
\mathcal{C}_{S}^{\sigma }}^{\sigma }f\equiv \sum_{A\in \mathcal{A}}\mathsf{P}%
_{\mathcal{C}_{A}}^{\sigma }f,  \label{double corona}
\end{equation}%
where $\mathcal{A}\equiv \mathcal{S}$ is the double stopping collection for $%
f$. We are relabeling the double corona as $\mathcal{A}$ here so as to
minimize confusion. We now record the main facts proved above for the double
corona.

\begin{lemma}
The data $\mathcal{A}$ and $\left\{ \alpha _{\mathcal{A}}\left( A\right)
\right\} _{A\in \mathcal{A}}$ satisfy properties (1), (2), (3) and (4) in
Definition \ref{general stopping data}.
\end{lemma}

To bound $\mathsf{B}_{\Subset _{\mathbf{\rho },\varepsilon }}\left(
f,g\right) $ we fix the stopping data $\mathcal{A}$ and $\left\{ \alpha _{%
\mathcal{A}}\left( A\right) \right\} _{A\in \mathcal{A}}$ constructed above
with the double iterated corona. We now consider the following \emph{%
canonical splitting} of the form $\mathsf{B}_{\Subset _{\mathbf{\rho }%
,\varepsilon }}\left( f,g\right) $ that involves the quasiHaar corona
projections $\mathsf{P}_{\mathcal{C}_{A}}^{\sigma }$ acting on $f$ and the $%
\mathbf{\tau }$\emph{-shifted} quasiHaar corona projections $\mathsf{P}_{%
\mathcal{C}_{B}^{\mathbf{\tau }-\limfunc{shift}}}^{\omega }$ acting on $g$.
Here the $\mathbf{\tau }$-shifted corona $\mathcal{C}_{B}^{\mathbf{\tau }-%
\limfunc{shift}}$ is defined to include only those quasicubes $J\in \mathcal{%
C}_{B}$ that are \emph{not} $\mathbf{\tau }$-nearby $B$, and to include also
such quasicubes $J$ which in addition \emph{are} $\mathbf{\tau }$-nearby in
the children $B^{\prime }$ of $B$.

\begin{definition}
\label{def parameters}The parameters $\mathbf{\tau }$ and $\mathbf{\rho }$
are now fixed to satisfy 
\begin{equation*}
\mathbf{\tau }>\mathbf{r}\text{ and }\mathbf{\rho }>\mathbf{r}+\mathbf{\tau }%
,
\end{equation*}%
where $\mathbf{r}$ is the goodness parameter already fixed.
\end{definition}

\begin{definition}
\label{shifted corona}For $B\in \mathcal{A}$ we define%
\begin{equation*}
\mathcal{C}_{B}^{\mathbf{\tau }-\limfunc{shift}}=\left\{ J\in \mathcal{C}%
_{B}:J\Subset _{\mathbf{\tau },\varepsilon }B\right\} \cup
\dbigcup\limits_{B^{\prime }\in \mathfrak{C}_{\mathcal{A}}\left( B\right)
}\left\{ J\in \Omega \mathcal{D}:J\Subset _{\mathbf{\tau },\varepsilon }B%
\text{ and }J\text{ \emph{is} }\mathbf{\tau }\text{-nearby in }B^{\prime
}\right\} .
\end{equation*}
\end{definition}

We will use repeatedly the fact that the $\mathbf{\tau }$-shifted coronas $%
\mathcal{C}_{B}^{\mathbf{\tau }-\limfunc{shift}}$ have overlap bounded by $%
\mathbf{\tau }$:%
\begin{equation}
\sum_{B\in \mathcal{A}}\mathbf{1}_{\mathcal{C}_{B}^{\mathbf{\tau }-\limfunc{%
shift}}}\left( J\right) \leq \mathbf{\tau },\ \ \ \ \ J\in \Omega \mathcal{D}%
.  \label{tau overlap}
\end{equation}%
The forms $\mathsf{B}_{\Subset _{\mathbf{\rho },\varepsilon }}\left(
f,g\right) $ are no longer linear in $f$ and $g$ as the `cut' is determined
by the coronas $\mathcal{C}_{F}$ and $\mathcal{C}_{G}^{\mathbf{\tau }-%
\limfunc{shift}}$, which depend on $f$ as well as the measures $\sigma $ and 
$\omega $. However, if the coronas are held fixed, then the forms can be
considered bilinear in $f$ and $g$. It is convenient at this point to
introduce the following shorthand notation:%
\begin{equation*}
\left\langle T_{\sigma }^{\alpha }\left( \mathsf{P}_{\mathcal{C}%
_{A}}^{\sigma }f\right) ,\mathsf{P}_{\mathcal{C}_{B}^{\mathbf{\tau }-%
\limfunc{shift}}}^{\omega }g\right\rangle _{\omega }^{\Subset _{\mathbf{\rho 
},\varepsilon }}\equiv \sum_{\substack{ I\in \mathcal{C}_{A}\text{ and }J\in 
\mathcal{C}_{B}^{\mathbf{\tau }-\limfunc{shift}}  \\ J\Subset _{\mathbf{\rho 
},\varepsilon }I}}\left\langle T_{\sigma }^{\alpha }\left( \bigtriangleup
_{I}^{\sigma }f\right) ,\left( \bigtriangleup _{J}^{\omega }g\right)
\right\rangle _{\omega }\ .
\end{equation*}%
We then have the canonical splitting,%
\begin{eqnarray}
&&\mathsf{B}_{\Subset _{\mathbf{\rho },\varepsilon }}\left( f,g\right)
\label{parallel corona decomp'} \\
&=&\sum_{A,B\in \mathcal{A}}\left\langle T_{\sigma }^{\alpha }\left( \mathsf{%
P}_{\mathcal{C}_{A}}^{\sigma }f\right) ,\mathsf{P}_{\mathcal{C}_{B}^{\mathbf{%
\tau }-\limfunc{shift}}}^{\omega }g\right\rangle _{\omega }^{\Subset _{%
\mathbf{\rho },\varepsilon }}  \notag \\
&=&\sum_{A\in \mathcal{A}}\left\langle T_{\sigma }^{\alpha }\left( \mathsf{P}%
_{\mathcal{C}_{A}}^{\sigma }f\right) ,\mathsf{P}_{\mathcal{C}_{A}^{\mathbf{%
\tau }-\limfunc{shift}}}^{\omega }g\right\rangle _{\omega }^{\Subset _{%
\mathbf{\rho },\varepsilon }}+\sum_{\substack{ A,B\in \mathcal{A}  \\ %
B\subsetneqq A}}\left\langle T_{\sigma }^{\alpha }\left( \mathsf{P}_{%
\mathcal{C}_{A}}^{\sigma }f\right) ,\mathsf{P}_{\mathcal{C}_{B}^{\mathbf{%
\tau }-\limfunc{shift}}}^{\omega }g\right\rangle _{\omega }^{\Subset _{%
\mathbf{\rho },\varepsilon }}  \notag \\
&&+\sum_{\substack{ A,B\in \mathcal{A}  \\ B\supsetneqq A}}\left\langle
T_{\sigma }^{\alpha }\left( \mathsf{P}_{\mathcal{C}_{A}}^{\sigma }f\right) ,%
\mathsf{P}_{\mathcal{C}_{B}^{\mathbf{\tau }-\limfunc{shift}}}^{\omega
}g\right\rangle _{\omega }^{\Subset _{\mathbf{\rho },\varepsilon }}+\sum 
_{\substack{ A,B\in \mathcal{A}  \\ A\cap B=\emptyset }}\left\langle
T_{\sigma }^{\alpha }\left( \mathsf{P}_{\mathcal{C}_{A}}^{\sigma }f\right) ,%
\mathsf{P}_{\mathcal{C}_{B}^{\mathbf{\tau }-\limfunc{shift}}}^{\omega
}g\right\rangle _{\omega }^{\Subset _{\mathbf{\rho },\varepsilon }}  \notag
\\
&\equiv &\mathsf{T}_{\limfunc{diagonal}}\left( f,g\right) +\mathsf{T}_{%
\limfunc{far}\limfunc{below}}\left( f,g\right) +\mathsf{T}_{\limfunc{far}%
\limfunc{above}}\left( f,g\right) +\mathsf{T}_{\limfunc{disjoint}}\left(
f,g\right) .  \notag
\end{eqnarray}%
Now the final two terms $\mathsf{T}_{\limfunc{far}\limfunc{above}}\left(
f,g\right) $ and $\mathsf{T}_{\limfunc{disjoint}}\left( f,g\right) $ each
vanish since there are no pairs $\left( I,J\right) \in \mathcal{C}_{A}\times 
\mathcal{C}_{B}^{\mathbf{\tau }-\limfunc{shift}}$ with both (\textbf{i}) $%
J\Subset _{\mathbf{\rho },\varepsilon }I$ and (\textbf{ii}) either $%
B\subsetneqq A$ or $B\cap A=\emptyset $.

The \emph{far below} term $\mathsf{T}_{\limfunc{far}\limfunc{below}}\left(
f,g\right) $ is bounded using the Intertwining Proposition and the control
of functional energy condition by the energy condition given in the next two
sections. Indeed, assuming these two results, we have from $\mathbf{\tau }<%
\mathbf{\rho }$ that%
\begin{eqnarray*}
\mathsf{T}_{\limfunc{far}\limfunc{below}}\left( f,g\right) &=&\sum 
_{\substack{ A,B\in \mathcal{A}  \\ B\subsetneqq A}}\sum_{\substack{ I\in 
\mathcal{C}_{A}\text{ and }J\in \mathcal{C}_{B}^{\mathbf{\tau }-\limfunc{%
shift}}  \\ J\Subset _{\mathbf{\rho },\varepsilon }I}}\left\langle T_{\sigma
}^{\alpha }\left( \bigtriangleup _{I}^{\sigma }f\right) ,\left(
\bigtriangleup _{J}^{\omega }g\right) \right\rangle _{\omega } \\
&=&\sum_{B\in \mathcal{A}}\sum_{A\in \mathcal{A}:\ B\subsetneqq A}\sum 
_{\substack{ I\in \mathcal{C}_{A}\text{ and }J\in \mathcal{C}_{B}^{\mathbf{%
\tau }-\limfunc{shift}}  \\ J\Subset _{\mathbf{\rho },\varepsilon }I}}%
\left\langle T_{\sigma }^{\alpha }\left( \bigtriangleup _{I}^{\sigma
}f\right) ,\left( \bigtriangleup _{J}^{\omega }g\right) \right\rangle
_{\omega } \\
&=&\sum_{B\in \mathcal{A}}\sum_{A\in \mathcal{A}:\ B\subsetneqq A}\sum_{I\in 
\mathcal{C}_{A}\text{ and }J\in \mathcal{C}_{B}^{\mathbf{\tau }-\limfunc{%
shift}}}\left\langle T_{\sigma }^{\alpha }\left( \bigtriangleup _{I}^{\sigma
}f\right) ,\left( \bigtriangleup _{J}^{\omega }g\right) \right\rangle
_{\omega } \\
&&-\sum_{B\in \mathcal{A}}\sum_{A\in \mathcal{A}:\ B\subsetneqq A}\sum 
_{\substack{ I\in \mathcal{C}_{A}\text{ and }J\in \mathcal{C}_{B}^{\mathbf{%
\tau }-\limfunc{shift}}  \\ J\Subset _{\mathbf{\rho },\varepsilon }I}}%
\left\langle T_{\sigma }^{\alpha }\left( \bigtriangleup _{I}^{\sigma
}f\right) ,\left( \bigtriangleup _{J}^{\omega }g\right) \right\rangle
_{\omega } \\
&=&\mathsf{T}_{\limfunc{far}\limfunc{below}}^{1}\left( f,g\right) -\mathsf{T}%
_{\limfunc{far}\limfunc{below}}^{2}\left( f,g\right) .
\end{eqnarray*}%
Now $\mathsf{T}_{\limfunc{far}\text{\ }\limfunc{below}}^{2}\left( f,g\right) 
$ is bounded by $\mathcal{NTV}_{\alpha }$ by Lemma \ref{standard delta}
since $J$ is good if $\bigtriangleup _{J}^{\omega }g\neq 0$.

The form $\mathsf{T}_{\limfunc{far}\limfunc{below}}^{1}\left( f,g\right) $
can be written as%
\begin{eqnarray*}
\mathsf{T}_{\limfunc{far}\limfunc{below}}^{1}\left( f,g\right) &=&\sum_{B\in 
\mathcal{A}}\sum_{I\in \Omega \mathcal{D}:\ B\subsetneqq I}\left\langle
T_{\sigma }^{\alpha }\left( \bigtriangleup _{I}^{\sigma }f\right)
,g_{B}\right\rangle _{\omega }; \\
\text{where }g_{B} &\equiv &\sum_{J\in \mathcal{C}_{B}^{\mathbf{\tau }-%
\limfunc{shift}}}\bigtriangleup _{J}^{\omega }g\ .
\end{eqnarray*}%
The Intertwining Proposition \ref{strongly adapted} applies to this latter
form and shows that it is bounded by $\mathcal{NTV}_{\alpha }+\mathfrak{F}%
_{\alpha }$. Then Proposition \ref{func ener control} shows that $\mathfrak{F%
}_{\alpha }\lesssim \mathcal{NTV}_{\alpha }+\mathcal{E}_{\alpha }$, which
completes the proof that%
\begin{equation}
\left\vert \mathsf{T}_{\limfunc{far}\limfunc{below}}\left( f,g\right)
\right\vert \lesssim \left( \mathcal{NTV}_{\alpha }+\mathcal{E}_{\alpha
}\right) \left\Vert f\right\Vert _{L^{2}\left( \sigma \right) }\left\Vert
g\right\Vert _{L^{2}\left( \omega \right) }\ .  \label{far below bound}
\end{equation}

The boundedness of the diagonal term $\mathsf{T}_{\limfunc{diagonal}}\left(
f,g\right) $ will then be reduced to the forms in the
paraproduct/neighbour/stopping form decomposition of NTV. The stopping form
is then further split into two sublinear forms in (\ref{def split}) below,
where the boundedness of the more difficult of the two is treated by
adapting the stopping time and recursion of M. Lacey \cite{Lac}. More
precisely, to handle the diagonal term $\mathsf{T}_{\limfunc{diagonal}%
}\left( f,g\right) $, it is enough to consider the individual corona pieces 
\begin{equation*}
\mathsf{B}_{\Subset _{\mathbf{\rho },\varepsilon }}^{A}\left( f,g\right)
\equiv \left\langle T_{\sigma }^{\alpha }\left( \mathsf{P}_{\mathcal{C}%
_{A}}^{\sigma }f\right) ,\mathsf{P}_{\mathcal{C}_{A}^{\mathbf{\tau }-%
\limfunc{shift}}}^{\omega }g\right\rangle _{\omega }^{\Subset }\ ,
\end{equation*}%
and to prove the following estimate:%
\begin{equation*}
\left\vert \mathsf{B}_{\Subset _{\mathbf{\rho },\varepsilon }}^{A}\left(
f,g\right) \right\vert \lesssim \left( \mathcal{NTV}_{\alpha }+\mathcal{E}%
_{\alpha }\right) \ \left( \alpha _{\mathcal{A}}\left( A\right) \sqrt{%
\left\vert A\right\vert _{\sigma }}+\left\Vert \mathsf{P}_{\mathcal{C}%
_{A}}^{\sigma }f\right\Vert _{L^{2}\left( \sigma \right) }\right) \
\left\Vert \mathsf{P}_{\mathcal{C}_{A}^{\mathbf{\tau }-\limfunc{shift}%
}}^{\omega }g\right\Vert _{L^{2}\left( \omega \right) }\ .
\end{equation*}%
Indeed, we then have from Cauchy-Schwarz that%
\begin{eqnarray*}
&&\sum_{A\in \mathcal{A}}\left\vert \mathsf{B}_{\Subset _{\mathbf{\rho }%
,\varepsilon }}^{A}\left( f,g\right) \right\vert =\sum_{A\in \mathcal{A}%
}\left\vert \mathsf{B}_{\Subset _{\mathbf{\rho },\varepsilon }}^{A}\left( 
\mathsf{P}_{\mathcal{C}_{A}}^{\sigma }f,\mathsf{P}_{\mathcal{C}_{A}^{\mathbf{%
\tau }-\limfunc{shift}}}^{\omega }g\right) \right\vert \\
&\lesssim &\left( \mathcal{NTV}_{\alpha }+\mathcal{E}_{\alpha }\right) \
\left( \sum_{A\in \mathcal{A}}\alpha _{\mathcal{A}}\left( A\right)
^{2}\left\vert A\right\vert _{\sigma }+\left\Vert \mathsf{P}_{\mathcal{C}%
_{A}}^{\sigma }f\right\Vert _{L^{2}\left( \sigma \right) }^{2}\right) ^{%
\frac{1}{2}}\ \left( \sum_{A\in \mathcal{A}}\left\Vert \mathsf{P}_{\mathcal{C%
}_{A}^{\mathbf{\tau }-\limfunc{shift}}}^{\omega }g\right\Vert _{L^{2}\left(
\omega \right) }^{2}\right) ^{\frac{1}{2}} \\
&\lesssim &\left( \mathcal{NTV}_{\alpha }+\mathcal{E}_{\alpha }\right) \
\left\Vert f\right\Vert _{L^{2}\left( \sigma \right) }\left\Vert
g\right\Vert _{L^{2}\left( \omega \right) }\ ,
\end{eqnarray*}%
where the last line uses `quasi' orthogonality in $f$ and orthogonality in
both $f$ and $g$.

Following arguments in \cite{NTV3}, \cite{Vol} and \cite{LaSaShUr}, we now
use the paraproduct / neighbour / stopping splitting of NTV to reduce
boundedness of $\mathsf{B}_{\Subset _{\mathbf{\rho },\varepsilon
}}^{A}\left( f,g\right) $ to boundedness of the associated stopping form 
\begin{equation}
\mathsf{B}_{stop}^{A}\left( f,g\right) \equiv \sum_{I\in \limfunc{supp}%
\widehat{f}}\sum_{J:\ J\Subset _{\mathbf{\rho },\varepsilon }I\text{ and }%
I_{J}\notin \mathcal{A}}\left( \mathbb{E}_{I_{J}}^{\sigma }\bigtriangleup
_{I}^{\sigma }f\right) \ \left\langle T_{\sigma }^{\alpha }\mathbf{1}%
_{A\setminus I_{J}},\bigtriangleup _{J}^{\omega }g\right\rangle _{\omega }\ ,
\label{bounded stopping form}
\end{equation}%
where $f$ is supported in the quasicube $A$ and its expectations $\mathbb{E}%
_{I}^{\sigma }\left\vert f\right\vert $ are bounded by $\alpha _{\mathcal{A}%
}\left( A\right) $ for $I\in \mathcal{C}_{A}^{\sigma }$, the quasiHaar
support of $f$ is contained in the corona $\mathcal{C}_{A}^{\sigma }$, and
the quasiHaar support of $g$\ is contained in $\mathcal{C}_{A}^{\mathbf{\tau 
}-\limfunc{shift}}$, and where $I_{J}$ is the $\Omega \mathcal{D}$-child of $%
I$ that contains $J$. Indeed, to see this, we note that $\bigtriangleup
_{I}^{\sigma }f=\mathbf{1}_{I}\bigtriangleup _{I}^{\sigma }f$ and write both%
\begin{eqnarray*}
\mathbf{1}_{I} &=&\mathbf{1}_{I_{J}}+\sum_{\theta \left( I_{J}\right) \in 
\mathfrak{C}_{\Omega \mathcal{D}}\left( I\right) \setminus \left\{
I_{J}\right\} }\mathbf{1}_{\theta \left( I_{J}\right) }\ , \\
\mathbf{1}_{I_{J}} &=&\mathbf{1}_{A}-\mathbf{1}_{A\setminus I_{J}}\ ,
\end{eqnarray*}%
where $\theta \left( I_{J}\right) \in \mathfrak{C}_{\Omega \mathcal{D}%
}\left( I\right) \setminus \left\{ I_{J}\right\} $ ranges over the $2^{n}-1$ 
$\Omega \mathcal{D}$-children of $I$ other than the child $I_{J}$ that
contains $J$. Then we obtain%
\begin{eqnarray*}
\left\langle T_{\sigma }^{\alpha }\bigtriangleup _{I}^{\sigma
}f,\bigtriangleup _{J}^{\omega }g\right\rangle _{\omega } &=&\left\langle
T_{\sigma }^{\alpha }\left( \mathbf{1}_{I_{J}}\bigtriangleup _{I}^{\sigma
}f\right) ,\bigtriangleup _{J}^{\omega }g\right\rangle _{\omega
}+\sum_{\theta \left( I_{J}\right) \in \mathfrak{C}_{\Omega \mathcal{D}%
}\left( I\right) \setminus \left\{ I_{J}\right\} }\left\langle T_{\sigma
}^{\alpha }\left( \mathbf{1}_{\theta \left( I_{J}\right) }\bigtriangleup
_{I}^{\sigma }f\right) ,\bigtriangleup _{J}^{\omega }g\right\rangle _{\omega
} \\
&=&\left( \mathbb{E}_{I_{J}}^{\sigma }\bigtriangleup _{I}^{\sigma }f\right)
\left\langle T_{\sigma }^{\alpha }\left( \mathbf{1}_{I_{J}}\right)
,\bigtriangleup _{J}^{\omega }g\right\rangle _{\omega }+\sum_{\theta \left(
I_{J}\right) \in \mathfrak{C}_{\Omega \mathcal{D}}\left( I\right) \setminus
\left\{ I_{J}\right\} }\left\langle T_{\sigma }^{\alpha }\left( \mathbf{1}%
_{\theta \left( I_{J}\right) }\bigtriangleup _{I}^{\sigma }f\right)
,\bigtriangleup _{J}^{\omega }g\right\rangle _{\omega } \\
&=&\left( \mathbb{E}_{I_{J}}^{\sigma }\bigtriangleup _{I}^{\sigma }f\right)
\left\langle T_{\sigma }^{\alpha }\mathbf{1}_{A},\bigtriangleup _{J}^{\omega
}g\right\rangle _{\omega } \\
&&-\left( \mathbb{E}_{I_{J}}^{\sigma }\bigtriangleup _{I}^{\sigma }f\right)
\left\langle T_{\sigma }^{\alpha }\mathbf{1}_{A\setminus
I_{J}},\bigtriangleup _{J}^{\omega }g\right\rangle _{\omega } \\
&&+\sum_{\theta \left( I_{J}\right) \in \mathfrak{C}_{\Omega \mathcal{D}%
}\left( I\right) \setminus \left\{ I_{J}\right\} }\left\langle T_{\sigma
}^{\alpha }\left( \mathbf{1}_{\theta \left( I_{J}\right) }\bigtriangleup
_{I}^{\sigma }f\right) ,\bigtriangleup _{J}^{\omega }g\right\rangle _{\omega
}\ ,
\end{eqnarray*}%
and the corresponding NTV splitting of $\mathsf{B}_{\Subset _{\mathbf{\rho }%
,\varepsilon }}^{A}\left( f,g\right) $:%
\begin{eqnarray*}
\mathsf{B}_{\Subset _{\mathbf{\rho },\varepsilon }}^{A}\left( f,g\right)
&=&\left\langle T_{\sigma }^{\alpha }\left( \mathsf{P}_{\mathcal{C}%
_{A}}^{\sigma }f\right) ,\mathsf{P}_{\mathcal{C}_{A}^{\mathbf{\tau }-%
\limfunc{shift}}}^{\omega }g\right\rangle _{\omega }^{\Subset _{\mathbf{\rho 
},\varepsilon }}=\sum_{\substack{ I\in \mathcal{C}_{A}\text{ and }J\in 
\mathcal{C}_{A}^{\mathbf{\tau }-\limfunc{shift}}  \\ J\Subset _{\mathbf{\rho 
},\varepsilon }I}}\left\langle T_{\sigma }^{\alpha }\left( \bigtriangleup
_{I}^{\sigma }f\right) ,\bigtriangleup _{J}^{\omega }g\right\rangle _{\omega
} \\
&=&\sum_{\substack{ I\in \mathcal{C}_{A}\text{ and }J\in \mathcal{C}_{A}^{%
\mathbf{\tau }-\limfunc{shift}}  \\ J\Subset _{\mathbf{\rho },\varepsilon }I 
}}\left( \mathbb{E}_{I_{J}}^{\sigma }\bigtriangleup _{I}^{\sigma }f\right)
\left\langle T_{\sigma }^{\alpha }\mathbf{1}_{A},\bigtriangleup _{J}^{\omega
}g\right\rangle _{\omega } \\
&&-\sum_{\substack{ I\in \mathcal{C}_{A}\text{ and }J\in \mathcal{C}_{A}^{%
\mathbf{\tau }-\limfunc{shift}}  \\ J\Subset _{\mathbf{\rho },\varepsilon }I 
}}\left( \mathbb{E}_{I_{J}}^{\sigma }\bigtriangleup _{I}^{\sigma }f\right)
\left\langle T_{\sigma }^{\alpha }\mathbf{1}_{A\setminus
I_{J}},\bigtriangleup _{J}^{\omega }g\right\rangle _{\omega } \\
&&+\sum_{\substack{ I\in \mathcal{C}_{A}\text{ and }J\in \mathcal{C}_{A}^{%
\mathbf{\tau }-\limfunc{shift}}  \\ J\Subset _{\mathbf{\rho },\varepsilon }I 
}}\sum_{\theta \left( I_{J}\right) \in \mathfrak{C}_{\Omega \mathcal{D}%
}\left( I\right) \setminus \left\{ I_{J}\right\} }\left\langle T_{\sigma
}^{\alpha }\left( \mathbf{1}_{\theta \left( I_{J}\right) }\bigtriangleup
_{I}^{\sigma }f\right) ,\bigtriangleup _{J}^{\omega }g\right\rangle _{\omega
} \\
&\equiv &\mathsf{B}_{paraproduct}^{A}\left( f,g\right) -\mathsf{B}%
_{stop}^{A}\left( f,g\right) +\mathsf{B}_{neighbour}^{A}\left( f,g\right) .
\end{eqnarray*}%
The paraproduct form $\mathsf{B}_{paraproduct}^{A}\left( f,g\right) $ is
easily controlled by the testing condition for $T^{\alpha }$. Indeed, we have%
\begin{eqnarray*}
\mathsf{B}_{paraproduct}^{A}\left( f,g\right) &=&\sum_{\substack{ I\in 
\mathcal{C}_{A}\text{ and }J\in \mathcal{C}_{A}^{\mathbf{\tau }-\limfunc{%
shift}}  \\ J\Subset _{\mathbf{\rho },\varepsilon }I}}\left( \mathbb{E}%
_{I_{J}}^{\sigma }\bigtriangleup _{I}^{\sigma }f\right) \left\langle
T_{\sigma }^{\alpha }\mathbf{1}_{A},\bigtriangleup _{J}^{\omega
}g\right\rangle _{\omega } \\
&=&\sum_{J\in \mathcal{C}_{A}^{\mathbf{\tau }-\limfunc{shift}}}\left\langle
T_{\sigma }^{\alpha }\mathbf{1}_{A},\bigtriangleup _{J}^{\omega
}g\right\rangle _{\omega }\left\{ \sum_{I\in \mathcal{C}_{A}\text{:\ }%
J\Subset _{\mathbf{\rho },\varepsilon }I}\left( \mathbb{E}_{I_{J}}^{\sigma
}\bigtriangleup _{I}^{\sigma }f\right) \right\} \\
&=&\sum_{J\in \mathcal{C}_{A}^{\mathbf{\tau }-\limfunc{shift}}}\left\langle
T_{\sigma }^{\alpha }\mathbf{1}_{A},\bigtriangleup _{J}^{\omega
}g\right\rangle _{\omega }\left\{ \mathbb{E}_{I^{\natural }\left( J\right)
_{J}}^{\sigma }f-\mathbb{E}_{A}^{\sigma }f\right\} \\
&=&\left\langle T_{\sigma }^{\alpha }\mathbf{1}_{A},\sum_{J\in \mathcal{C}%
_{A}^{\mathbf{\tau }-\limfunc{shift}}}\left\{ \mathbb{E}_{I^{\natural
}\left( J\right) _{J}}^{\sigma }f-\mathbb{E}_{A}^{\sigma }f\right\}
\bigtriangleup _{J}^{\omega }g\right\rangle _{\omega }\ ,
\end{eqnarray*}%
where $I^{\natural }\left( J\right) $ denotes the smallest quasicube $I\in 
\mathcal{C}_{A}$ such that $J\Subset _{\mathbf{\rho },\varepsilon }I$, and
of course $I^{\natural }\left( J\right) _{J}$ denotes its child containing $%
J $. We claim that by construction of the corona we have $I^{\natural
}\left( J\right) _{J}\notin \mathcal{A}$, and so $\left\vert \mathbb{E}%
_{I^{\natural }\left( J\right) _{J}}^{\sigma }f\right\vert \lesssim \mathbb{E%
}_{A}^{\sigma }\left\vert f\right\vert \leq \alpha _{\mathcal{A}}\left(
A\right) $. Indeed, in our application of the stopping form we have $f=%
\mathsf{P}_{\mathcal{C}_{A}}^{\sigma }f$ and $g=\mathsf{P}_{\mathcal{C}_{A}^{%
\mathbf{\tau }-\limfunc{shift}}}^{\omega }g$, and the definitions of the
coronas $\mathcal{C}_{A}$ and $\mathcal{C}_{A}^{\mathbf{\tau }-\limfunc{shift%
}}$ together with $\mathbf{r}<\mathbf{\tau }<\mathbf{\rho }$ imply that $%
I^{\natural }\left( J\right) _{J}\notin \mathcal{A}$ for $J\in \mathcal{C}%
_{A}^{\mathbf{\tau }-\limfunc{shift}}$.

Thus from the orthogonality of the quasiHaar projections $\bigtriangleup
_{J}^{\omega }g$ and the bound on the coefficients $\left\vert \mathbb{E}%
_{I^{\natural }\left( J\right) _{J}}^{\sigma }f-\mathbb{E}_{A}^{\sigma
}f\right\vert \lesssim \alpha _{\mathcal{A}}\left( A\right) $ we have%
\begin{eqnarray*}
\left\vert \mathsf{B}_{paraproduct}^{A}\left( f,g\right) \right\vert
&=&\left\vert \left\langle T_{\sigma }^{\alpha }\mathbf{1}_{A},\sum_{J\in 
\mathcal{C}_{A}^{\mathbf{\tau }-\limfunc{shift}}}\left\{ \mathbb{E}%
_{I^{\natural }\left( J\right) _{J}}^{\sigma }f-\mathbb{E}_{A}^{\sigma
}f\right\} \bigtriangleup _{J}^{\omega }g\right\rangle _{\omega }\right\vert
\\
&\lesssim &\alpha _{\mathcal{A}}\left( A\right) \ \left\Vert \mathbf{1}%
_{A}T_{\sigma }^{\alpha }\mathbf{1}_{A}\right\Vert _{L^{2}\left( \omega
\right) }\ \left\Vert \mathsf{P}_{\mathcal{C}_{A}^{\mathbf{\tau }-\limfunc{%
shift}}}^{\omega }g\right\Vert _{L^{2}\left( \omega \right) } \\
&\leq &\mathfrak{T}_{T^{\alpha }}\ \alpha _{\mathcal{A}}\left( A\right) \ 
\sqrt{\left\vert A\right\vert _{\sigma }}\ \left\Vert \mathsf{P}_{\mathcal{C}%
_{A}^{\mathbf{\tau }-\limfunc{shift}}}^{\omega }g\right\Vert _{L^{2}\left(
\omega \right) },
\end{eqnarray*}%
because $\left\Vert \sum_{J\in \mathcal{C}_{A}^{\mathbf{\tau }-\limfunc{shift%
}}}\lambda _{J}\bigtriangleup _{J}^{\omega }g\right\Vert _{L^{2}\left(
\omega \right) }\leq \left( \sup_{J}\left\vert \lambda _{J}\right\vert
\right) \left\Vert \sum_{J\in \mathcal{C}_{A}^{\mathbf{\tau }-\limfunc{shift}%
}}\bigtriangleup _{J}^{\omega }g\right\Vert _{L^{2}\left( \omega \right) }$.

Next, the neighbour form $\mathsf{B}_{neighbour}^{A}\left( f,g\right) $ is
easily controlled by the $A_{2}^{\alpha }$ condition using the Energy Lemma %
\ref{ener} and the fact that the quasicubes $J$ are good. In particular, the
information encoded in the stopping tree $\mathcal{A}$ plays no role here.
We have%
\begin{equation*}
\mathsf{B}_{neighbour}^{A}\left( f,g\right) =\sum_{\substack{ I\in \mathcal{C%
}_{A}\text{ and }J\in \mathcal{C}_{A}^{\mathbf{\tau }-\limfunc{shift}}  \\ %
J\Subset _{\mathbf{\rho },\varepsilon }I}}\sum_{\theta \left( I_{J}\right)
\in \mathfrak{C}_{\Omega \mathcal{D}}\left( I\right) \setminus \left\{
I_{J}\right\} }\left\langle T_{\sigma }^{\alpha }\left( \mathbf{1}_{\theta
\left( I_{J}\right) }\bigtriangleup _{I}^{\sigma }f\right) ,\bigtriangleup
_{J}^{\omega }g\right\rangle _{\omega }.
\end{equation*}%
Recall that $I_{J}$ is the child of $I$ that contains $J$. Fix $\theta
\left( I_{J}\right) \in \mathfrak{C}_{\Omega \mathcal{D}}\left( I\right)
\setminus \left\{ I_{J}\right\} $ momentarily, and an integer $s\geq \mathbf{%
r}$. The inner product to be estimated is 
\begin{equation*}
\left\langle T_{\sigma }^{\alpha }(\mathbf{1}_{\theta \left( I_{J}\right)
}\sigma \Delta _{I}^{\sigma }f),\Delta _{J}^{\omega }g\right\rangle _{\omega
},
\end{equation*}%
i.e. 
\begin{equation*}
\left\langle T_{\sigma }^{\alpha }\left( \mathbf{1}_{\theta \left(
I_{J}\right) }\bigtriangleup _{I}^{\sigma }f\right) ,\bigtriangleup
_{J}^{\omega }g\right\rangle _{\omega }=\mathbb{E}_{\theta \left(
I_{J}\right) }^{\sigma }\Delta _{I}^{\sigma }f\cdot \left\langle T_{\sigma
}^{\alpha }\left( \mathbf{1}_{\theta \left( I_{J}\right) }\right)
,\bigtriangleup _{J}^{\omega }g\right\rangle _{\omega }.
\end{equation*}%
Thus we can write%
\begin{equation}
\mathsf{B}_{neighbour}^{A}\left( f,g\right) =\sum_{\substack{ I\in \mathcal{C%
}_{A}\text{ and }J\in \mathcal{C}_{A}^{\mathbf{\tau }-\limfunc{shift}}  \\ %
J\Subset _{\mathbf{\rho },\varepsilon }I}}\sum_{\theta \left( I_{J}\right)
\in \mathfrak{C}_{\Omega \mathcal{D}}\left( I\right) \setminus \left\{
I_{J}\right\} }\left( \mathbb{E}_{\theta \left( I_{J}\right) }^{\sigma
}\Delta _{I}^{\sigma }f\right) \ \left\langle T_{\sigma }^{\alpha }\left( 
\mathbf{1}_{\theta \left( I_{J}\right) }\sigma \right) ,\Delta _{J}^{\omega
}g\right\rangle _{\omega }  \label{neighbour term}
\end{equation}

Now we will use the following fractional analogue of the Poisson inequality
in \cite{Vol}. We remind the reader that there are absolute positive
constants $c,C\,$\ such that $c\left\vert J\right\vert ^{\frac{1}{n}}\leq
\ell \left( J\right) \leq C\left\vert J\right\vert ^{\frac{1}{n}}$ for all
quasicubes $J$, and that we defined the quasidistance $\limfunc{qdist}\left(
E,F\right) $ between two sets $E$ and $F$ to be the Euclidean distance $%
\limfunc{dist}\left( \Omega ^{-1}E,\Omega ^{-1}F\right) $ between the
preimages under $\Omega $ of the sets $E$ and $F$.

\begin{lemma}
\label{Poisson inequality}Suppose that $J\subset I\subset K$ and that $%
\limfunc{qdist}\left( J,\partial I\right) >\tfrac{1}{2}\ell \left( J\right)
^{\varepsilon }\ell \left( I\right) ^{1-\varepsilon }$. Then%
\begin{equation}
\mathrm{P}^{\alpha }(J,\sigma \mathbf{1}_{K\setminus I})\lesssim \left( 
\frac{\ell \left( J\right) }{\ell \left( I\right) }\right) ^{1-\varepsilon
\left( n+1-\alpha \right) }\mathrm{P}^{\alpha }(I,\sigma \mathbf{1}%
_{K\setminus I}).  \label{e.Jsimeq}
\end{equation}
\end{lemma}

\begin{proof}
We have%
\begin{equation*}
\mathrm{P}^{\alpha }\left( J,\sigma \chi _{K\setminus I}\right) \approx
\sum_{k=0}^{\infty }2^{-k}\frac{1}{\left\vert 2^{k}J\right\vert ^{1-\frac{%
\alpha }{n}}}\int_{\left( 2^{k}J\right) \cap \left( K\setminus I\right)
}d\sigma ,
\end{equation*}%
and $\left( 2^{k}J\right) \cap \left( K\setminus I\right) \neq \emptyset $
requires%
\begin{equation*}
\limfunc{qdist}\left( J,K\setminus I\right) \leq c2^{k}\ell \left( J\right) ,
\end{equation*}%
for some dimensional constant $c>0$. Let $k_{0}$ be the smallest such $k$.
By our distance assumption we must then have%
\begin{equation*}
\tfrac{1}{2}\ell \left( J\right) ^{\varepsilon }\ell \left( I\right)
^{1-\varepsilon }\leq \limfunc{qdist}\left( J,\partial I\right) \leq
c2^{k_{0}}\ell \left( J\right) ,
\end{equation*}%
or%
\begin{equation*}
2^{-k_{0}-1}\leq c\left( \frac{\ell \left( J\right) }{\ell \left( I\right) }%
\right) ^{1-\varepsilon }.
\end{equation*}%
Now let $k_{1}$ be defined by $2^{k_{1}}\equiv \frac{\ell \left( I\right) }{%
\ell \left( J\right) }$. Then assuming $k_{1}>k_{0}$ (the case $k_{1}\leq
k_{0}$ is similar) we have%
\begin{eqnarray*}
\mathrm{P}^{\alpha }\left( J,\sigma \chi _{K\setminus I}\right) &\approx
&\left\{ \sum_{k=k_{0}}^{k_{1}}+\sum_{k=k_{1}}^{\infty }\right\} 2^{-k}\frac{%
1}{\left\vert 2^{k}J\right\vert ^{1-\frac{\alpha }{n}}}\int_{\left(
2^{k}J\right) \cap \left( K\setminus I\right) }d\sigma \\
&\lesssim &2^{-k_{0}}\frac{\left\vert I\right\vert ^{1-\frac{\alpha }{n}}}{%
\left\vert 2^{k_{0}}J\right\vert ^{1-\frac{\alpha }{n}}}\left( \frac{1}{%
\left\vert I\right\vert ^{1-\frac{\alpha }{n}}}\int_{\left(
2^{k_{1}}J\right) \cap \left( K\setminus I\right) }d\sigma \right)
+2^{-k_{1}}\mathrm{P}^{\alpha }\left( I,\sigma \chi _{\setminus I}\right) \\
&\lesssim &\left( \frac{\ell \left( J\right) }{\ell \left( I\right) }\right)
^{\left( 1-\varepsilon \right) \left( n+1-\alpha \right) }\left( \frac{\ell
\left( I\right) }{\ell \left( J\right) }\right) ^{n-\alpha }\mathrm{P}%
^{\alpha }\left( I,\sigma \chi _{K\setminus I}\right) +\frac{\ell \left(
J\right) }{\ell \left( I\right) }\mathrm{P}^{\alpha }\left( I,\sigma \chi
_{K\setminus I}\right) ,
\end{eqnarray*}%
which is the inequality (\ref{e.Jsimeq}).
\end{proof}

Now fix $I_{0},I_{\theta }\in \mathfrak{C}_{\Omega \mathcal{D}}\left(
I\right) $ with $I_{0}\neq I_{\theta }$ and assume that $J\Subset _{\mathbf{r%
},\varepsilon }I_{0}$. Let $\frac{\ell \left( J\right) }{\ell \left(
I_{0}\right) }=2^{-s}$ in the pivotal estimate in the Energy Lemma \ref{ener}
with $J\subset I_{0}\subset I$ to obtain 
\begin{align*}
\left\vert \langle T_{\sigma }^{\alpha }\left( \mathbf{1}_{I_{\theta
}}\sigma \right) ,\Delta _{J}^{\omega }g\rangle _{\omega }\right\vert &
\lesssim \left\Vert \Delta _{J}^{\omega }g\right\Vert _{L^{2}\left( \omega
\right) }\sqrt{\left\vert J\right\vert _{\omega }}\mathrm{P}^{\alpha }\left(
J,\mathbf{1}_{I_{\theta }}\sigma \right) \\
& \lesssim \left\Vert \Delta _{J}^{\omega }g\right\Vert _{L^{2}\left( \omega
\right) }\sqrt{\left\vert J\right\vert _{\omega }}\cdot 2^{-\left(
1-\varepsilon \left( n+1-\alpha \right) \right) s}\mathrm{P}^{\alpha }\left(
I_{0},\mathbf{1}_{I_{\theta }}\sigma \right) .
\end{align*}%
Here we are using (\ref{e.Jsimeq}), which applies since $J\subset I_{0}$.

In the sum below, we keep the side lengths of the quasicubes $J$ fixed at $%
2^{-s}$ times that of $I$, and of course take $J\subset I_{0}$. We estimate 
\begin{align*}
A(I,I_{0},I_{\theta },s)& \equiv \sum_{J\;:\;2^{s}\ell \left( J\right) =\ell
\left( I\right) :J\subset I_{0}}\left\vert \langle T_{\sigma }^{\alpha
}\left( \mathbf{1}_{I_{\theta }}\sigma \Delta _{I}^{\sigma }f\right) ,\Delta
_{J}^{\omega }g\rangle _{\omega }\right\vert \\
& \leq 2^{-\left( 1-\varepsilon \left( n+1-\alpha \right) \right) s}|\mathbb{%
E}_{I_{\theta }}^{\sigma }\Delta _{I}^{\sigma }f|\ \mathrm{P}^{\alpha
}(I_{0},\mathbf{1}_{I_{\theta }}\sigma )\sum_{J\;:\;2^{s}\ell \left(
J\right) =\ell \left( I\right) :\ J\subset I_{0}}\left\Vert \Delta
_{J}^{\omega }g\right\Vert _{L^{2}\left( \omega \right) }\sqrt{\left\vert
J\right\vert _{\omega }} \\
& \leq 2^{-\left( 1-\varepsilon \left( n+1-\alpha \right) \right) s}|\mathbb{%
E}_{I_{\theta }}^{\sigma }\Delta _{I}^{\sigma }f|\ \mathrm{P}^{\alpha
}(I_{0},\mathbf{1}_{I_{\theta }}\sigma )\sqrt{\left\vert I_{0}\right\vert
_{\omega }}\Lambda (I,I_{0},I_{\theta },s), \\
\Lambda (I,I_{0},I_{\theta },s)^{2}& \equiv \sum_{J\in \mathcal{C}_{A}^{%
\mathbf{\tau }-\limfunc{shift}}:\;2^{s}\ell \left( J\right) =\ell \left(
I\right) :\ J\subset I_{0}}\left\Vert \Delta _{J}^{\omega }g\right\Vert
_{L^{2}\left( \omega \right) }^{2}\,.
\end{align*}%
The last line follows upon using the Cauchy-Schwarz inequality and the fact
that $\Delta _{J}^{\omega }g=0$ if $J\notin \mathcal{C}_{A}^{\mathbf{\tau }-%
\limfunc{shift}}$. We also note that since $2^{s+1}\ell \left( J\right)
=\ell \left( I\right) $, 
\begin{eqnarray}
\sum_{I_{0}\in \mathfrak{C}_{\Omega \mathcal{D}}\left( I\right) }\Lambda
(I,I_{0},I_{\theta },s)^{2} &\equiv &\sum_{J\in \mathcal{C}_{A}^{\mathbf{%
\tau }-\limfunc{shift}}:\;2^{s+1}\ell \left( J\right) =\ell \left( I\right)
:\ J\subset I}\left\Vert \Delta _{J}^{\omega }g\right\Vert _{L^{2}\left(
\omega \right) }^{2}\ ;  \label{g} \\
\sum_{I\in \mathcal{C}_{A}}\sum_{I_{0}\in \mathfrak{C}_{\Omega \mathcal{D}%
}\left( I\right) }\Lambda (I,I_{0},I_{\theta },s)^{2} &\leq &\left\Vert 
\mathsf{P}_{\mathcal{C}_{A}^{\mathbf{\tau }-\limfunc{shift}}}^{\omega
}g\right\Vert _{L^{2}(\omega )}^{2}\ .  \notag
\end{eqnarray}

Using 
\begin{equation}
\left\vert \mathbb{E}_{I_{\theta }}^{\sigma }\Delta _{I}^{\sigma
}f\right\vert \leq \sqrt{\mathbb{E}_{I_{\theta }}^{\sigma }\left\vert \Delta
_{I}^{\sigma }f\right\vert ^{2}}\leq \left\Vert \Delta _{I}^{\sigma
}f\right\Vert _{L^{2}\left( \sigma \right) }\ \left\vert I_{\theta
}\right\vert _{\sigma }^{-\frac{1}{2}},  \label{e.haarAvg}
\end{equation}%
we can thus estimate $A(I,I_{0},I_{\theta },s)$ as follows, in which we use
the $A_{2}^{\alpha }$ hypothesis $\sup_{I}\frac{\left\vert I\right\vert
_{\sigma }\left\vert I\right\vert _{\omega }}{\left\vert I\right\vert
^{2\left( 1-\frac{\alpha }{n}\right) }}=A_{2}^{\alpha }<\infty $: 
\begin{eqnarray*}
A(I,I_{0},I_{\theta },s) &\lesssim &2^{-\left( 1-\varepsilon \left(
n+1-\alpha \right) \right) s}\left\Vert \Delta _{I}^{\sigma }f\right\Vert
_{L^{2}\left( \sigma \right) }\Lambda (I,I_{0},I_{\theta },s)\cdot
\left\vert I_{\theta }\right\vert _{\sigma }^{-\frac{1}{2}}\mathrm{P}%
^{\alpha }(I_{0},\mathbf{1}_{I_{\theta }}\sigma )\sqrt{\left\vert
I_{0}\right\vert _{\omega }} \\
&\lesssim &\sqrt{A_{2}^{\alpha }}2^{-\left( 1-\varepsilon \left( n+1-\alpha
\right) \right) s}\left\Vert \Delta _{I}^{\sigma }f\right\Vert _{L^{2}\left(
\sigma \right) }\Lambda (I,I_{0},I_{\theta },s)\,,
\end{eqnarray*}%
since $\mathrm{P}^{\alpha }(I_{0},\mathbf{1}_{I_{\theta }}\sigma )\lesssim 
\frac{\left\vert I_{\theta }\right\vert _{\sigma }}{\left\vert I_{\theta
}\right\vert ^{1-\frac{\alpha }{n}}}$ shows that 
\begin{equation*}
\left\vert I_{\theta }\right\vert _{\sigma }^{-\frac{1}{2}}\mathrm{P}%
^{\alpha }(I_{0},\mathbf{1}_{I_{\theta }}\sigma )\ \sqrt{\left\vert
I_{0}\right\vert _{\omega }}\lesssim \frac{\sqrt{\left\vert I_{\theta
}\right\vert _{\sigma }}\sqrt{\left\vert I_{0}\right\vert _{\omega }}}{%
\left\vert I\right\vert ^{1-\frac{\alpha }{n}}}\lesssim \sqrt{A_{2}^{\alpha }%
}.
\end{equation*}

An application of Cauchy-Schwarz to the sum over $I$ using (\ref{g}) then
shows that 
\begin{eqnarray*}
&&\sum_{I\in \mathcal{C}_{A}}\sum_{\substack{ I_{0},I_{\theta }\in \mathfrak{%
C}_{\Omega \mathcal{D}}\left( I\right)  \\ I_{0}\neq I_{\theta }}}%
A(I,I_{0},I_{\theta },s) \\
&\lesssim &\sqrt{A_{2}^{\alpha }}2^{-\left( 1-\varepsilon \left( n+1-\alpha
\right) \right) s}\sqrt{\sum_{I\in \mathcal{C}_{A}}\left\Vert \Delta
_{I}^{\sigma }f\right\Vert _{L^{2}\left( \sigma \right) }^{2}}\sqrt{%
\sum_{I\in \mathcal{C}_{A}}\left( \sum_{\substack{ I_{0},I_{\theta }\in 
\mathfrak{C}_{\Omega \mathcal{D}}\left( I\right)  \\ I_{0}\neq I_{\theta }}}%
\Lambda (I,I_{0},I_{\theta },s)\right) ^{2}} \\
&\lesssim &\sqrt{A_{2}^{\alpha }}2^{-\left( 1-\varepsilon \left( n+1-\alpha
\right) \right) s}\lVert \mathsf{P}_{\mathcal{C}_{A}}^{\sigma }f\rVert
_{L^{2}(\sigma )}\sqrt{2^{n}}\sqrt{\sum_{I\in \mathcal{C}_{A}}\left( \sum 
_{\substack{ I_{0}\in \mathfrak{C}_{\Omega \mathcal{D}}\left( I\right)  \\ %
I_{0}\neq I_{\theta }}}\Lambda (I,I_{0},I_{\theta },s)\right) ^{2}} \\
&\lesssim &\sqrt{A_{2}^{\alpha }}2^{-\left( 1-\varepsilon \left( n+1-\alpha
\right) \right) s}\lVert \mathsf{P}_{\mathcal{C}_{A}}^{\sigma }f\rVert
_{L^{2}(\sigma )}\left\Vert \mathsf{P}_{\mathcal{C}_{A}^{\mathbf{\tau }-%
\limfunc{shift}}}^{\omega }g\right\Vert _{L^{2}(\omega )}\,.
\end{eqnarray*}%
This estimate is summable in $s\geq \mathbf{r}$, and so the proof of 
\begin{eqnarray*}
\left\vert \mathsf{B}_{neighbour}^{A}\left( f,g\right) \right\vert
&=&\left\vert \sum_{\substack{ I\in \mathcal{C}_{A}\text{ and }J\in \mathcal{%
C}_{A}^{\mathbf{\tau }-\limfunc{shift}}  \\ J\Subset _{\mathbf{\rho }%
,\varepsilon }I}}\sum_{\theta \left( I_{J}\right) \in \mathfrak{C}_{\Omega 
\mathcal{D}}\left( I\right) \setminus \left\{ I_{J}\right\} }\left\langle
T_{\sigma }^{\alpha }\left( \mathbf{1}_{\theta \left( I_{J}\right)
}\bigtriangleup _{I}^{\sigma }f\right) ,\bigtriangleup _{J}^{\omega
}g\right\rangle _{\omega }\right\vert \\
&=&\left\vert \sum_{I\in \mathcal{C}_{A}}\sum_{\substack{ I_{0},I_{\theta
}\in \mathfrak{C}_{\Omega \mathcal{D}}\left( I\right)  \\ I_{0}\neq
I_{\theta }}}\sum_{s=\mathbf{r}}^{\infty }A(I,I_{0},I_{\theta },s)\right\vert
\\
&\lesssim &\sqrt{A_{2}^{\alpha }}\left\Vert \mathsf{P}_{\mathcal{C}%
_{A}}^{\sigma }f\right\Vert _{L^{2}(\sigma )}\left\Vert \mathsf{P}_{\mathcal{%
C}_{A}^{\mathbf{\tau }-\limfunc{shift}}}^{\omega }g\right\Vert
_{L^{2}(\omega )}
\end{eqnarray*}%
is complete.

\bigskip

It is to the sublinear form on the left side of (\ref{First inequality})
below, derived from the stopping form $\mathsf{B}_{stop}^{A}\left(
f,g\right) $, that the argument of M. Lacey in \cite{Lac} will be adapted.
This will result in the inequality%
\begin{equation}
\left\vert \mathsf{B}_{stop}^{A}\left( f,g\right) \right\vert \lesssim
\left( \mathcal{E}_{\alpha }^{\limfunc{deep}}+\sqrt{A_{2}^{\alpha }}+\sqrt{%
A_{2}^{\alpha ,\limfunc{punct}}}\right) \ \left( \alpha _{\mathcal{A}}\left(
A\right) \sqrt{\left\vert A\right\vert _{\sigma }}+\left\Vert f\right\Vert
_{L^{2}\left( \sigma \right) }\right) \ \left\Vert g\right\Vert
_{L^{2}\left( \omega \right) }\ ,\ \ \ \ \ A\in \mathcal{A},
\label{B stop form 3}
\end{equation}%
where the bounded quasiaverages of $f$ in $\mathsf{B}_{stop}^{A}\left(
f,g\right) $ will prove crucial. But first we turn to completing the proof
of the bound (\ref{far below bound}) for the far below form $\mathsf{T}_{%
\limfunc{far}\limfunc{below}}\left( f,g\right) $ using the Intertwining
Proposition and the control of functional energy by the $\mathcal{A}%
_{2}^{\alpha }$ condition and the energy condition $\mathcal{E}_{\alpha }$.

\section{Intertwining proposition}

Here we generalize the Intertwining Proposition (see e.g. \cite{Saw} and 
\cite{SaShUr5}) to higher dimensions. The main principle here says that,
modulo terms that are controlled by the $\gamma $-functional quasienergy
constant $\mathfrak{F}_{\alpha }$ and the quasiNTV constant $\mathcal{NTV}%
_{\alpha }$ (see below), we can pass the shifted $\omega $-corona projection 
$\mathsf{P}_{\mathcal{C}_{B}^{\mathbf{\tau }-\limfunc{shift}}}^{\omega }$
through the operator $T^{\alpha }$ to become the shifted $\sigma $-corona
projection $\mathsf{P}_{\mathcal{C}_{B}^{\mathbf{\tau }-\limfunc{shift}%
}}^{\sigma }$, provided the goodness parameters $\mathbf{r},\varepsilon $
are chosen sufficiently large and small respectively depending on $\gamma $.
More precisely, the idea is that with $T_{\sigma }^{\alpha }f\equiv
T^{\alpha }\left( f\sigma \right) $, the intertwining operator 
\begin{equation*}
\mathsf{P}_{\mathcal{C}_{B}^{\mathbf{\tau }-\limfunc{shift}}}^{\omega }\left[
\mathsf{P}_{\mathcal{C}_{B}^{\mathbf{\tau }-\limfunc{shift}}}^{\omega
}T_{\sigma }^{\alpha }-T_{\sigma }^{\alpha }\mathsf{P}_{\mathcal{C}_{B}^{%
\mathbf{\tau }-\limfunc{shift}}}^{\sigma }\right] \mathsf{P}_{\mathcal{C}%
_{A}}^{\sigma }
\end{equation*}%
is bounded with constant $\mathfrak{F}_{\alpha }+\mathcal{NTV}_{\alpha }$,
provided $\gamma \leq c_{n}2^{\left( 1-\varepsilon \right) \mathbf{r}}$. In
those cases where the coronas $\mathcal{C}_{B}^{\mathbf{\tau }-\limfunc{shift%
}}$ and $\mathcal{C}_{A}$ are (almost) disjoint, the intertwining operator
reduces (essentially) to $\mathsf{P}_{\mathcal{C}_{B}^{\mathbf{\tau }-%
\limfunc{shift}}}^{\omega }T_{\sigma }^{\alpha }\mathsf{P}_{\mathcal{C}%
_{A}}^{\sigma }$, and then combined with the control of the functional
quasienergy constant $\mathfrak{F}_{\alpha }$ by the quasienergy condition
constant $\mathcal{E}_{\alpha }$ and $\mathcal{A}_{2}^{\alpha }+\mathcal{A}%
_{2}^{\alpha ,\ast }+A_{2}^{\alpha ,\limfunc{punct}}$, provided $\gamma $ is
sufficiently large depending only on $n$ and $\alpha $, we obtain the
required bound (\ref{far below bound}) for $\mathsf{T}_{\limfunc{far}%
\limfunc{below}}\left( f,g\right) $ above.

To describe the quantities we use to bound these forms, we need to adapt to
higher dimensions three definitions used for the Hilbert transform that are
relevant to functional energy.

\begin{definition}
\label{sigma carleson n}A collection $\mathcal{F}$ of dyadic quasicubes is $%
\sigma $\emph{-Carleson} if%
\begin{equation*}
\sum_{F\in \mathcal{F}:\ F\subset S}\left\vert F\right\vert _{\sigma }\leq
C_{\mathcal{F}}\left\vert S\right\vert _{\sigma },\ \ \ \ \ S\in \mathcal{F}.
\end{equation*}%
The constant $C_{\mathcal{F}}$ is referred to as the Carleson norm of $%
\mathcal{F}$.
\end{definition}

\begin{definition}
Let $\mathcal{F}$ be a collection of dyadic quasicubes. The good $\tau $%
-shifted corona corresponding to $F$ is defined by%
\begin{equation*}
\mathcal{C}_{F}^{\limfunc{good},\mathbf{\tau }-\limfunc{shift}}\equiv
\left\{ J\in \Omega \mathcal{D}_{\limfunc{good}}^{\omega }:J\Subset _{%
\mathbf{\tau },\varepsilon }F\text{ and }J\not\Subset _{\mathbf{\tau }%
,\varepsilon }F^{\prime }\text{ for any }F^{\prime }\in \mathfrak{C}_{%
\mathcal{F}}\left( F\right) \right\} .
\end{equation*}
\end{definition}

Note that $\mathcal{C}_{F}^{\limfunc{good},\mathbf{\tau }-\limfunc{shift}}=%
\mathcal{C}_{F}^{\mathbf{\tau }-\limfunc{shift}}\cap \Omega \mathcal{D}_{%
\limfunc{good}}^{\omega }$, and that the collections $\mathcal{C}_{F}^{%
\limfunc{good},\mathbf{\tau }-\limfunc{shift}}$ have bounded overlap $%
\mathbf{\tau }$ since for fixed $J$, there are at most $\mathbf{\tau }$
quasicubes $F\in \mathcal{F}$ with the property that $J$ is good and $%
J\Subset _{\mathbf{\tau },\varepsilon }F$ and $J\not\Subset _{\mathbf{\tau }%
,\varepsilon }F^{\prime }$ for any $F^{\prime }\in \mathfrak{C}_{\mathcal{F}%
}\left( F\right) $. Here $\mathfrak{C}_{\mathcal{F}}\left( F\right) $
denotes the set of $\mathcal{F}$-children of $F$. Given any collection $%
\mathcal{H}\subset \Omega \mathcal{D}$ of quasicubes, and a dyadic quasicube 
$J$, we define the corresponding quasiHaar projection $\mathsf{P}_{\mathcal{H%
}}^{\omega }$ and its localization $\mathsf{P}_{\mathcal{H};J}^{\omega }$ to 
$J$ by%
\begin{equation}
\mathsf{P}_{\mathcal{H}}^{\omega }=\sum_{H\in \mathcal{H}}\bigtriangleup
_{H}^{\omega }\text{ and }\mathsf{P}_{\mathcal{H};J}^{\omega }=\sum_{H\in 
\mathcal{H}:\ H\subset J}\bigtriangleup _{H}^{\omega }\ .
\label{def localization}
\end{equation}

\begin{definition}
\label{functional energy n}Let $\mathfrak{F}_{\alpha }$ be the smallest
constant in the `\textbf{f}unctional quasienergy' inequality below, holding
for all $h\in L^{2}\left( \sigma \right) $ and all $\sigma $-Carleson
collections $\mathcal{F}$ with Carleson norm $C_{\mathcal{F}}$ bounded by a
fixed constant $C$: 
\begin{equation}
\sum_{F\in \mathcal{F}}\sum_{J\in \mathcal{M}_{\left( \mathbf{r},\varepsilon
\right) -\limfunc{deep}}\left( F\right) }\left( \frac{\mathrm{P}^{\alpha
}\left( J,h\sigma \right) }{\left\vert J\right\vert ^{\frac{1}{n}}}\right)
^{2}\left\Vert \mathsf{P}_{\mathcal{C}_{F}^{\limfunc{good},\mathbf{\tau }-%
\limfunc{shift}};J}^{\omega }\mathbf{x}\right\Vert _{L^{2}\left( \omega
\right) }^{2}\leq \mathfrak{F}_{\alpha }\lVert h\rVert _{L^{2}\left( \sigma
\right) }\,.  \label{e.funcEnergy n}
\end{equation}
\end{definition}

This definition of $\mathfrak{F}_{\alpha }$ depends also on the choice of
the fixed constant $C$, but it will be clear from the arguments below that $%
C $ may be taken to depend only on $n$ and $\alpha $, and we do not compute
its value here. There is a similar definition of the dual constant $%
\mathfrak{F}_{\alpha }^{\ast }$.

\begin{remark}
If in (\ref{e.funcEnergy n}), we take $h=\mathbf{1}_{I}$ and $\mathcal{F}$
to be the trivial Carleson collection $\left\{ I_{r}\right\} _{r=1}^{\infty
} $ where the quasicubes $I_{r}$ are pairwise disjoint in $I$, then we
obtain the deep quasienergy condition in Definition \ref{energy condition},
but with $\mathsf{P}_{\mathcal{C}_{F}^{\limfunc{good},\mathbf{\tau }-%
\limfunc{shift}};J}^{\omega }$ in place of $\mathsf{P}_{J}^{\limfunc{subgood}%
,\omega } $. However, the projection $\mathsf{P}_{J}^{\limfunc{subgood}%
,\omega }$ is larger than $\mathsf{P}_{\mathcal{C}_{F}^{\limfunc{good},%
\mathbf{\tau }-\limfunc{shift}};J}^{\omega }$ by the finite projection $%
\sum_{2^{-\mathbf{\tau }}\ell \left( J\right) \leq \ell \left( J^{\prime
}\right) \leq 2^{-\mathbf{r}}\ell \left( J\right) }\bigtriangleup
_{J^{\prime }}^{\omega }$, and so we just miss obtaining the deep
quasienergy condition as a consequence of the functional quasienergy
condition. Nevertheless, this near miss with $h=\mathbf{1}_{I}$ explains the
terminology `functional' quasienergy.
\end{remark}

We now show that the functional quasienergy inequality (\ref{e.funcEnergy n}%
) suffices to prove an $\alpha $-fractional $n$-dimensional analogue of the
Intertwining Proposition (see e.g. \cite{Saw}). Let $\mathcal{F}$ be any
subset of $\Omega \mathcal{D}$. For any $J\in \Omega \mathcal{D}$, we define 
$\pi _{\mathcal{F}}^{0}J$ to be the smallest $F\in \mathcal{F}$ that
contains $J$. Then for $s\geq 1$, we recursively define $\pi _{\mathcal{F}%
}^{s}J$ to be the smallest $F\in \mathcal{F}$ that \emph{strictly} contains $%
\pi _{\mathcal{F}}^{s-1}J$. This definition satisfies $\pi _{\mathcal{F}%
}^{s+t}J=\pi _{\mathcal{F}}^{s}\pi _{\mathcal{F}}^{t}J$ for all $s,t\geq 0$
and $J\in \Omega \mathcal{D}$. In particular $\pi _{\mathcal{F}}^{s}J=\pi _{%
\mathcal{F}}^{s}F$ where $F=\pi _{\mathcal{F}}^{0}J$. In the special case $%
\mathcal{F}=\Omega \mathcal{D}$ we often suppress the subscript $\mathcal{F}$
and simply write $\pi ^{s}$ for $\pi _{\Omega \mathcal{D}}^{s}$. Finally,
for $F\in \mathcal{F}$, we write $\mathfrak{C}_{\mathcal{F}}\left( F\right)
\equiv \left\{ F^{\prime }\in \mathcal{F}:\pi _{\mathcal{F}}^{1}F^{\prime
}=F\right\} $ for the collection of $\mathcal{F}$-children of $F$. Let 
\begin{equation*}
\mathcal{NTV}_{\alpha }\equiv \sqrt{\mathcal{A}_{2}^{\alpha }}+\mathfrak{T}%
_{\alpha }+\mathcal{WBP}_{\alpha },
\end{equation*}%
where we remind the reader that $\mathfrak{T}_{\alpha }$ and $\mathcal{WBP}%
_{\alpha }$ refer to the quasitesting condition and quasiweak boundedness
property respectively.

\begin{proposition}[The Intertwining Proposition]
\label{strongly adapted}Suppose that $\mathcal{F}$ is $\sigma $-Carleson.
Then%
\begin{equation*}
\left\vert \sum_{F\in \mathcal{F}}\ \sum_{I:\ I\supsetneqq F}\ \left\langle
T_{\sigma }^{\alpha }\bigtriangleup _{I}^{\sigma }f,\mathsf{P}_{\mathcal{C}%
_{F}^{\limfunc{good},\mathbf{\tau }-\limfunc{shift}}}^{\omega
}g\right\rangle _{\omega }\right\vert \lesssim \left( \mathfrak{F}_{\alpha }+%
\mathcal{E}_{\alpha }+\mathcal{NTV}_{\alpha }\right) \ \left\Vert
f\right\Vert _{L^{2}\left( \sigma \right) }\left\Vert g\right\Vert
_{L^{2}\left( \omega \right) }.
\end{equation*}
\end{proposition}

\begin{proof}
We write the left hand side of the display above as%
\begin{equation*}
\sum_{F\in \mathcal{F}}\ \sum_{I:\ I\supsetneqq F}\ \left\langle T_{\sigma
}^{\alpha }\bigtriangleup _{I}^{\sigma }f,g_{F}\right\rangle _{\omega
}=\sum_{F\in \mathcal{F}}\ \left\langle T_{\sigma }^{\alpha }\left(
\sum_{I:\ I\supsetneqq F}\bigtriangleup _{I}^{\sigma }f\right)
,g_{F}\right\rangle _{\omega }\equiv \sum_{F\in \mathcal{F}}\ \left\langle
T_{\sigma }^{\alpha }f_{F},g_{F}\right\rangle _{\omega }\ ,
\end{equation*}%
where%
\begin{equation*}
g_{F}=\mathsf{P}_{\mathcal{C}_{F}^{\limfunc{good},\mathbf{\tau }-\limfunc{%
shift}}}^{\omega }g\text{ and }f_{F}\equiv \sum_{I:\ I\supsetneqq
F}\bigtriangleup _{I}^{\sigma }f\ .
\end{equation*}%
Note that $g_{F}$ is supported in $F$, and that $f_{F}$ is constant on $F$.
We note that the quasicubes $I$ occurring in this sum are linearly and
consecutively ordered by inclusion, along with the quasicubes $F^{\prime
}\in \mathcal{F}$ that contain $F$. More precisely, we can write%
\begin{equation*}
F\equiv F_{0}\subsetneqq F_{1}\subsetneqq F_{2}\subsetneqq ...\subsetneqq
F_{n}\subsetneqq F_{n+1}\subsetneqq ...F_{N}
\end{equation*}%
where $F_{m}=\pi _{\mathcal{F}}^{m}F$ for all $m\geq 1$. We can also write%
\begin{equation*}
F=F_{0}\subsetneqq I_{1}\subsetneqq I_{2}\subsetneqq ...\subsetneqq
I_{k}\subsetneqq I_{k+1}\subsetneqq ...\subsetneqq I_{K}=F_{N}
\end{equation*}%
where $I_{k}=\pi _{\Omega \mathcal{D}}^{k}F$ for all $k\geq 1$. There is a
(unique) subsequence $\left\{ k_{m}\right\} _{m=1}^{N}$ such that%
\begin{equation*}
F_{m}=I_{k_{m}},\ \ \ \ \ 1\leq m\leq N.
\end{equation*}

Define%
\begin{equation*}
f_{F}\left( x\right) =\sum_{\ell =1}^{\infty }\bigtriangleup _{I_{\ell
}}^{\sigma }f\left( x\right) .
\end{equation*}%
Assume now that $k_{m}\leq k<k_{m+1}$. We denote the $2^{n}-1$ siblings of $%
I $ by $\theta \left( I\right) $, $\theta \in \Theta $, i.e. $\left\{ \theta
\left( I\right) \right\} _{\theta \in \Theta }=\mathfrak{C}_{\Omega \mathcal{%
D}}\left( \pi _{\Omega \mathcal{D}}I\right) \setminus \left\{ I\right\} $.
There are two cases to consider here:%
\begin{equation*}
\theta \left( I_{k}\right) \notin \mathcal{F}\text{ and }\theta \left(
I_{k}\right) \in \mathcal{F}.
\end{equation*}%
Suppose first that $\theta \left( I_{k}\right) \notin \mathcal{F}$. Then $%
\theta \left( I_{k}\right) \in \mathcal{C}_{F_{m+1}}^{\sigma }$ and using a
telescoping sum, we compute that for 
\begin{equation*}
x\in \theta \left( I_{k}\right) \subset I_{k+1}\setminus I_{k}\subset
F_{m+1}\setminus F_{m},
\end{equation*}%
we have 
\begin{equation*}
\left\vert f_{F}\left( x\right) \right\vert =\left\vert \sum_{\ell
=k}^{\infty }\bigtriangleup _{I_{\ell }}^{\sigma }f\left( x\right)
\right\vert =\left\vert \mathbb{E}_{\theta \left( I_{k}\right) }^{\sigma }f-%
\mathbb{E}_{I_{K}}^{\sigma }f\right\vert \lesssim \mathbb{E}%
_{F_{m+1}}^{\sigma }\left\vert f\right\vert \ .
\end{equation*}%
On the other hand, if $\theta \left( I_{k}\right) \in \mathcal{F}$, then $%
I_{k+1}\in \mathcal{C}_{F_{m+1}}^{\sigma }$ and we have%
\begin{equation*}
\left\vert f_{F}\left( x\right) -\bigtriangleup _{\theta \left( I_{k}\right)
}^{\sigma }f\left( x\right) \right\vert =\left\vert \sum_{\ell =k+1}^{\infty
}\bigtriangleup _{I_{\ell }}^{\sigma }f\left( x\right) \right\vert
=\left\vert \mathbb{E}_{I_{k+1}}^{\sigma }f-\mathbb{E}_{I_{K}}^{\sigma
}f\right\vert \lesssim \mathbb{E}_{F_{m+1}}^{\sigma }\left\vert f\right\vert
\ .
\end{equation*}%
Now we write%
\begin{eqnarray*}
f_{F} &=&\varphi _{F}+\psi _{F}, \\
\varphi _{F} &\equiv &\sum_{k,\theta :\ \theta \left( I_{k}\right) \in 
\mathcal{F}}^{\infty }\mathbf{1}_{\theta \left( I_{k}\right) }\bigtriangleup
_{I_{k}}^{\sigma }f\text{ and }\psi _{F}=f_{F}-\varphi _{F}\ ; \\
\sum_{F\in \mathcal{F}}\ \left\langle T_{\sigma }^{\alpha
}f_{F},g_{F}\right\rangle _{\omega } &=&\sum_{F\in \mathcal{F}}\
\left\langle T_{\sigma }^{\alpha }\varphi _{F},g_{F}\right\rangle _{\omega
}+\sum_{F\in \mathcal{F}}\ \left\langle T_{\sigma }^{\alpha }\psi
_{F},g_{F}\right\rangle _{\omega }\ ,
\end{eqnarray*}%
and note that both $\varphi _{F}$ and $\psi _{F}$ are constant on $F$. We
can apply (\ref{indicator far}) using $\theta \left( I_{k}\right) \in 
\mathcal{F}$ to the first sum here to obtain%
\begin{eqnarray*}
\left\vert \sum_{F\in \mathcal{F}}\ \left\langle T_{\sigma }^{\alpha
}\varphi _{F},g_{F}\right\rangle _{\omega }\right\vert &\lesssim &\mathcal{%
NTV}_{\alpha }\ \left\Vert \sum_{F\in \mathcal{F}}\varphi _{F}\right\Vert
_{L^{2}\left( \sigma \right) }\left\Vert \sum_{F\in \mathcal{F}%
}g_{F}\right\Vert _{L^{2}\left( \omega \right) }^{2} \\
&\lesssim &\mathcal{NTV}_{\alpha }\ \left\Vert f\right\Vert _{L^{2}\left(
\sigma \right) }\left[ \sum_{F\in \mathcal{F}}\left\Vert g_{F}\right\Vert
_{L^{2}\left( \omega \right) }^{2}\right] ^{\frac{1}{2}}.
\end{eqnarray*}

Turning to the second sum we note that%
\begin{eqnarray*}
\left\vert \psi _{F}\right\vert &\leq &\sum_{m=0}^{N}\left( \mathbb{E}%
_{F_{m+1}}^{\sigma }\left\vert f\right\vert \right) \ \mathbf{1}%
_{F_{m+1}\setminus F_{m}}=\left( \mathbb{E}_{F}^{\sigma }\left\vert
f\right\vert \right) \ \mathbf{1}_{F}+\sum_{m=0}^{N}\left( \mathbb{E}_{\pi _{%
\mathcal{F}}^{m+1}F}^{\sigma }\left\vert f\right\vert \right) \ \mathbf{1}%
_{\pi _{\mathcal{F}}^{m+1}F\setminus \pi _{\mathcal{F}}^{m}F} \\
&=&\left( \mathbb{E}_{F}^{\sigma }\left\vert f\right\vert \right) \ \mathbf{1%
}_{F}+\sum_{F^{\prime }\in \mathcal{F}:\ F\subset F^{\prime }}\left( \mathbb{%
E}_{\pi _{\mathcal{F}}F^{\prime }}^{\sigma }\left\vert f\right\vert \right)
\ \mathbf{1}_{\pi _{\mathcal{F}}F^{\prime }\setminus F^{\prime }} \\
&\leq &\alpha _{\mathcal{F}}\left( F\right) \ \mathbf{1}_{F}+\sum_{F^{\prime
}\in \mathcal{F}:\ F\subset F^{\prime }}\alpha _{\mathcal{F}}\left( \pi _{%
\mathcal{F}}F^{\prime }\right) \ \mathbf{1}_{\pi _{\mathcal{F}}F^{\prime
}\setminus F^{\prime }} \\
&\leq &\alpha _{\mathcal{F}}\left( F\right) \ \mathbf{1}_{F}+\sum_{F^{\prime
}\in \mathcal{F}:\ F\subset F^{\prime }}\alpha _{\mathcal{F}}\left( \pi _{%
\mathcal{F}}F^{\prime }\right) \ \mathbf{1}_{\pi _{\mathcal{F}}F^{\prime }}\ 
\mathbf{1}_{F^{c}} \\
&=&\alpha _{\mathcal{F}}\left( F\right) \ \mathbf{1}_{F}+\Phi \ \mathbf{1}%
_{F^{c}}\ ,\ \ \ \ \ \text{\ for all }F\in \mathcal{F},
\end{eqnarray*}%
where%
\begin{equation*}
\Phi \equiv \sum_{F^{\prime \prime }\in \mathcal{F}}\ \alpha _{\mathcal{F}%
}\left( F^{\prime \prime }\right) \ \mathbf{1}_{F^{\prime \prime }}\ .
\end{equation*}

Now we write%
\begin{equation*}
\sum_{F\in \mathcal{F}}\ \left\langle T_{\sigma }^{\alpha }\psi
_{F},g_{F}\right\rangle _{\omega }=\sum_{F\in \mathcal{F}}\ \left\langle
T_{\sigma }^{\alpha }\left( \mathbf{1}_{F}\psi _{F}\right)
,g_{F}\right\rangle _{\omega }+\sum_{F\in \mathcal{F}}\ \left\langle
T_{\sigma }^{\alpha }\left( \mathbf{1}_{F^{c}}\psi _{F}\right)
,g_{F}\right\rangle _{\omega }\equiv I+II.
\end{equation*}%
Then quasicube testing $\left\vert \left\langle T_{\sigma }^{\alpha }\mathbf{%
1}_{F},g_{F}\right\rangle _{\omega }\right\vert =\left\vert \left\langle 
\mathbf{1}_{F}T_{\sigma }^{\alpha }\mathbf{1}_{F},g_{F}\right\rangle
_{\omega }\right\vert \leq \mathfrak{T}_{T^{\alpha }}\sqrt{\left\vert
F\right\vert _{\sigma }}\left\Vert g_{F}\right\Vert _{L^{2}\left( \omega
\right) }$ and `quasi' orthogonality, together with the fact that $\psi _{F}$
is a constant on $F$ bounded by $\alpha _{\mathcal{F}}\left( F\right) $, give%
\begin{eqnarray*}
&&\left\vert I\right\vert \leq \sum_{F\in \mathcal{F}}\ \left\vert
\left\langle T_{\sigma }^{\alpha }\mathbf{1}_{F}\psi _{F},g_{F}\right\rangle
_{\omega }\right\vert \lesssim \sum_{F\in \mathcal{F}}\ \alpha _{\mathcal{F}%
}\left( F\right) \ \left\vert \left\langle T_{\sigma }^{\alpha }\mathbf{1}%
_{F},g_{F}\right\rangle _{\omega }\right\vert \\
&\lesssim &\sum_{F\in \mathcal{F}}\ \alpha _{\mathcal{F}}\left( F\right) 
\mathcal{NTV}_{\alpha }\sqrt{\left\vert F\right\vert _{\sigma }}\left\Vert
g_{F}\right\Vert _{L^{2}\left( \omega \right) }\lesssim \mathcal{NTV}%
_{\alpha }\left\Vert f\right\Vert _{L^{2}\left( \sigma \right) }\left[
\sum_{F\in \mathcal{F}}\left\Vert g_{F}\right\Vert _{L^{2}\left( \omega
\right) }^{2}\right] ^{\frac{1}{2}}.
\end{eqnarray*}%
Now $\mathbf{1}_{F^{c}}\psi _{F}$ is supported outside $F$, and each $J$ in
the quasiHaar support of $g_{F}$ is $\mathbf{r}$-deeply embedded in $F$,
i.e. $J\Subset _{\mathbf{r},\varepsilon }F$. Thus we can apply the Energy
Lemma \ref{ener} to obtain%
\begin{eqnarray*}
\left\vert II\right\vert &=&\left\vert \sum_{F\in \mathcal{F}}\left\langle
T_{\sigma }^{\alpha }\left( \mathbf{1}_{F^{c}}\psi _{F}\right)
,g_{F}\right\rangle _{\omega }\right\vert \\
&\lesssim &\sum_{F\in \mathcal{F}}\sum_{J\in \mathcal{M}_{\left( \mathbf{r}%
,\varepsilon \right) -\limfunc{deep}}\left( F\right) }\frac{\mathrm{P}%
^{\alpha }\left( J,\mathbf{1}_{F^{c}}\Phi \sigma \right) }{\left\vert
J\right\vert ^{\frac{1}{n}}}\left\Vert \mathsf{P}_{\mathcal{C}_{F}^{\limfunc{%
good},\mathbf{\tau }-\limfunc{shift}};J}^{\omega }\mathbf{x}\right\Vert
_{L^{2}\left( \omega \right) }\left\Vert \mathsf{P}_{J}^{\omega
}g_{F}\right\Vert _{L^{2}\left( \omega \right) } \\
&&+\sum_{F\in \mathcal{F}}\sum_{J\in \mathcal{M}_{\left( \mathbf{r}%
,\varepsilon \right) -\limfunc{deep}}\left( F\right) }\frac{\mathrm{P}%
_{1+\delta ^{\prime }}^{\alpha }\left( J,\mathbf{1}_{F^{c}}\Phi \sigma
\right) }{\left\vert J\right\vert ^{\frac{1}{n}}}\left\Vert \mathsf{P}%
_{\left( \mathcal{C}_{F}^{\limfunc{good},\mathbf{\tau }-\limfunc{shift}%
}\right) ^{\ast };J}^{\omega }\mathbf{x}\right\Vert _{L^{2}\left( \omega
\right) }\left\Vert \mathsf{P}_{J}^{\omega }g_{F}\right\Vert _{L^{2}\left(
\omega \right) } \\
&\equiv &II_{G}+II_{B}\ .
\end{eqnarray*}

Then\ from Cauchy-Schwarz, the functional quasienergy condition, and $%
\left\Vert \Phi \right\Vert _{L^{2}\left( \sigma \right) }\lesssim
\left\Vert f\right\Vert _{L^{2}\left( \sigma \right) }$ we obtain%
\begin{eqnarray*}
\left\vert II_{G}\right\vert &\leq &\left( \sum_{F\in \mathcal{F}}\sum_{J\in 
\mathcal{M}_{\left( \mathbf{r},\varepsilon \right) -\limfunc{deep}}\left(
F\right) }\left( \frac{\mathrm{P}^{\alpha }\left( J,\mathbf{1}_{F^{c}}\Phi
\sigma \right) }{\left\vert J\right\vert ^{\frac{1}{n}}}\right)
^{2}\left\Vert \mathsf{P}_{\mathcal{C}_{F}^{\limfunc{good},\mathbf{\tau }-%
\limfunc{shift}};J}^{\omega }\mathbf{x}\right\Vert _{L^{2}\left( \omega
\right) }^{2}\right) ^{\frac{1}{2}} \\
&&\times \left( \sum_{F\in \mathcal{F}}\sum_{J\in \mathcal{M}_{\left( 
\mathbf{r},\varepsilon \right) -\limfunc{deep}}\left( F\right) }\left\Vert 
\mathsf{P}_{J}^{\omega }g_{F}\right\Vert _{L^{2}\left( \omega \right)
}^{2}\right) ^{\frac{1}{2}} \\
&\lesssim &\mathfrak{F}_{\alpha }\left\Vert \Phi \right\Vert _{L^{2}\left(
\sigma \right) }\left[ \sum_{F\in \mathcal{F}}\left\Vert g_{F}\right\Vert
_{L^{2}\left( \omega \right) }^{2}\right] ^{\frac{1}{2}}\lesssim \mathbf{%
\tau }\mathfrak{F}_{\alpha }\left\Vert f\right\Vert _{L^{2}\left( \sigma
\right) }\left\Vert g\right\Vert _{L^{2}\left( \omega \right) },
\end{eqnarray*}%
by the $\mathbf{\tau }$-overlap (\ref{tau overlap}) of the shifted coronas $%
\mathcal{C}_{F}^{\limfunc{good},\mathbf{\tau }-\limfunc{shift}}$.

In term $II_{B}$ the projections $\mathsf{P}_{\left( \mathcal{C}_{F}^{%
\limfunc{good},\mathbf{\tau }-\limfunc{shift}}\right) ^{\ast };J}^{\omega }$
are no longer almost orthogonal, and we must instead exploit the decay in
the Poisson integral $\mathrm{P}_{1+\delta ^{\prime }}^{\alpha }$ along with
goodness of the quasicubes $J$. This idea was already used by M. Lacey and
B. Wick in \cite{LaWi} in a similar situation. As a consequence of this
decay we will be able to bound $II_{B}$ \emph{directly} by the quasienergy
condition, without having to invoke the more difficult functional
quasienergy condition. For the decay we compute%
\begin{eqnarray*}
\frac{\mathrm{P}_{1+\delta ^{\prime }}^{\alpha }\left( J,\Phi \sigma \right) 
}{\left\vert J\right\vert ^{\frac{1}{n}}} &=&\int_{F^{c}}\frac{\left\vert
J\right\vert ^{\frac{\delta ^{\prime }}{n}}}{\left\vert y-c_{J}\right\vert
^{n+1+\delta -\alpha }}\Phi \left( y\right) d\sigma \left( y\right) \\
&\leq &\sum_{t=0}^{\infty }\int_{\pi _{\mathcal{F}}^{t+1}F\setminus \pi _{%
\mathcal{F}}^{t}F}\left( \frac{\left\vert J\right\vert ^{\frac{1}{n}}}{%
\limfunc{dist}\left( c_{J},\left( \pi _{\mathcal{F}}^{t}F\right) ^{c}\right) 
}\right) ^{\delta ^{\prime }}\frac{1}{\left\vert y-c_{J}\right\vert
^{n+1-\alpha }}\Phi \left( y\right) d\sigma \left( y\right) \\
&\leq &\sum_{t=0}^{\infty }\left( \frac{\left\vert J\right\vert ^{\frac{1}{n}%
}}{\limfunc{dist}\left( c_{J},\left( \pi _{\mathcal{F}}^{t}F\right)
^{c}\right) }\right) ^{\delta ^{\prime }}\frac{\mathrm{P}^{\alpha }\left( J,%
\mathbf{1}_{\pi _{\mathcal{F}}^{t+1}F\setminus \pi _{\mathcal{F}}^{t}F}\Phi
\sigma \right) }{\left\vert J\right\vert ^{\frac{1}{n}}},
\end{eqnarray*}%
and then use the goodness inequality%
\begin{equation*}
\limfunc{dist}\left( c_{J},\left( \pi _{\mathcal{F}}^{t}F\right) ^{c}\right)
\geq \frac{1}{2}\ell \left( \pi _{\mathcal{F}}^{t}F\right) ^{1-\varepsilon
}\ell \left( J\right) ^{\varepsilon }\geq \frac{1}{2}2^{t\left(
1-\varepsilon \right) }\ell \left( F\right) ^{1-\varepsilon }\ell \left(
J\right) ^{\varepsilon }\geq 2^{t\left( 1-\varepsilon \right) -1}\ell \left(
J\right) ,
\end{equation*}%
to conclude that%
\begin{eqnarray}
\left( \frac{\mathrm{P}_{1+\delta ^{\prime }}^{\alpha }\left( J,\mathbf{1}%
_{F^{c}}\Phi \sigma \right) }{\left\vert J\right\vert ^{\frac{1}{n}}}\right)
^{2} &\lesssim &\left( \sum_{t=0}^{\infty }2^{-t\delta ^{\prime }\left(
1-\varepsilon \right) }\frac{\mathrm{P}^{\alpha }\left( J,\mathbf{1}_{\pi _{%
\mathcal{F}}^{t+1}F\setminus \pi _{\mathcal{F}}^{t}F}\Phi \sigma \right) }{%
\left\vert J\right\vert ^{\frac{1}{n}}}\right) ^{2}  \label{decay in t} \\
&\lesssim &\sum_{t=0}^{\infty }2^{-t\delta ^{\prime }\left( 1-\varepsilon
\right) }\left( \frac{\mathrm{P}^{\alpha }\left( J,\mathbf{1}_{\pi _{%
\mathcal{F}}^{t+1}F\setminus \pi _{\mathcal{F}}^{t}F}\Phi \sigma \right) }{%
\left\vert J\right\vert ^{\frac{1}{n}}}\right) ^{2}.  \notag
\end{eqnarray}%
Now we apply Cauchy-Schwarz to obtain%
\begin{eqnarray*}
II_{B} &=&\sum_{F\in \mathcal{F}}\sum_{J\in \mathcal{M}_{\left( \mathbf{r}%
,\varepsilon \right) -\limfunc{deep}}\left( F\right) }\frac{\mathrm{P}%
_{1+\delta ^{\prime }}^{\alpha }\left( J,\mathbf{1}_{F^{c}}\Phi \sigma
\right) }{\left\vert J\right\vert ^{\frac{1}{n}}}\left\Vert \mathsf{P}%
_{\left( \mathcal{C}_{F}^{\limfunc{good},\mathbf{\tau }-\limfunc{shift}%
}\right) ^{\ast };J}^{\omega }\mathbf{x}\right\Vert _{L^{2}\left( \omega
\right) }\left\Vert \mathsf{P}_{J}^{\omega }g_{F}\right\Vert _{L^{2}\left(
\omega \right) } \\
&\leq &\left( \sum_{F\in \mathcal{F}}\sum_{J\in \mathcal{M}_{\left( \mathbf{r%
},\varepsilon \right) -\limfunc{deep}}\left( F\right) }\left( \frac{\mathrm{P%
}_{1+\delta ^{\prime }}^{\alpha }\left( J,\mathbf{1}_{F^{c}}\Phi \sigma
\right) }{\left\vert J\right\vert ^{\frac{1}{n}}}\right) ^{2}\left\Vert 
\mathsf{P}_{\left( \mathcal{C}_{F}^{\limfunc{good},\mathbf{\tau }-\limfunc{%
shift}}\right) ^{\ast };J}^{\omega }\mathbf{x}\right\Vert _{L^{2}\left(
\omega \right) }^{2}\right) ^{\frac{1}{2}}\left[ \sum_{F}\left\Vert
g_{F}\right\Vert _{L^{2}\left( \omega \right) }^{2}\right] ^{\frac{1}{2}} \\
&\equiv &\sqrt{II_{\limfunc{energy}}}\left[ \sum_{F}\left\Vert
g_{F}\right\Vert _{L^{2}\left( \omega \right) }^{2}\right] ^{\frac{1}{2}},
\end{eqnarray*}%
and it remains to estimate $II_{\limfunc{energy}}$. From (\ref{decay in t})
and the plugged deep quasienergy condition we have%
\begin{eqnarray*}
&&II_{\limfunc{energy}} \\
&\leq &\sum_{F\in \mathcal{F}}\sum_{J\in \mathcal{M}_{\left( \mathbf{r}%
,\varepsilon \right) -\limfunc{deep}}\left( F\right) }\sum_{t=0}^{\infty
}2^{-t\delta ^{\prime }\left( 1-\varepsilon \right) }\left( \frac{\mathrm{P}%
^{\alpha }\left( J,\mathbf{1}_{\pi _{\mathcal{F}}^{t+1}F\setminus \pi _{%
\mathcal{F}}^{t}F}\Phi \sigma \right) }{\left\vert J\right\vert ^{\frac{1}{n}%
}}\right) ^{2}\left\Vert \mathsf{P}_{\left( \mathcal{C}_{F}^{\limfunc{good},%
\mathbf{\tau }-\limfunc{shift}}\right) ^{\ast };J}^{\omega }\mathbf{x}%
\right\Vert _{L^{2}\left( \omega \right) }^{2} \\
&=&\sum_{t=0}^{\infty }2^{-t\delta ^{\prime }\left( 1-\varepsilon \right)
}\sum_{G\in \mathcal{F}}\sum_{F\in \mathfrak{C}_{\mathcal{F}}^{\left(
t+1\right) }\left( G\right) }\sum_{J\in \mathcal{M}_{\left( \mathbf{r}%
,\varepsilon \right) -\limfunc{deep}}\left( F\right) }\left( \frac{\mathrm{P}%
^{\alpha }\left( J,\mathbf{1}_{G\setminus \pi _{\mathcal{F}}^{t}F}\Phi
\sigma \right) }{\left\vert J\right\vert ^{\frac{1}{n}}}\right)
^{2}\left\Vert \mathsf{P}_{\left( \mathcal{C}_{F}^{\limfunc{good},\mathbf{%
\tau }-\limfunc{shift}}\right) ^{\ast };J}^{\omega }\mathbf{x}\right\Vert
_{L^{2}\left( \omega \right) }^{2} \\
&\lesssim &\sum_{t=0}^{\infty }2^{-t\delta ^{\prime }\left( 1-\varepsilon
\right) }\sum_{G\in \mathcal{F}}\alpha _{\mathcal{F}}\left( G\right)
^{2}\sum_{F\in \mathfrak{C}_{\mathcal{F}}^{\left( t+1\right) }\left(
G\right) }\sum_{J\in \mathcal{M}_{\left( \mathbf{r},\varepsilon \right) -%
\limfunc{deep}}\left( F\right) }\left( \frac{\mathrm{P}^{\alpha }\left( J,%
\mathbf{1}_{G\setminus \pi _{\mathcal{F}}^{t}F}\sigma \right) }{\left\vert
J\right\vert ^{\frac{1}{n}}}\right) ^{2}\left\Vert \mathsf{P}_{J}^{\limfunc{%
subgood},\omega }\mathbf{x}\right\Vert _{L^{2}\left( \omega \right) }^{2} \\
&\lesssim &\sum_{t=0}^{\infty }2^{-t\delta ^{\prime }\left( 1-\varepsilon
\right) }\sum_{G\in \mathcal{F}}\alpha _{\mathcal{F}}\left( G\right)
^{2}\left( \mathcal{E}_{\alpha }^{2}+A_{2}^{\alpha }\right) \left\vert
G\right\vert _{\sigma }\lesssim \left( \mathcal{E}_{\alpha
}^{2}+A_{2}^{\alpha }\right) \left\Vert f\right\Vert _{L^{2}\left( \sigma
\right) }^{2}.
\end{eqnarray*}

This completes the proof of the Intertwining Proposition \ref{strongly
adapted}.
\end{proof}

\section{Control of functional energy by energy modulo $\mathcal{A}_{2}^{%
\protect\alpha }$ and $A_{2}^{\protect\alpha ,\limfunc{punct}}\label{equiv}$}

Now we arrive at one of our main propositions in the proof of our theorem.
We show that the functional quasienergy constants $\mathfrak{F}_{\alpha }$
as in (\ref{e.funcEnergy n}) are controlled by $\mathcal{A}_{2}^{\alpha }$, $%
A_{2}^{\alpha ,\limfunc{punct}}$ and both the \emph{deep} and \emph{refined}
quasienergy constants $\mathcal{E}_{\alpha }^{\limfunc{deep}}$ and $\mathcal{%
E}_{\alpha }^{\limfunc{refined}}$\ defined in Definition \ref{energy
condition}. Recall $\left( \mathcal{E}_{\alpha }\right) ^{2}=\left( \mathcal{%
E}_{\alpha }^{\limfunc{deep}}\right) ^{2}+\left( \mathcal{E}_{\alpha }^{%
\limfunc{refined}}\right) ^{2}$. The proof of this fact is further
complicated when common point masses are permitted, accounting for the
inclusion of the punctured Muckenhoupt condition $A_{2}^{\alpha ,\limfunc{%
punct}}$.

\begin{proposition}
\label{func ener control}We have%
\begin{equation*}
\mathfrak{F}_{\alpha }\lesssim \mathcal{E}_{\alpha }^{\limfunc{plug}}+\sqrt{%
\mathcal{A}_{2}^{\alpha }}+\sqrt{\mathcal{A}_{2}^{\alpha ,\ast }}+\sqrt{%
A_{2}^{\alpha ,\limfunc{punct}}}\text{ and }\mathfrak{F}_{\alpha }^{\ast
}\lesssim \mathcal{E}_{\alpha }^{\limfunc{plug},\ast }+\sqrt{\mathcal{A}%
_{2}^{\alpha }}+\sqrt{\mathcal{A}_{2}^{\alpha ,\ast }}+\sqrt{A_{2}^{\alpha
,\ast ,\limfunc{punct}}}\ .
\end{equation*}
\end{proposition}

To prove this proposition, we fix $\mathcal{F}$ as in (\ref{e.funcEnergy n}%
), and set 
\begin{equation}
\mu \equiv \sum_{F\in \mathcal{F}}\sum_{J\in \mathcal{M}_{\left( \mathbf{r}%
,\varepsilon \right) -\limfunc{deep}}\left( F\right) }\left\Vert \mathsf{P}%
_{F,J}^{\omega }\mathbf{x}\right\Vert _{L^{2}\left( \omega \right)
}^{2}\cdot \delta _{\left( c_{J},\ell \left( J\right) \right) }\text{ and }d%
\overline{\mu }\left( x,t\right) \equiv \frac{1}{t^{2}}d\mu \left(
x,t\right) \ ,  \label{def mu n}
\end{equation}%
where $\mathcal{M}_{\left( \mathbf{r},\varepsilon \right) -\limfunc{deep}%
}\left( F\right) $ consists of the maximal $\mathbf{r}$-deeply embedded
subquasicubes of $F$, and where $\delta _{\left( c_{J},\ell \left( J\right)
\right) }$ denotes the Dirac unit mass at the point $\left( c_{J},\ell
\left( J\right) \right) $ in the upper half-space $\mathbb{R}_{+}^{n+1}$.
Here $J$ is a dyadic quasicube with center $c_{J}$ and side length $\ell
\left( J\right) $. For convenience in notation, we denote for any dyadic
quasicube $J$ the localized projection $\mathsf{P}_{\mathcal{C}_{F}^{%
\limfunc{good},\mathbf{\tau }-\limfunc{shift}};J}^{\omega }$ given in (\ref%
{def localization}) by%
\begin{equation*}
\mathsf{P}_{F,J}^{\omega }\equiv \mathsf{P}_{\mathcal{C}_{F}^{\limfunc{good},%
\mathbf{\tau }-\limfunc{shift}};J}^{\omega }=\sum_{J^{\prime }\subset J:\
J^{\prime }\in \mathcal{C}_{F}^{\limfunc{good},\mathbf{\tau }-\limfunc{shift}%
}}\bigtriangleup _{J^{\prime }}^{\omega }.
\end{equation*}%
We emphasize that the quasicubes $J\in \mathcal{M}_{\left( \mathbf{r}%
,\varepsilon \right) -\limfunc{deep}}\left( F\right) $ are not necessarily
good, but that the subquasicubes $J^{\prime }\subset J$ arising in the
projection $\mathsf{P}_{F,J}^{\omega }$ are good. We can replace $\mathbf{x}$
by $\mathbf{x}-\mathbf{c}$ inside the projection for any choice of $\mathbf{c%
}$ we wish; the projection is unchanged. More generally, $\delta _{q}$
denotes a Dirac unit mass at a point $q$ in the upper half-space $\mathbb{R}%
_{+}^{n+1}$.

We prove the two-weight inequality 
\begin{equation}
\left\Vert \mathbb{P}^{\alpha }\left( f\sigma \right) \right\Vert _{L^{2}(%
\mathbb{R}_{+}^{n+1},\overline{\mu })}\lesssim \left( \mathcal{E}_{\alpha }^{%
\limfunc{plug}}+\sqrt{\mathcal{A}_{2}^{\alpha }}+\sqrt{\mathcal{A}%
_{2}^{\alpha ,\ast }}+\sqrt{A_{2}^{\alpha ,\limfunc{punct}}}\right) \lVert
f\rVert _{L^{2}\left( \sigma \right) }\,,  \label{two weight Poisson n}
\end{equation}%
for all nonnegative $f$ in $L^{2}\left( \sigma \right) $, noting that $%
\mathcal{F}$ and $f$ are \emph{not} related here. Above, $\mathbb{P}^{\alpha
}(\cdot )$ denotes the $\alpha $-fractional Poisson extension to the upper
half-space $\mathbb{R}_{+}^{n+1}$,

\begin{equation*}
\mathbb{P}^{\alpha }\nu \left( x,t\right) \equiv \int_{\mathbb{R}^{n}}\frac{t%
}{\left( t^{2}+\left\vert x-y\right\vert ^{2}\right) ^{\frac{n+1-\alpha }{2}}%
}d\nu \left( y\right) ,
\end{equation*}%
so that in particular 
\begin{equation*}
\left\Vert \mathbb{P}^{\alpha }(f\sigma )\right\Vert _{L^{2}(\mathbb{R}%
_{+}^{n+1},\overline{\mu })}^{2}=\sum_{F\in \mathcal{F}}\sum_{J\in \mathcal{M%
}_{\mathbf{r}-\limfunc{deep}}\left( F\right) }\mathbb{P}^{\alpha }\left(
f\sigma \right) (c(J),\ell \left( J\right) )^{2}\left\Vert \mathsf{P}%
_{F,J}^{\omega }\frac{x}{\left\vert J\right\vert ^{\frac{1}{n}}}\right\Vert
_{L^{2}\left( \omega \right) }^{2}\,,
\end{equation*}%
and so (\ref{two weight Poisson n}) proves the first line in Proposition \ref%
{func ener control} upon inspecting (\ref{e.funcEnergy n}). Note also that
we can equivalently write $\left\Vert \mathbb{P}^{\alpha }\left( f\sigma
\right) \right\Vert _{L^{2}(\mathbb{R}_{+}^{n+1},\overline{\mu }%
)}=\left\Vert \widetilde{\mathbb{P}}^{\alpha }\left( f\sigma \right)
\right\Vert _{L^{2}(\mathbb{R}_{+}^{n+1},\mu )}$ where $\widetilde{\mathbb{P}%
}^{\alpha }\nu \left( x,t\right) \equiv \frac{1}{t}\mathbb{P}^{\alpha }\nu
\left( x,t\right) $ is the renormalized Poisson operator. Here we have
simply shifted the factor $\frac{1}{t^{2}}$ in $\overline{\mu }$ to $%
\left\vert \widetilde{\mathbb{P}}^{\alpha }\left( f\sigma \right)
\right\vert ^{2}$ instead, and we will do this shifting often throughout the
proof when it is convenient to do so.

The characterization of the two-weight inequality for fractional and Poisson
integrals in \cite{Saw3} was stated in terms of the collection $\mathcal{P}%
^{n}$ of cubes in $\mathbb{R}^{n}$ with sides parallel to the coordinate
axes. It is a routine matter to pullback the Poisson inequality under a
globally biLipschitz map $\Omega :\mathbb{R}^{n}\rightarrow \mathbb{R}^{n}$,
then apply the theorem in \cite{Saw3} (as a black box), and then to
pushforward the conclusions of the theorems so as to extend these
characterizations of fractional and Poisson integral inequalities to the
setting of quasicubes $Q\in \Omega \mathcal{P}^{n}$ and quasitents $Q\times %
\left[ 0,\ell \left( Q\right) \right] \subset \mathbb{R}_{+}^{n+1}$ with $%
Q\in \Omega \mathcal{P}^{n}$. Using this extended theorem for the two-weight
Poisson inequality, we see that inequality (\ref{two weight Poisson n})
requires checking these two inequalities for dyadic quasicubes $I\in \Omega 
\mathcal{D}$ and quasiboxes $\widehat{I}=I\times \left[ 0,\ell \left(
I\right) \right) $ in the upper half-space$\mathbb{R}_{+}^{n+1}$: 
\begin{equation}
\int_{\mathbb{R}_{+}^{n+1}}\mathbb{P}^{\alpha }\left( \mathbf{1}_{I}\sigma
\right) \left( x,t\right) ^{2}d\overline{\mu }\left( x,t\right) \equiv
\left\Vert \mathbb{P}^{\alpha }\left( \mathbf{1}_{I}\sigma \right)
\right\Vert _{L^{2}(\widehat{I},\overline{\mu })}^{2}\lesssim \left( \left( 
\mathcal{E}_{\alpha }^{\limfunc{plug}}\right) ^{2}+\mathcal{A}_{2}^{\alpha }+%
\mathcal{A}_{2}^{\alpha ,\ast }+A_{2}^{\alpha ,\limfunc{punct}}\right)
\sigma (I)\,,  \label{e.t1 n}
\end{equation}%
\begin{equation}
\int_{\mathbb{R}^{n}}[\mathbb{Q}^{\alpha }(t\mathbf{1}_{\widehat{I}}%
\overline{\mu })]^{2}d\sigma (x)\lesssim \left( \left( \mathcal{E}_{\alpha
}^{\limfunc{plug}}\right) ^{2}+\mathcal{A}_{2}^{\alpha }+\mathcal{A}%
_{2}^{\alpha ,\ast }+A_{2}^{\alpha ,\limfunc{punct}}\right) \int_{\widehat{I}%
}t^{2}d\overline{\mu }(x,t),  \label{e.t2 n}
\end{equation}%
for all \emph{dyadic} quasicubes $I\in \Omega \mathcal{D}$, and where the
dual Poisson operator $\mathbb{Q}^{\alpha }$ is given by 
\begin{equation*}
\mathbb{Q}^{\alpha }(t\mathbf{1}_{\widehat{I}}\overline{\mu })\left(
x\right) =\int_{\widehat{I}}\frac{t^{2}}{\left( t^{2}+\lvert x-y\rvert
^{2}\right) ^{\frac{n+1-\alpha }{2}}}d\overline{\mu }\left( y,t\right) \,.
\end{equation*}%
It is important to note that we can choose for $\Omega \mathcal{D}$ any
fixed dyadic quasigrid, the compensating point being that the integrations
on the left sides of (\ref{e.t1 n}) and (\ref{e.t2 n}) are taken over the
entire spaces $\mathbb{R}_{+}^{n+1}$ and $\mathbb{R}^{n}$ respectively.

\begin{remark}
There is a gap in the proof of the Poisson inequality at the top of page 542
in \cite{Saw3}. However, this gap can be fixed as in \cite{SaWh} or \cite%
{LaSaUr1}.
\end{remark}

\subsection{Poisson testing}

We now turn to proving the Poisson testing conditions (\ref{e.t1 n}) and (%
\ref{e.t2 n}). The same testing conditions have been considered in \cite%
{SaShUr5} but in the setting of no common point masses, and the proofs there
carry over to the situation here, but careful attention must now be paid to
the possibility of common point masses. In \cite{Hyt2} Hyt\"{o}nen
circumvented this difficulty by introducing a Poisson operator `with holes',
which was then analyzed using shifted dyadic grids, but part of his argument
was heavily dependent on the dimension being $n=1$, and the extension of
this argument to higher dimensions is feasible (see earlier versions of this
paper on the \textit{arXiv}), but technically very involved. We circumvent
the difficulty of permitting common point masses here instead by using the
energy Muckenhoupt constants $A_{2}^{\alpha ,\limfunc{energy}}$ and $%
A_{2}^{\alpha ,\ast ,\limfunc{energy}}$, which require control by the
punctured Muckenhoupt constants $A_{2}^{\alpha ,\limfunc{punct}}$ and $%
A_{2}^{\alpha ,\ast ,\limfunc{punct}}$. The following elementary Poisson
inequalities (see e.g. \cite{Vol}) will be used extensively.

\begin{lemma}
\label{Poisson inequalities}Suppose that $J,K,I$ are quasicubes in $\mathbb{R%
}^{n}$, and that $\mu $ is a positive measure supported in $\mathbb{R}%
^{n}\setminus I$. If $J\subset K\subset 2K\subset I$, then%
\begin{equation*}
\frac{\mathrm{P}^{\alpha }\left( J,\mu \right) }{\left\vert J\right\vert ^{%
\frac{1}{n}}}\lesssim \frac{\mathrm{P}^{\alpha }\left( K,\mu \right) }{%
\left\vert K\right\vert ^{\frac{1}{n}}}\lesssim \frac{\mathrm{P}^{\alpha
}\left( J,\mu \right) }{\left\vert J\right\vert ^{\frac{1}{n}}},
\end{equation*}%
while if $2J\subset K\subset I$, then%
\begin{equation*}
\frac{\mathrm{P}^{\alpha }\left( K,\mu \right) }{\left\vert K\right\vert ^{%
\frac{1}{n}}}\lesssim \frac{\mathrm{P}^{\alpha }\left( J,\mu \right) }{%
\left\vert J\right\vert ^{\frac{1}{n}}}.
\end{equation*}
\end{lemma}

\begin{proof}
We have%
\begin{equation*}
\frac{\mathrm{P}^{\alpha }\left( J,\mu \right) }{\left\vert J\right\vert ^{%
\frac{1}{n}}}=\frac{1}{\left\vert J\right\vert ^{\frac{1}{n}}}\int \frac{%
\left\vert J\right\vert ^{\frac{1}{n}}}{\left( \left\vert J\right\vert ^{%
\frac{1}{n}}+\left\vert x-c_{J}\right\vert \right) ^{n+1-\alpha }}d\mu
\left( x\right) ,
\end{equation*}%
where $J\subset K\subset 2K\subset I$ implies that%
\begin{equation*}
\left\vert J\right\vert ^{\frac{1}{n}}+\left\vert x-c_{J}\right\vert \approx
\left\vert K\right\vert ^{\frac{1}{n}}+\left\vert x-c_{K}\right\vert ,\ \ \
\ \ x\in \mathbb{R}^{n}\setminus I,
\end{equation*}%
and where $2J\subset K\subset I$ implies that%
\begin{equation*}
\left\vert J\right\vert ^{\frac{1}{n}}+\left\vert x-c_{J}\right\vert
\lesssim \left\vert K\right\vert ^{\frac{1}{n}}+\left\vert
x-c_{K}\right\vert ,\ \ \ \ \ x\in \mathbb{R}^{n}\setminus I.
\end{equation*}
\end{proof}

Now we record the bounded overlap of the projections $\mathsf{P}%
_{F,J}^{\omega }$.

\begin{lemma}
\label{tau ovelap}Suppose $\mathsf{P}_{F,J}^{\omega }$ is as above and fix
any $I_{0}\in \Omega \mathcal{D}$, so that $I_{0}$, $F$ and $J$ all lie in a
common quasigrid. If $J\in \mathcal{M}_{\left( \mathbf{r},\varepsilon
\right) -\limfunc{deep}}\left( F\right) $ for some $F\in \mathcal{F}$ with $%
F\supsetneqq I_{0}\supset J$ and $\mathsf{P}_{F,J}^{\omega }\neq 0$, then 
\begin{equation*}
F=\pi _{\mathcal{F}}^{\left( \ell \right) }I_{0}\text{ for some }0\leq \ell
\leq \mathbf{\tau }.
\end{equation*}%
As a consequence we have the bounded overlap,%
\begin{equation*}
\#\left\{ F\in \mathcal{F}:J\subset I_{0}\subsetneqq F\text{ for some }J\in 
\mathcal{M}_{\left( \mathbf{r},\varepsilon \right) -\limfunc{deep}}\left(
F\right) \text{ with }\mathsf{P}_{F,J}^{\omega }\neq 0\right\} \leq \mathbf{%
\tau }.
\end{equation*}
\end{lemma}

\begin{proof}
Indeed, if $J^{\prime }\in \mathcal{C}_{\pi _{\mathcal{F}}^{\left( \ell
\right) }I_{0}}^{\limfunc{good},\mathbf{\tau }-\limfunc{shift}}$ for some $%
\ell >\mathbf{\tau }$, then either $J^{\prime }\cap \pi _{\mathcal{F}%
}^{\left( 0\right) }I_{0}=\emptyset $ or $J^{\prime }\supset \pi _{\mathcal{F%
}}^{\left( 0\right) }I_{0}$. Since $J\subset I_{0}\subset \pi _{\mathcal{F}%
}^{\left( 0\right) }I_{0}$, we cannot have $J^{\prime }$ contained in $J$,
and this shows that $\mathsf{P}_{\pi _{\mathcal{F}}^{\left( \ell \right)
}I_{0},J}^{\omega }=0$.
\end{proof}

Finally we record the only places in the proof where the \emph{refined}
quasienergy conditions are used. This lemma will be used in bounding both of
the local Poisson testing conditions. Recall that $\mathcal{A}\Omega 
\mathcal{D}$ consists of all alternate $\Omega \mathcal{D}$-dyadic
quasicubes where $K$ is alternate dyadic if it is a union of $2^{n}$ $\Omega 
\mathcal{D}$-dyadic quasicubes $K^{\prime }$ with $\ell \left( K^{\prime
}\right) =\frac{1}{2}\ell \left( K\right) $.

\begin{remark}
The following lemma is another of the key results on the way to the proof of
our theorem, and is an analogue of the corresponding lemma from \cite%
{SaShUr5}, but with the right hand side involving only the plugged energy
constants and the energy Muckenhoupt constants.
\end{remark}

\begin{lemma}
\label{refined lemma}Let $\Omega \mathcal{D},\mathcal{F\subset }\Omega 
\mathcal{D}$ be quasigrids and $\left\{ \mathsf{P}_{F,J}^{\omega }\right\} 
_{\substack{ F\in \mathcal{F}  \\ J\in \mathcal{M}_{\left( \mathbf{r}%
,\varepsilon \right) -\limfunc{deep}}\left( F\right) }}$ be as above with $%
J,F$ in the dyadic quasigrid $\Omega \mathcal{D}$. For any alternate
quasicube $I\in \mathcal{A}\Omega \mathcal{D}$ define%
\begin{equation}
B\left( I\right) \equiv \sum_{F\in \mathcal{F}:\ F\supsetneqq I^{\prime }%
\text{ for some }I^{\prime }\in \mathfrak{C}\left( I\right) }\sum_{J\in 
\mathcal{M}_{\left( \mathbf{r},\varepsilon \right) -\limfunc{deep}}\left(
F\right) :\ J\subset I}\left( \frac{\mathrm{P}^{\alpha }\left( J,\mathbf{1}%
_{I}\sigma \right) }{\left\vert J\right\vert ^{\frac{1}{n}}}\right)
^{2}\left\Vert \mathsf{P}_{F,J}^{\omega }\mathbf{x}\right\Vert _{L^{2}\left(
\omega \right) }^{2}\ .  \label{term B}
\end{equation}%
Then%
\begin{equation}
B\left( I\right) \lesssim \mathbf{\tau }\left( \left( \mathcal{E}_{\alpha }^{%
\limfunc{plug}}\right) ^{2}+A_{2}^{\alpha ,\limfunc{energy}}\right)
\left\vert I\right\vert _{\sigma }\ .  \label{B bound}
\end{equation}
\end{lemma}

\begin{proof}
We first prove the bound (\ref{B bound}) for $B\left( I\right) $ ignoring
for the moment the possible case when $J=I$ in the sum defining $B\left(
I\right) $. So suppose that $I\in \mathcal{A}\Omega \mathcal{D}$ is an
alternate $\Omega \mathcal{D}$-dyadic quasicube. Define%
\begin{equation*}
\Lambda ^{\ast }\left( I\right) \equiv \left\{ J\subsetneqq I:J\in \mathcal{M%
}_{\left( \mathbf{r},\varepsilon \right) -\limfunc{deep}}\left( F\right) 
\text{ for some }F\supsetneqq I^{\prime }\text{, }I^{\prime }\in \mathfrak{C}%
\left( I\right) \text{ with }\mathsf{P}_{F,J}^{\omega }\neq 0\right\} ,
\end{equation*}%
and pigeonhole this collection as $\Lambda ^{\ast }\left( I\right)
=\dbigcup\limits_{I_{0}^{\prime }\in \mathfrak{C}\left( I\right) }\Lambda
\left( I^{\prime }\right) $, where for each $I^{\prime }\in \mathfrak{C}%
\left( I\right) $ we define 
\begin{equation*}
\Lambda \left( I^{\prime }\right) \equiv \left\{ J\subset I^{\prime }:J\in 
\mathcal{M}_{\left( \mathbf{r},\varepsilon \right) -\limfunc{deep}}\left(
F\right) \text{ for some }F\supsetneqq I^{\prime }\text{ with }\mathsf{P}%
_{F,J}^{\omega }\neq 0\right\} .
\end{equation*}%
By Lemma \ref{tau ovelap} we may further pigeonhole (possibly with some
duplication) the quasicubes $J$ in $\Lambda \left( I^{\prime }\right) $ as
follows:%
\begin{equation*}
\Lambda \left( I^{\prime }\right) \subset \dbigcup\limits_{\ell =0}^{\mathbf{%
\tau }}\Lambda _{\ell }\left( I^{\prime }\right) ;\ \ \ \ \ \Lambda _{\ell
}\left( I^{\prime }\right) \equiv \left\{ J\subset I^{\prime }:J\in \mathcal{%
M}_{\left( \mathbf{r},\varepsilon \right) -\limfunc{deep}}\left( \pi _{%
\mathcal{F}}^{\ell }I^{\prime }\right) \text{ with }\mathsf{P}_{\pi _{%
\mathcal{F}}^{\ell }I^{\prime },J}^{\omega }\neq 0\right\} ,
\end{equation*}%
since every $F\supsetneqq I^{\prime }$ is of the form $\pi _{\mathcal{F}%
}^{\ell }I^{\prime }$ for some $\ell \geq 0$. Altogether then, we have
pigeonholed $\Lambda ^{\ast }\left( I\right) $ as%
\begin{equation*}
\Lambda ^{\ast }\left( I\right) =\dbigcup\limits_{I^{\prime }\in \mathfrak{C}%
\left( I\right) }\dbigcup\limits_{\ell =0}^{\mathbf{\tau }}\Lambda _{\ell
}\left( I^{\prime }\right) .
\end{equation*}

Now fix $I^{\prime }\in \mathfrak{C}\left( I\right) $ and $0\leq \ell \leq 
\mathbf{\tau }$, and for each $J$ in $\Lambda _{\ell }\left( I^{\prime
}\right) $, note that \emph{either} $J$ must contain some $K\in \mathcal{M}%
_{\left( \mathbf{r},\varepsilon \right) -\limfunc{deep}}\left( I^{\prime
}\right) $ \emph{or} $J\subset K$ for some $K\in \mathcal{M}_{\left( \mathbf{%
r},\varepsilon \right) -\limfunc{deep}}\left( I^{\prime }\right) $ (or both
if equality); define%
\begin{eqnarray*}
\Lambda _{\ell }\left( I^{\prime }\right) &=&\Lambda _{\ell }^{\limfunc{big}%
}\left( I^{\prime }\right) \cup \Lambda _{\ell }^{\limfunc{small}}\left(
I^{\prime }\right) ; \\
\Lambda _{\ell }^{\limfunc{small}}\left( I^{\prime }\right) &\equiv &\left\{
J\in \Lambda _{\ell }\left( I^{\prime }\right) :J\subset K\text{ for some }%
K\in \mathcal{M}_{\left( \mathbf{r},\varepsilon \right) -\limfunc{deep}%
}\left( I^{\prime }\right) \right\} ,
\end{eqnarray*}%
and we make the corresponding decomposition of $B\left( I\right) $ (again
with possible duplication);%
\begin{eqnarray*}
B\left( I\right) &=&B^{\limfunc{big}}\left( I\right) +B^{\limfunc{small}%
}\left( I\right) ; \\
B^{\limfunc{big}/\limfunc{small}}\left( I\right) &\equiv &\sum_{I^{\prime
}\in \mathfrak{C}\left( I\right) }\dsum\limits_{\ell =0}^{\mathbf{\tau }%
}\sum_{J\in \Lambda _{\ell }^{\limfunc{big}/\limfunc{small}}\left( I^{\prime
}\right) }\left( \frac{\mathrm{P}^{\alpha }\left( J,\mathbf{1}_{I}\sigma
\right) }{\left\vert J\right\vert ^{\frac{1}{n}}}\right) ^{2}\left\Vert 
\mathsf{P}_{\pi _{\mathcal{F}}^{\ell }I^{\prime },J}^{\omega }\mathbf{x}%
\right\Vert _{L^{2}\left( \omega \right) }^{2}\ .
\end{eqnarray*}

Turning first to $B^{\limfunc{small}}\left( I\right) $, we define $\sigma
\left( \ell \right) $ by $\pi _{\mathcal{F}}^{\left( \ell \right) }\left(
I^{\prime }\right) =\pi _{\Omega \mathcal{D}}^{\left( \sigma \left( \ell
\right) \right) }\left( I^{\prime }\right) $, so that $\Lambda _{\ell }^{%
\limfunc{small}}\left( I^{\prime }\right) \subset \mathcal{M}_{\left( 
\mathbf{r},\varepsilon \right) -\limfunc{deep}}^{\sigma \left( \ell \right)
}\left( I\right) $, and we obtain%
\begin{eqnarray}
&&B^{\limfunc{small}}\left( I\right) \leq \sum_{I^{\prime }\in \mathfrak{C}%
\left( I\right) }\dsum\limits_{\ell =0}^{\mathbf{\tau }}\sum_{J\in \Lambda
_{\ell }^{\limfunc{small}}\left( I^{\prime }\right) }\left( \frac{\mathrm{P}%
^{\alpha }\left( J,\mathbf{1}_{I}\sigma \right) }{\left\vert J\right\vert ^{%
\frac{1}{n}}}\right) ^{2}\left\Vert \mathsf{P}_{J}^{\limfunc{good},\omega }%
\mathbf{x}\right\Vert _{L^{2}\left( \omega \right) }^{2}  \label{B small} \\
&\leq &\sum_{I^{\prime }\in \mathfrak{C}\left( I\right) }\dsum\limits_{\ell
=0}^{\mathbf{\tau }}\sum_{J\in \mathcal{M}_{\left( \mathbf{r},\varepsilon
\right) -\limfunc{deep}}^{\sigma \left( \ell \right) }\left( I\right)
}\left( \frac{\mathrm{P}^{\alpha }\left( J,\mathbf{1}_{I}\sigma \right) }{%
\left\vert J\right\vert ^{\frac{1}{n}}}\right) ^{2}\left\Vert \mathsf{P}%
_{J}^{\limfunc{good},\omega }\mathbf{x}\right\Vert _{L^{2}\left( \omega
\right) }^{2}  \notag \\
&\lesssim &2^{n}\mathbf{\tau }\left( \mathcal{E}_{\alpha }^{\limfunc{refined}%
\limfunc{plug}}\right) ^{2}\left\vert I\right\vert _{\sigma }\ .  \notag
\end{eqnarray}

This is the only point in the proof of Theorem \ref{T1 theorem} that a
refined quasienergy constant is used.

Turning now to the more delicate term $B^{\limfunc{big}}\left( I\right) $,
we write for $J\in \Lambda _{\ell }^{\limfunc{big}}\left( I^{\prime }\right) 
$,%
\begin{eqnarray*}
\left\Vert \mathsf{P}_{J}^{\limfunc{good},\omega }\mathbf{x}\right\Vert
_{L^{2}\left( \omega \right) }^{2} &=&\sum_{J^{\prime }\subset J\text{:\ }%
J^{\prime }\text{ good}}\left\Vert \bigtriangleup _{J^{\prime }}^{\omega }%
\mathbf{x}\right\Vert _{L^{2}\left( \omega \right) }^{2} \\
&\leq &\sum_{J^{\prime }\in \mathcal{N}_{\mathbf{r}}\left( I^{\prime
}\right) \text{:\ }J^{\prime }\subset J}\left\Vert \bigtriangleup
_{J^{\prime }}^{\omega }\mathbf{x}\right\Vert _{L^{2}\left( \omega \right)
}^{2}+\sum_{K\in \mathcal{M}_{\left( \mathbf{r},\varepsilon \right) -%
\limfunc{deep}}\left( I^{\prime }\right) \text{:\ }K\subset J}\left\Vert 
\mathsf{P}_{K}^{\limfunc{good},\omega }\mathbf{x}\right\Vert _{L^{2}\left(
\omega \right) }^{2}\ ,
\end{eqnarray*}%
where for any $I$, we define $\mathcal{N}_{\mathbf{r}}\left( I\right) \equiv
\left\{ J^{\prime }\subset I:\ell \left( J^{\prime }\right) \geq 2^{-\mathbf{%
r}}\ell \left( I\right) \right\} $ to be the set of $r$-near quasicubes in $%
I $. Then we estimate%
\begin{eqnarray*}
B^{\limfunc{big}}\left( I\right) &=&\sum_{I^{\prime }\in \mathfrak{C}\left(
I\right) }\dsum\limits_{\ell =0}^{\mathbf{\tau }}\sum_{J\in \Lambda _{\ell
}^{\limfunc{big}}\left( I^{\prime }\right) }\left( \frac{\mathrm{P}^{\alpha
}\left( J,\mathbf{1}_{I}\sigma \right) }{\left\vert J\right\vert ^{\frac{1}{n%
}}}\right) ^{2}\left\Vert \mathsf{P}_{\pi _{\mathcal{F}}^{\ell }I^{\prime
},J}^{\omega }\mathbf{x}\right\Vert _{L^{2}\left( \omega \right) }^{2} \\
&\leq &\sum_{I^{\prime }\in \mathfrak{C}\left( I\right) }\dsum\limits_{\ell
=0}^{\mathbf{\tau }}\sum_{J\in \Lambda _{\ell }^{\limfunc{big}}\left(
I^{\prime }\right) }\left( \frac{\mathrm{P}^{\alpha }\left( J,\mathbf{1}%
_{I}\sigma \right) }{\left\vert J\right\vert ^{\frac{1}{n}}}\right)
^{2}\left\Vert \mathsf{P}_{J}^{\limfunc{good},\omega }\mathbf{x}\right\Vert
_{L^{2}\left( \omega \right) }^{2} \\
&=&\sum_{I^{\prime }\in \mathfrak{C}\left( I\right) }\dsum\limits_{\ell =0}^{%
\mathbf{\tau }}\sum_{J\in \Lambda _{\ell }^{\limfunc{big}}\left( I^{\prime
}\right) }\left( \frac{\mathrm{P}^{\alpha }\left( J,\mathbf{1}_{I}\sigma
\right) }{\left\vert J\right\vert ^{\frac{1}{n}}}\right) ^{2}\sum_{J^{\prime
}\in \mathcal{N}_{\mathbf{r}}\left( I^{\prime }\right) \text{:\ }J^{\prime
}\subset J}\left\Vert \bigtriangleup _{J^{\prime }}^{\omega }\mathbf{x}%
\right\Vert _{L^{2}\left( \omega \right) }^{2} \\
&&+\sum_{I^{\prime }\in \mathfrak{C}\left( I\right) }\dsum\limits_{\ell =0}^{%
\mathbf{\tau }}\sum_{J\in \Lambda _{\ell }^{\limfunc{big}}\left( I^{\prime
}\right) }\left( \frac{\mathrm{P}^{\alpha }\left( J,\mathbf{1}_{I}\sigma
\right) }{\left\vert J\right\vert ^{\frac{1}{n}}}\right) ^{2}\sum_{K\in 
\mathcal{M}_{\left( \mathbf{r},\varepsilon \right) -\limfunc{deep}}\left(
I^{\prime }\right) \text{:\ }K\subset J}\left\Vert \mathsf{P}_{K}^{\limfunc{%
subgood},\omega }\mathbf{x}\right\Vert _{L^{2}\left( \omega \right) }^{2} \\
&\equiv &B_{1}^{\limfunc{big}}\left( I\right) +B_{2}^{\limfunc{big}}\left(
I\right) \ .
\end{eqnarray*}%
Now using that the quasicubes $J$ in each $\Lambda _{\ell }^{\limfunc{big}%
}\left( I^{\prime }\right) $ that arise in the term $B_{1}^{\limfunc{big}%
}\left( I\right) $ are both $\mathbf{r}$-nearby in $I^{\prime }$ and
pairwise disjoint, and that there are altogether only $C2^{n\mathbf{r}}$
quasicubes in $\mathcal{N}_{\mathbf{r}}\left( I^{\prime }\right) $, we have 
\begin{eqnarray}
&&  \label{big 1} \\
B_{1}^{\limfunc{big}}\left( I\right) &=&\sum_{I^{\prime }\in \mathfrak{C}%
\left( I\right) }\dsum\limits_{\ell =0}^{\mathbf{\tau }}\sum_{J\in \Lambda
_{\ell }^{\limfunc{big}}\left( I^{\prime }\right) \cap \mathcal{N}_{\mathbf{r%
}}\left( I^{\prime }\right) }\left( \frac{\mathrm{P}^{\alpha }\left( J,%
\mathbf{1}_{I}\sigma \right) }{\left\vert J\right\vert ^{\frac{1}{n}}}%
\right) ^{2}\sum_{J^{\prime }\in \mathcal{N}_{\mathbf{r}}\left( I^{\prime
}\right) \text{:\ }J^{\prime }\subset J}\left\Vert \bigtriangleup
_{J^{\prime }}^{\omega }\mathbf{x}\right\Vert _{L^{2}\left( \omega \right)
}^{2}  \notag \\
&\lesssim &\sum_{I^{\prime }\in \mathfrak{C}\left( I\right)
}\dsum\limits_{\ell =0}^{\mathbf{\tau }}\sum_{J\in \Lambda _{\ell }^{%
\limfunc{big}}\left( I^{\prime }\right) \cap \mathcal{N}_{\mathbf{r}}\left(
I^{\prime }\right) }\left( \frac{\mathrm{P}^{\alpha }\left( J,\mathbf{1}%
_{I}\sigma \right) }{\left\vert J\right\vert ^{\frac{1}{n}}}\right)
^{2}\left\Vert \mathsf{P}_{J}^{\omega }\mathbf{x}\right\Vert _{L^{2}\left(
\omega \right) }^{2}  \notag \\
&\lesssim &\mathbf{\tau }2^{n\mathbf{r}}\sup_{I^{\prime }\in \mathfrak{C}%
\left( I\right) }\sup_{J\in \mathcal{N}_{\mathbf{r}}\left( I^{\prime
}\right) }\left( \frac{\mathrm{P}^{\alpha }\left( J,\mathbf{1}_{I}\sigma
\right) }{\left\vert J\right\vert ^{\frac{1}{n}}}\right) ^{2}\left\Vert 
\mathsf{P}_{J}^{\omega }\mathbf{x}\right\Vert _{L^{2}\left( \omega \right)
}^{2}  \notag \\
&\lesssim &\mathbf{\tau }2^{n\mathbf{r}}\left( \frac{\left\vert I\right\vert
_{\sigma }}{\left\vert I\right\vert ^{1-\frac{\alpha }{n}}}\right)
^{2}\left\Vert \mathsf{P}_{I}^{\omega }\frac{\mathbf{x}}{\ell \left(
I\right) }\right\Vert _{L^{2}\left( \omega \right) }^{2}\lesssim \mathbf{%
\tau }2^{n\mathbf{r}}A_{2}^{\alpha ,\limfunc{energy}}\left\vert I\right\vert
_{\sigma }\ .  \notag
\end{eqnarray}%
At this point we can estimate the missing case $J=I$ in the same way, namely%
\begin{eqnarray*}
&&\sum_{F\in \mathcal{F}:\ I\in \mathcal{M}_{\left( \mathbf{r},\varepsilon
\right) -\limfunc{deep}}\left( F\right) }\left( \frac{\mathrm{P}^{\alpha
}\left( I,\mathbf{1}_{I}\sigma \right) }{\left\vert I\right\vert ^{\frac{1}{n%
}}}\right) ^{2}\left\Vert \mathsf{P}_{F,I}^{\omega }\mathbf{x}\right\Vert
_{L^{2}\left( \omega \right) }^{2} \\
&\lesssim &\mathbf{\tau }\left( \frac{\mathrm{P}^{\alpha }\left( I,\mathbf{1}%
_{I}\sigma \right) }{\left\vert I\right\vert ^{\frac{1}{n}}}\right)
^{2}\left\Vert \mathsf{P}_{I}^{\omega }\mathbf{x}\right\Vert _{L^{2}\left(
\omega \right) }^{2}\lesssim \mathbf{\tau }A_{2}^{\alpha ,\limfunc{energy}%
}\left\vert I\right\vert _{\sigma }\ .
\end{eqnarray*}

Since $\frac{\mathrm{P}^{\alpha }\left( J,\mathbf{1}_{I}\sigma \right) }{%
\left\vert J\right\vert ^{\frac{1}{n}}}\lesssim \frac{\mathrm{P}^{\alpha
}\left( K,\mathbf{1}_{I}\sigma \right) }{\left\vert K\right\vert ^{\frac{1}{n%
}}}$ for $K\subset J$, and since the quasicubes $J\in \Lambda _{\ell }^{%
\limfunc{big}}\left( I^{\prime }\right) $ are pairwise disjoint, we have%
\begin{eqnarray*}
B_{2}^{\limfunc{big}}\left( I\right) &=&\sum_{I^{\prime }\in \mathfrak{C}%
\left( I\right) }\dsum\limits_{\ell =0}^{\mathbf{\tau }}\sum_{J\in \Lambda
_{\ell }^{\limfunc{big}}\left( I^{\prime }\right) }\sum_{K\in \mathcal{M}%
_{\left( \mathbf{r},\varepsilon \right) -\limfunc{deep}}\left( I^{\prime
}\right) \text{:\ }K\subset J}\left( \frac{\mathrm{P}^{\alpha }\left( J,%
\mathbf{1}_{I}\sigma \right) }{\left\vert J\right\vert ^{\frac{1}{n}}}%
\right) ^{2}\left\Vert \mathsf{P}_{K}^{\limfunc{subgood},\omega }\mathbf{x}%
\right\Vert _{L^{2}\left( \omega \right) }^{2} \\
&\lesssim &\sum_{I^{\prime }\in \mathfrak{C}\left( I\right)
}\dsum\limits_{\ell =0}^{\mathbf{\tau }}\sum_{J\in \Lambda _{\ell }^{%
\limfunc{big}}\left( I^{\prime }\right) }\sum_{K\in \mathcal{M}_{\left( 
\mathbf{r},\varepsilon \right) -\limfunc{deep}}\left( I^{\prime }\right) 
\text{:\ }K\subset J}\left( \frac{\mathrm{P}^{\alpha }\left( K,\mathbf{1}%
_{I}\sigma \right) }{\left\vert K\right\vert ^{\frac{1}{n}}}\right)
^{2}\left\Vert \mathsf{P}_{K}^{\limfunc{subgood},\omega }\mathbf{x}%
\right\Vert _{L^{2}\left( \omega \right) }^{2} \\
&\leq &\sum_{I^{\prime }\in \mathfrak{C}\left( I\right) }\dsum\limits_{\ell
=0}^{\mathbf{\tau }}\sum_{K\in \mathcal{M}_{\left( \mathbf{r},\varepsilon
\right) -\limfunc{deep}}\left( I^{\prime }\right) }\left( \frac{\mathrm{P}%
^{\alpha }\left( K,\mathbf{1}_{I}\sigma \right) }{\left\vert K\right\vert ^{%
\frac{1}{n}}}\right) ^{2}\left\Vert \mathsf{P}_{K}^{\limfunc{subgood},\omega
}\mathbf{x}\right\Vert _{L^{2}\left( \omega \right) }^{2} \\
&\lesssim &\mathbf{\tau }\left( \mathcal{E}_{\alpha }^{\limfunc{deep}%
\limfunc{plug}}\right) ^{2}\left\vert I\right\vert _{\sigma }\ ,
\end{eqnarray*}%
where the final line follows from the plugged deep energy condition with the
trivial outer decomposition $I=\dbigcup\limits_{I^{\prime }\in \mathfrak{C}%
\left( I\right) }I^{\prime }$. This completes the proof of Lemma \ref%
{refined lemma}.
\end{proof}

\subsection{The forward Poisson testing inequality}

Fix $I\in \Omega \mathcal{D}$. We split the integration on the left side of (%
\ref{e.t1 n}) into a local and global piece:%
\begin{equation*}
\int_{\mathbb{R}_{+}^{n+1}}\mathbb{P}^{\alpha }\left( \mathbf{1}_{I}\sigma
\right) ^{2}d\overline{\mu }=\int_{\widehat{I}}\mathbb{P}^{\alpha }\left( 
\mathbf{1}_{I}\sigma \right) ^{2}d\overline{\mu }+\int_{\mathbb{R}%
_{+}^{n+1}\setminus \widehat{I}}\mathbb{P}^{\alpha }\left( \mathbf{1}%
_{I}\sigma \right) ^{2}d\overline{\mu }\equiv \mathbf{Local}\left( I\right) +%
\mathbf{Global}\left( I\right) ,
\end{equation*}%
where more explicitly,%
\begin{eqnarray}
&&\mathbf{Local}\left( I\right) \equiv \int_{\widehat{I}}\left[ \mathbb{P}%
^{\alpha }\left( \mathbf{1}_{I}\sigma \right) \left( x,t\right) \right] ^{2}d%
\overline{\mu }\left( x,t\right) ;\ \ \ \ \ \overline{\mu }\equiv \frac{1}{%
t^{2}}\mu ,  \label{def local forward} \\
\text{i.e. }\overline{\mu } &\equiv &\sum_{J\in \Omega \mathcal{D}}\frac{1}{%
\ell \left( J\right) ^{2}}\ \sum_{F\in \mathcal{F}}\sum_{J\in \mathcal{M}%
_{\left( \mathbf{r},\varepsilon \right) -\limfunc{deep}}\left( F\right)
}\left\Vert \mathsf{P}_{F,J}^{\omega }\mathbf{x}\right\Vert _{L^{2}\left(
\omega \right) }^{2}\cdot \delta _{\left( c_{J},\ell \left( J\right) \right)
}.  \notag
\end{eqnarray}%
Here is a brief schematic diagram of the decompositions, with bounds in $%
\fbox{}$, used in this subsection:%
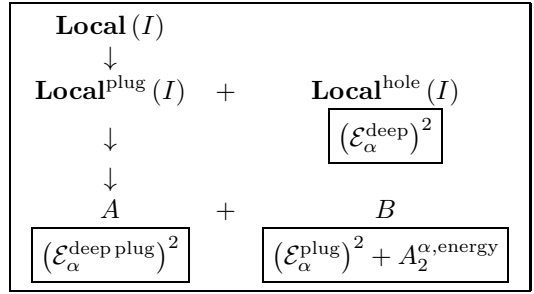
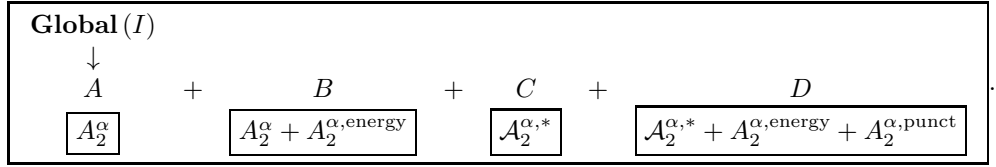
\begin{equation*}
\fbox{$%
\begin{array}{ccc}
\mathbf{Local}\left( I\right) &  &  \\ 
\downarrow &  &  \\ 
\mathbf{Local}^{\limfunc{plug}}\left( I\right) & + & \mathbf{Local}^{%
\limfunc{hole}}\left( I\right) \\ 
\downarrow &  & \fbox{$\left( \mathcal{E}_{\alpha }^{\limfunc{deep}}\right)
^{2}$} \\ 
\downarrow &  &  \\ 
A & + & B \\ 
\fbox{$\left( \mathcal{E}_{\alpha }^{\limfunc{deep}\limfunc{plug}}\right)
^{2}$} &  & \fbox{$\left( \mathcal{E}_{\alpha }^{\limfunc{plug}}\right)
^{2}+A_{2}^{\alpha ,\limfunc{energy}}$}%
\end{array}%
$}
\end{equation*}%
and%
\begin{equation*}
\fbox{$%
\begin{array}{ccccccc}
\mathbf{Global}\left( I\right) &  &  &  &  &  &  \\ 
\downarrow &  &  &  &  &  &  \\ 
A & + & B & + & C & + & D \\ 
\fbox{$A_{2}^{\alpha }$} &  & \fbox{$A_{2}^{\alpha }+A_{2}^{\alpha ,\limfunc{%
energy}}$} &  & \fbox{$\mathcal{A}_{2}^{\alpha ,\ast }$} &  & \fbox{$%
\mathcal{A}_{2}^{\alpha ,\ast }+A_{2}^{\alpha ,\limfunc{energy}%
}+A_{2}^{\alpha ,\limfunc{punct}}$}%
\end{array}%
$}.
\end{equation*}

An important consequence of the fact that $I$ and $J$ lie in the same
quasigrid $\Omega \mathcal{D}=\Omega \mathcal{D}^{\omega }$, is that%
\begin{equation}
\left( c\left( J\right) ,\ell \left( J\right) \right) \in \widehat{I}\text{ 
\textbf{if and only if} }J\subset I.  \label{tent consequence}
\end{equation}%
We thus have

\begin{eqnarray*}
&&\mathbf{Local}\left( I\right) =\int_{\widehat{I}}\mathbb{P}^{\alpha
}\left( \mathbf{1}_{I}\sigma \right) \left( x,t\right) ^{2}d\overline{\mu }%
\left( x,t\right) \\
&=&\sum_{F\in \mathcal{F}}\sum_{J\in \mathcal{M}_{\mathbf{r}-\limfunc{deep}%
}(F):\ J\subset I}\mathbb{P}^{\alpha }\left( \mathbf{1}_{I}\sigma \right)
\left( c_{J},\left\vert J\right\vert ^{\frac{1}{n}}\right) ^{2}\left\Vert 
\mathsf{P}_{F,J}^{\omega }\frac{\mathbf{x}}{\left\vert J\right\vert ^{\frac{1%
}{n}}}\right\Vert _{L^{2}\left( \omega \right) }^{2} \\
&\approx &\sum_{F\in \mathcal{F}}\sum_{J\in \mathcal{M}_{\mathbf{r}-\limfunc{%
deep}}\left( F\right) :\ J\subset I}\mathrm{P}^{\alpha }\left( J,\mathbf{1}%
_{I}\sigma \right) ^{2}\lVert \mathsf{P}_{F,J}^{\omega }\frac{\mathbf{x}}{%
\left\vert J\right\vert ^{\frac{1}{n}}}\rVert _{L^{2}\left( \omega \right)
}^{2} \\
&\lesssim &\mathbf{Local}^{\limfunc{plug}}\left( I\right) +\mathbf{Local}^{%
\func{hole}}\left( I\right) ,
\end{eqnarray*}%
where the `plugged' local sum $\mathbf{Local}^{\limfunc{plug}}\left(
I\right) $ is given by 
\begin{align*}
& \mathbf{Local}^{\limfunc{plug}}\left( I\right) \equiv \sum_{F\in \mathcal{F%
}}\sum_{J\in \mathcal{M}_{\mathbf{r}-\limfunc{deep}}\left( F\right) :\
J\subset I}\left( \frac{\mathrm{P}^{\alpha }\left( J,\mathbf{1}_{F\cap
I}\sigma \right) }{\left\vert J\right\vert ^{\frac{1}{n}}}\right)
^{2}\left\Vert \mathsf{P}_{F,J}^{\omega }\mathbf{x}\right\Vert _{L^{2}\left(
\omega \right) }^{2} \\
& =\left\{ \sum_{F\in \mathcal{F}:\ F\subset I}+\sum_{F\in \mathcal{F}:\
F\supsetneqq I}\right\} \sum_{J\in \mathcal{M}_{\mathbf{r}-\limfunc{deep}%
}\left( F\right) :\ J\subset I}\left( \frac{\mathrm{P}^{\alpha }\left( J,%
\mathbf{1}_{F\cap I}\sigma \right) }{\left\vert J\right\vert ^{\frac{1}{n}}}%
\right) ^{2}\left\Vert \mathsf{P}_{F,J}^{\omega }\mathbf{x}\right\Vert
_{L^{2}\left( \omega \right) }^{2} \\
& =A+B.
\end{align*}%
Then a \emph{trivial} application of the deep quasienergy condition (where
`trivial' means that the outer decomposition is just a single quasicube)
gives 
\begin{eqnarray*}
A &\leq &\sum_{F\in \mathcal{F}:\ F\subset I}\sum_{J\in \mathcal{M}_{\mathbf{%
r}-\limfunc{deep}}\left( F\right) }\left( \frac{\mathrm{P}^{\alpha }\left( J,%
\mathbf{1}_{F}\sigma \right) }{\left\vert J\right\vert ^{\frac{1}{n}}}%
\right) ^{2}\left\Vert \mathsf{P}_{F,J}^{\omega }\mathbf{x}\right\Vert
_{L^{2}\left( \omega \right) }^{2} \\
&\leq &\sum_{F\in \mathcal{F}:\ F\subset I}\left( \mathcal{E}_{\alpha }^{%
\limfunc{deep}\limfunc{plug}}\right) ^{2}\left\vert F\right\vert _{\sigma
}\lesssim \left( \mathcal{E}_{\alpha }^{\limfunc{deep}\limfunc{plug}}\right)
^{2}\left\vert I\right\vert _{\sigma }\,,
\end{eqnarray*}%
since $\left\Vert \mathsf{P}_{F,J}^{\omega }x\right\Vert _{L^{2}\left(
\omega \right) }^{2}\leq \left\Vert \mathsf{P}_{J}^{\limfunc{good},\omega }%
\mathbf{x}\right\Vert _{L^{2}\left( \omega \right) }^{2}$, where we recall
that the quasienergy constant $\mathcal{E}_{\alpha }^{\limfunc{deep}\limfunc{%
plug}}$ is defined in (\ref{def deep plug}). We also used here that the
stopping quasicubes $\mathcal{F}$ satisfy a $\sigma $-Carleson measure
estimate, 
\begin{equation*}
\sum_{F\in \mathcal{F}:\ F\subset F_{0}}\left\vert F\right\vert _{\sigma
}\lesssim \left\vert F_{0}\right\vert _{\sigma }.
\end{equation*}%
Lemma \ref{refined lemma} applies to the remaining term $B$ to obtain the
bound%
\begin{equation*}
B\lesssim \mathbf{\tau }\left( \left( \mathcal{E}_{\alpha }^{\limfunc{plug}%
}\right) ^{2}+A_{2}^{\alpha ,\limfunc{energy}}\right) \left\vert
I\right\vert _{\sigma }\ .
\end{equation*}

It remains then to show the inequality with `holes', where the support of $%
\sigma $ is restricted to the complement of the quasicube $F$. Thus for $%
J\in \mathcal{M}_{\left( \mathbf{r},\varepsilon \right) -\limfunc{deep}%
}\left( F\right) $ we may use $I\setminus F$ in the argument of the Poisson
integral. We consider%
\begin{equation*}
\mathbf{Local}^{\func{hole}}\left( I\right) =\sum_{F\in \mathcal{F}%
}\sum_{J\in \mathcal{M}_{\left( \mathbf{r},\varepsilon \right) -\limfunc{deep%
}}\left( F\right) :\ J\subset I}\left( \frac{\mathrm{P}^{\alpha }\left( J,%
\mathbf{1}_{I\setminus F}\sigma \right) }{\left\vert J\right\vert ^{\frac{1}{%
n}}}\right) ^{2}\left\Vert \mathsf{P}_{F,J}^{\omega }\mathbf{x}\right\Vert
_{L^{2}\left( \omega \right) }^{2}\ .
\end{equation*}

\begin{lemma}
\label{local hole}We have 
\begin{equation}
\mathbf{Local}^{\func{hole}}\left( I\right) \lesssim \left( \mathcal{E}%
_{\alpha }^{\limfunc{deep}}\right) ^{2}\left\vert I\right\vert _{\sigma }\,.
\label{RTS n}
\end{equation}
\end{lemma}

\begin{proof}
Fix $I\in \Omega \mathcal{D}$ and define%
\begin{equation*}
\mathcal{F}_{I}\equiv \left\{ F\in \mathcal{F}:F\subset I\right\} \cup
\left\{ I\right\} ,
\end{equation*}%
and denote by $\pi F$, for this proof only, the parent of $F$ in the tree $%
\mathcal{F}_{I}$. We estimate%
\begin{equation*}
S\equiv \sum_{F\in \mathcal{F}_{I}}\sum_{J\in \mathcal{M}_{\left( \mathbf{r}%
,\varepsilon \right) -\limfunc{deep}}\left( F\right) :\ J\subset I}\left( 
\frac{\mathrm{P}^{\alpha }\left( J,\mathbf{1}_{I\setminus F}\sigma \right) }{%
\left\vert J\right\vert ^{\frac{1}{n}}}\right) ^{2}\left\Vert \mathsf{P}%
_{F,J}^{\omega }\mathbf{x}\right\Vert _{L^{2}\left( \omega \right) }^{2}
\end{equation*}%
by 
\begin{eqnarray*}
S &=&\sum_{F\in \mathcal{F}_{I}}\sum_{J\in \mathcal{M}_{\left( \mathbf{r}%
,\varepsilon \right) -\limfunc{deep}}\left( F\right) :\ J\subset
I}\sum_{F^{\prime }\in \mathcal{F}:\ F\subset F^{\prime }\subsetneqq
I}\left( \frac{\mathrm{P}^{\alpha }\left( J,\mathbf{1}_{\pi F^{\prime
}\setminus F^{\prime }}\sigma \right) }{\left\vert J\right\vert ^{\frac{1}{n}%
}}\right) ^{2}\left\Vert \mathsf{P}_{F,J}^{\omega }\mathbf{x}\right\Vert
_{L^{2}\left( \omega \right) }^{2} \\
&=&\sum_{F^{\prime }\in \mathcal{F}_{I}}\sum_{F\in \mathcal{F}:\ F\subset
F^{\prime }}\sum_{J\in \mathcal{M}_{\left( \mathbf{r},\varepsilon \right) -%
\limfunc{deep}}\left( F\right) :\ J\subset I}\left( \frac{\mathrm{P}^{\alpha
}\left( J,\mathbf{1}_{\pi F^{\prime }\setminus F^{\prime }}\sigma \right) }{%
\left\vert J\right\vert ^{\frac{1}{n}}}\right) ^{2}\left\Vert \mathsf{P}%
_{F,J}^{\omega }\mathbf{x}\right\Vert _{L^{2}\left( \omega \right) }^{2} \\
&=&\sum_{F^{\prime }\in \mathcal{F}_{I}}\sum_{K\in \mathcal{M}_{\left( 
\mathbf{r},\varepsilon \right) -\limfunc{deep}}\left( F^{\prime }\right)
}\sum_{F\in \mathcal{F}:\ F\subset F^{\prime }}\sum_{J\in \mathcal{M}%
_{\left( \mathbf{r},\varepsilon \right) -\limfunc{deep}}\left( F\right) :\
J\subset I}\left( \frac{\mathrm{P}^{\alpha }\left( J,\mathbf{1}_{\pi
F^{\prime }\setminus F^{\prime }}\sigma \right) }{\left\vert J\right\vert ^{%
\frac{1}{n}}}\right) ^{2}\left\Vert \mathsf{P}_{F,J\cap K}^{\omega }\mathbf{x%
}\right\Vert _{L^{2}\left( \omega \right) }^{2} \\
&\lesssim &\sum_{F^{\prime }\in \mathcal{F}_{I}}\sum_{K\in \mathcal{M}%
_{\left( \mathbf{r},\varepsilon \right) -\limfunc{deep}}\left( F^{\prime
}\right) }\left( \frac{\mathrm{P}^{\alpha }\left( K,\mathbf{1}_{\pi
F^{\prime }\setminus F^{\prime }}\sigma \right) }{\left\vert K\right\vert ^{%
\frac{1}{n}}}\right) ^{2}\sum_{F\in \mathcal{F}:\ F\subset F^{\prime
}}\sum_{J\in \mathcal{M}_{\left( \mathbf{r},\varepsilon \right) -\limfunc{%
deep}}\left( F\right) :\ J\subset I}\left\Vert \mathsf{P}_{F,J\cap
K}^{\omega }\mathbf{x}\right\Vert _{L^{2}\left( \omega \right) }^{2},
\end{eqnarray*}%
where in the third line we have used that each $J^{\prime }$ appearing in $%
\mathsf{P}_{F,J}^{\omega }$ occurs in one of the $\mathsf{P}_{F,J\cap
K}^{\omega }$ by goodness, and where in the fourth line we have used the
Poisson inequalities in Lemma \ref{Poisson inequalities}. We now invoke%
\begin{equation*}
\sum_{F\in \mathcal{F}:\ F\subset F^{\prime }}\sum_{J\in \mathcal{M}_{\left( 
\mathbf{r},\varepsilon \right) -\limfunc{deep}}\left( F\right) :\ J\subset
I}\left\Vert \mathsf{P}_{F,J\cap K}^{\omega }\mathbf{x}\right\Vert
_{L^{2}\left( \omega \right) }^{2}\lesssim \mathbf{\tau }\left\Vert \widehat{%
\mathsf{P}}_{F^{\prime },K}^{\omega }\mathbf{x}\right\Vert _{L^{2}\left(
\omega \right) }^{2}\ ,
\end{equation*}%
where for $K\in \mathcal{M}_{\left( \mathbf{r},\varepsilon \right) -\limfunc{%
deep}}\left( F^{\prime }\right) $,%
\begin{eqnarray*}
\widehat{\mathsf{P}}_{F^{\prime },K}^{\omega } &\equiv &\sum_{J^{\prime
}\subset J\cap K:\ J^{\prime }\in \mathcal{C}_{F^{\prime };I}^{\ast
}}\bigtriangleup _{J^{\prime }}^{\omega }\ , \\
\text{where }\mathcal{C}_{F^{\prime };I}^{\ast } &\equiv
&\dbigcup\limits_{F\in \mathcal{F}:\ F\subset F^{\prime
}}\dbigcup\limits_{J\in \mathcal{M}_{\left( \mathbf{r},\varepsilon \right) -%
\limfunc{deep}}\left( F\right) :\ J\subset I}\mathcal{C}_{F}^{\limfunc{good},%
\mathbf{\tau }-\limfunc{shift}}\ .
\end{eqnarray*}

Now denote by $d\left( F,F^{\prime }\right) \equiv d_{\mathcal{F}_{I}}\left(
F,F^{\prime }\right) $ the distance from $F$ to $F^{\prime }$ in the tree $%
\mathcal{F}_{I}$, and denote by $d\left( F\right) \equiv d_{\mathcal{F}%
_{I}}\left( F,I\right) $ the distance of $F$ from the root $I$. Since the
collection $\mathcal{F}$ satisfies a Carleson condition, namely $\sum_{F\in 
\mathcal{F}_{I}}\left\vert F\cap I^{\prime }\right\vert _{\sigma }\leq
C\left\vert I^{\prime }\right\vert _{\sigma }$ for all $I^{\prime }$, we
have geometric decay in generations:%
\begin{equation}
\sum_{F\in \mathcal{F}_{I}:\ d\left( F\right) =k}\left\vert F\right\vert
_{\sigma }\lesssim 2^{-\delta k}\left\vert I\right\vert _{\sigma }\ ,\ \ \ \
\ k\geq 0.  \label{geometric decay}
\end{equation}%
Indeed, with $m>2C$ we have%
\begin{equation}
\sum_{F\in \mathcal{F}_{I}:\ F\subset F^{\prime }\text{ and }d\left(
F,F^{\prime }\right) =m}\left\vert F^{\prime }\cap I^{\prime }\right\vert
_{\sigma }<\frac{1}{2}\left\vert I^{\prime }\right\vert _{\sigma }\ ,
\label{half}
\end{equation}%
since otherwise%
\begin{equation*}
\sum_{F\in \mathcal{F}_{I}:\ F\subset F^{\prime }\text{ and }d\left(
F,F^{\prime }\right) \leq m}\left\vert F^{\prime }\cap I^{\prime
}\right\vert _{\sigma }\geq m\frac{1}{2}\left\vert I^{\prime }\right\vert
_{\sigma }\ ,
\end{equation*}%
a contradiction. Now iterate (\ref{half}) to obtain (\ref{geometric decay}).
Thus we can write%
\begin{eqnarray*}
S &\lesssim &\sum_{F^{\prime }\in \mathcal{F}_{I}}\sum_{K\in \mathcal{M}%
_{\left( \mathbf{r},\varepsilon \right) -\limfunc{deep}}\left( F^{\prime
}\right) }\left( \frac{\mathrm{P}^{\alpha }\left( K,\mathbf{1}_{\pi
F^{\prime }\setminus F^{\prime }}\sigma \right) }{\left\vert K\right\vert ^{%
\frac{1}{n}}}\right) ^{2}\left\Vert \widehat{\mathsf{P}}_{F^{\prime
},K}^{\omega }\mathbf{x}\right\Vert _{L^{2}\left( \omega \right) }^{2} \\
&=&\sum_{k=1}^{\infty }\sum_{F^{\prime }\in \mathcal{F}_{I}:\ d\left(
F^{\prime }\right) =k}\sum_{K\in \mathcal{M}_{\left( \mathbf{r},\varepsilon
\right) -\limfunc{deep}}\left( F^{\prime }\right) }\left( \frac{\mathrm{P}%
^{\alpha }\left( K,\mathbf{1}_{\pi F^{\prime }\setminus F^{\prime }}\sigma
\right) }{\left\vert K\right\vert ^{\frac{1}{n}}}\right) ^{2}\left\Vert 
\widehat{\mathsf{P}}_{F^{\prime },K}^{\omega }\mathbf{x}\right\Vert
_{L^{2}\left( \omega \right) }^{2}\equiv \sum_{k=1}^{\infty }A_{k}.
\end{eqnarray*}%
Now we have $\left\Vert \widehat{\mathsf{P}}_{F^{\prime },K}^{\omega }%
\mathbf{x}\right\Vert _{L^{2}\left( \omega \right) }^{2}\leq \left\Vert 
\mathsf{P}_{K}^{\limfunc{good},\omega }\mathbf{x}\right\Vert _{L^{2}\left(
\omega \right) }^{2}$, and hence by the deep energy condition,%
\begin{eqnarray*}
A_{k} &=&\sum_{F^{\prime }\in \mathcal{F}_{I}:\ d\left( F^{\prime }\right)
=k}\sum_{K\in \mathcal{M}_{\left( \mathbf{r},\varepsilon \right) -\limfunc{%
deep}}\left( F^{\prime }\right) }\left( \frac{\mathrm{P}^{\alpha }\left( K,%
\mathbf{1}_{\pi _{\mathcal{F}}F^{\prime }\setminus F^{\prime }}\sigma
\right) }{\left\vert K\right\vert ^{\frac{1}{n}}}\right) ^{2}\left\Vert 
\widehat{\mathsf{P}}_{F^{\prime },K}^{\omega }\mathbf{x}\right\Vert
_{L^{2}\left( \omega \right) }^{2} \\
&\lesssim &\left( \mathcal{E}_{\alpha }^{\limfunc{deep}}\right)
^{2}\sum_{F^{\prime \prime }\in \mathcal{F}_{I}:\ d\left( F^{\prime \prime
}\right) =k-1}\left\vert F^{\prime \prime }\right\vert _{\sigma }\lesssim
\left( \mathcal{E}_{\alpha }^{\limfunc{deep}}\right) ^{2}2^{-\delta
k}\left\vert I\right\vert _{\sigma }\ ,
\end{eqnarray*}%
where we have applied the deep energy condition for each $F^{\prime \prime
}\in \mathcal{F}_{I}$ with $d\left( F^{\prime \prime }\right) =k-1$ to obtain%
\begin{equation}
\sum_{F^{\prime }\in \mathcal{F}_{I}:\ \pi F^{\prime }=F^{\prime \prime
}}\sum_{K\in \mathcal{M}_{\left( \mathbf{r},\varepsilon \right) -\limfunc{%
deep}}\left( F^{\prime }\right) }\left( \frac{\mathrm{P}^{\alpha }\left( K,%
\mathbf{1}_{F^{\prime \prime }\setminus F^{\prime }}\sigma \right) }{%
\left\vert K\right\vert ^{\frac{1}{n}}}\right) ^{2}\left\Vert \widehat{%
\mathsf{P}}_{F^{\prime },K}^{\omega }\mathbf{x}\right\Vert _{L^{2}\left(
\omega \right) }^{2}\leq \left( \mathcal{E}_{\alpha }^{\limfunc{deep}%
}\right) ^{2}\left\vert F^{\prime \prime }\right\vert _{\sigma }\ .
\label{to obtain}
\end{equation}%
Finally then we obtain%
\begin{equation*}
S\lesssim \sum_{k=1}^{\infty }\left( \mathcal{E}_{\alpha }^{\limfunc{deep}%
}\right) ^{2}2^{-\delta k}\left\vert I\right\vert _{\sigma }\lesssim \left( 
\mathcal{E}_{\alpha }^{\limfunc{deep}}\right) ^{2}\left\vert I\right\vert
_{\sigma }\ ,
\end{equation*}%
which is (\ref{RTS n}).
\end{proof}

Altogether we have now proved the estimate $\mathbf{Local}\left( I\right)
\lesssim \left( \left( \mathcal{E}_{\alpha }^{\limfunc{plug}}\right)
^{2}+A_{2}^{\alpha ,\limfunc{energy}}\right) \left\vert I\right\vert
_{\sigma }$ when $I\in \Omega \mathcal{D}$, i.e. for every dyadic quasicube $%
L\in \Omega \mathcal{D}$,%
\begin{eqnarray}
&&  \label{local} \\
\mathbf{Local}\left( L\right) &\approx &\sum_{F\in \mathcal{F}}\sum_{J\in 
\mathcal{M}_{\left( \mathbf{r},\varepsilon \right) -\limfunc{deep}}\left(
F\right) :\ J\subset L}\left( \frac{\mathrm{P}^{\alpha }\left( J,\mathbf{1}%
_{L}\sigma \right) }{\left\vert J\right\vert ^{\frac{1}{n}}}\right)
^{2}\left\Vert \mathsf{P}_{F,J}^{\omega }\mathbf{x}\right\Vert _{L^{2}\left(
\omega \right) }^{2}  \notag \\
&\lesssim &\left( \left( \mathcal{E}_{\alpha }^{\limfunc{plug}}\right)
^{2}+A_{2}^{\alpha ,\limfunc{energy}}\right) \left\vert L\right\vert
_{\sigma },\ \ \ L\in \Omega \mathcal{D}.  \notag
\end{eqnarray}

\subsubsection{The alternate local estimate}

For future use, we prove a strengthening of the local estimate $\mathbf{Local%
}\left( L\right) $ to \emph{alternate} quasicubes $M\in \mathcal{A}\Omega 
\mathcal{D}$.

\begin{lemma}
\label{shifted}With notation as above and $M\in \mathcal{A}\Omega \mathcal{D}
$ an alternate quasicube, we have 
\begin{eqnarray}
&&  \label{shifted local} \\
\mathbf{Local}\left( M\right) &\equiv &\sum_{F\in \mathcal{F}}\sum_{J\in 
\mathcal{M}_{\left( \mathbf{r},\varepsilon \right) -\limfunc{deep}}\left(
F\right) :\ J\subset M}\left( \frac{\mathrm{P}^{\alpha }\left( J,\mathbf{1}%
_{M}\sigma \right) }{\left\vert J\right\vert ^{\frac{1}{n}}}\right)
^{2}\left\Vert \mathsf{P}_{F,J}^{\omega }\mathbf{x}\right\Vert _{L^{2}\left(
\omega \right) }^{2}  \notag \\
&\lesssim &\mathbf{\tau }\left( \left( \mathcal{E}_{\alpha }^{\limfunc{plug}%
}\right) ^{2}+A_{2}^{\alpha ,\limfunc{energy}}\right) \left\vert
M\right\vert _{\sigma },\ \ \ M\in \mathcal{A}\Omega \mathcal{D}.  \notag
\end{eqnarray}
\end{lemma}

\begin{proof}
We prove (\ref{shifted local}) by repeating the above proof of (\ref{local})
and noting the points requiring change. First we decompose 
\begin{equation*}
\mathbf{Local}\left( M\right) \lesssim \mathbf{Local}^{\limfunc{plug}}\left(
M\right) +\mathbf{Local}^{\func{hole}}\left( M\right) +\mathbf{Local}^{%
\limfunc{offset}}\left( M\right)
\end{equation*}%
where $\mathbf{Local}^{\limfunc{plug}}\left( M\right) $ and $\mathbf{Local}^{%
\func{hole}}\left( M\right) $ are analogous to $\mathbf{Local}^{\limfunc{plug%
}}\left( I\right) $ and $\mathbf{Local}^{\func{hole}}\left( I\right) $
above, and where $\mathbf{Local}^{\limfunc{offset}}\left( M\right) $ is an
additional term arising because $M\setminus F$ need not be empty when $M\cap
F\neq \emptyset $ and $F$ is not contained in $M$:%
\begin{eqnarray*}
\mathbf{Local}^{\limfunc{plug}}\left( M\right) &\equiv &\sum_{F\in \mathcal{F%
}}\sum_{J\in \mathcal{M}_{\left( \mathbf{r},\varepsilon \right) -\limfunc{%
deep}}\left( F\right) :\ J\subset M}\left( \frac{\mathrm{P}^{\alpha }\left(
J,\mathbf{1}_{M\cap F}\sigma \right) }{\left\vert J\right\vert ^{\frac{1}{n}}%
}\right) ^{2}\left\Vert \mathsf{P}_{F,J}^{\omega }\mathbf{x}\right\Vert
_{L^{2}\left( \omega \right) }^{2}\ , \\
\mathbf{Local}^{\func{hole}}\left( M\right) &\equiv &\sum_{F\in \mathcal{F}%
:\ F\subset M}\sum_{J\in \mathcal{M}_{\left( \mathbf{r},\varepsilon \right) -%
\limfunc{deep}}\left( F\right) :\ J\subset M}\left( \frac{\mathrm{P}^{\alpha
}\left( J,\mathbf{1}_{M\setminus F}\sigma \right) }{\left\vert J\right\vert
^{\frac{1}{n}}}\right) ^{2}\left\Vert \mathsf{P}_{F,J}^{\omega }\mathbf{x}%
\right\Vert _{L^{2}\left( \omega \right) }^{2}\ , \\
\mathbf{Local}^{\limfunc{offset}}\left( M\right) &\equiv &\sum_{F\in 
\mathcal{F}:\ F\not\subset M}\sum_{J\in \mathcal{M}_{\left( \mathbf{r}%
,\varepsilon \right) -\limfunc{deep}}\left( F\right) :\ J\subset M}\left( 
\frac{\mathrm{P}^{\alpha }\left( J,\mathbf{1}_{M\setminus F}\sigma \right) }{%
\left\vert J\right\vert ^{\frac{1}{n}}}\right) ^{2}\left\Vert \mathsf{P}%
_{F,J}^{\omega }\mathbf{x}\right\Vert _{L^{2}\left( \omega \right) }^{2}\ .
\end{eqnarray*}%
We have%
\begin{align*}
& \mathbf{Local}^{\limfunc{plug}}\left( M\right) =\left\{ \sum_{F\in 
\mathcal{F}:\ F\subset \text{ some }M^{\prime }\in \mathfrak{C}_{\Omega 
\mathcal{D}}\left( M\right) }+\sum_{F\in \mathcal{F}:\ F\supsetneqq \text{
some }M^{\prime }\in \mathfrak{C}_{\Omega \mathcal{D}}\left( M\right)
}\right\} \sum_{J\in \mathcal{M}_{\left( \mathbf{r},\varepsilon \right) -%
\limfunc{deep}}\left( F\right) :\ J\subset M} \\
& \ \ \ \ \ \ \ \ \ \ \ \ \ \ \ \ \ \ \ \ \ \ \ \ \ \ \ \ \ \ \times \left( 
\frac{\mathrm{P}^{\alpha }\left( J,\mathbf{1}_{F\cap M}\sigma \right) }{%
\left\vert J\right\vert ^{\frac{1}{n}}}\right) ^{2}\left\Vert \mathsf{P}%
_{F,J}^{\omega }\mathbf{x}\right\Vert _{L^{2}\left( \omega \right) }^{2} \\
& =A+B.
\end{align*}

Term $A$ satisfies%
\begin{equation*}
A\lesssim \left( \mathcal{E}_{\alpha }^{\limfunc{deep}\limfunc{plug}}\right)
^{2}\left\vert M\right\vert _{\sigma }\ ,
\end{equation*}%
just as above using $\left\Vert \mathsf{P}_{F,J}^{\omega }x\right\Vert
_{L^{2}\left( \omega \right) }^{2}\leq \left\Vert \mathsf{P}_{J}^{\omega }%
\mathbf{x}\right\Vert _{L^{2}\left( \omega \right) }^{2}$, and the fact that
the stopping quasicubes $\mathcal{F}$ satisfy a $\sigma $-Carleson measure
estimate, 
\begin{equation*}
\sum_{F\in \mathcal{F}:\ F\subset M}\left\vert F\right\vert _{\sigma
}\lesssim \left\vert M\right\vert _{\sigma }.
\end{equation*}

Term $B$ is handled directly by Lemma \ref{refined lemma} with the alternate
quasicube $I=M$ to obtain%
\begin{equation*}
B\lesssim \left( \left( \mathcal{E}_{\alpha }^{\limfunc{plug}}\right)
^{2}+A_{2}^{\alpha ,\limfunc{energy}}\right) \left\vert M\right\vert
_{\sigma }\ .
\end{equation*}

To extend Lemma \ref{local hole} to alternate quasicubes $M\in \mathcal{A}%
\Omega \mathcal{D}$, we define%
\begin{equation*}
\mathcal{F}_{M}\equiv \left\{ F\in \mathcal{F}:F\subset M\right\} \cup
\left\{ M\right\} ,
\end{equation*}%
and follow along the proof there with only trivial changes. The analogue of (%
\ref{to obtain}) is now%
\begin{equation*}
\sum_{F^{\prime }\in \mathcal{F}_{M}:\ \pi F^{\prime }=F^{\prime \prime
}}\sum_{K\in \mathcal{M}_{\left( \mathbf{r},\varepsilon \right) -\limfunc{%
deep}}\left( F^{\prime }\right) }\left( \frac{\mathrm{P}^{\alpha }\left( K,%
\mathbf{1}_{F^{\prime \prime }\setminus F^{\prime }}\sigma \right) }{%
\left\vert K\right\vert ^{\frac{1}{n}}}\right) ^{2}\left\Vert \widehat{%
\mathsf{P}}_{F^{\prime },K}^{\omega }\mathbf{x}\right\Vert _{L^{2}\left(
\omega \right) }^{2}\leq \left( \mathcal{E}_{\alpha }^{\limfunc{deep}%
}\right) ^{2}\left\vert F^{\prime \prime }\right\vert _{\sigma }\ ,
\end{equation*}%
the only change being that $\mathcal{F}_{M}$ now appears in place of $%
\mathcal{F}_{I}$, so that the deep energy condition still applies. We
conclude that 
\begin{equation*}
\mathbf{Local}^{\func{hole}}\left( M\right) \lesssim \left( \mathcal{E}%
_{\alpha }^{\limfunc{deep}}\right) ^{2}\left\vert M\right\vert _{\sigma }\ .
\end{equation*}

Finally, the additional term $\mathbf{Local}^{\limfunc{offset}}\left(
M\right) $ is handled directly by Lemma \ref{refined lemma}, and this
completes the proof of the estimate (\ref{shifted local}) in Lemma \ref%
{shifted}.
\end{proof}

\subsubsection{The global estimate}

Now we turn to proving the following estimate for the global part of the
first testing condition \eqref{e.t1 n}:%
\begin{equation*}
\mathbf{Global}\left( I\right) =\int_{\mathbb{R}_{+}^{n+1}\setminus \widehat{%
I}}\mathbb{P}^{\alpha }\left( \mathbf{1}_{I}\sigma \right) ^{2}d\overline{%
\mu }\lesssim \left( \mathcal{A}_{2}^{\alpha ,\ast }+A_{2}^{\alpha ,\limfunc{%
energy}}+A_{2}^{\alpha ,\limfunc{punct}}\right) \left\vert I\right\vert
_{\sigma }.
\end{equation*}%
We begin by decomposing the integral on the right into four pieces. As a
particular consequence of Lemma \ref{tau ovelap}, we note that given $J$,
there are at most a fixed number $\mathbf{\tau }$ of $F\in \mathcal{F}$ such
that $J\in \mathcal{M}_{\mathbf{r}-\limfunc{deep}}\left( F\right) $. We have:%
\begin{eqnarray*}
&&\int_{\mathbb{R}_{+}^{n+1}\setminus \widehat{I}}\mathbb{P}^{\alpha }\left( 
\mathbf{1}_{I}\sigma \right) ^{2}d\mu \leq \sum_{J:\ \left( c_{J},\ell
\left( J\right) \right) \in \mathbb{R}_{+}^{n+1}\setminus \widehat{I}}%
\mathbb{P}^{\alpha }\left( \mathbf{1}_{I}\sigma \right) \left( c_{J},\ell
\left( J\right) \right) ^{2}\sum_{\substack{ F\in \mathcal{F}  \\ J\in 
\mathcal{M}_{\left( \mathbf{r},\varepsilon \right) -\limfunc{deep}}\left(
F\right) }}\left\Vert \mathsf{P}_{F,J}^{\omega }\frac{\mathbf{x}}{\left\vert
J\right\vert ^{\frac{1}{n}}}\right\Vert _{L^{2}\left( \omega \right) }^{2} \\
&=&\left\{ \sum_{\substack{ J\cap 3I=\emptyset  \\ \ell \left( J\right) \leq
\ell \left( I\right) }}+\sum_{J\subset 3I\setminus I}+\sum_{\substack{ J\cap
I=\emptyset  \\ \ell \left( J\right) >\ell \left( I\right) }}%
+\sum_{J\supsetneqq I}\right\} \mathbb{P}^{\alpha }\left( \mathbf{1}%
_{I}\sigma \right) \left( c_{J},\ell \left( J\right) \right) ^{2}\sum 
_{\substack{ F\in \mathcal{F}:  \\ J\in \mathcal{M}_{\left( \mathbf{r}%
,\varepsilon \right) -\limfunc{deep}}\left( F\right) }}\left\Vert \mathsf{P}%
_{F,J}^{\omega }\frac{\mathbf{x}}{\left\vert J\right\vert ^{\frac{1}{n}}}%
\right\Vert _{L^{2}\left( \omega \right) }^{2} \\
&=&A+B+C+D.
\end{eqnarray*}

We further decompose term $A$ according to the length of $J$ and its
distance from $I$, and then use Lemma \ref{tau ovelap} with $I_{0}=J$ to
obtain:%
\begin{eqnarray*}
A &\lesssim &\sum_{m=0}^{\infty }\sum_{k=1}^{\infty }\sum_{\substack{ %
J\subset 3^{k+1}I\setminus 3^{k}I  \\ \ell \left( J\right) =2^{-m}\ell
\left( I\right) }}\left( \frac{2^{-m}\left\vert I\right\vert ^{\frac{1}{n}}}{%
\limfunc{quasidist}\left( J,I\right) ^{n+1-\alpha }}\left\vert I\right\vert
_{\sigma }\right) ^{2}\mathbf{\tau }\left\vert J\right\vert _{\omega } \\
&\lesssim &\sum_{m=0}^{\infty }2^{-2m}\sum_{k=1}^{\infty }\frac{\left\vert
I\right\vert ^{\frac{2}{n}}\left\vert I\right\vert _{\sigma }\left\vert
3^{k+1}I\setminus 3^{k}I\right\vert _{\omega }}{\left\vert 3^{k}I\right\vert
^{2\left( 1+\frac{1}{n}-\frac{\alpha }{n}\right) }}\left\vert I\right\vert
_{\sigma } \\
&\lesssim &\sum_{m=0}^{\infty }2^{-2m}\sum_{k=1}^{\infty }3^{-2k}\left\{ 
\frac{\left\vert 3^{k+1}I\setminus 3^{k}I\right\vert _{\omega }\left\vert
3^{k}I\right\vert _{\sigma }}{\left\vert 3^{k}I\right\vert ^{2\left( 1-\frac{%
\alpha }{n}\right) }}\right\} \left\vert I\right\vert _{\sigma }\lesssim
A_{2}^{\alpha }\left\vert I\right\vert _{\sigma },
\end{eqnarray*}%
where the offset Muckenhoupt constant $A_{2}^{\alpha }$ applies because $%
3^{k+1}I$ has only three times the side length of $3^{k}I$.

For term $B$ we first dispose of the nearby sum $B_{\limfunc{nearby}}$ that
consists of the sum over those $J$ which satisfy in addition $2^{-\mathbf{r}%
}\ell \left( I\right) \leq \ell \left( J\right) \leq \ell \left( I\right) $.
But it is a straightforward task to bound $B_{\limfunc{nearby}}$ by $%
CA_{2}^{\alpha ,\limfunc{energy}}\left\vert I\right\vert _{\sigma }$ as
there are at most $2^{n\left( \mathbf{r}+1\right) }$ such quasicubes $J$. In
order to bound $B-B_{\limfunc{nearby}}$, let 
\begin{equation*}
\mathcal{J}^{\ast }\equiv \dbigcup\limits_{F\in \mathcal{F}}\dbigcup\limits 
_{\substack{ J\in \mathcal{M}_{\left( \mathbf{r},\varepsilon \right) -%
\limfunc{deep}}\left( F\right)  \\ J\subset 3I\setminus I\text{ and }\ell
\left( J\right) \leq 2^{-\mathbf{r}}\ell \left( I\right) }}\left\{ K\in 
\mathcal{C}_{F}^{\limfunc{good},\mathbf{\tau }-\limfunc{shift}}:K\subset
J\right\} ,
\end{equation*}%
which is the union of all quasicubes $K$ for which the projection $%
\bigtriangleup _{K}^{\omega }$ occurs in one of the projections $\mathsf{P}%
_{F,J}^{\omega }$ with $\ell \left( J\right) <2^{-\mathbf{r}}\ell \left(
I\right) $ in term $B$. We further decompose term $B$ according to the
length of $J$ and use the fractional Poisson inequality (\ref{e.Jsimeq}) in
Lemma \ref{Poisson inequality} on the neighbour $I^{\prime }$ of $I$
containing $K$,%
\begin{equation*}
\mathrm{P}^{\alpha }\left( K,\mathbf{1}_{I}\sigma \right) ^{2}\lesssim
\left( \frac{\ell \left( K\right) }{\ell \left( I\right) }\right)
^{2-2\left( n+1-\alpha \right) \varepsilon }\mathrm{P}^{\alpha }\left( I,%
\mathbf{1}_{I}\sigma \right) ^{2},\ \ \ \ \ K\in \mathcal{J}^{\ast
},K\subset 3I\setminus I,
\end{equation*}%
where we have used that $\mathrm{P}^{\alpha }\left( I^{\prime },\mathbf{1}%
_{I}\sigma \right) \approx \mathrm{P}^{\alpha }\left( I,\mathbf{1}_{I}\sigma
\right) $ and that the quasicubes $K\in \mathcal{J}^{\ast }$ are good and
have side length at most $2^{-\mathbf{r}}\ell \left( I\right) $. We then
obtain from Lemma \ref{Poisson inequalities} and Lemma \ref{tau ovelap} with 
$I_{0}=J$,%
\begin{eqnarray*}
B-B_{\limfunc{nearby}} &\approx &\sum_{J\subset 3I\setminus I}\left( \frac{%
\mathrm{P}^{\alpha }\left( J,\mathbf{1}_{I}\sigma \right) }{\left\vert
J\right\vert ^{\frac{1}{n}}}\right) ^{2}\sum_{\substack{ F\in \mathcal{F}:\
J\in \mathcal{M}_{\left( \mathbf{r},\varepsilon \right) -\limfunc{deep}%
}\left( F\right)  \\ \ell \left( J\right) \leq 2^{-\mathbf{r}}\ell \left(
I\right) }}\left\Vert \mathsf{P}_{F,J}^{\omega }x\right\Vert _{L^{2}\left(
\omega \right) }^{2} \\
&\lesssim &\sum_{K\in \mathcal{J}^{\ast }}\left( \frac{\mathrm{P}^{\alpha
}\left( K,\mathbf{1}_{I}\sigma \right) }{\left\vert K\right\vert ^{\frac{1}{n%
}}}\right) ^{2}\mathbf{\tau \ }\left\Vert \bigtriangleup _{K}^{\omega }%
\mathbf{x}\right\Vert _{L^{2}\left( \omega \right) }^{2} \\
&\lesssim &\mathbf{\tau }\sum_{m=\mathbf{r}}^{\infty }\sum_{\substack{ %
K\subset 3I\setminus I  \\ \ell \left( K\right) =2^{-m}\ell \left( I\right) 
}}\left( 2^{-m}\right) ^{2-2\left( n+1-\alpha \right) \varepsilon }\mathrm{P}%
^{\alpha }\left( I,\mathbf{1}_{I}\sigma \right) ^{2}\left\vert K\right\vert
_{\omega } \\
&\lesssim &\mathbf{\tau }\sum_{m=\mathbf{r}}^{\infty }\left( 2^{-m}\right)
^{2-2\left( n+1-\alpha \right) \varepsilon }\left( \frac{\left\vert
I\right\vert _{\sigma }}{\left\vert I\right\vert ^{1-\frac{\alpha }{n}}}%
\right) ^{2}\sum_{\substack{ K\subset 3I\setminus I  \\ \ell \left( K\right)
=2^{-m}\ell \left( I\right) }}\left\vert K\right\vert _{\omega } \\
&\lesssim &\mathbf{\tau }\sum_{m=\mathbf{r}}^{\infty }\left( 2^{-m}\right)
^{2-2\left( n+1-\alpha \right) \varepsilon }\frac{\left\vert I\right\vert
_{\sigma }\left\vert 3I\setminus I\right\vert _{\omega }}{\left\vert
3I\right\vert ^{2\left( 1-\frac{\alpha }{n}\right) }}\left\vert I\right\vert
_{\sigma }\lesssim \mathbf{\tau }A_{2}^{\alpha }\left\vert I\right\vert
_{\sigma }\ .
\end{eqnarray*}

For term $C$ we will have to group the quasicubes $J$ into blocks $B_{i}$,
and then exploit Lemma \ref{tau ovelap}. We first split the sum according to
whether or not $I$ intersects the triple of $J$:%
\begin{eqnarray*}
C &\approx &\left\{ \sum_{\substack{ J:\ I\cap 3J=\emptyset  \\ \ell \left(
J\right) >\ell \left( I\right) }}+\sum_{\substack{ J:\ I\subset 3J\setminus
J  \\ \ell \left( J\right) >\ell \left( I\right) }}\right\} \left( \frac{%
\left\vert J\right\vert ^{\frac{1}{n}}}{\left( \left\vert J\right\vert ^{%
\frac{1}{n}}+\limfunc{quasidist}\left( J,I\right) \right) ^{n+1-\alpha }}%
\left\vert I\right\vert _{\sigma }\right) ^{2}\sum_{\substack{ F\in \mathcal{%
F}:  \\ J\in \mathcal{M}_{\left( \mathbf{r},\varepsilon \right) -\limfunc{%
deep}}\left( F\right) }}\left\Vert \mathsf{P}_{F,J}^{\omega }\frac{\mathbf{x}%
}{\left\vert J\right\vert ^{\frac{1}{n}}}\right\Vert _{L^{2}\left( \omega
\right) }^{2} \\
&=&C_{1}+C_{2}.
\end{eqnarray*}%
We first consider $C_{1}$. Let $\mathcal{M}$ be the maximal dyadic
quasicubes in $\left\{ Q:3Q\cap I=\emptyset \right\} $, and then let $%
\left\{ B_{i}\right\} _{i=1}^{\infty }$ be an enumeration of those $Q\in 
\mathcal{M}$ whose side length is at least $\ell \left( I\right) $. Now we
further decompose the sum in $C_{1}$ by grouping the quasicubes $J$ into the
Whitney quasicubes $B_{i}$, and then using Lemma \ref{tau ovelap} with $%
I_{0}=J$: 
\begin{eqnarray*}
C_{1} &\leq &\sum_{i=1}^{\infty }\sum_{J:\ J\subset B_{i}\setminus I}\left( 
\frac{1}{\left( \left\vert J\right\vert ^{\frac{1}{n}}+\limfunc{quasidist}%
\left( J,I\right) \right) ^{n+1-\alpha }}\left\vert I\right\vert _{\sigma
}\right) ^{2}\sum_{\substack{ F\in \mathcal{F}:  \\ J\in \mathcal{M}_{\left( 
\mathbf{r},\varepsilon \right) -\limfunc{deep}}\left( F\right) }}\left\Vert 
\mathsf{P}_{F,J}^{\omega }\mathbf{x}\right\Vert _{L^{2}\left( \omega \right)
}^{2} \\
&\lesssim &\sum_{i=1}^{\infty }\left( \frac{1}{\left( \left\vert
B_{i}\right\vert ^{\frac{1}{n}}+\limfunc{quasidist}\left( B_{i},I\right)
\right) ^{n+1-\alpha }}\left\vert I\right\vert _{\sigma }\right)
^{2}\sum_{J:\ J\subset B_{i}\setminus I}\sum_{\substack{ F\in \mathcal{F}: 
\\ J\in \mathcal{M}_{\left( \mathbf{r},\varepsilon \right) -\limfunc{deep}%
}\left( F\right) }}\left\Vert \mathsf{P}_{F,J}^{\omega }\mathbf{x}%
\right\Vert _{L^{2}\left( \omega \right) }^{2} \\
&\lesssim &\sum_{i=1}^{\infty }\left( \frac{1}{\left( \left\vert
B_{i}\right\vert ^{\frac{1}{n}}+\limfunc{quasidist}\left( B_{i},I\right)
\right) ^{n+1-\alpha }}\left\vert I\right\vert _{\sigma }\right)
^{2}\sum_{J:\ J\subset B_{i}\setminus I}\mathbf{\tau \;}\left\vert
J\right\vert ^{\frac{2}{n}}\left\vert J\right\vert _{\omega } \\
&\lesssim &\sum_{i=1}^{\infty }\left( \frac{1}{\left( \left\vert
B_{i}\right\vert ^{\frac{1}{n}}+\limfunc{quasidist}\left( B_{i},I\right)
\right) ^{n+1-\alpha }}\left\vert I\right\vert _{\sigma }\right) ^{2}\mathbf{%
\tau \ }\left\vert B_{i}\right\vert ^{\frac{2}{n}}\left\vert B_{i}\setminus
I\right\vert _{\omega } \\
&\lesssim &\mathbf{\tau }\left\{ \sum_{i=1}^{\infty }\frac{\left\vert
B_{i}\setminus I\right\vert _{\omega }\left\vert I\right\vert _{\sigma }}{%
\left\vert B_{i}\right\vert ^{2\left( 1-\frac{\alpha }{n}\right) }}\right\}
\left\vert I\right\vert _{\sigma }\ ,
\end{eqnarray*}%
and since $B_{i}\cap I\neq \emptyset $, 
\begin{eqnarray*}
\sum_{i=1}^{\infty }\frac{\left\vert B_{i}\setminus I\right\vert _{\omega
}\left\vert I\right\vert _{\sigma }}{\left\vert B_{i}\right\vert ^{2\left( 1-%
\frac{\alpha }{n}\right) }} &=&\frac{\left\vert I\right\vert _{\sigma }}{%
\left\vert I\right\vert ^{1-\frac{\alpha }{n}}}\sum_{i=1}^{\infty }\frac{%
\left\vert I\right\vert ^{1-\frac{\alpha }{n}}}{\left\vert B_{i}\right\vert
^{2\left( 1-\frac{\alpha }{n}\right) }}\left\vert B_{i}\setminus
I\right\vert _{\omega } \\
&\approx &\frac{\left\vert I\right\vert _{\sigma }}{\left\vert I\right\vert
^{1-\frac{\alpha }{n}}}\sum_{i=1}^{\infty }\int_{B_{i}\setminus I}\frac{%
\left\vert I\right\vert ^{1-\frac{\alpha }{n}}}{\limfunc{quasidist}\left(
x,I\right) ^{2\left( n-\alpha \right) }}d\omega \left( x\right) \\
&\approx &\frac{\left\vert I\right\vert _{\sigma }}{\left\vert I\right\vert
^{1-\frac{\alpha }{n}}}\sum_{i=1}^{\infty }\int_{B_{i}\setminus I}\left( 
\frac{\left\vert I\right\vert ^{\frac{1}{n}}}{\left[ \left\vert I\right\vert
^{\frac{1}{n}}+\limfunc{quasidist}\left( x,I\right) \right] ^{2}}\right)
^{n-\alpha }d\omega \left( x\right) \\
&\leq &\frac{\left\vert I\right\vert _{\sigma }}{\left\vert I\right\vert ^{1-%
\frac{\alpha }{n}}}\mathcal{P}^{\alpha }\left( I,\mathbf{1}_{I^{c}}\omega
\right) \leq \mathcal{A}_{2}^{\alpha ,\ast },
\end{eqnarray*}%
we obtain%
\begin{equation*}
C_{1}\lesssim \mathbf{\tau }\mathcal{A}_{2}^{\alpha ,\ast }\left\vert
I\right\vert _{\sigma }\ .
\end{equation*}

Next we turn to estimating term $C_{2}$ where the triple of $J$ contains $I$
but $J$ itself does not. Note that there are at most $2^{2n}-2^{n}$ such
quasicubes $J$ of a given side length, at most $2^{n}-1$ in each
`generalized octant' relative to $I$. So with this in mind we sum over the
quasicubes $J$ according to their lengths to obtain%
\begin{eqnarray*}
C_{2} &=&\sum_{m=1}^{\infty }\sum_{\substack{ J:\ I\subset 3J\setminus J  \\ %
\ell \left( J\right) =2^{m}\ell \left( I\right) }}\left( \frac{\left\vert
J\right\vert ^{\frac{1}{n}}}{\left( \left\vert J\right\vert ^{\frac{1}{n}}+%
\limfunc{dist}\left( J,I\right) \right) ^{n+1-\alpha }}\left\vert
I\right\vert _{\sigma }\right) ^{2}\sum_{\substack{ F\in \mathcal{F}:  \\ %
J\in \mathcal{M}_{\left( \mathbf{r},\varepsilon \right) -\limfunc{deep}%
}\left( F\right) }}\left\Vert \mathsf{P}_{F,J}^{\omega }\frac{\mathbf{x}}{%
\left\vert J\right\vert ^{\frac{1}{n}}}\right\Vert _{L^{2}\left( \omega
\right) }^{2} \\
&\lesssim &\sum_{m=1}^{\infty }\left( \frac{\left\vert I\right\vert _{\sigma
}}{\left\vert 2^{m}I\right\vert ^{1-\frac{\alpha }{n}}}\right) ^{2}\mathbf{%
\tau \ }\left\vert \left( 5\cdot 2^{m}I\right) \setminus I\right\vert
_{\omega }=\mathbf{\tau }\left\{ \frac{\left\vert I\right\vert _{\sigma }}{%
\left\vert I\right\vert ^{1-\frac{\alpha }{n}}}\sum_{m=1}^{\infty }\frac{%
\left\vert I\right\vert ^{1-\frac{\alpha }{n}}\left\vert \left( 5\cdot
2^{m}I\right) \setminus I\right\vert _{\omega }}{\left\vert
2^{m}I\right\vert ^{2\left( 1-\frac{\alpha }{n}\right) }}\right\} \left\vert
I\right\vert _{\sigma } \\
&\lesssim &\mathbf{\tau }\left\{ \frac{\left\vert I\right\vert _{\sigma }}{%
\left\vert I\right\vert ^{1-\frac{\alpha }{n}}}\mathcal{P}^{\alpha }\left( I,%
\mathbf{1}_{I^{c}}\omega \right) \right\} \left\vert I\right\vert _{\sigma
}\leq \mathbf{\tau }\mathcal{A}_{2}^{\alpha ,\ast }\left\vert I\right\vert
_{\sigma },
\end{eqnarray*}%
since in analogy with the corresponding estimate above,%
\begin{equation*}
\sum_{m=1}^{\infty }\frac{\left\vert I\right\vert ^{1-\frac{\alpha }{n}%
}\left\vert \left( 5\cdot 2^{m}I\right) \setminus I\right\vert _{\omega }}{%
\left\vert 2^{m}I\right\vert ^{2\left( 1-\frac{\alpha }{n}\right) }}=\int
\sum_{m=1}^{\infty }\frac{\left\vert I\right\vert ^{1-\frac{\alpha }{n}}}{%
\left\vert 2^{m}I\right\vert ^{2\left( 1-\frac{\alpha }{n}\right) }}\mathbf{1%
}_{\left( 5\cdot 2^{m}I\right) \setminus I}\left( x\right) \ d\omega \left(
x\right) \lesssim \mathcal{P}^{\alpha }\left( I,\mathbf{1}_{I^{c}}\omega
\right) .
\end{equation*}

Finally, we turn to term $D$. The quasicubes $J$ occurring here are included
in the set of ancestors $A_{k}\equiv \pi _{\Omega \mathcal{D}}^{\left(
k\right) }I$ of $I$, $1\leq k<\infty $.%
\begin{eqnarray*}
D &=&\sum_{k=1}^{\infty }\mathbb{P}^{\alpha }\left( \mathbf{1}_{I}\sigma
\right) \left( c\left( A_{k}\right) ,\left\vert A_{k}\right\vert ^{\frac{1}{n%
}}\right) ^{2}\sum_{\substack{ F\in \mathcal{F}:  \\ A_{k}\in \mathcal{M}%
_{\left( \mathbf{r},\varepsilon \right) -\limfunc{deep}}\left( F\right) }}%
\left\Vert \mathsf{P}_{F,A_{k}}^{\omega }\frac{\mathbf{x}}{\lvert
A_{k}\rvert ^{\frac{1}{n}}}\right\Vert _{L^{2}\left( \omega \right) }^{2} \\
&=&\sum_{k=1}^{\infty }\mathbb{P}^{\alpha }\left( \mathbf{1}_{I}\sigma
\right) \left( c\left( A_{k}\right) ,\left\vert A_{k}\right\vert ^{\frac{1}{n%
}}\right) ^{2}\sum_{\substack{ F\in \mathcal{F}:  \\ A_{k}\in \mathcal{M}%
_{\left( \mathbf{r},\varepsilon \right) -\limfunc{deep}}\left( F\right) }}%
\sum_{J^{\prime }\in \mathcal{C}_{F}^{\limfunc{good},\mathbf{\tau }-\limfunc{%
shift}}:\ J^{\prime }\subset A_{k}\setminus I}\left\Vert \bigtriangleup
_{J^{\prime }}^{\omega }\frac{\mathbf{x}}{\lvert A_{k}\rvert ^{\frac{1}{n}}}%
\right\Vert _{L^{2}\left( \omega \right) }^{2} \\
&&+\sum_{k=1}^{\infty }\mathbb{P}^{\alpha }\left( \mathbf{1}_{I}\sigma
\right) \left( c\left( A_{k}\right) ,\left\vert A_{k}\right\vert ^{\frac{1}{n%
}}\right) ^{2}\sum_{\substack{ F\in \mathcal{F}:  \\ A_{k}\in \mathcal{M}%
_{\left( \mathbf{r},\varepsilon \right) -\limfunc{deep}}\left( F\right) }}%
\sum_{J^{\prime }\in \mathcal{C}_{F}^{\limfunc{good},\mathbf{\tau }-\limfunc{%
shift}}:\ J^{\prime }\subset I}\left\Vert \bigtriangleup _{J^{\prime
}}^{\omega }\frac{\mathbf{x}}{\lvert A_{k}\rvert ^{\frac{1}{n}}}\right\Vert
_{L^{2}\left( \omega \right) }^{2} \\
&&+\sum_{k=1}^{\infty }\mathbb{P}^{\alpha }\left( \mathbf{1}_{I}\sigma
\right) \left( c\left( A_{k}\right) ,\left\vert A_{k}\right\vert ^{\frac{1}{n%
}}\right) ^{2}\sum_{\substack{ F\in \mathcal{F}:  \\ A_{k}\in \mathcal{M}%
_{\left( \mathbf{r},\varepsilon \right) -\limfunc{deep}}\left( F\right) }}%
\sum_{J^{\prime }\in \mathcal{C}_{F}^{\limfunc{good},\mathbf{\tau }-\limfunc{%
shift}}:\ I\subsetneqq J^{\prime }\subset A_{k}}\left\Vert \bigtriangleup
_{J^{\prime }}^{\omega }\frac{\mathbf{x}}{\lvert A_{k}\rvert ^{\frac{1}{n}}}%
\right\Vert _{L^{2}\left( \omega \right) }^{2} \\
&\equiv &D_{\limfunc{disjoint}}+D_{\limfunc{descendent}}+D_{\limfunc{ancestor%
}}\ .
\end{eqnarray*}%
We thus have from Lemma \ref{tau ovelap} again,%
\begin{eqnarray*}
D_{\limfunc{disjoint}} &=&\sum_{k=1}^{\infty }\mathbb{P}^{\alpha }\left( 
\mathbf{1}_{I}\sigma \right) \left( c\left( A_{k}\right) ,\left\vert
A_{k}\right\vert ^{\frac{1}{n}}\right) ^{2} \\
&&\ \ \ \ \ \ \ \ \ \ \ \ \ \ \ \times \sum_{\substack{ F\in \mathcal{F}: 
\\ A_{k}\in \mathcal{M}_{\left( \mathbf{r},\varepsilon \right) -\limfunc{deep%
}}\left( F\right) }}\sum_{J^{\prime }\in \mathcal{C}_{F}^{\limfunc{good},%
\mathbf{\tau }-\limfunc{shift}}:\ J^{\prime }\subset A_{k}\setminus
I}\left\Vert \bigtriangleup _{J^{\prime }}^{\omega }\frac{\mathbf{x}}{\lvert
A_{k}\rvert ^{\frac{1}{n}}}\right\Vert _{L^{2}\left( \omega \right) }^{2} \\
&\lesssim &\sum_{k=1}^{\infty }\left( \frac{\left\vert I\right\vert _{\sigma
}\left\vert A_{k}\right\vert ^{\frac{1}{n}}}{\left\vert A_{k}\right\vert ^{1+%
\frac{1-\alpha }{n}}}\right) ^{2}\mathbf{\tau \;}\left\vert A_{k}\setminus
I\right\vert _{\omega }=\mathbf{\tau }\left\{ \frac{\left\vert I\right\vert
_{\sigma }}{\left\vert I\right\vert ^{1-\frac{\alpha }{n}}}%
\sum_{k=1}^{\infty }\frac{\left\vert I\right\vert ^{1-\frac{\alpha }{n}}}{%
\left\vert A_{k}\right\vert ^{2\left( 1-\frac{\alpha }{n}\right) }}%
\left\vert A_{k}\setminus I\right\vert _{\omega }\right\} \left\vert
I\right\vert _{\sigma } \\
&\lesssim &\mathbf{\tau }\left\{ \frac{\left\vert I\right\vert _{\sigma }}{%
\left\vert I\right\vert ^{1-\frac{\alpha }{n}}}\mathcal{P}^{\alpha }\left( I,%
\mathbf{1}_{I^{c}}\omega \right) \right\} \left\vert I\right\vert _{\sigma
}\lesssim \mathbf{\tau }\mathcal{A}_{2}^{\alpha ,\ast }\left\vert
I\right\vert _{\sigma },
\end{eqnarray*}%
since%
\begin{eqnarray*}
\sum_{k=1}^{\infty }\frac{\left\vert I\right\vert ^{1-\frac{\alpha }{n}}}{%
\left\vert A_{k}\right\vert ^{2\left( 1-\frac{\alpha }{n}\right) }}%
\left\vert A_{k}\setminus I\right\vert _{\omega } &=&\int \sum_{k=1}^{\infty
}\frac{\left\vert I\right\vert ^{1-\frac{\alpha }{n}}}{\left\vert
A_{k}\right\vert ^{2\left( 1-\frac{\alpha }{n}\right) }}\mathbf{1}%
_{A_{k}\setminus I}\left( x\right) d\omega \left( x\right) \\
&=&\int \sum_{k=1}^{\infty }\frac{1}{2^{2\left( 1-\frac{\alpha }{n}\right) k}%
}\frac{\left\vert I\right\vert ^{1-\frac{\alpha }{n}}}{\left\vert
I\right\vert ^{2\left( 1-\frac{\alpha }{n}\right) }}\mathbf{1}%
_{A_{k}\setminus I}\left( x\right) d\omega \left( x\right) \\
&\lesssim &\int_{I^{c}}\left( \frac{\left\vert I\right\vert ^{\frac{1}{n}}}{%
\left[ \left\vert I\right\vert ^{\frac{1}{n}}+\limfunc{quasidist}\left(
x,I\right) \right] ^{2}}\right) ^{n-\alpha }d\omega \left( x\right) =%
\mathcal{P}^{\alpha }\left( I,\mathbf{1}_{I^{c}}\omega \right) .
\end{eqnarray*}%
The next term $D_{\limfunc{descendent}}$ satisfies%
\begin{eqnarray*}
D_{\limfunc{descendent}} &\lesssim &\sum_{k=1}^{\infty }\left( \frac{%
\left\vert I\right\vert _{\sigma }\left\vert A_{k}\right\vert ^{\frac{1}{n}}%
}{\left\vert A_{k}\right\vert ^{1+\frac{1-\alpha }{n}}}\right) ^{2}\mathbf{%
\tau \;}\left\Vert \mathsf{P}_{I}^{\limfunc{good},\omega }\frac{\mathbf{x}}{%
2^{k}\lvert I\rvert ^{\frac{1}{n}}}\right\Vert _{L^{2}\left( \omega \right)
}^{2} \\
&=&\mathbf{\tau }\sum_{k=1}^{\infty }2^{-2k\left( n-\alpha +1\right) }\left( 
\frac{\left\vert I\right\vert _{\sigma }}{\left\vert I\right\vert ^{1-\frac{%
\alpha }{n}}}\right) ^{2}\left\Vert \mathsf{P}_{I}^{\limfunc{good},\omega }%
\frac{\mathbf{x}}{\lvert I\rvert ^{\frac{1}{n}}}\right\Vert _{L^{2}\left(
\omega \right) }^{2} \\
&\lesssim &\mathbf{\tau }\left\{ \frac{\left\vert I\right\vert _{\sigma
}\left\Vert \mathsf{P}_{I}^{\limfunc{good},\omega }\frac{\mathbf{x}}{\lvert
I\rvert ^{\frac{1}{n}}}\right\Vert _{L^{2}\left( \omega \right) }^{2}}{%
\left\vert I\right\vert ^{2\left( 1-\frac{\alpha }{n}\right) }}\right\}
\left\vert I\right\vert _{\sigma }\lesssim \mathbf{\tau }A_{2}^{\alpha ,%
\limfunc{energy}}\left\vert I\right\vert _{\sigma }\ .
\end{eqnarray*}

Finally for $D_{\limfunc{ancestor}}$ we note that each $J^{\prime }$ is of
the form $J^{\prime }=A_{\ell }\equiv \pi _{\Omega \mathcal{D}}^{\left( \ell
\right) }I$ for some $\ell \geq 1$, and that there are at most $C\mathbf{%
\tau }$ pairs $\left( F,A_{k}\right) $ with $k\geq \ell $ such that $%
A_{k}\in \mathcal{M}_{\left( \mathbf{r},\varepsilon \right) -\limfunc{deep}%
}\left( F\right) $ and $J^{\prime }=A_{\ell }\in \mathcal{C}_{F}^{\limfunc{%
good},\mathbf{\tau }-\limfunc{shift}}$. Now we write%
\begin{eqnarray*}
D_{\limfunc{ancestor}} &=&\sum_{k=1}^{\infty }\mathbb{P}^{\alpha }\left( 
\mathbf{1}_{I}\sigma \right) \left( c\left( A_{k}\right) ,\left\vert
A_{k}\right\vert ^{\frac{1}{n}}\right) ^{2}\sum_{\substack{ F\in \mathcal{F}%
:  \\ A_{k}\in \mathcal{M}_{\left( \mathbf{r},\varepsilon \right) -\limfunc{%
deep}}\left( F\right) }}\sum_{J^{\prime }\in \mathcal{C}_{F}^{\limfunc{good},%
\mathbf{\tau }-\limfunc{shift}}:\ I\subsetneqq J^{\prime }\subset
A_{k}}\left\Vert \bigtriangleup _{J^{\prime }}^{\omega }\frac{\mathbf{x}}{%
\lvert A_{k}\rvert ^{\frac{1}{n}}}\right\Vert _{L^{2}\left( \omega \right)
}^{2} \\
&\lesssim &\mathbf{\tau }\sum_{k=1}^{\infty }\left( \frac{\left\vert
I\right\vert _{\sigma }\left\vert A_{k}\right\vert ^{\frac{1}{n}}}{%
\left\vert A_{k}\right\vert ^{1+\frac{1-\alpha }{n}}}\right) ^{2}\sum_{\ell
=1}^{k}\left\Vert \bigtriangleup _{A_{\ell }}^{\omega }\frac{\mathbf{x}}{%
\lvert A_{k}\rvert ^{\frac{1}{n}}}\right\Vert _{L^{2}\left( \omega \right)
}^{2} \\
&\leq &\mathbf{\tau }\sum_{k=1}^{\infty }\left( \frac{\left\vert
I\right\vert _{\sigma }\left\vert A_{k}\right\vert ^{\frac{1}{n}}}{%
\left\vert A_{k}\right\vert ^{1+\frac{1-\alpha }{n}}}\right) ^{2}\left\Vert 
\mathsf{P}_{A_{k}}^{\limfunc{good},\omega }\frac{\mathbf{x}}{\lvert
A_{k}\rvert ^{\frac{1}{n}}}\right\Vert _{L^{2}\left( \omega \right) }^{2}.
\end{eqnarray*}%
At this point we need a \emph{`prepare to puncture'} argument, as we will
want to derive geometric decay from $\left\Vert \mathsf{P}_{J^{\prime
}}^{\omega }\mathbf{x}\right\Vert _{L^{2}\left( \omega \right) }^{2}$ by
dominating it by the `nonenergy' term $\left\vert J^{\prime }\right\vert ^{%
\frac{2}{n}}\left\vert J^{\prime }\cap I\right\vert _{\omega }$, as well as
using the Muckenhoupt energy constant. For this we define $\widetilde{\omega 
}=\omega -\omega \left( \left\{ p\right\} \right) \delta _{p}$ where $p$ is
an atomic point in $I$ for which 
\begin{equation*}
\omega \left( \left\{ p\right\} \right) =\sup_{q\in \mathfrak{P}_{\left(
\sigma ,\omega \right) }:\ q\in I}\omega \left( \left\{ q\right\} \right) .
\end{equation*}%
(If $\omega $ has no atomic point in common with $\sigma $ in $I$ set $%
\widetilde{\omega }=\omega $.) Then we have $\left\vert I\right\vert _{%
\widetilde{\omega }}=\omega \left( I,\mathfrak{P}_{\left( \sigma ,\omega
\right) }\right) $ and%
\begin{equation*}
\frac{\left\vert I\right\vert _{\widetilde{\omega }}}{\left\vert
I\right\vert ^{\left( 1-\frac{\alpha }{n}\right) }}\frac{\left\vert
I\right\vert _{\sigma }}{\left\vert I\right\vert ^{\left( 1-\frac{\alpha }{n}%
\right) }}=\frac{\omega \left( I,\mathfrak{P}_{\left( \sigma ,\omega \right)
}\right) }{\left\vert I\right\vert ^{\left( 1-\frac{\alpha }{n}\right) }}%
\frac{\left\vert I\right\vert _{\sigma }}{\left\vert I\right\vert ^{\left( 1-%
\frac{\alpha }{n}\right) }}\leq A_{2}^{\alpha ,\limfunc{punct}}.
\end{equation*}%
A key observation, already noted in the proof of Lemma \ref{energy A2}
above, is that%
\begin{equation}
\left\Vert \bigtriangleup _{K}^{\omega }\mathbf{x}\right\Vert _{L^{2}\left(
\omega \right) }^{2}=\left\{ 
\begin{array}{ccc}
\left\Vert \bigtriangleup _{K}^{\omega }\left( \mathbf{x}-\mathbf{p}\right)
\right\Vert _{L^{2}\left( \omega \right) }^{2} & \text{ if } & p\in K \\ 
\left\Vert \bigtriangleup _{K}^{\omega }\mathbf{x}\right\Vert _{L^{2}\left( 
\widetilde{\omega }\right) }^{2} & \text{ if } & p\notin K%
\end{array}%
\right. \leq \ell \left( K\right) ^{2}\left\vert K\right\vert _{\widetilde{%
\omega }},\ \ \ \ \ \text{for all }K\in \Omega \mathcal{D}\ ,
\label{key obs}
\end{equation}%
and so, as in the proof of Lemma \ref{energy A2},%
\begin{equation*}
\left\Vert \mathsf{P}_{A_{k}}^{\limfunc{good},\omega }\frac{\mathbf{x}}{%
\left\vert A_{k}\right\vert ^{\frac{1}{n}}}\right\Vert _{L^{2}\left( \omega
\right) }^{2}\leq 3\left\vert A_{k}\right\vert _{\widetilde{\omega }}\ .
\end{equation*}%
Then we continue with%
\begin{eqnarray*}
&&\mathbf{\tau }\sum_{k=1}^{\infty }\left( \frac{\left\vert I\right\vert
_{\sigma }\left\vert A_{k}\right\vert ^{\frac{1}{n}}}{\left\vert
A_{k}\right\vert ^{1+\frac{1-\alpha }{n}}}\right) ^{2}\left\Vert \mathsf{P}%
_{A_{k}}^{\limfunc{good},\omega }\frac{\mathbf{x}}{\lvert A_{k}\rvert ^{%
\frac{1}{n}}}\right\Vert _{L^{2}\left( \omega \right) }^{2} \\
&\lesssim &\mathbf{\tau }\sum_{k=1}^{\infty }\left( \frac{\left\vert
I\right\vert _{\sigma }\left\vert A_{k}\right\vert ^{\frac{1}{n}}}{%
\left\vert A_{k}\right\vert ^{1+\frac{1-\alpha }{n}}}\right) ^{2}\left\vert
A_{k}\right\vert _{\widetilde{\omega }} \\
&=&\mathbf{\tau }\sum_{k=1}^{\infty }\left( \frac{\left\vert I\right\vert
_{\sigma }}{\left\vert A_{k}\right\vert ^{1-\frac{\alpha }{n}}}\right)
^{2}\left\vert A_{k}\setminus I\right\vert _{\omega }+\mathbf{\tau }%
\sum_{k=1}^{\infty }\left( \frac{\left\vert I\right\vert _{\sigma }}{%
2^{k\left( n-\alpha \right) }\left\vert I\right\vert ^{1-\frac{\alpha }{n}}}%
\right) ^{2}\left\vert I\right\vert _{\widetilde{\omega }} \\
&\lesssim &\mathbf{\tau }\left( \mathcal{A}_{2}^{\alpha ,\ast
}+A_{2}^{\alpha ,\limfunc{punct}}\right) \left\vert I\right\vert _{\sigma },
\end{eqnarray*}%
where the inequality $\sum_{k=1}^{\infty }\left( \frac{\left\vert
I\right\vert _{\sigma }}{\left\vert A_{k}\right\vert ^{1-\frac{\alpha }{n}}}%
\right) ^{2}\left\vert A_{k}\setminus I\right\vert _{\omega }\lesssim 
\mathcal{A}_{2}^{\alpha ,\ast }\left\vert I\right\vert _{\sigma }$ is
already proved above in the estimate for $D_{\limfunc{disjoint}}$.

\subsection{The backward Poisson testing inequality}

Fix $I\in \Omega \mathcal{D}$. It suffices to prove%
\begin{equation}
\mathbf{Back}\left( \widehat{I}\right) \equiv \int_{\mathbb{R}^{n}}\left[ 
\mathbb{Q}^{\alpha }\left( t\mathbf{1}_{\widehat{I}}\overline{\mu }\right)
\left( y\right) \right] ^{2}d\sigma (y)\lesssim \left\{ \mathcal{A}%
_{2}^{\alpha }+\left( \mathcal{E}_{\alpha }^{\limfunc{plug}}+\sqrt{%
A_{2}^{\alpha ,\limfunc{energy}}}\right) \sqrt{A_{2}^{\alpha ,\limfunc{punct}%
}}\right\} \int_{\widehat{I}}t^{2}d\overline{\mu }(x,t).  \label{e.t2 n'}
\end{equation}%
Note that in dimension $n=1$, Hyt\"{o}nen obtained in \cite{Hyt2} the
simpler bound $A_{2}^{\alpha }$ for the term analogous to (\ref{e.t2 n'}).
Here is a brief schematic diagram of the decompositions, with bounds in $%
\fbox{}$, used in this subsection:%
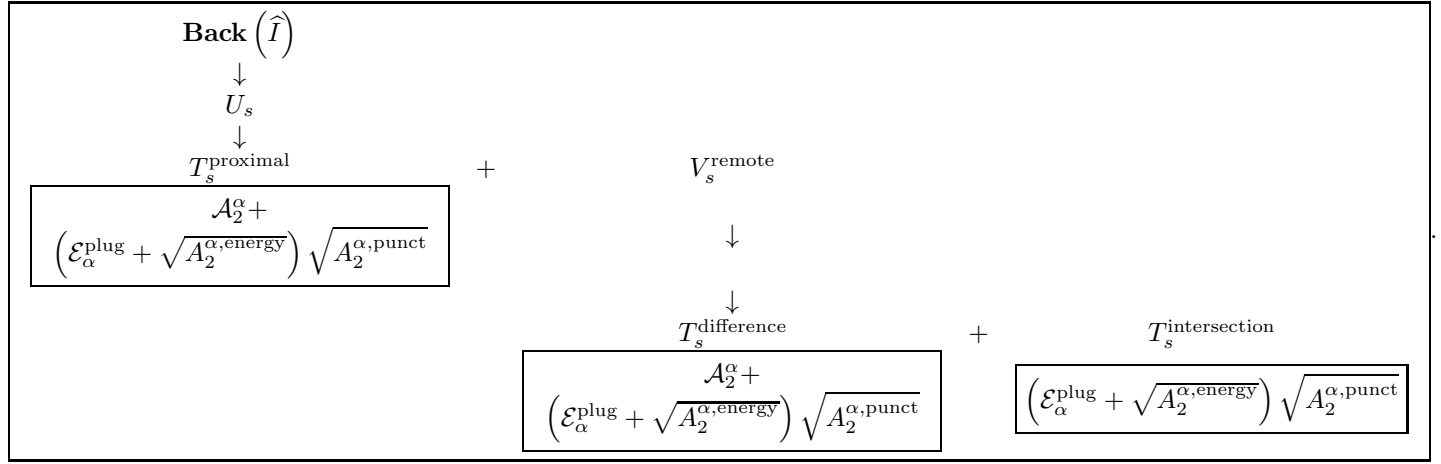
\begin{equation*}
\fbox{$%
\begin{array}{ccccc}
\mathbf{Back}\left( \widehat{I}\right) &  &  &  &  \\ 
\downarrow &  &  &  &  \\ 
U_{s} &  &  &  &  \\ 
\downarrow &  &  &  &  \\ 
T_{s}^{\limfunc{proximal}} & + & V_{s}^{\limfunc{remote}} &  &  \\ 
\fbox{$%
\begin{array}{c}
\mathcal{A}_{2}^{\alpha }+ \\ 
\left( \mathcal{E}_{\alpha }^{\limfunc{plug}}+\sqrt{A_{2}^{\alpha ,\limfunc{%
energy}}}\right) \sqrt{A_{2}^{\alpha ,\limfunc{punct}}}%
\end{array}%
$} &  & \downarrow &  &  \\ 
&  & \downarrow &  &  \\ 
&  & T_{s}^{\limfunc{difference}} & + & T_{s}^{\limfunc{intersection}} \\ 
&  & \fbox{$%
\begin{array}{c}
\mathcal{A}_{2}^{\alpha }+ \\ 
\left( \mathcal{E}_{\alpha }^{\limfunc{plug}}+\sqrt{A_{2}^{\alpha ,\limfunc{%
energy}}}\right) \sqrt{A_{2}^{\alpha ,\limfunc{punct}}}%
\end{array}%
$} &  & \fbox{$\left( \mathcal{E}_{\alpha }^{\limfunc{plug}}+\sqrt{%
A_{2}^{\alpha ,\limfunc{energy}}}\right) \sqrt{A_{2}^{\alpha ,\limfunc{punct}%
}}$}%
\end{array}%
$}.
\end{equation*}%
Using (\ref{tent consequence}) we see that the integral on the right hand
side of (\ref{e.t2 n'}) is 
\begin{equation}
\int_{\widehat{I}}t^{2}d\overline{\mu }=\sum_{F\in \mathcal{F}}\sum_{J\in 
\mathcal{M}_{\left( \mathbf{r},\varepsilon \right) -\limfunc{deep}}\left(
F\right) :\ J\subset I}\lVert \mathsf{P}_{F,J}^{\omega }\mathbf{x}\rVert
_{L^{2}\left( \omega \right) }^{2}\,.  \label{mu I hat}
\end{equation}%
where $\mathsf{P}_{F,J}^{\omega }$ was defined earlier.

We now compute using (\ref{tent consequence}) again that 
\begin{eqnarray}
\mathbb{Q}^{\alpha }\left( t\mathbf{1}_{\widehat{I}}\overline{\mu }\right)
\left( y\right) &=&\int_{\widehat{I}}\frac{t^{2}}{\left( t^{2}+\left\vert
x-y\right\vert ^{2}\right) ^{\frac{n+1-\alpha }{2}}}d\overline{\mu }\left(
x,t\right)  \label{PI hat} \\
&\approx &\sum_{F\in \mathcal{F}}\sum_{\substack{ J\in \mathcal{M}_{\left( 
\mathbf{r},\varepsilon \right) -\limfunc{deep}}\left( F\right)  \\ J\subset
I }}\frac{\left\Vert \mathsf{P}_{F,J}^{\omega }\mathbf{x}\right\Vert
_{L^{2}\left( \omega \right) }^{2}}{\left( \left\vert J\right\vert ^{\frac{1%
}{n}}+\left\vert y-c_{J}\right\vert \right) ^{n+1-\alpha }},  \notag
\end{eqnarray}%
and then expand the square and integrate to obtain that the term $\mathbf{%
Back}\left( \widehat{I}\right) $ is 
\begin{equation*}
\sum_{\substack{ F\in \mathcal{F}  \\ J\in \mathcal{M}_{\left( \mathbf{r}%
,\varepsilon \right) -\limfunc{deep}}\left( F\right)  \\ J\subset I}}\sum 
_{\substack{ F^{\prime }\in \mathcal{F}:  \\ J^{\prime }\in \mathcal{M}%
_{\left( \mathbf{r},\varepsilon \right) -\limfunc{deep}}\left( F^{\prime
}\right)  \\ J^{\prime }\subset I}}\int_{\mathbb{R}^{n}}\frac{\left\Vert 
\mathsf{P}_{F,J}^{\omega }\mathbf{x}\right\Vert _{L^{2}\left( \omega \right)
}^{2}}{\left( \left\vert J\right\vert ^{\frac{1}{n}}+\left\vert
y-c_{J}\right\vert \right) ^{n+1-\alpha }}\frac{\left\Vert \mathsf{P}%
_{F^{\prime },J^{\prime }}^{\omega }\mathbf{x}\right\Vert _{L^{2}\left(
\omega \right) }^{2}}{\left( \left\vert J^{\prime }\right\vert ^{\frac{1}{n}%
}+\left\vert y-c_{J^{\prime }}\right\vert \right) ^{n+1-\alpha }}d\sigma
\left( y\right) .
\end{equation*}

By symmetry we may assume that $\ell \left( J^{\prime }\right) \leq \ell
\left( J\right) $. We fix a nonnegative integer $s$, and consider those
quasicubes $J$ and $J^{\prime }$ with $\ell \left( J^{\prime }\right)
=2^{-s}\ell \left( J\right) $. For fixed $s$ we will control the expression 
\begin{eqnarray*}
U_{s} &\equiv &\sum_{\substack{ F,F^{\prime }\in \mathcal{F}}}\sum 
_{\substack{ J\in \mathcal{M}_{\left( \mathbf{r},\varepsilon \right) -%
\limfunc{deep}}\left( F\right) ,\ J^{\prime }\in \mathcal{M}_{\left( \mathbf{%
r},\varepsilon \right) -\limfunc{deep}}\left( F^{\prime }\right)  \\ %
J,J^{\prime }\subset I,\ \ell \left( J^{\prime }\right) =2^{-s}\ell \left(
J\right) }} \\
&&\times \int_{\mathbb{R}^{n}}\frac{\left\Vert \mathsf{P}_{F,J}^{\omega }%
\mathbf{x}\right\Vert _{L^{2}\left( \omega \right) }^{2}}{\left( \left\vert
J\right\vert ^{\frac{1}{n}}+\left\vert y-c_{J}\right\vert \right)
^{n+1-\alpha }}\frac{\left\Vert \mathsf{P}_{F^{\prime },J^{\prime }}^{\omega
}\mathbf{x}\right\Vert _{L^{2}\left( \omega \right) }^{2}}{\left( \left\vert
J^{\prime }\right\vert ^{\frac{1}{n}}+\left\vert y-c_{J^{\prime
}}\right\vert \right) ^{n+1-\alpha }}d\sigma \left( y\right) ,
\end{eqnarray*}%
by proving that%
\begin{equation}
U_{s}\lesssim 2^{-\delta s}\left\{ \mathcal{A}_{2}^{\alpha }+\left( \mathcal{%
E}_{\alpha }^{\limfunc{plug}}+\sqrt{A_{2}^{\alpha ,\limfunc{energy}}}\right) 
\sqrt{A_{2}^{\alpha ,\limfunc{punct}}}\right\} \int_{\widehat{I}}t^{2}d%
\overline{\mu },\ \ \ \ \ \text{where }\delta =\frac{1}{2n}.
\label{Us bound}
\end{equation}%
With this accomplished, we can sum in $s\geq 0$ to control the term $\mathbf{%
Back}\left( \widehat{I}\right) $. We now decompose $U_{s}=T_{s}^{\limfunc{%
proximal}}+T_{s}^{\limfunc{difference}}+T_{s}^{\limfunc{intersection}}$ into
three pieces.

Our first decomposition is to write%
\begin{equation}
U_{s}=T_{s}^{\limfunc{proximal}}+V_{s}^{\limfunc{remote}}\ ,
\label{initial decomp}
\end{equation}%
where in the `proximal' term $T_{s}^{\limfunc{proximal}}$ we restrict the
summation over pairs of quasicubes $J,J^{\prime }$ to those satisfying $%
\limfunc{qdist}\left( c_{J},c_{J^{\prime }}\right) <2^{s\delta }\ell \left(
J\right) $; while in the `remote' term $V_{s}^{\limfunc{remote}}$ we
restrict the summation over pairs of quasicubes $J,J^{\prime }$ to those
satisfying the opposite inequality $\limfunc{qdist}\left( c_{J},c_{J^{\prime
}}\right) \geq 2^{s\delta }\ell \left( J\right) $. Then we further decompose 
\begin{equation*}
V_{s}^{\limfunc{remote}}=T_{s}^{\limfunc{difference}}+T_{s}^{\limfunc{%
intersection}},
\end{equation*}%
where in the `difference' term $T_{s}^{\limfunc{difference}}$ we restict
integration in $y$ to the difference $\mathbb{R}^{n}\setminus B\left(
J,J^{\prime }\right) $ of $\mathbb{R}^{n}$ and 
\begin{equation*}
B\left( J,J^{\prime }\right) \equiv B\left( c_{J},\frac{1}{2}\limfunc{qdist}%
\left( c_{J},c_{J^{\prime }}\right) \right) ,
\end{equation*}%
the quasiball centered at $c_{J}$ with radius $\frac{1}{2}\limfunc{qdist}%
\left( c_{J},c_{J^{\prime }}\right) $; while in the `intersection' term $%
T_{s}^{\limfunc{intersection}}$ we restict integration in $y$ to the
intersection $\mathbb{R}^{n}\cap B\left( J,J^{\prime }\right) $ of $\mathbb{R%
}^{n}$ with the quasiball $B\left( J,J^{\prime }\right) $; i.e. 
\begin{eqnarray*}
T_{s}^{\limfunc{intersection}} &\equiv &\sum_{\substack{ F,F^{\prime }\in 
\mathcal{F}}}\sum_{\substack{ J\in \mathcal{M}_{\left( \mathbf{r}%
,\varepsilon \right) -\limfunc{deep}}\left( F\right) ,\ J^{\prime }\in 
\mathcal{M}_{\left( \mathbf{r},\varepsilon \right) -\limfunc{deep}}\left(
F^{\prime }\right)  \\ J,J^{\prime }\subset I,\ \ell \left( J^{\prime
}\right) =2^{-s}\ell \left( J\right)  \\ \limfunc{qdist}\left(
c_{J},c_{J^{\prime }}\right) \geq 2^{s\left( 1+\delta \right) }\ell \left(
J^{\prime }\right) }} \\
&&\times \int_{B\left( J,J^{\prime }\right) }\frac{\left\Vert \mathsf{P}%
_{F,J}^{\omega }\mathbf{x}\right\Vert _{L^{2}\left( \omega \right) }^{2}}{%
\left( \left\vert J\right\vert ^{\frac{1}{n}}+\left\vert y-c_{J}\right\vert
\right) ^{n+1-\alpha }}\frac{\left\Vert \mathsf{P}_{F^{\prime },J^{\prime
}}^{\omega }\mathbf{x}\right\Vert _{L^{2}\left( \omega \right) }^{2}}{\left(
\left\vert J^{\prime }\right\vert ^{\frac{1}{n}}+\left\vert y-c_{J^{\prime
}}\right\vert \right) ^{n+1-\alpha }}d\sigma \left( y\right) .
\end{eqnarray*}%
Here is a schematic reminder of the these decompositions with the
distinguishing points of the definitions boxed:{}

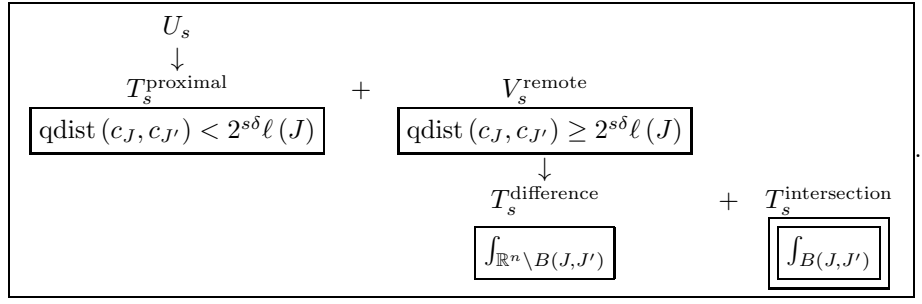
\begin{equation*}
\fbox{$%
\begin{array}{ccccc}
U_{s} &  &  &  &  \\ 
\downarrow &  &  &  &  \\ 
T_{s}^{\limfunc{proximal}} & + & V_{s}^{\limfunc{remote}} &  &  \\ 
\fbox{$\limfunc{qdist}\left( c_{J},c_{J^{\prime }}\right) <2^{s\delta }\ell
\left( J\right) $} &  & \fbox{$\limfunc{qdist}\left( c_{J},c_{J^{\prime
}}\right) \geq 2^{s\delta }\ell \left( J\right) $} &  &  \\ 
&  & \downarrow &  &  \\ 
&  & T_{s}^{\limfunc{difference}} & + & T_{s}^{\limfunc{intersection}} \\ 
&  & \fbox{$\int_{\mathbb{R}^{n}\setminus B\left( J,J^{\prime }\right) }$} & 
& \fbox{$\fbox{$\int_{B\left( J,J^{\prime }\right) }$}$}%
\end{array}%
$}.
\end{equation*}

We will exploit the restriction of integration to $B\left( J,J^{\prime
}\right) $, together with the condition 
\begin{equation*}
\limfunc{qdist}\left( c_{J},c_{J^{\prime }}\right) \geq 2^{s\left( 1+\delta
\right) }\ell \left( J^{\prime }\right) =2^{s\delta }\ell \left( J\right) ,
\end{equation*}%
in establishing (\ref{important}) below, which will then give an estimate
for the term $T_{s}^{\limfunc{intersection}}$ using an argument dual to that
used for the other terms $T_{s}^{\limfunc{proximal}}$ and $T_{s}^{\limfunc{%
difference}}$, to which we now turn.

\subsubsection{The proximal and difference terms}

We have%
\begin{align*}
T_{s}^{\limfunc{proximal}}& \equiv \sum_{\substack{ F,F^{\prime }\in 
\mathcal{F}}}\sum_{\substack{ J\in \mathcal{M}_{\left( \mathbf{r}%
,\varepsilon \right) -\limfunc{deep}}\left( F\right) ,\ J^{\prime }\in 
\mathcal{M}_{\left( \mathbf{r},\varepsilon \right) -\limfunc{deep}}\left(
F^{\prime }\right)  \\ J,J^{\prime }\subset I,\ \ell \left( J^{\prime
}\right) =2^{-s}\ell \left( J\right) \text{ and }\limfunc{qdist}\left(
c_{J},c_{J^{\prime }}\right) <2^{s\delta }\ell \left( J\right) }} \\
& \times \int_{\mathbb{R}^{n}}\frac{\left\Vert \mathsf{P}_{F,J}^{\omega }%
\mathbf{x}\right\Vert _{L^{2}\left( \omega \right) }^{2}}{\left( \left\vert
J\right\vert ^{\frac{1}{n}}+\left\vert y-c_{J}\right\vert \right)
^{n+1-\alpha }}\frac{\left\Vert \mathsf{P}_{F^{\prime },J^{\prime }}^{\omega
}\mathbf{x}\right\Vert _{L^{2}\left( \omega \right) }^{2}}{\left( \left\vert
J^{\prime }\right\vert ^{\frac{1}{n}}+\left\vert y-c_{J^{\prime
}}\right\vert \right) ^{n+1-\alpha }}d\sigma \left( y\right) \\
& \leq M_{s}^{\limfunc{proximal}}\sum_{F\in \mathcal{F}}\sum_{\substack{ 
\mathcal{M}_{\left( \mathbf{r},\varepsilon \right) -\limfunc{deep}}\left(
F\right)  \\ J\subset I}}\lVert \mathsf{P}_{F,J}^{\omega }\mathbf{z}\rVert
_{\omega }^{2}=M_{s}^{\limfunc{proximal}}\int_{\widehat{I}}t^{2}d\overline{%
\mu },
\end{align*}%
where%
\begin{align*}
M_{s}^{\limfunc{proximal}}& \equiv \sup_{F\in \mathcal{F}}\sup_{\substack{ %
J\in \mathcal{M}_{\left( \mathbf{r},\varepsilon \right) -\limfunc{deep}%
}\left( F\right) }}A_{s}^{\limfunc{proximal}}\left( J\right) ; \\
A_{s}^{\limfunc{proximal}}\left( J\right) & \equiv \sum_{F^{\prime }\in 
\mathcal{F}}\sum_{\substack{ J^{\prime }\in \mathcal{M}_{\left( \mathbf{r}%
,\varepsilon \right) -\limfunc{deep}}\left( F^{\prime }\right)  \\ J^{\prime
}\subset I,\ \ell \left( J^{\prime }\right) =2^{-s}\ell \left( J\right) 
\text{ and }\limfunc{qdist}\left( c_{J},c_{J^{\prime }}\right) <2^{s\delta
}\ell \left( J\right) }}\int_{\mathbb{R}^{n}}S_{\left( J^{\prime },J\right)
}^{F^{\prime }}\left( y\right) d\sigma \left( y\right) ; \\
S_{\left( J^{\prime },J\right) }^{F^{\prime }}\left( x\right) & \equiv \frac{%
1}{\left( \left\vert J\right\vert ^{\frac{1}{n}}+\left\vert
y-c_{J}\right\vert \right) ^{n+1-\alpha }}\frac{\left\Vert \mathsf{P}%
_{F^{\prime },J^{\prime }}^{\omega }\mathbf{x}\right\Vert _{L^{2}\left(
\omega \right) }^{2}}{\left( \left\vert J^{\prime }\right\vert ^{\frac{1}{n}%
}+\left\vert y-c_{J^{\prime }}\right\vert \right) ^{n+1-\alpha }},
\end{align*}%
and similarly%
\begin{align*}
T_{s}^{\limfunc{difference}}& \equiv \sum_{\substack{ F,F^{\prime }\in 
\mathcal{F}}}\sum_{\substack{ J\in \mathcal{M}_{\left( \mathbf{r}%
,\varepsilon \right) -\limfunc{deep}}\left( F\right) ,\ J^{\prime }\in 
\mathcal{M}_{\left( \mathbf{r},\varepsilon \right) -\limfunc{deep}}\left(
F^{\prime }\right)  \\ J,J^{\prime }\subset I,\ \ell \left( J^{\prime
}\right) =2^{-s}\ell \left( J\right) \text{ and }\limfunc{qdist}\left(
c_{J},c_{J^{\prime }}\right) \geq 2^{s\delta }\ell \left( J\right) }} \\
& \times \int_{\mathbb{R}^{n}\setminus B\left( J,J^{\prime }\right) }\frac{%
\left\Vert \mathsf{P}_{F,J}^{\omega }\mathbf{x}\right\Vert _{L^{2}\left(
\omega \right) }^{2}}{\left( \left\vert J\right\vert ^{\frac{1}{n}%
}+\left\vert y-c_{J}\right\vert \right) ^{n+1-\alpha }}\frac{\left\Vert 
\mathsf{P}_{F^{\prime },J^{\prime }}^{\omega }\mathbf{x}\right\Vert
_{L^{2}\left( \omega \right) }^{2}}{\left( \left\vert J^{\prime }\right\vert
^{\frac{1}{n}}+\left\vert y-c_{J^{\prime }}\right\vert \right) ^{n+1-\alpha }%
}d\sigma \left( y\right) \\
& \leq M_{s}^{\limfunc{difference}}\sum_{F\in \mathcal{F}}\sum_{\substack{ 
\mathcal{M}_{\left( \mathbf{r},\varepsilon \right) -\limfunc{deep}}\left(
F\right)  \\ J\subset I}}\lVert \mathsf{P}_{F,J}^{\omega }\mathbf{z}\rVert
_{\omega }^{2}=M_{s}^{\limfunc{difference}}\int_{\widehat{I}}t^{2}d\overline{%
\mu };
\end{align*}%
where%
\begin{eqnarray*}
M_{s}^{\limfunc{difference}} &\equiv &\sup_{F\in \mathcal{F}}\sup_{\substack{
J\in \mathcal{M}_{\left( \mathbf{r},\varepsilon \right) -\limfunc{deep}%
}\left( F\right) }}A_{s}^{\func{difference}}\left( J\right) ; \\
A_{s}^{\limfunc{difference}}\left( J\right) &\equiv &\sum_{F^{\prime }\in 
\mathcal{F}}\sum_{\substack{ J^{\prime }\in \mathcal{M}_{\left( \mathbf{r}%
,\varepsilon \right) -\limfunc{deep}}\left( F^{\prime }\right)  \\ J^{\prime
}\subset I,\ \ell \left( J^{\prime }\right) =2^{-s}\ell \left( J\right) 
\text{ and }\limfunc{qdist}\left( c_{J},c_{J^{\prime }}\right) \geq
2^{s\delta }\ell \left( J\right) }}\int_{\mathbb{R}^{n}\setminus B\left(
J,J^{\prime }\right) }S_{\left( J^{\prime },J\right) }^{F^{\prime }}\left(
y\right) d\sigma \left( y\right) .
\end{eqnarray*}%
The restriction of integration in $A_{s}^{\limfunc{difference}}$ to $\mathbb{%
R}^{n}\setminus B\left( J,J^{\prime }\right) $ will be used to establish (%
\ref{vanishing close}) below.

\begin{notation}
Since the quasicubes $F,J,F^{\prime },J^{\prime }$ that arise in all of the
sums here satisfy (recall (\ref{tent consequence})) 
\begin{equation*}
J\in \mathcal{M}_{\left( \mathbf{r},\varepsilon \right) -\limfunc{deep}%
}\left( F\right) ,\ J^{\prime }\in \mathcal{M}_{\left( \mathbf{r}%
,\varepsilon \right) -\limfunc{deep}}\left( F^{\prime }\right) \text{ and }%
\ell \left( J^{\prime }\right) =2^{-s}\ell \left( J\right) \text{ and }%
J,J^{\prime }\subset I,
\end{equation*}%
we will often employ the notation $\overset{\ast }{\sum }$ to remind the
reader that, as applicable, these four conditions are in force even when
they are\ not explictly mentioned.
\end{notation}

Now fix $J$ as in $M_{s}^{\limfunc{proximal}}$ respectively $M_{s}^{\limfunc{%
difference}}$, and decompose the sum over $J^{\prime }$ in $A_{s}^{\limfunc{%
proximal}}\left( J\right) $ respectively $A_{s}^{\limfunc{difference}}\left(
J\right) $ by%
\begin{eqnarray*}
&&A_{s}^{\limfunc{proximal}}\left( J\right) =\sum_{F^{\prime }\in \mathcal{F}%
}\sum_{\substack{ J^{\prime }\in \mathcal{M}_{\left( \mathbf{r},\varepsilon
\right) -\limfunc{deep}}\left( F^{\prime }\right)  \\ J^{\prime }\subset I,\
\ell \left( J^{\prime }\right) =2^{-s}\ell \left( J\right) \text{ and }%
\limfunc{qdist}\left( c_{J},c_{J^{\prime }}\right) <2^{s\delta }\ell \left(
J\right) }}\int_{\mathbb{R}^{n}}S_{\left( J^{\prime },J\right) }^{F^{\prime
}}\left( y\right) d\sigma \left( y\right)  \\
&=&\sum_{F^{\prime }\in \mathcal{F}}\overset{\ast }{\sum_{\substack{ %
c_{J^{\prime }}\in 2J \\ \limfunc{qdist}\left( c_{J},c_{J^{\prime }}\right)
<2^{s\delta }\ell \left( J\right) }}}\int_{\mathbb{R}^{n}}S_{\left(
J^{\prime },J\right) }^{F^{\prime }}\left( y\right) d\sigma \left( y\right)
+\sum_{F^{\prime }\in \mathcal{F}}\sum_{\ell =1}^{\infty }\overset{\ast }{%
\sum_{\substack{ c_{J^{\prime }}\in 2^{\ell +1}J\setminus 2^{\ell }J \\ 
\limfunc{qdist}\left( c_{J},c_{J^{\prime }}\right) <2^{s\delta }\ell \left(
J\right) }}}\int_{\mathbb{R}^{n}}S_{\left( J^{\prime },J\right) }^{F^{\prime
}}\left( y\right) d\sigma \left( y\right)  \\
&\equiv &\sum_{\ell =0}^{\infty }A_{s}^{\limfunc{proximal},\ell }\left(
J\right) ,
\end{eqnarray*}%
respectively%
\begin{eqnarray*}
&&A_{s}^{\limfunc{difference}}\left( J\right) =\sum_{F^{\prime }\in \mathcal{%
F}}\sum_{\substack{ J^{\prime }\in \mathcal{M}_{\left( \mathbf{r}%
,\varepsilon \right) -\limfunc{deep}}\left( F^{\prime }\right)  \\ J^{\prime
}\subset I,\ \ell \left( J^{\prime }\right) =2^{-s}\ell \left( J\right) 
\text{ and }\limfunc{qdist}\left( c_{J},c_{J^{\prime }}\right) \geq
2^{s\delta }\ell \left( J\right) }}\int_{\mathbb{R}^{n}\setminus B\left(
J,J^{\prime }\right) }S_{\left( J^{\prime },J\right) }^{F^{\prime }}\left(
y\right) d\sigma \left( y\right)  \\
&=&\sum_{F^{\prime }\in \mathcal{F}}\overset{\ast }{\sum_{\substack{ %
c_{J^{\prime }}\in 2J \\ \limfunc{qdist}\left( c_{J},c_{J^{\prime }}\right)
\geq 2^{s\delta }\ell \left( J\right) }}}\int_{\mathbb{R}^{n}\setminus
B\left( J,J^{\prime }\right) }S_{\left( J^{\prime },J\right) }^{F^{\prime
}}\left( y\right) d\sigma \left( y\right)  \\
&&+\sum_{\ell =1}^{\infty }\sum_{F^{\prime }\in \mathcal{F}}\overset{\ast }{%
\sum_{\substack{ c_{J^{\prime }}\in 2^{\ell +1}J\setminus 2^{\ell }J \\ 
\limfunc{qdist}\left( c_{J},c_{J^{\prime }}\right) \geq 2^{s\delta }\ell
\left( J\right) }}}\int_{\mathbb{R}^{n}\setminus B\left( J,J^{\prime
}\right) }S_{\left( J^{\prime },J\right) }^{F^{\prime }}\left( y\right)
d\sigma \left( y\right)  \\
&\equiv &\sum_{\ell =0}^{\infty }A_{s}^{\limfunc{difference},\ell }\left(
J\right) .
\end{eqnarray*}%
Let $m$ be the smallest integer for which 
\begin{equation}
2^{-m}\sqrt{n}\leq \frac{1}{3}.  \label{smallest m}
\end{equation}%
Now decompose the integrals over $I$ in $A_{s}^{\limfunc{proximal},\ell
}\left( J\right) $ by%
\begin{eqnarray*}
A_{s}^{\limfunc{proximal},0}\left( J\right)  &=&\sum_{F^{\prime }\in 
\mathcal{F}}\overset{\ast }{\sum_{\substack{ c_{J^{\prime }}\in 2J \\ 
\limfunc{qdist}\left( c_{J},c_{J^{\prime }}\right) <2^{s\delta }\ell \left(
J\right) }}}\int_{\mathbb{R}^{n}\setminus 4J}S_{\left( J^{\prime },J\right)
}^{F^{\prime }}\left( y\right) d\sigma \left( y\right)  \\
&&+\sum_{F^{\prime }\in \mathcal{F}}\overset{\ast }{\sum_{\substack{ %
c_{J^{\prime }}\in 2J \\ \limfunc{qdist}\left( c_{J},c_{J^{\prime }}\right)
<2^{s\delta }\ell \left( J\right) }}}\int_{4J}S_{\left( J^{\prime },J\right)
}^{F^{\prime }}\left( y\right) d\sigma \left( y\right)  \\
&\equiv &A_{s,far}^{\limfunc{proximal},0}\left( J\right) +A_{s,near}^{%
\limfunc{proximal},0}\left( J\right) , \\
A_{s}^{\limfunc{proximal},\ell }\left( J\right)  &=&\sum_{F^{\prime }\in 
\mathcal{F}}\overset{\ast }{\sum_{\substack{ c_{J^{\prime }}\in 2^{\ell
+1}J\setminus 2^{\ell }J \\ \limfunc{qdist}\left( c_{J},c_{J^{\prime
}}\right) <2^{s\delta }\ell \left( J\right) }}}\int_{\mathbb{R}^{n}\setminus
2^{\ell +2}J}S_{\left( J^{\prime },J\right) }^{F^{\prime }}\left( y\right)
d\sigma \left( y\right)  \\
&&+\sum_{F^{\prime }\in \mathcal{F}}\overset{\ast }{\sum_{\substack{ %
c_{J^{\prime }}\in 2^{\ell +1}J\setminus 2^{\ell }J \\ \limfunc{qdist}\left(
c_{J},c_{J^{\prime }}\right) <2^{s\delta }\ell \left( J\right) }}}%
\int_{2^{\ell +2}J\setminus 2^{\ell -m}J}S_{\left( J^{\prime },J\right)
}^{F^{\prime }}\left( y\right) d\sigma \left( y\right)  \\
&&+\sum_{F^{\prime }\in \mathcal{F}}\overset{\ast }{\sum_{\substack{ %
c_{J^{\prime }}\in 2^{\ell +1}J\setminus 2^{\ell }J \\ \limfunc{qdist}\left(
c_{J},c_{J^{\prime }}\right) <2^{s\delta }\ell \left( J\right) }}}%
\int_{2^{\ell -m}J}S_{\left( J^{\prime },J\right) }^{F^{\prime }}\left(
y\right) d\sigma \left( y\right)  \\
&\equiv &A_{s,far}^{\limfunc{proximal},\ell }\left( J\right) +A_{s,near}^{%
\limfunc{proximal},\ell }\left( J\right) +A_{s,close}^{\limfunc{proximal}%
,\ell }\left( J\right) ,\ \ \ \ \ \ell \geq 1.
\end{eqnarray*}%
Similarly we decompose the integrals over the difference 
\begin{equation*}
I^{\ast }\equiv \mathbb{R}^{n}\setminus B\left( J,J^{\prime }\right) 
\end{equation*}%
in $A_{s}^{\limfunc{difference},\ell }\left( J\right) $ by%
\begin{eqnarray*}
A_{s}^{\limfunc{difference},0}\left( J\right)  &=&\sum_{F^{\prime }\in 
\mathcal{F}}\overset{\ast }{\sum_{\substack{ c_{J^{\prime }}\in 2J \\ 
\limfunc{qdist}\left( c_{J},c_{J^{\prime }}\right) \geq 2^{s\delta }\ell
\left( J\right) }}}\int_{I^{\ast }\setminus 4J}S_{\left( J^{\prime
},J\right) }^{F^{\prime }}\left( y\right) d\sigma \left( y\right)  \\
&&+\sum_{F^{\prime }\in \mathcal{F}}\overset{\ast }{\sum_{\substack{ %
c_{J^{\prime }}\in 2J \\ \limfunc{qdist}\left( c_{J},c_{J^{\prime }}\right)
\geq 2^{s\delta }\ell \left( J\right) }}}\int_{I^{\ast }\cap 4J}S_{\left(
J^{\prime },J\right) }^{F^{\prime }}\left( y\right) d\sigma \left( y\right) 
\\
&\equiv &A_{s,far}^{\limfunc{difference},0}\left( J\right) +A_{s,near}^{%
\limfunc{difference},0}\left( J\right) , \\
A_{s}^{\limfunc{difference},\ell }\left( J\right)  &=&\sum_{F^{\prime }\in 
\mathcal{F}}\overset{\ast }{\sum_{\substack{ c_{J^{\prime }}\in 2^{\ell
+1}J\setminus 2^{\ell }J \\ \limfunc{qdist}\left( c_{J},c_{J^{\prime
}}\right) \geq 2^{s\delta }\ell \left( J\right) }}}\int_{I^{\ast }\setminus
2^{\ell +2}J}S_{\left( J^{\prime },J\right) }^{F^{\prime }}\left( y\right)
d\sigma \left( y\right)  \\
&&+\sum_{F^{\prime }\in \mathcal{F}}\overset{\ast }{\sum_{\substack{ %
c_{J^{\prime }}\in 2^{\ell +1}J\setminus 2^{\ell }J \\ \limfunc{qdist}\left(
c_{J},c_{J^{\prime }}\right) \geq 2^{s\delta }\ell \left( J\right) }}}%
\int_{I^{\ast }\cap \left( 2^{\ell +2}J\setminus 2^{\ell -m}J\right)
}S_{\left( J^{\prime },J\right) }^{F^{\prime }}\left( y\right) d\sigma
\left( y\right)  \\
&&+\sum_{F^{\prime }\in \mathcal{F}}\overset{\ast }{\sum_{\substack{ %
c_{J^{\prime }}\in 2^{\ell +1}J\setminus 2^{\ell }J \\ \limfunc{qdist}\left(
c_{J},c_{J^{\prime }}\right) \geq 2^{s\delta }\ell \left( J\right) }}}%
\int_{I^{\ast }\cap 2^{\ell -m}J}S_{\left( J^{\prime },J\right) }^{F^{\prime
}}\left( y\right) d\sigma \left( y\right)  \\
&\equiv &A_{s,far}^{\limfunc{difference},\ell }\left( J\right) +A_{s,near}^{%
\limfunc{difference},\ell }\left( J\right) +A_{s,close}^{\limfunc{difference}%
,\ell }\left( J\right) ,\ \ \ \ \ \ell \geq 1.
\end{eqnarray*}

We now note the important point that the close terms $A_{s,close}^{\limfunc{%
proximal},\ell }\left( J\right) $ and $A_{s,close}^{\limfunc{difference}%
,\ell }\left( J\right) $ both $\emph{vanish}$ for $\ell >\delta s$ because
of the decomposition (\ref{initial decomp}):%
\begin{equation}
A_{s,close}^{\limfunc{proximal},\ell }\left( J\right) =A_{s,close}^{\limfunc{%
difference},\ell }\left( J\right) =0,\ \ \ \ \ \ell \geq 1+\delta s.
\label{vanishing close}
\end{equation}%
Indeed, if $c_{J^{\prime }}\in 2^{\ell +1}J\setminus 2^{\ell }J$, then we
have%
\begin{equation}
\frac{1}{2}2^{\ell }\ell \left( J\right) \leq \limfunc{qdist}\left(
c_{J},c_{J^{\prime }}\right) ,  \label{distJJ'}
\end{equation}%
and if $\ell \geq 1+\delta s$, then%
\begin{equation*}
\limfunc{qdist}\left( c_{J},c_{J^{\prime }}\right) \geq 2^{\delta s}\ell
\left( J\right) =2^{\left( 1+\delta \right) s}\ell \left( J^{\prime }\right)
.
\end{equation*}%
It now follows from the definition of $V_{s}^{\limfunc{remote}}$ and $T_{s}^{%
\limfunc{proximal}}$ in (\ref{initial decomp}), that $A_{s,close}^{\limfunc{%
proximal},\ell }\left( J\right) =0$, and so we are left to consider the term 
$A_{s,close}^{\limfunc{difference},\ell }\left( J\right) $, where the
integration is taken over the set $\mathbb{R}^{n}\setminus B\left(
J,J^{\prime }\right) $. But we are also restricted in $A_{s,close}^{\limfunc{%
difference},\ell }\left( J\right) $ to integrating over the quasicube $%
2^{\ell -m}J$, which is contained in $B\left( J,J^{\prime }\right) $ by (\ref%
{distJJ'}). Indeed, the smallest\ ball centered at $c_{J}$ that contains $%
2^{\ell -m}J$ has radius $\sqrt{n}\frac{1}{2}2^{\ell -m}\ell \left( J\right) 
$, which by (\ref{smallest m}) and (\ref{distJJ'}) is at most $\frac{1}{4}%
2^{\ell }\ell \left( J\right) \leq \frac{1}{2}\limfunc{qdist}\left(
c_{J},c_{J^{\prime }}\right) $, the radius of $B\left( J,J^{\prime }\right) $%
. Thus the range of integration in the term $A_{s,close}^{\limfunc{difference%
},\ell }\left( J\right) $ is the empty set, and so $A_{s,close}^{\limfunc{%
difference},\ell }\left( J\right) =0$ as well as $A_{s,close}^{\limfunc{%
proximal},\ell }\left( J\right) =0$. This proves (\ref{vanishing close}).

From now on we treat $T_{s}^{\limfunc{proximal}}$ and $T_{s}^{\limfunc{%
difference}}$ in the same way since the terms $A_{s,close}^{\limfunc{proximal%
},\ell }\left( J\right) $ and $A_{s,close}^{\limfunc{difference},\ell
}\left( J\right) $ both vanish for $\ell \geq 1+\delta s$. Thus we will
suppress the superscripts $\limfunc{proximal}$ and $\limfunc{difference}$ in
the $far$, $near$ and $close$ decomposition of $A_{s}^{\limfunc{proximal}%
,\ell }\left( J\right) $ and $A_{s}^{\limfunc{difference},\ell }\left(
J\right) $, and we will also suppress the conditions $\limfunc{qdist}\left(
c_{J},c_{J^{\prime }}\right) <2^{s\delta }\ell \left( J\right) $ and $%
\limfunc{qdist}\left( c_{J},c_{J^{\prime }}\right) \geq 2^{s\delta }\ell
\left( J\right) $ in the proximal and difference terms since they no longer
play a role. Using the bounded overlap of the shifted coronas $\mathcal{C}%
_{F}^{\limfunc{good},\mathbf{\tau }-\limfunc{shift}}$, we have 
\begin{equation*}
\sum_{F^{\prime }\in \mathcal{F}}\left\Vert \mathsf{P}_{F^{\prime
},J^{\prime }}^{\omega }\mathbf{x}\right\Vert _{L^{2}\left( \omega \right)
}^{2}\lesssim \mathbf{\tau }\left\vert J^{\prime }\right\vert ^{\frac{2}{n}%
}\left\vert J^{\prime }\right\vert _{\omega }\ ,
\end{equation*}%
and so we have%
\begin{eqnarray*}
A_{s,far}^{0}\left( J\right) &\leq &\sum_{F^{\prime }\in \mathcal{F}}\overset%
{\ast }{\sum_{c_{J^{\prime }}\in 2J}}\int_{\mathbb{R}^{n}\setminus
4J}S_{\left( J^{\prime },J\right) }^{F^{\prime }}\left( y\right) d\sigma
\left( y\right) \\
&\lesssim &\mathbf{\tau }\sum_{c_{J^{\prime }}\in 2J}\int_{\mathbb{R}%
^{n}\setminus 4J}\frac{\left\vert J^{\prime }\right\vert ^{\frac{2}{n}%
}\left\vert J^{\prime }\right\vert _{\omega }}{\left( \left\vert
J\right\vert ^{\frac{1}{n}}+\left\vert y-c_{J}\right\vert \right) ^{2\left(
n+1-\alpha \right) }}d\sigma \left( y\right) \\
&=&\mathbf{\tau }2^{-2s}\left( \sum_{c_{J^{\prime }}\in 2J}\left\vert
J^{\prime }\right\vert _{\omega }\right) \int_{\mathbb{R}^{n}\setminus 4J}%
\frac{\left\vert J\right\vert ^{\frac{2}{n}}}{\left( \left\vert J\right\vert
^{\frac{1}{n}}+\left\vert y-c_{J}\right\vert \right) ^{2\left( n+1-\alpha
\right) }}d\sigma \left( y\right) ,
\end{eqnarray*}%
which is dominated by%
\begin{eqnarray*}
&&\mathbf{\tau }2^{-2s}\left\vert 4J\right\vert _{\omega }\int_{\mathbb{R}%
^{n}\setminus 4J}\frac{1}{\left( \left\vert J\right\vert ^{\frac{1}{n}%
}+\left\vert y-c_{J}\right\vert \right) ^{2\left( n-\alpha \right) }}d\sigma
\left( y\right) \\
&\approx &\mathbf{\tau }2^{-2s}\frac{\left\vert 4J\right\vert _{\omega }}{%
\left\vert 4J\right\vert ^{1-\frac{\alpha }{n}}}\int_{\mathbb{R}%
^{n}\setminus 4J}\left( \frac{\left\vert J\right\vert ^{\frac{1}{n}}}{\left(
\left\vert J\right\vert ^{\frac{1}{n}}+\left\vert y-c_{J}\right\vert \right)
^{2}}\right) ^{n-\alpha }d\sigma \left( y\right) \\
&\lesssim &\mathbf{\tau }2^{-2s}\frac{\left\vert 4J\right\vert _{\omega }}{%
\left\vert 4J\right\vert ^{1-\frac{\alpha }{n}}}\mathcal{P}^{\alpha }\left(
4J,\mathbf{1}_{\mathbb{R}^{n}\setminus 4J}\sigma \right) \lesssim \mathbf{%
\tau }2^{-2s}\mathcal{A}_{2}^{\alpha }\ .
\end{eqnarray*}

To estimate the near term $A_{s,near}^{0}\left( J\right) $, we initially
keep the energy $\left\Vert \mathsf{P}_{F^{\prime },J^{\prime }}^{\omega }%
\mathbf{z}\right\Vert _{L^{2}\left( \omega \right) }^{2}$ and write 
\begin{eqnarray}
A_{s,near}^{0}\left( J\right) &\leq &\sum_{F^{\prime }\in \mathcal{F}}%
\overset{\ast }{\sum_{c_{J^{\prime }}\in 2J}}\int_{4J}S_{\left( J^{\prime
},J\right) }^{F^{\prime }}\left( y\right) d\sigma \left( y\right)
\label{A0snear} \\
&\approx &\sum_{F^{\prime }\in \mathcal{F}}\overset{\ast }{%
\sum_{c_{J^{\prime }}\in 2J}}\int_{4J}\frac{1}{\left\vert J\right\vert ^{%
\frac{1}{n}\left( n+1-\alpha \right) }}\frac{\left\Vert \mathsf{P}%
_{F^{\prime },J^{\prime }}^{\omega }\mathbf{x}\right\Vert _{L^{2}\left(
\omega \right) }^{2}}{\left( \left\vert J^{\prime }\right\vert ^{\frac{1}{n}%
}+\left\vert y-c_{J^{\prime }}\right\vert \right) ^{n+1-\alpha }}d\sigma
\left( y\right)  \notag \\
&=&\sum_{F^{\prime }\in \mathcal{F}}\frac{1}{\left\vert J\right\vert ^{\frac{%
1}{n}\left( n+1-\alpha \right) }}\overset{\ast }{\sum_{c_{J^{\prime }}\in 2J}%
}\left\Vert \mathsf{P}_{F^{\prime },J^{\prime }}^{\omega }\mathbf{x}%
\right\Vert _{L^{2}\left( \omega \right) }^{2}\int_{4J}\frac{1}{\left(
\left\vert J^{\prime }\right\vert ^{\frac{1}{n}}+\left\vert y-c_{J^{\prime
}}\right\vert \right) ^{n+1-\alpha }}d\sigma \left( y\right)  \notag \\
&=&\sum_{F^{\prime }\in \mathcal{F}}\frac{1}{\left\vert J\right\vert ^{\frac{%
1}{n}\left( n+1-\alpha \right) }}\overset{\ast }{\sum_{c_{J^{\prime }}\in 2J}%
}\left\Vert \mathsf{P}_{F^{\prime },J^{\prime }}^{\omega }\mathbf{x}%
\right\Vert _{L^{2}\left( \omega \right) }^{2}\frac{\mathrm{P}^{\alpha
}\left( J^{\prime },\mathbf{1}_{I\cap \left( 4J\right) }\sigma \right) }{%
\left\vert J^{\prime }\right\vert ^{\frac{1}{n}}}.  \notag
\end{eqnarray}%
In order to estimate the final sum above, we must invoke the `prepare to
puncture' argument above, as we will want to derive geometric decay from $%
\left\Vert \mathsf{P}_{J^{\prime }}^{\omega }\mathbf{x}\right\Vert
_{L^{2}\left( \omega \right) }^{2}$ by dominating it by the `nonenergy' term 
$\left\vert J^{\prime }\right\vert ^{\frac{2}{n}}\left\vert J^{\prime
}\right\vert _{\omega }$, as well as using the Muckenhoupt energy constant.
Choose an alternate quasicube $\widetilde{J}\in \mathcal{A}\Omega \mathcal{D}
$ satisfying $\dbigcup\limits_{c_{J^{\prime }}\in 2J}J^{\prime }\subset
4J\subset \widetilde{J}$ and $\ell \left( \widetilde{J}\right) \leq C\ell
\left( J\right) $. Define $\widetilde{\omega }=\omega -\omega \left( \left\{
p\right\} \right) \delta _{p}$ where $p$ is an atomic point in $\widetilde{J}
$ for which 
\begin{equation*}
\omega \left( \left\{ p\right\} \right) =\sup_{q\in \mathfrak{P}_{\left(
\sigma ,\omega \right) }:\ q\in \widetilde{J}}\omega \left( \left\{
q\right\} \right) .
\end{equation*}%
(If $\omega $ has no atomic point in common with $\sigma $ in $\widetilde{J}$
set $\widetilde{\omega }=\omega $.) Then we have $\left\vert \widetilde{J}%
\right\vert _{\widetilde{\omega }}=\omega \left( \widetilde{J},\mathfrak{P}%
_{\left( \sigma ,\omega \right) }\right) $ and%
\begin{equation*}
\frac{\left\vert \widetilde{J}\right\vert _{\widetilde{\omega }}}{\left\vert 
\widetilde{J}\right\vert ^{\left( 1-\frac{\alpha }{n}\right) }}\frac{%
\left\vert \widetilde{J}\right\vert _{\sigma }}{\left\vert \widetilde{J}%
\right\vert ^{\left( 1-\frac{\alpha }{n}\right) }}=\frac{\omega \left( 
\widetilde{J},\mathfrak{P}_{\left( \sigma ,\omega \right) }\right) }{%
\left\vert \widetilde{J}\right\vert ^{\left( 1-\frac{\alpha }{n}\right) }}%
\frac{\left\vert \widetilde{J}\right\vert _{\sigma }}{\left\vert \widetilde{J%
}\right\vert ^{\left( 1-\frac{\alpha }{n}\right) }}\leq A_{2}^{\alpha ,%
\limfunc{punct}}.
\end{equation*}%
From (\ref{key obs}) we also have%
\begin{equation*}
\sum_{F^{\prime }\in \mathcal{F}}\left\Vert \mathsf{P}_{F^{\prime
},J^{\prime }}^{\omega }\mathbf{x}\right\Vert _{L^{2}\left( \omega \right)
}^{2}\lesssim \mathbf{\tau }\ell \left( J^{\prime }\right) ^{2}\left\vert
J^{\prime }\right\vert _{\widetilde{\omega }}\ ,
\end{equation*}%
for all $J^{\prime }$ arising in the sum in the final line of (\ref{A0snear}%
) above.

Now by Cauchy-Schwarz and the alternate local estimate (\ref{shifted local})
in Lemma \ref{shifted} with $M=\widetilde{J}$ applied to the second line
below, the last sum in (\ref{A0snear}) is dominated by%
\begin{eqnarray}
&&\frac{1}{\left\vert J\right\vert ^{\frac{1}{n}\left( n+1-\alpha \right) }}%
\left( \sum_{F^{\prime }\in \mathcal{F}}\overset{\ast }{\sum_{c\left(
J^{\prime }\right) \in 2J}}\left\Vert \mathsf{P}_{F^{\prime },J^{\prime
}}^{\omega }\mathbf{x}\right\Vert _{L^{2}\left( \omega \right) }^{2}\right)
^{\frac{1}{2}}  \label{dom by} \\
&&\ \ \ \ \ \ \ \ \ \ \times \left( \sum_{F^{\prime }\in \mathcal{F}}\overset%
{\ast }{\sum_{c_{J^{\prime }}\in 2J}}\left\Vert \mathsf{P}_{F^{\prime
},J^{\prime }}^{\omega }\mathbf{x}\right\Vert _{L^{2}\left( \omega \right)
}^{2}\left( \frac{\mathrm{P}^{\alpha }\left( J^{\prime },\mathbf{1}_{I\cap
\left( 4J\right) }\sigma \right) }{\left\vert J^{\prime }\right\vert ^{\frac{%
1}{n}}}\right) ^{2}\right) ^{\frac{1}{2}}  \notag \\
&\lesssim &\frac{1}{\left\vert J\right\vert ^{\frac{1}{n}\left( n+1-\alpha
\right) }}\left( \mathbf{\tau }\sum_{c_{J^{\prime }}\in 2J}\left\vert
J^{\prime }\right\vert ^{\frac{2}{n}}\left\vert J^{\prime }\right\vert _{%
\widetilde{\omega }}\right) ^{\frac{1}{2}}\sqrt{\left( \mathcal{E}_{\alpha
}^{\limfunc{plug}}\right) ^{2}+A_{2}^{\alpha ,\limfunc{energy}}}\sqrt{%
\mathbf{\tau }\left\vert \widetilde{J}\right\vert _{\sigma }}  \notag \\
&\lesssim &\mathbf{\tau }\frac{2^{-s}\left\vert J\right\vert ^{\frac{1}{n}}}{%
\left\vert J\right\vert ^{\frac{1}{n}\left( n+1-\alpha \right) }}\sqrt{%
\left\vert 4J\right\vert _{\widetilde{\omega }}}\sqrt{\left( \mathcal{E}%
_{\alpha }^{\limfunc{plug}}\right) ^{2}+A_{2}^{\alpha ,\limfunc{energy}}}%
\sqrt{\left\vert \widetilde{J}\right\vert _{\sigma }}  \notag \\
&\lesssim &\mathbf{\tau }2^{-s}\sqrt{\left( \mathcal{E}_{\alpha }^{\limfunc{%
plug}}\right) ^{2}+A_{2}^{\alpha ,\limfunc{energy}}}\sqrt{\frac{\left\vert 
\widetilde{J}\right\vert _{\widetilde{\omega }}}{\left\vert \widetilde{J}%
\right\vert ^{\frac{1}{n}\left( n-\alpha \right) }}\frac{\left\vert 
\widetilde{J}\right\vert _{\sigma }}{\left\vert \widetilde{J}\right\vert ^{%
\frac{1}{n}\left( n-\alpha \right) }}}\lesssim \mathbf{\tau }2^{-s}\sqrt{%
\left( \mathcal{E}_{\alpha }^{\limfunc{plug}}\right) ^{2}+A_{2}^{\alpha ,%
\limfunc{energy}}}\sqrt{A_{2}^{\alpha ,\limfunc{punct}}}\ .  \notag
\end{eqnarray}

Similarly, for $\ell \geq 1$, we can estimate the far term

\begin{eqnarray*}
A_{s,far}^{\ell }\left( J\right) &\leq &\sum_{F^{\prime }\in \mathcal{F}}%
\overset{\ast }{\sum_{c_{J^{\prime }}\in \left( 2^{\ell +1}J\right)
\setminus \left( 2^{\ell }J\right) }}\int_{\mathbb{R}^{n}\setminus 2^{\ell
+2}J}S_{\left( J^{\prime },J\right) }^{F^{\prime }}\left( y\right) d\sigma
\left( y\right) \\
&\lesssim &\mathbf{\tau }\sum_{c_{J^{\prime }}\in \left( 2^{\ell +1}J\right)
\setminus \left( 2^{\ell }J\right) }\int_{\mathbb{R}^{n}\setminus 2^{\ell
+2}J}\frac{\left\vert J^{\prime }\right\vert ^{\frac{2}{n}}\left\vert
J^{\prime }\right\vert _{\omega }}{\left( \left\vert J\right\vert ^{\frac{1}{%
n}}+\left\vert y-c_{J}\right\vert \right) ^{2\left( n+1-\alpha \right) }}%
d\sigma \left( y\right) \\
&=&\mathbf{\tau }2^{-2s}\left( \sum_{c_{J^{\prime }}\in \left( 2^{\ell
+1}J\right) }\left\vert J^{\prime }\right\vert _{\omega }\right) \int_{%
\mathbb{R}^{n}\setminus 2^{\ell +2}J}\frac{\left\vert J\right\vert ^{\frac{2%
}{n}}}{\left( \left\vert J\right\vert ^{\frac{1}{n}}+\left\vert
y-c_{J}\right\vert \right) ^{2\left( n+1-\alpha \right) }}d\sigma \left(
y\right) \\
&\approx &\mathbf{\tau }2^{-2s}2^{-2\ell }\left( \sum_{c_{J^{\prime }}\in
\left( 2^{\ell +1}J\right) }\left\vert J^{\prime }\right\vert _{\omega
}\right) \int_{\mathbb{R}^{n}\setminus 2^{\ell +2}J}\frac{\left\vert 2^{\ell
}J\right\vert ^{\frac{2}{n}}}{\left( \left\vert 2^{\ell }J\right\vert ^{%
\frac{1}{n}}+\left\vert y-c_{2^{\ell }J}\right\vert \right) ^{2\left(
n+1-\alpha \right) }}d\sigma \left( y\right) ,
\end{eqnarray*}%
which is at most%
\begin{eqnarray*}
&&\mathbf{\tau }2^{-2s}2^{-2\ell }\left\vert 2^{\ell +2}J\right\vert
_{\omega }\int_{\mathbb{R}^{n}\setminus 2^{\ell +2}J}\frac{1}{\left(
\left\vert 2^{\ell }J\right\vert ^{\frac{1}{n}}+\left\vert y-c_{2^{\ell
}J}\right\vert \right) ^{2\left( n-\alpha \right) }}d\sigma \left( y\right)
\\
&\approx &\mathbf{\tau }2^{-2s}2^{-2\ell }\frac{\left\vert 2^{\ell
+2}J\right\vert _{\omega }}{\left\vert 2^{\ell }J\right\vert ^{1-\frac{%
\alpha }{n}}}\int_{\mathbb{R}^{n}\setminus 2^{\ell +2}J}\left( \frac{%
\left\vert 2^{\ell }J\right\vert ^{\frac{1}{n}}}{\left( \left\vert 2^{\ell
}J\right\vert ^{\frac{1}{n}}+\left\vert y-c_{2^{\ell }J}\right\vert \right)
^{2}}\right) ^{n-\alpha }d\sigma \left( y\right) \\
&\lesssim &\mathbf{\tau }2^{-2s}2^{-2\ell }\left\{ \frac{\left\vert 2^{\ell
+2}J\right\vert _{\omega }}{\left\vert 2^{\ell }J\right\vert ^{1-\frac{%
\alpha }{n}}}\mathcal{P}^{\alpha }\left( 2^{\ell +2}J,1_{\mathbb{R}%
^{n}\setminus 2^{\ell +2}J}\sigma \right) \right\} \lesssim \mathbf{\tau }%
2^{-2s}2^{-2\ell }\mathcal{A}_{2}^{\alpha }\ .
\end{eqnarray*}

To estimate the near term $A_{s,near}^{\ell }\left( J\right) $ we must again
invoke the \emph{`prepare to puncture'} argument. Choose an alternate
quasicube $\widetilde{J}\in \mathcal{A}\Omega \mathcal{D}$ such that $%
\dbigcup\limits_{c_{J^{\prime }}\in 2^{\ell +1}J\setminus 2^{\ell
}J}J^{\prime }\subset 2^{\ell +2}J\subset \widetilde{J}$ and $\ell \left( 
\widetilde{J}\right) \leq C2^{\ell }\ell \left( J\right) $. Define $%
\widetilde{\omega }=\omega -\omega \left( \left\{ p\right\} \right) \delta
_{p}$ where $p$ is an atomic point in $\widetilde{J}$ for which 
\begin{equation*}
\omega \left( \left\{ p\right\} \right) =\sup_{q\in \mathfrak{P}_{\left(
\sigma ,\omega \right) }:\ q\in \widetilde{J}}\omega \left( \left\{
q\right\} \right) .
\end{equation*}%
(If $\omega $ has no atomic point in common with $\sigma $ in $\widetilde{J}$
set $\widetilde{\omega }=\omega $.) Then we have $\left\vert \widetilde{J}%
\right\vert _{\widetilde{\omega }}=\omega \left( \widetilde{J},\mathfrak{P}%
_{\left( \sigma ,\omega \right) }\right) $, and just as in the argument
above following (\ref{A0snear}), we have from (\ref{key obs}) both%
\begin{equation*}
\frac{\left\vert \widetilde{J}\right\vert _{\widetilde{\omega }}}{\left\vert 
\widetilde{J}\right\vert ^{\left( 1-\frac{\alpha }{n}\right) }}\frac{%
\left\vert \widetilde{J}\right\vert _{\sigma }}{\left\vert \widetilde{J}%
\right\vert ^{\left( 1-\frac{\alpha }{n}\right) }}\leq A_{2}^{\alpha ,%
\limfunc{punct}}\text{ and }\sum_{F^{\prime }\in \mathcal{F}}\left\Vert 
\mathsf{P}_{F^{\prime },J^{\prime }}^{\omega }\mathbf{x}\right\Vert
_{L^{2}\left( \omega \right) }^{2}\lesssim \mathbf{\tau }\ell \left(
J^{\prime }\right) ^{2}\left\vert J^{\prime }\right\vert _{\widetilde{\omega 
}}\ .
\end{equation*}%
Thus using that $m=\left[ \log _{2}\left( 3\sqrt{n}\right) \right] +1$ in
the definition of $A_{s,near}^{\ell }\left( J\right) $, we see that

\begin{eqnarray*}
A_{s,near}^{\ell }\left( J\right)  &\leq &\sum_{F^{\prime }\in \mathcal{F}}%
\overset{\ast }{\sum_{c_{J^{\prime }}\in 2^{\ell +1}J\setminus 2^{\ell }J}}%
\int_{2^{\ell +2}J\setminus 2^{\ell -m}J}S_{\left( J^{\prime },J\right)
}^{F^{\prime }}\left( y\right) d\sigma \left( y\right)  \\
&\approx &\sum_{F^{\prime }\in \mathcal{F}}\overset{\ast }{%
\sum_{c_{J^{\prime }}\in 2^{\ell +1}J\setminus 2^{\ell }J}}\int_{2^{\ell
+2}J\setminus 2^{\ell -m}J}\frac{1}{\left\vert 2^{\ell }J\right\vert ^{\frac{%
1}{n}\left( n+1-\alpha \right) }}\frac{\left\Vert \mathsf{P}_{F^{\prime
},J^{\prime }}^{\omega }\mathbf{x}\right\Vert _{L^{2}\left( \omega \right)
}^{2}}{\left( \left\vert J^{\prime }\right\vert ^{\frac{1}{n}}+\left\vert
y-c_{J^{\prime }}\right\vert \right) ^{n+1-\alpha }}d\sigma \left( y\right) 
\\
&\lesssim &\frac{1}{\left\vert 2^{\ell }J\right\vert ^{\frac{1}{n}\left(
n+1-\alpha \right) }}\sum_{F^{\prime }\in \mathcal{F}}\overset{\ast }{%
\sum_{c_{J^{\prime }}\in 2^{\ell +1}J\setminus 2^{\ell }J}}\left\Vert 
\mathsf{P}_{F^{\prime },J^{\prime }}^{\omega }\mathbf{x}\right\Vert
_{L^{2}\left( \omega \right) }^{2} \\
&&\ \ \ \ \ \ \ \ \ \ \ \ \ \ \ \times \int_{2^{\ell +2}J}\frac{1}{\left(
\left\vert J^{\prime }\right\vert ^{\frac{1}{n}}+\left\vert y-c_{J^{\prime
}}\right\vert \right) ^{n+1-\alpha }}d\sigma \left( y\right) ,
\end{eqnarray*}%
is dominated by%
\begin{eqnarray*}
&&\frac{1}{\left\vert 2^{\ell }J\right\vert ^{\frac{1}{n}\left( n+1-\alpha
\right) }}\sum_{F^{\prime }\in \mathcal{F}}\overset{\ast }{%
\sum_{c_{J^{\prime }}\in 2^{\ell +1}J\setminus 2^{\ell }J}}\left\Vert 
\mathsf{P}_{F^{\prime },J^{\prime }}^{\omega }\mathbf{x}\right\Vert
_{L^{2}\left( \omega \right) }^{2}\frac{\mathrm{P}^{\alpha }\left( J^{\prime
},\mathbf{1}_{2^{\ell +2}J}\sigma \right) }{\left\vert J^{\prime
}\right\vert ^{\frac{1}{n}}} \\
&\leq &\frac{1}{\left\vert 2^{\ell }J\right\vert ^{\frac{1}{n}\left(
n+1-\alpha \right) }}\left( \sum_{F^{\prime }\in \mathcal{F}}\overset{\ast }{%
\sum_{c_{J^{\prime }}\in 2^{\ell +1}J\setminus 2^{\ell }J}}\left\Vert 
\mathsf{P}_{F^{\prime },J^{\prime }}^{\omega }\mathbf{x}\right\Vert
_{L^{2}\left( \omega \right) }^{2}\right) ^{\frac{1}{2}} \\
&&\times \left( \sum_{F^{\prime }\in \mathcal{F}}\overset{\ast }{%
\sum_{c_{J^{\prime }}\in 2^{\ell +1}J\setminus 2^{\ell }J}}\left\Vert 
\mathsf{P}_{F^{\prime },J^{\prime }}^{\omega }\mathbf{x}\right\Vert
_{L^{2}\left( \omega \right) }^{2}\left( \frac{\mathrm{P}^{\alpha }\left(
J^{\prime },\mathbf{1}_{2^{\ell +2}J}\sigma \right) }{\left\vert J^{\prime
}\right\vert ^{\frac{1}{n}}}\right) ^{2}\right) ^{\frac{1}{2}}.
\end{eqnarray*}

This can now be estimated using $\sum_{F^{\prime }\in \mathcal{F}}\left\Vert 
\mathsf{P}_{F^{\prime },J^{\prime }}^{\omega }\mathbf{x}\right\Vert
_{L^{2}\left( \omega \right) }^{2}\lesssim \mathbf{\tau }\left\vert
J^{\prime }\right\vert ^{\frac{2}{n}}\left\vert J^{\prime }\right\vert _{%
\widetilde{\omega }}=\mathbf{\tau }2^{-2s}\left\vert J\right\vert ^{\frac{2}{%
n}}\left\vert J^{\prime }\right\vert _{\widetilde{\omega }}$ along with the
alternate local estimate (\ref{shifted local}) in Lemma \ref{shifted} with $%
M=\widetilde{J}$ applied to the final line above to get 
\begin{eqnarray*}
A_{s,near}^{\ell }\left( J\right) &\lesssim &\mathbf{\tau }2^{-s}2^{-\ell }%
\frac{\left\vert 2^{\ell }J\right\vert ^{\frac{1}{n}}}{\left\vert 2^{\ell
}J\right\vert ^{\frac{1}{n}\left( n+1-\alpha \right) }}\sqrt{\left\vert 
\widetilde{J}\right\vert _{\widetilde{\omega }}}\sqrt{\left( \mathcal{E}%
_{\alpha }^{\limfunc{plug}}\right) ^{2}+A_{2}^{\alpha ,\limfunc{energy}}}%
\sqrt{\left\vert \widetilde{J}\right\vert _{\sigma }} \\
&\lesssim &\mathbf{\tau }2^{-s}2^{-\ell }\sqrt{\left( \mathcal{E}_{\alpha }^{%
\limfunc{plug}}\right) ^{2}+A_{2}^{\alpha ,\limfunc{energy}}}\sqrt{\frac{%
\left\vert \widetilde{J}\right\vert _{\widetilde{\omega }}}{\left\vert 
\widetilde{J}\right\vert ^{1-\frac{\alpha }{n}}}\frac{\left\vert \widetilde{J%
}\right\vert _{\sigma }}{\left\vert \widetilde{J}\right\vert ^{1-\frac{%
\alpha }{n}}}} \\
&\lesssim &\mathbf{\tau }2^{-s}2^{-\ell }\sqrt{\left( \mathcal{E}_{\alpha }^{%
\limfunc{plug}}\right) ^{2}+A_{2}^{\alpha ,\limfunc{energy}}}\sqrt{%
A_{2}^{\alpha ,\limfunc{punct}}}\ .
\end{eqnarray*}%
These estimates are summable in both $s$ and $\ell $.

Now we turn to the terms $A_{s,close}^{\ell }\left( J\right) $, and recall
from (\ref{vanishing close}) that $A_{s,close}^{\ell }\left( J\right) =0$ if 
$\ell \geq 1+\delta s$. So we now suppose that $\ell \leq \delta s$. We
have, with $m$ as in (\ref{smallest m}),%
\begin{eqnarray*}
&&A_{s,close}^{\ell }\left( J\right) \leq \sum_{F^{\prime }\in \mathcal{F}}%
\overset{\ast }{\sum_{c_{J^{\prime }}\in 2^{\ell +1}J\setminus 2^{\ell }J}}%
\int_{2^{\ell -m}J}S_{\left( J^{\prime },J\right) }^{F^{\prime }}\left(
y\right) d\sigma \left( y\right) \\
&\approx &\sum_{F^{\prime }\in \mathcal{F}}\overset{\ast }{%
\sum_{c_{J^{\prime }}\in 2^{\ell +1}J\setminus 2^{\ell }J}}\int_{2^{\ell
-m}J}\frac{1}{\left( \left\vert J\right\vert ^{\frac{1}{n}}+\left\vert
y-c_{J}\right\vert \right) ^{n+1-\alpha }}\frac{\left\Vert \mathsf{P}%
_{F^{\prime },J^{\prime }}^{\omega }\mathbf{x}\right\Vert _{L^{2}\left(
\omega \right) }^{2}}{\left\vert 2^{\ell }J\right\vert ^{\frac{1}{n}\left(
n+1-\alpha \right) }}d\sigma \left( y\right) \\
&=&\left( \sum_{F^{\prime }\in \mathcal{F}}\overset{\ast }{%
\sum_{c_{J^{\prime }}\in 2^{\ell +1}J\setminus 2^{\ell }J}}\left\Vert 
\mathsf{P}_{F^{\prime },J^{\prime }}^{\omega }\mathbf{x}\right\Vert
_{L^{2}\left( \omega \right) }^{2}\right) \frac{1}{\left\vert 2^{\ell
}J\right\vert ^{\frac{1}{n}\left( n+1-\alpha \right) }} \\
&&\ \ \ \ \ \ \ \ \ \ \ \ \ \ \ \ \ \ \ \ \times \int_{2^{\ell -m}J}\frac{1}{%
\left( \left\vert J\right\vert ^{\frac{1}{n}}+\left\vert y-c_{J}\right\vert
\right) ^{n+1-\alpha }}d\sigma \left( y\right) .
\end{eqnarray*}%
Now we use the inequality $\sum_{F^{\prime }\in \mathcal{F}}\left\Vert 
\mathsf{P}_{F^{\prime },J^{\prime }}^{\omega }\mathbf{z}\right\Vert
_{L^{2}\left( \omega \right) }^{2}\leq \mathbf{\tau }\left\vert J^{\prime
}\right\vert ^{\frac{2}{n}}\left\vert J^{\prime }\right\vert _{\omega }$ and
we get the relatively crude estimate%
\begin{eqnarray*}
&&A_{s,close}^{\ell }\left( J\right) \\
&\lesssim &\mathbf{\tau }2^{-2s}\left\vert J\right\vert ^{\frac{2}{n}%
}\left\vert 2^{\ell +2}J\setminus 2^{\ell -1}J\right\vert _{\omega }\frac{1}{%
\left\vert 2^{\ell }J\right\vert ^{\frac{1}{n}\left( n+1-\alpha \right) }}%
\int_{2^{\ell -m}J}\frac{1}{\left( \left\vert J\right\vert ^{\frac{1}{n}%
}+\left\vert y-c_{J}\right\vert \right) ^{n+1-\alpha }}d\sigma \left(
y\right) \\
&\lesssim &\mathbf{\tau }2^{-2s}\left\vert J\right\vert ^{\frac{2}{n}}\frac{%
\left\vert 2^{\ell +2}J\setminus 2^{\ell -1}J\right\vert _{\omega }}{%
\left\vert 2^{\ell }J\right\vert ^{\frac{1}{n}\left( n+1-\alpha \right) }}%
\frac{\left\vert 2^{\ell -m}J\right\vert _{\sigma }}{\left\vert J\right\vert
^{\frac{1}{n}\left( n+1-\alpha \right) }}\lesssim \mathbf{\tau }2^{-2s}\frac{%
\left\vert 2^{\ell +2}J\setminus 2^{\ell -1}J\right\vert _{\omega }}{%
\left\vert 2^{\ell +2}J\right\vert ^{1-\frac{\alpha }{n}}}\frac{\left\vert
2^{\ell -m}J\right\vert _{\sigma }}{\left\vert 2^{\ell -m}J\right\vert ^{1-%
\frac{\alpha }{n}}}2^{\ell \left( n-1-\alpha \right) } \\
&\lesssim &\mathbf{\tau }2^{-2s}2^{\ell \left( n-1-\alpha \right)
}A_{2}^{\alpha }\lesssim \mathbf{\tau }2^{-s}A_{2}^{\alpha }\ ,
\end{eqnarray*}%
provided that $\ell \leq \frac{s}{n}$ and $m>1$. But we are assuming $\ell
\leq \delta s$ here, and so we obtain a suitable estimate for $%
A_{s,close}^{\ell }\left( J\right) $ from this crude estimate provided we
choose $0<\delta <\frac{1}{n}$. Indeed, for fixed $s$ we then have $\ell
\leq \delta s<\frac{s}{n}$, and so also%
\begin{equation*}
2^{-2s}2^{\ell \left( n-1-\alpha \right) }\leq 2^{-2s}2^{s\left( \frac{%
n-1-\alpha }{n}\right) }\leq 2^{-s\left( \frac{n+1+\alpha }{n}\right) }\ ,
\end{equation*}%
and hence%
\begin{equation*}
\sum_{l=1}^{\delta s}2^{-2s}2^{\ell \left( n-1-\alpha \right) }\leq C\frac{s%
}{n}2^{-s\left( \frac{n+1+\alpha }{n}\right) }\lesssim 2^{-s}\ .
\end{equation*}%
The above estimates prove%
\begin{equation*}
T_{s}^{\limfunc{proximal}}+T_{s}^{\limfunc{difference}}\lesssim 2^{-s}\left( 
\mathcal{A}_{2}^{\alpha }+\sqrt{\left( \mathcal{E}_{\alpha }^{\limfunc{plug}%
}\right) ^{2}+A_{2}^{\alpha ,\limfunc{energy}}}\sqrt{A_{2}^{\alpha ,\limfunc{%
punct}}}\right) ,
\end{equation*}%
which is summable in $s$.

\subsubsection{The intersection term}

Now we return to the term,%
\begin{eqnarray*}
T_{s}^{\limfunc{intersection}} &\equiv &\sum_{\substack{ F,F^{\prime }\in 
\mathcal{F}}}\sum_{\substack{ J\in \mathcal{M}_{\left( \mathbf{r}%
,\varepsilon \right) -\limfunc{deep}}\left( F\right) ,\ J^{\prime }\in 
\mathcal{M}_{\left( \mathbf{r},\varepsilon \right) -\limfunc{deep}}\left(
F^{\prime }\right)  \\ J,J^{\prime }\subset I,\ \ell \left( J^{\prime
}\right) =2^{-s}\ell \left( J\right)  \\ \limfunc{qdist}\left(
c_{J},c_{J^{\prime }}\right) \geq 2^{s\left( 1+\delta \right) }\ell \left(
J^{\prime }\right) }} \\
&&\times \int_{B\left( J,J^{\prime }\right) }\frac{\left\Vert \mathsf{P}%
_{F,J}^{\omega }\mathbf{x}\right\Vert _{L^{2}\left( \omega \right) }^{2}}{%
\left( \left\vert J\right\vert ^{\frac{1}{n}}+\left\vert y-c_{J}\right\vert
\right) ^{n+1-\alpha }}\frac{\left\Vert \mathsf{P}_{F^{\prime },J^{\prime
}}^{\omega }\mathbf{x}\right\Vert _{L^{2}\left( \omega \right) }^{2}}{\left(
\left\vert J^{\prime }\right\vert ^{\frac{1}{n}}+\left\vert y-c_{J^{\prime
}}\right\vert \right) ^{n+1-\alpha }}d\sigma \left( y\right) .
\end{eqnarray*}%
It will suffice to show that $T_{s}^{\limfunc{intersection}}$ satisfies the
estimate,%
\begin{eqnarray*}
T_{s}^{\limfunc{intersection}} &\lesssim &2^{-s\delta }\sqrt{\left( \mathcal{%
E}_{\alpha }^{\limfunc{plug}}\right) ^{2}+A_{2}^{\alpha ,\limfunc{energy}}}%
\sqrt{A_{2}^{\alpha ,\limfunc{punct}}}\sum_{F^{\prime }\in \mathcal{F}%
^{\prime }}\sum_{\substack{ J^{\prime }\in \mathcal{M}_{\left( \mathbf{r}%
,\varepsilon \right) -\limfunc{deep}}\left( F^{\prime }\right)  \\ J^{\prime
}\subset I}}\lVert \mathsf{P}_{F^{\prime },J^{\prime }}^{\omega }\mathbf{x}%
\rVert _{L^{2}\left( \omega \right) }^{2} \\
&=&2^{-s\delta }\sqrt{\left( \mathcal{E}_{\alpha }^{\limfunc{plug}}\right)
^{2}+A_{2}^{\alpha ,\limfunc{energy}}}\sqrt{A_{2}^{\alpha ,\limfunc{punct}}}%
\int_{\widehat{I}}t^{2}\overline{\mu }\ .
\end{eqnarray*}%
Using $B\left( J,J^{\prime }\right) =B\left( c_{J},\frac{1}{2}\limfunc{qdist}%
\left( c_{J},c_{J^{\prime }}\right) \right) $, we can write (suppressing
some notation for clarity),%
\begin{eqnarray*}
&&T_{s}^{\limfunc{intersection}} \\
&=&\sum_{F,F^{\prime }}\sum_{\substack{ J,J^{\prime }}}\int_{B\left(
J,J^{\prime }\right) }\frac{\left\Vert \mathsf{P}_{F,J}^{\omega }\mathbf{x}%
\right\Vert _{L^{2}\left( \omega \right) }^{2}}{\left( \left\vert
J\right\vert ^{\frac{1}{n}}+\left\vert y-c_{J}\right\vert \right)
^{n+1-\alpha }}\frac{\left\Vert \mathsf{P}_{F^{\prime },J^{\prime }}^{\omega
}\mathbf{x}\right\Vert _{L^{2}\left( \omega \right) }^{2}}{\left( \left\vert
J^{\prime }\right\vert ^{\frac{1}{n}}+\left\vert y-c_{J^{\prime
}}\right\vert \right) ^{n+1-\alpha }}d\sigma \left( y\right)  \\
&\approx &\sum_{F,F^{\prime }}\sum_{\substack{ J,J^{\prime }}}\left\Vert 
\mathsf{P}_{F,J}^{\omega }\mathbf{x}\right\Vert _{L^{2}\left( \omega \right)
}^{2}\left\Vert \mathsf{P}_{F^{\prime },J^{\prime }}^{\omega }\mathbf{x}%
\right\Vert _{L^{2}\left( \omega \right) }^{2}\frac{1}{\left\vert
c_{J}-c_{J^{\prime }}\right\vert ^{n+1-\alpha }}\int_{B\left( J,J^{\prime
}\right) }\frac{d\sigma \left( y\right) }{\left( \left\vert J\right\vert ^{%
\frac{1}{n}}+\left\vert y-c_{J}\right\vert \right) ^{n+1-\alpha }} \\
&\leq &\sum_{F^{\prime }}\sum_{\substack{ J^{\prime }}}\left\Vert \mathsf{P}%
_{F^{\prime },J^{\prime }}^{\omega }\mathbf{x}\right\Vert _{L^{2}\left(
\omega \right) }^{2}\sum_{F}\sum_{\substack{ J}}\frac{1}{\left\vert
c_{J}-c_{J^{\prime }}\right\vert ^{n+1-\alpha }}\left\Vert \mathsf{P}%
_{F,J}^{\omega }\mathbf{x}\right\Vert _{L^{2}\left( \omega \right)
}^{2}\int_{B\left( J,J^{\prime }\right) }\frac{d\sigma \left( y\right) }{%
\left( \left\vert J\right\vert ^{\frac{1}{n}}+\left\vert y-c_{J}\right\vert
\right) ^{n+1-\alpha }} \\
&\equiv &\sum_{F^{\prime }}\sum_{\substack{ J^{\prime }}}\left\Vert \mathsf{P%
}_{F^{\prime },J^{\prime }}^{\omega }\mathbf{x}\right\Vert _{L^{2}\left(
\omega \right) }^{2}S_{s}\left( J^{\prime }\right) ,
\end{eqnarray*}%
and since $\int_{B\left( J,J^{\prime }\right) }\frac{d\sigma \left( y\right) 
}{\left( \left\vert J\right\vert ^{\frac{1}{n}}+\left\vert
y-c_{J}\right\vert \right) ^{n+1-\alpha }}\approx \frac{\mathrm{P}^{\alpha
}\left( J,\mathbf{1}_{B\left( J,J^{\prime }\right) }\sigma \right) }{%
\left\vert J\right\vert ^{\frac{1}{n}}}$, it remains to show that for each
fixed $J^{\prime }$,%
\begin{eqnarray*}
S_{s}\left( J^{\prime }\right)  &\approx &\sum_{F}\overset{\ast }{\sum
_{\substack{ J:\ \limfunc{qdist}\left( c_{J},c_{J^{\prime }}\right) \geq
2^{s\left( 1+\delta \right) }\ell \left( J^{\prime }\right) }}}\frac{%
\left\Vert \mathsf{P}_{F,J}^{\omega }\mathbf{x}\right\Vert _{L^{2}\left(
\omega \right) }^{2}}{\left\vert c_{J}-c_{J^{\prime }}\right\vert
^{n+1-\alpha }}\frac{\mathrm{P}^{\alpha }\left( J,\mathbf{1}_{B\left(
J,J^{\prime }\right) }\sigma \right) }{\left\vert J\right\vert ^{\frac{1}{n}}%
} \\
&\lesssim &2^{-\delta s}\sqrt{\left( \mathcal{E}_{\alpha }^{\limfunc{plug}%
}\right) ^{2}+A_{2}^{\alpha ,\limfunc{energy}}}\sqrt{A_{2}^{\alpha }}\ .
\end{eqnarray*}

We write%
\begin{eqnarray}
S_{s}\left( J^{\prime }\right)  &\approx &\sum_{k\geq s\left( 1+\delta
\right) }\frac{1}{\left( 2^{k}\left\vert J^{\prime }\right\vert ^{\frac{1}{n}%
}\right) ^{n+1-\alpha }}\sum_{F}\overset{\ast }{\sum_{J:\ \limfunc{qdist}%
\left( c_{J},c_{J^{\prime }}\right) \approx 2^{k}\ell \left( J^{\prime
}\right) }}\left\Vert \mathsf{P}_{F,J}^{\omega }\mathbf{x}\right\Vert
_{L^{2}\left( \omega \right) }^{2}\frac{\mathrm{P}^{\alpha }\left( J,\mathbf{%
1}_{B\left( J,J^{\prime }\right) }\sigma \right) }{\left\vert J\right\vert ^{%
\frac{1}{n}}}  \label{def Sks} \\
&=&\sum_{k\geq s\left( 1+\delta \right) }\frac{1}{\left( 2^{k}\left\vert
J^{\prime }\right\vert ^{\frac{1}{n}}\right) ^{n+1-\alpha }}S_{s}^{k}\left(
J^{\prime }\right) \ ;  \notag \\
S_{s}^{k}\left( J^{\prime }\right)  &\equiv &\sum_{F}\overset{\ast }{%
\sum_{J:\ \limfunc{qdist}\left( c_{J},c_{J^{\prime }}\right) \approx
2^{k}\ell \left( J^{\prime }\right) }}\left\Vert \mathsf{P}_{F,J}^{\omega }%
\mathbf{x}\right\Vert _{L^{2}\left( \omega \right) }^{2}\frac{\mathrm{P}%
^{\alpha }\left( J,\mathbf{1}_{B\left( J,J^{\prime }\right) }\sigma \right) 
}{\left\vert J\right\vert ^{\frac{1}{n}}},  \notag
\end{eqnarray}%
where by $\limfunc{qdist}\left( c_{J},c_{J^{\prime }}\right) \approx
2^{k}\ell \left( J^{\prime }\right) $ we mean $2^{k}\ell \left( J^{\prime
}\right) \leq \limfunc{qdist}\left( c_{J},c_{J^{\prime }}\right) \leq
2^{k+1}\ell \left( J^{\prime }\right) $. Moreover, if $\limfunc{qdist}\left(
c_{J},c_{J^{\prime }}\right) \approx 2^{k}\ell \left( J^{\prime }\right) $,
then from the fact that the quasiradius of $B\left( J,J^{\prime }\right) $
is $\frac{1}{2}\limfunc{qdist}\left( c_{J},c_{J^{\prime }}\right) $, we
obtain 
\begin{equation*}
B\left( J,J^{\prime }\right) \subset C_{0}2^{k}J^{\prime },
\end{equation*}%
where $C_{0}$ is a positive constant ($C_{0}=6$ works).

For fixed $k\geq 1$, we invoke yet again the \emph{`prepare to puncture'}
argument. Choose an alternate quasicube $\widetilde{J^{\prime }}\in \mathcal{%
A}\Omega \mathcal{D}$ such that $C_{0}2^{k}J\subset \widetilde{J^{\prime }}$
and $\ell \left( \widetilde{J^{\prime }}\right) \leq C2^{k}\ell \left(
J^{\prime }\right) $. Define $\widetilde{\omega }=\omega -\omega \left(
\left\{ p\right\} \right) \delta _{p}$ where $p$ is an atomic point in $%
\widetilde{J^{\prime }}$ for which 
\begin{equation*}
\omega \left( \left\{ p\right\} \right) =\sup_{q\in \mathfrak{P}_{\left(
\sigma ,\omega \right) }:\ q\in \widetilde{J^{\prime }}}\omega \left(
\left\{ q\right\} \right) .
\end{equation*}%
(If $\omega $ has no atomic point in common with $\sigma $ in $\widetilde{%
J^{\prime }}$ set $\widetilde{\omega }=\omega $.) Then we have $\left\vert 
\widetilde{J^{\prime }}\right\vert _{\widetilde{\omega }}=\omega \left( 
\widetilde{J^{\prime }},\mathfrak{P}_{\left( \sigma ,\omega \right) }\right) 
$ and so from (\ref{key obs}),%
\begin{equation*}
\frac{\left\vert \widetilde{J^{\prime }}\right\vert _{\widetilde{\omega }}}{%
\left\vert \widetilde{J^{\prime }}\right\vert ^{1-\frac{\alpha }{n}}}\frac{%
\left\vert \widetilde{J^{\prime }}\right\vert _{\sigma }}{\left\vert 
\widetilde{J^{\prime }}\right\vert ^{1-\frac{\alpha }{n}}}\leq A_{2}^{\alpha
,\limfunc{punct}}\text{ and }\sum_{F\in \mathcal{F}}\left\Vert \mathsf{P}%
_{F,J}^{\omega }\mathbf{x}\right\Vert _{L^{2}\left( \omega \right)
}^{2}\lesssim \mathbf{\tau }\ell \left( J\right) ^{2}\left\vert J\right\vert
_{\widetilde{\omega }}\ .
\end{equation*}

Now we are ready to apply Cauchy-Schwarz and the alternate local estimate (%
\ref{shifted local}) in Lemma \ref{shifted} with $M=\widetilde{J^{\prime }}\ 
$to the second line below to get the following estimate for $S_{s}^{k}\left(
J^{\prime }\right) $ defined in (\ref{def Sks}) above:%
\begin{eqnarray*}
S_{s}^{k}\left( J^{\prime }\right) &\leq &\left( \sum_{F}\sum_{J:\ \limfunc{%
qdist}\left( c_{J},c_{J^{\prime }}\right) \approx 2^{k}\ell \left( J^{\prime
}\right) }\left\Vert \mathsf{P}_{F,J}^{\omega }\mathbf{x}\right\Vert
_{L^{2}\left( \omega \right) }^{2}\right) ^{\frac{1}{2}} \\
&&\times \left( \sum_{F}\sum_{J:\ \limfunc{qdist}\left( c_{J},c_{J^{\prime
}}\right) \approx 2^{k}\ell \left( J^{\prime }\right) }\left\Vert \mathsf{P}%
_{F,J}^{\omega }\mathbf{x}\right\Vert _{L^{2}\left( \omega \right)
}^{2}\left( \frac{\mathrm{P}^{\alpha }\left( J,\mathbf{1}_{B\left(
J,J^{\prime }\right) }\sigma \right) }{\left\vert J\right\vert ^{\frac{1}{n}}%
}\right) ^{2}\right) ^{\frac{1}{2}} \\
&\lesssim &\left( \mathbf{\tau }2^{2s}\left\vert J^{\prime }\right\vert ^{%
\frac{2}{n}}\left\vert \widetilde{J^{\prime }}\right\vert _{\widetilde{%
\omega }}\right) ^{\frac{1}{2}}\left( \mathbf{\tau }\left[ \left( \mathcal{E}%
_{\alpha }^{\limfunc{plug}}\right) ^{2}+A_{2}^{\alpha ,\limfunc{energy}}%
\right] \left\vert \widetilde{J^{\prime }}\right\vert _{\sigma }\right) ^{%
\frac{1}{2}} \\
&\lesssim &\mathbf{\tau }\sqrt{\left( \mathcal{E}_{\alpha }^{\limfunc{plug}%
}\right) ^{2}+A_{2}^{\alpha ,\limfunc{energy}}}2^{s}\left\vert J^{\prime
}\right\vert ^{\frac{1}{n}}\sqrt{\left\vert \widetilde{J^{\prime }}%
\right\vert _{\widetilde{\omega }}}\sqrt{\left\vert \widetilde{J^{\prime }}%
\right\vert _{\sigma }} \\
&\lesssim &\mathbf{\tau }\sqrt{\left( \mathcal{E}_{\alpha }^{\limfunc{plug}%
}\right) ^{2}+A_{2}^{\alpha ,\limfunc{energy}}}\sqrt{A_{2}^{\alpha ,\limfunc{%
punct}}}2^{s}\left\vert J^{\prime }\right\vert ^{\frac{1}{n}}\left\vert
2^{k}J^{\prime }\right\vert ^{1-\frac{\alpha }{n}} \\
&=&\mathbf{\tau }\sqrt{\left( \mathcal{E}_{\alpha }^{\limfunc{plug}}\right)
^{2}+A_{2}^{\alpha ,\limfunc{energy}}}\sqrt{A_{2}^{\alpha ,\limfunc{punct}}}%
2^{s}2^{k\left( n-\alpha \right) }\left\vert J^{\prime }\right\vert ^{\frac{1%
}{n}\left( n+1-\alpha \right) }.
\end{eqnarray*}

Altogether then we have%
\begin{eqnarray*}
S_{s}\left( J^{\prime }\right)  &=&\sum_{k\geq \left( 1+\delta \right) s}%
\frac{1}{\left( 2^{k}\left\vert J^{\prime }\right\vert ^{\frac{1}{n}}\right)
^{n+1-\alpha }}S_{s}^{k}\left( J^{\prime }\right)  \\
&\lesssim &\sqrt{\left( \mathcal{E}_{\alpha }^{\limfunc{plug}}\right)
^{2}+A_{2}^{\alpha ,\limfunc{energy}}}\sqrt{A_{2}^{\alpha ,\limfunc{punct}}}%
\sum_{k\geq \left( 1+\delta \right) s}\frac{1}{\left( 2^{k}\left\vert
J^{\prime }\right\vert ^{\frac{1}{n}}\right) ^{n+1-\alpha }}2^{s}2^{k\left(
n-\alpha \right) }\left\vert J^{\prime }\right\vert ^{\frac{1}{n}\left(
n+1-\alpha \right) } \\
&\lesssim &\sqrt{\left( \mathcal{E}_{\alpha }^{\limfunc{plug}}\right)
^{2}+A_{2}^{\alpha ,\limfunc{energy}}}\sqrt{A_{2}^{\alpha ,\limfunc{punct}}}%
\sum_{k\geq \left( 1+\delta \right) s}2^{s-k}\lesssim 2^{-\delta s}\sqrt{%
\left( \mathcal{E}_{\alpha }^{\limfunc{plug}}\right) ^{2}+A_{2}^{\alpha ,%
\limfunc{energy}}}\sqrt{A_{2}^{\alpha ,\limfunc{punct}}},
\end{eqnarray*}%
which is summable in $s$. This completes the proof of (\ref{Us bound}), and
hence of the estimate for $\mathbf{Back}\left( \widehat{I}\right) $ in (\ref%
{e.t2 n'}).

\section{The stopping form}

This section is virtually unchanged from the corresponding section in \cite%
{SaShUr5}. In the one-dimensional setting of the Hilbert transform, Hyt\"{o}%
nen \cite{Hyt2} observed that "...the innovative verification of the local
estimate by Lacey \cite{Lac} is already set up in such a way that it is
ready for us to borrow as a black box." The same observation carries over in
spirit here regarding the adaptation of Lacey's recursion and stopping time
to proving the local estimate in \cite{SaShUr5}. However, that adaptation
involves the splitting of the stopping form into two sublinear forms, the
first handled by methods in \cite{LaSaUr2}, and the second by methods in 
\cite{Lac}. So for the convenience of the reader, we repeat all the details
here, even though the arguments are little changed for common point masses.

In this section we adapt the argument of M. Lacey in \cite{Lac} to apply in
the setting of a general $\alpha $-fractional Calder\'{o}n-Zygmund operator $%
T^{\alpha }$ in $\mathbb{R}^{n}$ using the Monotonicity Lemma \ref{mono} and
our quasienergy condition in Definition \ref{energy condition}. We will
prove the bound (\ref{B stop form 3}) for the stopping form%
\begin{eqnarray}
\mathsf{B}_{\limfunc{stop}}^{A}\left( f,g\right) &\equiv &\sum_{\substack{ %
I\in \mathcal{C}_{A}\text{ and }J\in \mathcal{C}_{A}^{\mathbf{\tau }-%
\limfunc{shift}}  \\ J\Subset _{\mathbf{\rho },\varepsilon }I_{J}}}\left( 
\mathbb{E}_{I_{J}}^{\sigma }\bigtriangleup _{I}^{\sigma }f\right)
\left\langle T_{\sigma }^{\alpha }\mathbf{1}_{A\setminus
I_{J}},\bigtriangleup _{J}^{\omega }g\right\rangle _{\omega }  \label{dummy}
\\
&=&\sum_{\substack{ I:\ \pi I\in \mathcal{C}_{A}\text{ and }J\in \mathcal{C}%
_{A}^{\mathbf{\tau }-\limfunc{shift}}  \\ J\Subset _{\mathbf{\rho }%
,\varepsilon }I}}\left( \mathbb{E}_{I}^{\sigma }\bigtriangleup _{\pi
I}^{\sigma }f\right) \left\langle T_{\sigma }^{\alpha }\mathbf{1}%
_{A\setminus I},\bigtriangleup _{J}^{\omega }g\right\rangle _{\omega }, 
\notag
\end{eqnarray}%
where we have made the `change of dummy variable' $I_{J}\rightarrow I$ for
convenience in notation (recall that the child of $I$ that contains $J$ is
denoted $I_{J}$).

However, the Monotonicity Lemma of Lacey and Wick has an additional term on
the right hand side, and our quasienergy condition is not a direct
generalization of the one-dimensional energy condition. These differences in
higher dimension result in changes and complications that must be tracked
throughout the argument. In particular, we find it necessary to separate the
interaction of the two terms on the right side of the Monotonicity Lemma by
splitting the stopping form into the two corresponding sublinear forms in (%
\ref{def split}) below. Recall that for $A\in \mathcal{A}$ the \emph{shifted}
corona is given in Definition \ref{shifted corona} by 
\begin{equation*}
\mathcal{C}_{A}^{\mathbf{\tau }-\limfunc{shift}}=\left\{ J\in \mathcal{C}%
_{A}:J\Subset _{\mathbf{\tau },\varepsilon }A\right\} \cup
\dbigcup\limits_{A^{\prime }\in \mathfrak{C}_{\mathcal{A}}\left( A\right)
}\left\{ J\in \Omega \mathcal{D}:J\Subset _{\mathbf{\tau },\varepsilon }A%
\text{ and }J\text{ is }\mathbf{\tau }\text{-nearby in }A^{\prime }\right\} ,
\end{equation*}%
and in particular the $\mathbf{1}$-shifted corona is given by $\mathcal{C}%
_{A}^{\mathbf{1}-\limfunc{shift}}=\left( \mathcal{C}_{A}\setminus \left\{
A\right\} \right) \cup \mathfrak{C}_{\mathcal{A}}\left( A\right) $.

\begin{definition}
Suppose that $A\in \mathcal{A}$ and that $\mathcal{P}\subset \mathcal{C}%
_{A}^{\mathbf{1}-\limfunc{shift}}\times \mathcal{C}_{A}^{\mathbf{\tau }-%
\limfunc{shift}}$. We say that the collection of pairs $\mathcal{P}$ is $A$%
\emph{-admissible} if
\end{definition}

\begin{itemize}
\item (good and $\left( \mathbf{\rho -1},\varepsilon \right) $-deeply
embedded) $J$ is good and $J\Subset _{\mathbf{\rho -1},\varepsilon
}I\varsubsetneqq A$ for every $\left( I,J\right) \in \mathcal{P},$

\item (tree-connected in the first component) if $I_{1}\subset I_{2}$ and
both $\left( I_{1},J\right) \in \mathcal{P}$ and $\left( I_{2},J\right) \in 
\mathcal{P}$, then $\left( I,J\right) \in \mathcal{P}$ for every $I$ in the
geodesic $\left[ I_{1},I_{2}\right] =\left\{ I\in \Omega \mathcal{D}%
:I_{1}\subset I\subset I_{2}\right\} $.
\end{itemize}

However, since $\left( I,J\right) \in \mathcal{P}$ implies both $J\in 
\mathcal{C}_{A}^{\mathbf{\tau }-\limfunc{shift}}$ and $J\Subset _{\mathbf{%
\rho -1},\varepsilon }I$, the assumption $\mathbf{\rho >\tau }$ in
Definition \ref{def parameters} shows that $I$ is in the corona $\mathcal{C}%
_{A}$, and hence we may replace $\mathcal{C}_{A}^{\mathbf{1}-\limfunc{shift}%
} $ with the restricted corona $\mathcal{C}_{A}^{\prime }\equiv \mathcal{C}%
_{A}\setminus \left\{ A\right\} $ in the above definition of $A$\emph{%
-admissible}. The basic example of an admissible collection of pairs is
obtained from the pairs of quasicubes summed in the stopping form $\mathsf{B}%
_{stop}^{A}\left( f,g\right) $ in (\ref{dummy}), which occurs in (\ref{B
stop form 3}) above; 
\begin{equation}
\mathcal{P}^{A}\equiv \left\{ \left( I,J\right) :I\in \mathcal{C}%
_{A}^{\prime }\text{ and }J\in \mathcal{C}_{A}^{\mathbf{\tau }-\limfunc{shift%
}}\text{ where}\ J\text{ is }\mathbf{\tau }\text{-good and\ }J\Subset _{%
\mathbf{\rho -1},\varepsilon }I\right\} .  \label{initial P}
\end{equation}%
Recall also that $J$ is $\mathbf{\tau }$-good if $J\in \Omega \mathcal{D}%
_{\left( \mathbf{r},\varepsilon \right) -\limfunc{good}}^{\mathbf{\tau }}$
as in (\ref{extended good grid}), i.e. if $J$ and its children and its $\ell 
$-parents up to level $\mathbf{\tau }$ are all good. Recall that the
quasiHaar support of $g$ is contained in the collection of $\mathbf{\tau }$%
-good quasicubes.

\begin{definition}
Suppose that $A\in \mathcal{A}$ and that $\mathcal{P}$ is an $A$\emph{%
-admissible} collection of pairs. Define the associated \emph{stopping} form 
$\mathsf{B}_{\limfunc{stop}}^{A,\mathcal{P}}$ by%
\begin{equation*}
\mathsf{B}_{\limfunc{stop}}^{A,\mathcal{P}}\left( f,g\right) \equiv
\sum_{\left( I,J\right) \in \mathcal{P}}\left( \mathbb{E}_{I}^{\sigma
}\bigtriangleup _{\pi I}^{\sigma }f\right) \ \left\langle T_{\sigma
}^{\alpha }\mathbf{1}_{A\setminus I},\bigtriangleup _{J}^{\omega
}g\right\rangle _{\omega }\ ,
\end{equation*}%
where we may of course further restrict $I$ to $\pi I\in \limfunc{supp}%
\widehat{f}$ if we wish.
\end{definition}

Given an $A$-admissible collection $\mathcal{P}$ of pairs define the reduced
collection $\mathcal{P}^{\limfunc{red}}$ as follows. For each fixed $J$ let $%
I_{J}^{\limfunc{red}}$ be the largest good quasicube $I$ such that $\left(
I,J\right) \in \mathcal{P}$. Then set%
\begin{equation*}
\mathcal{P}^{\limfunc{red}}\equiv \left\{ \left( I,J\right) \in \mathcal{P}%
:I\subset I_{J}^{\limfunc{red}}\right\} .
\end{equation*}%
Clearly $\mathcal{P}^{\limfunc{red}}$ is $A$-admissible. Now recall our
assumption that the quasiHaar support of $f$ is contained in the set of $%
\mathbf{\tau }$-good quasicubes, which in particular requires that their
children are all good as well. This assumption has the important implication
that $\mathsf{B}_{\limfunc{stop}}^{A,\mathcal{P}}\left( f,g\right) =\mathsf{B%
}_{\limfunc{stop}}^{A,\mathcal{P}^{\limfunc{red}}}\left( f,g\right) $.
Indeed, if $\left( I,J\right) \in \mathcal{P}\setminus \mathcal{P}^{\limfunc{%
red}}$ then $\pi I\not\in \limfunc{supp}\widehat{f}$ and so $\mathbb{E}%
_{I}^{\sigma }\bigtriangleup _{\pi I}^{\sigma }f=0$. Thus for the purpose of
bounding the stopping form, we may assume that the following additional
property holds for any $A$-admissible collection of pairs $\mathcal{P}$:

\begin{itemize}
\item if $\left( I,J\right) \in \mathcal{P}$ is maximal in the sense that $%
I\supset I^{\prime }$ for all $I^{\prime }$ satisfying $\left( I^{\prime
},J\right) \in \mathcal{P}$, then $I$ is good.
\end{itemize}

Note that there is an asymmetry in our definition of $\mathcal{P}^{\limfunc{%
red}}$ here, namely that the second components $J$ are required to be $%
\mathbf{\tau }$-good, while the maximal first components $I$ are required to
be good. Of course the treatment of the dual stopping forms will use the
reversed requirements, and this accounts for our symmetric restrictions
imposed on the quasiHaar supports of $f$ and $g$ at the outset of the proof.

\begin{definition}
\label{def reduced}We say that an admissible collection $\mathcal{P}$ is 
\emph{reduced} if $\mathcal{P}=\mathcal{P}^{\limfunc{red}}$, so that the
additional property above holds.
\end{definition}

Note that $\mathsf{B}_{\limfunc{stop}}^{A,\mathcal{P}}\left( f,g\right) =%
\mathsf{B}_{\limfunc{stop}}^{A,\mathcal{P}}\left( \mathsf{P}_{\mathcal{C}%
_{A}}^{\sigma }f,\mathsf{P}_{\mathcal{C}_{A}^{\mathbf{\tau }-\limfunc{shift}%
}}^{\omega }g\right) $. Recall that the deep quasienergy condition constant $%
\mathcal{E}_{\alpha }^{\limfunc{deep}}$ is given by%
\begin{equation*}
\left( \mathcal{E}_{\alpha }^{\limfunc{deep}}\right) ^{2}\equiv \sup_{I=\dot{%
\cup}I_{r}}\frac{1}{\left\vert I\right\vert _{\sigma }}\sum_{r=1}^{\infty
}\sum_{J\in \mathcal{M}_{\left( \mathbf{r},\varepsilon \right) -\limfunc{deep%
}}\left( I_{r}\right) }\left( \frac{\mathrm{P}^{\alpha }\left( J,\mathbf{1}%
_{I\setminus \gamma J}\sigma \right) }{\left\vert J\right\vert ^{\frac{1}{n}}%
}\right) ^{2}\left\Vert \mathsf{P}_{J}^{\limfunc{subgood},\omega }\mathbf{x}%
\right\Vert _{L^{2}\left( \omega \right) }^{2}\ .
\end{equation*}

\begin{proposition}
\label{stopping bound}Suppose that $A\in \mathcal{A}$ and that $\mathcal{P}$
is an $A$\emph{-admissible} collection of pairs. Then the stopping form $%
\mathsf{B}_{\limfunc{stop}}^{A,\mathcal{P}}$ satisfies the bound%
\begin{equation}
\left\vert \mathsf{B}_{\limfunc{stop}}^{A,\mathcal{P}}\left( f,g\right)
\right\vert \lesssim \left( \mathcal{E}_{\alpha }^{\limfunc{deep}}+\sqrt{%
A_{2}^{\alpha }}\right) \left( \left\Vert f\right\Vert _{L^{2}\left( \sigma
\right) }+\alpha _{\mathcal{A}}\left( f\right) \sqrt{\left\vert A\right\vert
_{\sigma }}\right) \left\Vert g\right\Vert _{L^{2}\left( \omega \right) }\ .
\label{B stop form}
\end{equation}
\end{proposition}

With this proposition in hand, we can complete the proof of (\ref{B stop
form 3}), and hence of Theorem \ref{T1 theorem}, by summing over the
stopping quasicubes $A\in \mathcal{A}$ with the choice $\mathcal{P}^{A}$ of $%
A$-admissible pairs for each $A$:%
\begin{eqnarray*}
&&\sum_{A\in \mathcal{A}}\left\vert \mathsf{B}_{\limfunc{stop}}^{A,\mathcal{P%
}^{A}}\left( f,g\right) \right\vert \\
&\lesssim &\sum_{A\in \mathcal{A}}\left( \mathcal{E}_{\alpha }^{\limfunc{deep%
}}+\sqrt{A_{2}^{\alpha }}\right) \left( \left\Vert \mathsf{P}_{\mathcal{C}%
_{A}}f\right\Vert _{L^{2}\left( \sigma \right) }+\alpha _{\mathcal{A}}\left(
f\right) \sqrt{\left\vert A\right\vert _{\sigma }}\right) \left\Vert \mathsf{%
P}_{\mathcal{C}_{A}^{\mathbf{\tau }-\limfunc{shift}}}g\right\Vert
_{L^{2}\left( \omega \right) } \\
&\lesssim &\left( \mathcal{E}_{\alpha }^{\limfunc{deep}}+\sqrt{A_{2}^{\alpha
}}\right) \left( \sum_{A\in \mathcal{A}}\left( \left\Vert \mathsf{P}_{%
\mathcal{C}_{A}}f\right\Vert _{L^{2}\left( \sigma \right) }^{2}+\alpha _{%
\mathcal{A}}\left( f\right) ^{2}\left\vert A\right\vert _{\sigma }\right)
\right) ^{\frac{1}{2}}\left( \sum_{A\in \mathcal{A}}\left\Vert \mathsf{P}_{%
\mathcal{C}_{A}^{\mathbf{\tau }-\limfunc{shift}}}g\right\Vert _{L^{2}\left(
\omega \right) }^{2}\right) ^{\frac{1}{2}} \\
&\lesssim &\left( \mathcal{E}_{\alpha }^{\limfunc{deep}}+\sqrt{A_{2}^{\alpha
}}\right) \left\Vert f\right\Vert _{L^{2}\left( \sigma \right) }\left\Vert
g\right\Vert _{L^{2}\left( \omega \right) }\ ,
\end{eqnarray*}%
by orthogonality $\sum_{A\in \mathcal{A}}\left\Vert \mathsf{P}_{\mathcal{C}%
_{A}}f\right\Vert _{L^{2}\left( \sigma \right) }^{2}\leq \left\Vert
f\right\Vert _{L^{2}\left( \sigma \right) }^{2}$ in corona projections $%
\mathcal{C}_{A}^{\sigma }$, `quasi' orthogonality $\sum_{A\in \mathcal{A}%
}\alpha _{\mathcal{A}}\left( f\right) ^{2}\left\vert A\right\vert _{\sigma
}\lesssim \left\Vert f\right\Vert _{L^{2}\left( \sigma \right) }^{2}$ in the
stopping quasicubes $\mathcal{A}$, and by the bounded overlap of the shifted
coronas $\mathcal{C}_{A}^{\mathbf{\tau }-\limfunc{shift}}$: 
\begin{equation*}
\sum_{A\in \mathcal{A}}\mathbf{1}_{\mathcal{C}_{A}^{\mathbf{\tau }-\limfunc{%
shift}}}\leq \mathbf{\tau 1}_{\Omega \mathcal{D}}.
\end{equation*}

To prove Proposition \ref{stopping bound}, we begin by letting $\Pi _{2}%
\mathcal{P}$ consist of the second components of the pairs in $\mathcal{P}$
and writing 
\begin{eqnarray*}
\mathsf{B}_{\limfunc{stop}}^{A,\mathcal{P}}\left( f,g\right) &=&\sum_{J\in
\Pi _{2}\mathcal{P}}\left\langle T_{\sigma }^{\alpha }\varphi _{J}^{\mathcal{%
P}},\bigtriangleup _{J}^{\omega }g\right\rangle _{\omega }; \\
\text{where }\varphi _{J}^{\mathcal{P}} &\equiv &\sum_{I\in \mathcal{C}%
_{A}^{\prime }:\ \left( I,J\right) \in \mathcal{P}}\mathbb{E}_{I}^{\sigma
}\left( \bigtriangleup _{\pi I}^{\sigma }f\right) \ \mathbf{1}_{A\setminus
I}\ .
\end{eqnarray*}%
By the tree-connected property of $\mathcal{P}$, and the telescoping
property of martingale differences, together with the bound $\alpha _{%
\mathcal{A}}\left( A\right) $ on the quasiaverages of $f$ in the corona $%
\mathcal{C}_{A}$, we have%
\begin{equation}
\left\vert \varphi _{J}^{\mathcal{P}}\right\vert \lesssim \alpha _{\mathcal{A%
}}\left( A\right) 1_{A\setminus I_{\mathcal{P}}\left( J\right) },
\label{phi bound}
\end{equation}%
where $I_{\mathcal{P}}\left( J\right) \equiv \dbigcap \left\{ I:\left(
I,J\right) \in \mathcal{P}\right\} $ is the smallest quasicube $I$ for which 
$\left( I,J\right) \in \mathcal{P}$. Another important property of these
functions is the sublinearity:%
\begin{equation}
\left\vert \varphi _{J}^{\mathcal{P}}\right\vert \leq \left\vert \varphi
_{J}^{\mathcal{P}_{1}}\right\vert +\left\vert \varphi _{J}^{\mathcal{P}%
_{2}}\right\vert ,\ \ \ \ \ \mathcal{P}=\mathcal{P}_{1}\dot{\cup}\mathcal{P}%
_{2}\ .  \label{phi sublinear}
\end{equation}%
Now apply the Monotonicity Lemma \ref{mono} to the inner product $%
\left\langle T_{\sigma }^{\alpha }\varphi _{J},\bigtriangleup _{J}^{\omega
}g\right\rangle _{\omega }$ to obtain%
\begin{eqnarray*}
\left\vert \left\langle T_{\sigma }^{\alpha }\varphi _{J},\bigtriangleup
_{J}^{\omega }g\right\rangle _{\omega }\right\vert &\lesssim &\frac{\mathrm{P%
}^{\alpha }\left( J,\left\vert \varphi _{J}\right\vert \mathbf{1}%
_{A\setminus I_{\mathcal{P}}\left( J\right) }\sigma \right) }{\left\vert
J\right\vert ^{\frac{1}{n}}}\left\Vert \bigtriangleup _{J}^{\omega }\mathbf{x%
}\right\Vert _{L^{2}\left( \omega \right) }\left\Vert \bigtriangleup
_{J}^{\omega }g\right\Vert _{L^{2}\left( \omega \right) } \\
&&+\frac{\mathrm{P}_{1+\delta }^{\alpha }\left( J,\left\vert \varphi
_{J}\right\vert \mathbf{1}_{A\setminus I_{\mathcal{P}}\left( J\right)
}\sigma \right) }{\left\vert J\right\vert ^{\frac{1}{n}}}\left\Vert \mathsf{P%
}_{J}^{\omega }\mathbf{x}\right\Vert _{L^{2}\left( \omega \right)
}\left\Vert \bigtriangleup _{J}^{\omega }g\right\Vert _{L^{2}\left( \omega
\right) }.
\end{eqnarray*}%
Thus we have%
\begin{eqnarray}
&&  \label{def split} \\
\left\vert \mathsf{B}_{\limfunc{stop}}^{A,\mathcal{P}}\left( f,g\right)
\right\vert &\leq &\sum_{J\in \Pi _{2}\mathcal{P}}\frac{\mathrm{P}%
_{1}^{\alpha }\left( J,\left\vert \varphi _{J}\right\vert \mathbf{1}%
_{A\setminus I_{\mathcal{P}}\left( J\right) }\sigma \right) }{\left\vert
J\right\vert ^{\frac{1}{n}}}\left\Vert \bigtriangleup _{J}^{\omega }\mathbf{x%
}\right\Vert _{L^{2}\left( \omega \right) }\left\Vert \bigtriangleup
_{J}^{\omega }g\right\Vert _{L^{2}\left( \omega \right) }  \notag \\
&&+\sum_{J\in \Pi _{2}\mathcal{P}}\frac{\mathrm{P}_{1+\delta }^{\alpha
}\left( J,\left\vert \varphi _{J}\right\vert \mathbf{1}_{A\setminus I_{%
\mathcal{P}}\left( J\right) }\sigma \right) }{\left\vert J\right\vert ^{%
\frac{1}{n}}}\left\Vert \mathsf{P}_{J}^{\omega }\mathbf{x}\right\Vert
_{L^{2}\left( \omega \right) }\left\Vert \bigtriangleup _{J}^{\omega
}g\right\Vert _{L^{2}\left( \omega \right) }  \notag \\
&\equiv &\left\vert \mathsf{B}\right\vert _{\limfunc{stop},1,\bigtriangleup
^{\omega }}^{A,\mathcal{P}}\left( f,g\right) +\left\vert \mathsf{B}%
\right\vert _{\limfunc{stop},1+\delta ,\mathsf{P}^{\omega }}^{A,\mathcal{P}%
}\left( f,g\right) ,  \notag
\end{eqnarray}%
where we have dominated the stopping form by two sublinear stopping forms
that involve the Poisson integrals of order $1$ and $1+\delta $
respectively, and where the smaller Poisson integral $\mathrm{P}_{1+\delta
}^{\alpha }$ is multiplied by the larger projection $\left\Vert \mathsf{P}%
_{J}^{\omega }\mathbf{x}\right\Vert _{L^{2}\left( \omega \right) }$. This
splitting turns out to be successful in separating the two energy terms from
the right hand side of the Energy Lemma, because of the two properties (\ref%
{phi bound}) and (\ref{phi sublinear}) above. It remains to show the two
inequalities:%
\begin{equation}
\left\vert \mathsf{B}\right\vert _{\limfunc{stop},1,\bigtriangleup ^{\omega
}}^{A,\mathcal{P}}\left( f,g\right) \lesssim \left( \mathcal{E}_{\alpha }^{%
\limfunc{deep}}+\sqrt{A_{2}^{\alpha }}\right) \alpha _{\mathcal{A}}\left(
A\right) \sqrt{\left\vert A\right\vert _{\sigma }}\left\Vert g\right\Vert
_{L^{2}\left( \omega \right) },  \label{First inequality}
\end{equation}%
for $f\in L^{2}\left( \sigma \right) $ satisfying where $\mathbb{E}%
_{I}^{\sigma }\left\vert f\right\vert \leq \alpha _{\mathcal{A}}\left(
A\right) $ for all $I\in \mathcal{C}_{A}$; and%
\begin{equation}
\left\vert \mathsf{B}\right\vert _{\limfunc{stop},1+\delta ,\mathsf{P}%
^{\omega }}^{A,\mathcal{P}}\left( f,g\right) \lesssim \left( \mathcal{E}%
_{\alpha }^{\limfunc{deep}}+\sqrt{A_{2}^{\alpha }}\right) \left\Vert
f\right\Vert _{L^{2}\left( \sigma \right) }\left\Vert g\right\Vert
_{L^{2}\left( \omega \right) },  \label{Second inequality}
\end{equation}%
where we only need the case $\mathcal{P}=\mathcal{P}^{A}$ in this latter
inequality as there is no recursion involved in treating this second
sublinear form. We consider first the easier inequality (\ref{Second
inequality}) that does not require recursion. In the subsequent subsections
we will control the more difficult inequality (\ref{First inequality}) by
adapting the stopping time and recursion of M. Lacey to the sublinear form $%
\left\vert \mathsf{B}\right\vert _{\limfunc{stop},1,\bigtriangleup ^{\omega
}}^{A,\mathcal{P}}\left( f,g\right) $.

\subsection{The second inequality}

Now we turn to proving (\ref{Second inequality}), i.e.%
\begin{equation*}
\left\vert \mathsf{B}\right\vert _{\limfunc{stop},1+\delta ,\mathsf{P}%
^{\omega }}^{A,\mathcal{P}}\left( f,g\right) \lesssim \left( \mathcal{E}%
_{\alpha }^{\limfunc{deep}}+\sqrt{A_{2}^{\alpha }}\right) \left\Vert
f\right\Vert _{L^{2}\left( \sigma \right) }\left\Vert g\right\Vert
_{L^{2}\left( \omega \right) },
\end{equation*}%
where since 
\begin{equation*}
\left\vert \varphi _{J}\right\vert =\left\vert \sum_{I\in \mathcal{C}%
_{A}^{\prime }:\ \left( I,J\right) \in \mathcal{P}}\mathbb{E}_{I}^{\sigma
}\left( \bigtriangleup _{\pi I}^{\sigma }f\right) \ \mathbf{1}_{A\setminus
I}\right\vert \leq \sum_{I\in \mathcal{C}_{A}^{\prime }:\ \left( I,J\right)
\in \mathcal{P}}\left\vert \mathbb{E}_{I}^{\sigma }\left( \bigtriangleup
_{\pi I}^{\sigma }f\right) \ \mathbf{1}_{A\setminus I}\right\vert ,
\end{equation*}%
the sublinear form $\left\vert \mathsf{B}\right\vert _{\limfunc{stop}%
,1+\delta ,\mathsf{P}^{\omega }}^{A,\mathcal{P}}$ can be dominated and then
decomposed by pigeonholing the ratio of side lengths of $J$ and $I$:%
\begin{eqnarray*}
\left\vert \mathsf{B}\right\vert _{\limfunc{stop},1+\delta ,\mathsf{P}%
^{\omega }}^{A,\mathcal{P}}\left( f,g\right) &=&\sum_{J\in \Pi _{2}\mathcal{P%
}}\frac{\mathrm{P}_{1+\delta }^{\alpha }\left( J,\left\vert \varphi
_{J}\right\vert \mathbf{1}_{A\setminus I_{\mathcal{P}}\left( J\right)
}\sigma \right) }{\left\vert J\right\vert ^{\frac{1}{n}}}\left\Vert \mathsf{P%
}_{J}^{\omega }\mathbf{x}\right\Vert _{L^{2}\left( \omega \right)
}\left\Vert \bigtriangleup _{J}^{\omega }g\right\Vert _{L^{2}\left( \omega
\right) } \\
&\leq &\sum_{\left( I,J\right) \in \mathcal{P}}\frac{\mathrm{P}_{1+\delta
}^{\alpha }\left( J,\left\vert \mathbb{E}_{I}^{\sigma }\left( \bigtriangleup
_{\pi I}^{\sigma }f\right) \right\vert \mathbf{1}_{A\setminus I}\sigma
\right) }{\left\vert J\right\vert ^{\frac{1}{n}}}\left\Vert \mathsf{P}%
_{J}^{\omega }\mathbf{x}\right\Vert _{L^{2}\left( \omega \right) }\left\Vert
\bigtriangleup _{J}^{\omega }g\right\Vert _{L^{2}\left( \omega \right) } \\
&\equiv &\sum_{s=0}^{\infty }\left\vert \mathsf{B}\right\vert _{\limfunc{stop%
},1+\delta ,\mathsf{P}^{\omega }}^{A,\mathcal{P};s}\left( f,g\right) ; \\
\left\vert \mathsf{B}\right\vert _{\limfunc{stop},1+\delta ,\mathsf{P}%
^{\omega }}^{A,\mathcal{P};s}\left( f,g\right) &\equiv &\sum_{\substack{ %
\left( I,J\right) \in \mathcal{P}  \\ \ell \left( J\right) =2^{-s}\ell
\left( I\right) }}\frac{\mathrm{P}_{1+\delta }^{\alpha }\left( J,\left\vert 
\mathbb{E}_{I}^{\sigma }\left( \bigtriangleup _{\pi I}^{\sigma }f\right)
\right\vert \mathbf{1}_{A\setminus I}\sigma \right) }{\left\vert
J\right\vert ^{\frac{1}{n}}}\left\Vert \mathsf{P}_{J}^{\omega }\mathbf{x}%
\right\Vert _{L^{2}\left( \omega \right) }\left\Vert \bigtriangleup
_{J}^{\omega }g\right\Vert _{L^{2}\left( \omega \right) }.
\end{eqnarray*}%
Here we have the \emph{entire} projection $\mathsf{P}_{J}^{\omega }\mathbf{x}
$ onto all of the dyadic subquasicubes of $J$, but this is offset by the
smaller Poisson integral $\mathrm{P}_{1+\delta }^{\alpha }$. We will now
adapt the argument for the stopping term starting on page 42 of \cite%
{LaSaUr2}, where the geometric gain from the assumed Energy Hypothesis there
will be replaced by a geometric gain from the smaller Poisson integral $%
\mathrm{P}_{1+\delta }^{\alpha }$ used here.

First, we exploit the additional decay in the Poisson integral $\mathrm{P}%
_{1+\delta }^{\alpha }$ as follows. Suppose that $\left( I,J\right) \in 
\mathcal{P}$ with $\ell \left( J\right) =2^{-s}\ell \left( I\right) $. We
then compute%
\begin{eqnarray*}
\frac{\mathrm{P}_{1+\delta }^{\alpha }\left( J,\mathbf{1}_{A\setminus
I}\sigma \right) }{\left\vert J\right\vert ^{\frac{1}{n}}} &\approx
&\int_{A\setminus I}\frac{\left\vert J\right\vert ^{\frac{\delta }{n}}}{%
\left\vert y-c_{J}\right\vert ^{n+1+\delta -\alpha }}d\sigma \left( y\right)
\\
&\leq &\int_{A\setminus I}\left( \frac{\left\vert J\right\vert ^{\frac{1}{n}}%
}{\limfunc{qdist}\left( c_{J},I^{c}\right) }\right) ^{\delta }\frac{1}{%
\left\vert y-c_{J}\right\vert ^{n+1-\alpha }}d\sigma \left( y\right) \\
&\lesssim &\left( \frac{\left\vert J\right\vert ^{\frac{1}{n}}}{\limfunc{%
qdist}\left( c_{J},I^{c}\right) }\right) ^{\delta }\frac{\mathrm{P}^{\alpha
}\left( J,\mathbf{1}_{A\setminus I}\sigma \right) }{\left\vert J\right\vert
^{\frac{1}{n}}},
\end{eqnarray*}%
and use the goodness inequality,%
\begin{equation*}
\limfunc{qdist}\left( c_{J},I^{c}\right) \geq \frac{1}{2}\ell \left(
I\right) ^{1-\varepsilon }\ell \left( J\right) ^{\varepsilon }\geq \frac{1}{2%
}2^{s\left( 1-\varepsilon \right) }\ell \left( J\right) ,
\end{equation*}%
to conclude that%
\begin{equation}
\left( \frac{\mathrm{P}_{1+\delta }^{\alpha }\left( J,\mathbf{1}_{A\setminus
I}\sigma \right) }{\left\vert J\right\vert ^{\frac{1}{n}}}\right) \lesssim
2^{-s\delta \left( 1-\varepsilon \right) }\frac{\mathrm{P}^{\alpha }\left( J,%
\mathbf{1}_{A\setminus I}\sigma \right) }{\left\vert J\right\vert ^{\frac{1}{%
n}}}.  \label{Poisson decay}
\end{equation}

We next claim that for $s\geq 0$ an integer,%
\begin{eqnarray*}
\left\vert \mathsf{B}\right\vert _{\limfunc{stop},1+\delta ,\mathsf{P}%
^{\omega }}^{A,\mathcal{P};s}\left( f,g\right) &=&\sum_{\substack{ \left(
I,J\right) \in \mathcal{P}  \\ \ell \left( J\right) =2^{-s}\ell \left(
I\right) }}\frac{\mathrm{P}_{1+\delta }^{\alpha }\left( J,\left\vert \mathbb{%
E}_{I}^{\sigma }\left( \bigtriangleup _{\pi I}^{\sigma }f\right) \right\vert 
\mathbf{1}_{A\setminus I}\sigma \right) }{\left\vert J\right\vert ^{\frac{1}{%
n}}}\left\Vert \mathsf{P}_{J}^{\omega }\mathbf{x}\right\Vert _{L^{2}\left(
\omega \right) }\left\Vert \bigtriangleup _{J}^{\omega }g\right\Vert
_{L^{2}\left( \omega \right) } \\
&\lesssim &2^{-s\delta \left( 1-\varepsilon \right) }\ \left( \mathcal{E}%
_{\alpha }^{\limfunc{deep}}+\sqrt{A_{2}^{\alpha }}\right) \ \left\Vert
f\right\Vert _{L^{2}\left( \sigma \right) }\ \left\Vert g\right\Vert
_{L^{2}\left( \omega \right) }\,,
\end{eqnarray*}%
from which (\ref{Second inequality}) follows upon summing in $s\geq 0$. Now
using both%
\begin{eqnarray*}
\left\vert \mathbb{E}_{I}^{\sigma }\left( \bigtriangleup _{\pi I}^{\sigma
}f\right) \right\vert &=&\frac{1}{\left\vert I\right\vert _{\sigma }}%
\int_{I}\left\vert \bigtriangleup _{\pi I}^{\sigma }f\right\vert d\sigma
\leq \left\Vert \bigtriangleup _{\pi I}^{\sigma }f\right\Vert _{L^{2}\left(
\sigma \right) }\frac{1}{\sqrt{\left\vert I\right\vert _{\sigma }}}, \\
2^{n}\left\Vert f\right\Vert _{L^{2}(\sigma )}^{2} &=&\sum_{I\in \Omega 
\mathcal{D}}\left\Vert \bigtriangleup _{\pi I}^{\sigma }f\right\Vert
_{L^{2}\left( \sigma \right) }^{2}\ ,
\end{eqnarray*}%
we apply Cauchy-Schwarz in the $I$ variable above to see that 
\begin{eqnarray*}
&&\left[ \left\vert \mathsf{B}\right\vert _{\limfunc{stop},1+\delta ,\mathsf{%
P}^{\omega }}^{A,\mathcal{P};s}\left( f,g\right) \right] ^{2}\lesssim
\left\Vert f\right\Vert _{L^{2}(\sigma )}^{2} \\
&&\times \left[ \sum_{I\in \mathcal{C}_{A}^{\prime }}\left( \frac{1}{\sqrt{%
\left\vert I\right\vert _{\sigma }}}\sum_{\substack{ J:\ \left( I,J\right)
\in \mathcal{P}  \\ \ell \left( J\right) =2^{-s}\ell \left( I\right) }}\frac{%
\mathrm{P}_{1+\delta }^{\alpha }\left( J,\mathbf{1}_{A\setminus I}\sigma
\right) }{\left\vert J\right\vert ^{\frac{1}{n}}}\left\Vert \mathsf{P}%
_{J}^{\omega }\mathbf{x}\right\Vert _{L^{2}\left( \omega \right) }\left\Vert
\bigtriangleup _{J}^{\omega }g\right\Vert _{L^{2}\left( \omega \right)
}\right) ^{2}\right] ^{\frac{1}{2}}.
\end{eqnarray*}%
We can then estimate the sum inside the square brackets by%
\begin{equation*}
\sum_{I\in \mathcal{C}_{A}^{\prime }}\left\{ \sum_{\substack{ J:\ \left(
I,J\right) \in \mathcal{P}  \\ \ell \left( J\right) =2^{-s}\ell \left(
I\right) }}\left\Vert \bigtriangleup _{J}^{\omega }g\right\Vert
_{L^{2}\left( \omega \right) }^{2}\right\} \sum_{\substack{ J:\ \left(
I,J\right) \in \mathcal{P}  \\ \ell \left( J\right) =2^{-s}\ell \left(
I\right) }}\frac{1}{\left\vert I\right\vert _{\sigma }}\left( \frac{\mathrm{P%
}_{1+\delta }^{\alpha }\left( J,\mathbf{1}_{A\setminus I}\sigma \right) }{%
\left\vert J\right\vert ^{\frac{1}{n}}}\right) ^{2}\left\Vert \mathsf{P}%
_{J}^{\omega }\mathbf{x}\right\Vert _{L^{2}\left( \omega \right)
}^{2}\lesssim \left\Vert g\right\Vert _{L^{2}\left( \omega \right)
}^{2}A\left( s\right) ^{2},
\end{equation*}%
where%
\begin{equation*}
A\left( s\right) ^{2}\equiv \sup_{I\in \mathcal{C}_{A}^{\prime }}\sum 
_{\substack{ J:\ \left( I,J\right) \in \mathcal{P}  \\ \ell \left( J\right)
=2^{-s}\ell \left( I\right) }}\frac{1}{\left\vert I\right\vert _{\sigma }}%
\left( \frac{\mathrm{P}_{1+\delta }^{\alpha }\left( J,\mathbf{1}_{A\setminus
I}\sigma \right) }{\left\vert J\right\vert ^{\frac{1}{n}}}\right)
^{2}\left\Vert \mathsf{P}_{J}^{\omega }\mathbf{x}\right\Vert _{L^{2}\left(
\omega \right) }^{2}\,.
\end{equation*}%
Finally then we turn to the analysis of the supremum in last display. From
the Poisson decay (\ref{Poisson decay}) we have 
\begin{eqnarray*}
A\left( s\right) ^{2} &\lesssim &\sup_{I\in \mathcal{C}_{A}^{\prime }}\frac{1%
}{\left\vert I\right\vert _{\sigma }}2^{-2s\delta \left( 1-\varepsilon
\right) }\sum_{\substack{ J:\ \left( I,J\right) \in \mathcal{P}  \\ \ell
\left( J\right) =2^{-s}\ell \left( I\right) }}\left( \frac{\mathrm{P}%
^{\alpha }\left( J,\mathbf{1}_{A\setminus I}\sigma \right) }{\left\vert
J\right\vert ^{\frac{1}{n}}}\right) ^{2}\left\Vert \mathsf{P}_{J}^{\omega
}x\right\Vert _{L^{2}\left( \omega \right) }^{2} \\
&\lesssim &\sup_{I\in \mathcal{C}_{A}^{\prime }}\frac{1}{\left\vert
I\right\vert _{\sigma }}2^{-2s\delta \left( 1-\varepsilon \right)
}\sum_{K\in \mathcal{M}_{\left( \mathbf{r},\varepsilon \right) -\limfunc{deep%
}}\left( I\right) }\left( \frac{\mathrm{P}^{\alpha }\left( K,\mathbf{1}%
_{A\setminus I}\sigma \right) }{\left\vert K\right\vert ^{\frac{1}{n}}}%
\right) ^{2}\sum_{\substack{ J\subset K:\ \left( I,J\right) \in \mathcal{P} 
\\ \ell \left( J\right) =2^{-s}\ell \left( I\right) }}\left\Vert \mathsf{P}%
_{J}^{\omega }x\right\Vert _{L^{2}\left( \omega \right) }^{2} \\
&\lesssim &2^{-2s\delta \left( 1-\varepsilon \right) }\left[ \left( \mathcal{%
E}_{\alpha }^{\limfunc{deep}}\right) ^{2}+A_{2}^{\alpha }\right] \,,
\end{eqnarray*}%
where the last inequality is the one for which the definition of quasienergy
stopping quasicubes was designed. Indeed, from Definition~\ref{def energy
corona 3}, as $(I,J)\in \mathcal{P}$, we have that $I$ is \emph{not} a
stopping quasicube in $\mathcal{A}$, and hence that (\ref{def stop 3}) \emph{%
fails} to hold, delivering the estimate above since $J\Subset _{\mathbf{\rho
-1},\varepsilon }I$ good must be contained in some $K\in \mathcal{M}_{\left( 
\mathbf{r},\varepsilon \right) -\limfunc{deep}}\left( I\right) $, and since $%
\frac{\mathrm{P}^{\alpha }\left( J,\mathbf{1}_{A\setminus I}\sigma \right) }{%
\left\vert J\right\vert ^{\frac{1}{n}}}\approx \frac{\mathrm{P}^{\alpha
}\left( K,\mathbf{1}_{A\setminus I}\sigma \right) }{\left\vert K\right\vert
^{\frac{1}{n}}}$. The terms $\left\Vert \mathsf{P}_{J}^{\omega }x\right\Vert
_{L^{2}\left( \omega \right) }^{2}$ are additive since the $J^{\prime }s$
are pigeonholed by $\ell \left( J\right) =2^{-s}\ell \left( I\right) $.

\subsection{The first inequality and the recursion of M. Lacey}

Now we turn to proving the more difficult inequality (\ref{First inequality}%
). Recall that in dimension $n=1$ the energy condition%
\begin{equation*}
\sum_{n=1}^{\infty }\left\vert J_{n}\right\vert _{\omega }\mathsf{E}\left(
J_{n},\omega \right) ^{2}\mathrm{P}\left( J_{n},\mathbf{1}_{I}\sigma \right)
^{2}\lesssim \left( \mathcal{NTV}\right) \ \left\vert I\right\vert _{\sigma
},\ \ \ \ \ \overset{\cdot }{\mathop{\displaystyle \bigcup }}_{n=1}^{\infty
}J_{n}\subset I,
\end{equation*}%
could not be used in the NTV argument, because the set functional $%
J\rightarrow \left\vert J\right\vert _{\omega }\mathsf{E}\left( J,\omega
\right) ^{2}$ failed to be superadditive. On the other hand, the pivotal
condition of NTV,%
\begin{equation*}
\sum_{n=1}^{\infty }\left\vert J_{n}\right\vert _{\omega }\mathrm{P}\left(
J_{n},\mathbf{1}_{I}\sigma \right) ^{2}\lesssim \left\vert I\right\vert
_{\sigma },\ \ \ \ \ \overset{\cdot }{\mathop{\displaystyle \bigcup }}%
_{n=1}^{\infty }J_{n}\subset I,
\end{equation*}%
succeeded in the NTV argument because the set functional $J\rightarrow
\left\vert J\right\vert _{\omega }$ is trivially superadditive, indeed
additive. The final piece of the argument needed to prove the NTV conjecture
was found by M. Lacey in \cite{Lac}, and amounts to first replacing the
additivity of the functional $J\rightarrow \left\vert J\right\vert _{\omega
} $ with the additivity of the projection functional $\mathcal{H}\rightarrow
\left\Vert \mathsf{P}_{\mathcal{H}}^{\omega }x\right\Vert _{L^{2}\left(
\omega \right) }^{2}$ defined on subsets $\mathcal{H}$ of the dyadic
quasigrid $\Omega \mathcal{D}$. Then a stopping time argument relative to
this more subtle functional, together with a clever recursion, constitute
the main new ingredients in Lacey's argument \cite{Lac}.

To begin the extension to a more general Calder\'{o}n-Zygmund operator $%
T^{\alpha }$, we also recall the stopping quasienergy generalized to higher
dimensions by 
\begin{equation*}
\mathbf{X}^{\alpha }\left( \mathcal{C}_{A}\right) ^{2}\equiv \sup_{I\in 
\mathcal{C}_{A}}\frac{1}{\left\vert I\right\vert _{\sigma }}\sum_{J\in 
\mathcal{M}_{\left( \mathbf{r},\varepsilon \right) -\limfunc{deep}}\left(
I\right) }\left( \frac{\mathrm{P}^{\alpha }\left( J,\mathbf{1}_{A\setminus
\gamma J}\sigma \right) }{\left\vert J\right\vert ^{\frac{1}{n}}}\right)
^{2}\left\Vert \mathsf{P}_{J}^{\limfunc{subgood},\omega }\mathbf{x}%
\right\Vert _{L^{2}\left( \omega \right) }^{2}\ ,
\end{equation*}%
where $\mathcal{M}_{\left( \mathbf{r},\varepsilon \right) -\limfunc{deep}%
}\left( I\right) $ is the set of maximal $\mathbf{r}$-deeply embedded
subquasicubes of $I$ where $\mathbf{r}$ is the goodness parameter. What now
follows is an adaptation to our deep quasienergy condition and the sublinear
form $\left\vert \mathsf{B}\right\vert _{\limfunc{stop},1,\bigtriangleup
^{\omega }}^{A,\mathcal{P}}$ of the arguments of M. Lacey in \cite{Lac}. We
have the following Poisson inequality for quasicubes $B\subset A\subset I$:%
\begin{eqnarray}
\frac{\mathrm{P}^{\alpha }\left( A,\mathbf{1}_{I\setminus A}\sigma \right) }{%
\left\vert A\right\vert ^{\frac{1}{n}}} &\approx &\int_{I\setminus A}\frac{1%
}{\left( \left\vert y-c_{A}\right\vert \right) ^{n+1-\alpha }}d\sigma \left(
y\right)  \label{BAI} \\
&\lesssim &\int_{I\setminus A}\frac{1}{\left( \left\vert y-c_{B}\right\vert
\right) ^{n+1-\alpha }}d\sigma \left( y\right) \approx \frac{\mathrm{P}%
^{\alpha }\left( B,\mathbf{1}_{I\setminus A}\sigma \right) }{\left\vert
B\right\vert ^{\frac{1}{n}}}.  \notag
\end{eqnarray}

\subsection{The stopping energy}

Fix $A\in \mathcal{A}$. We will use a `decoupled' modification of the
stopping energy $\mathbf{X}\left( \mathcal{C}_{A}\right) $. Suppose that $%
\mathcal{P}$ is an $A$-admissible collection of pairs of quasicubes in the
product set $\Omega \mathcal{D}\times \Omega \mathcal{D}_{\limfunc{good}}$
of pairs of dyadic quasicubes in $\mathbb{R}^{n}$ with second component
good. For an admissible collection $\mathcal{P}$ let $\Pi _{1}\mathcal{P}$
and $\Pi _{2}\mathcal{P}$ be the quasicubes in the first and second
components of the pairs in $\mathcal{P}$ respectively, let $\Pi \mathcal{P}%
\equiv \Pi _{1}\mathcal{P}\cup \Pi _{2}\mathcal{P}$, and for $K\in \Pi 
\mathcal{P}$ define the $\mathbf{\tau }$-$\limfunc{deep}$ projection of $%
\mathcal{P}$ relative to $K$ by 
\begin{equation*}
\Pi _{2}^{K,\mathbf{\tau }-\limfunc{deep}}\mathcal{P}\equiv \left\{ J\in \Pi
_{2}\mathcal{P}:\ J\Subset _{\mathbf{\tau },\varepsilon }K\right\} .
\end{equation*}%
Now the quasicubes $J$ in$\ \Pi _{2}\mathcal{P}$ are of course \emph{always}
good, but this is \emph{not} the case for quasicubes $I$ in $\Pi _{1}%
\mathcal{P}$. Indeed, the collection $\mathcal{P}$ is tree-connected in the
first component, and it is clear that there can be many \emph{bad}
quasicubes in a connected geodesic in the tree $\Omega \mathcal{D}$. But the
quasiHaar support of $f$ is contained in \emph{good} quasicubes $I$, and we
have also assumed that the children of these quasicubes $I$ are good as
well. As a consequence we may always assume that our $A$-admissible
collections $\mathcal{P}$ are reduced in the sense of Definition \ref{def
reduced}. Thus we will use as our `size testing collection' of quasicubes
for $\mathcal{P}$ the collection 
\begin{equation*}
\Pi ^{\limfunc{goodbelow}}\mathcal{P}\equiv \left\{ K^{\prime }\in \Omega 
\mathcal{D}:K^{\prime }\text{ is }\limfunc{good}\text{ and }K^{\prime
}\subset K\text{ for some }K\in \Pi \mathcal{P}\right\} ,
\end{equation*}%
which consists of all the good subquasicubes of any quasicube in $\Pi 
\mathcal{P}$. Note that the maximal quasicubes in $\Pi \mathcal{P}=\Pi 
\mathcal{P}^{\limfunc{red}}$ are already good themselves, and so we have the
important property that%
\begin{equation}
\left( I,J\right) \in \mathcal{P}=\mathcal{P}^{\limfunc{red}}\text{ implies }%
I\subset K\text{ for some quasicube }K\in \Pi ^{\limfunc{goodbelow}}\mathcal{%
P}.  \label{inclusive}
\end{equation}%
Now define the `size functional' $\mathcal{S}_{\limfunc{size}}^{\alpha
,A}\left( \mathcal{P}\right) $ of $\mathcal{P}$ as follows. Recall that a
projection $\mathsf{P}_{\mathcal{H}}^{\omega }$ on $\mathbf{x}$ satisfies 
\begin{equation*}
\left\Vert \mathsf{P}_{\mathcal{H}}^{\omega }\mathbf{x}\right\Vert
_{L^{2}\left( \omega \right) }^{2}=\sum_{J\in \mathcal{H}}\left\Vert
\bigtriangleup _{J}^{\omega }\mathbf{x}\right\Vert _{L^{2}\left( \omega
\right) }^{2}.
\end{equation*}

\begin{definition}
\label{def ext size}If $\mathcal{P}$ is $A$-admissible, define%
\begin{equation}
\mathcal{S}_{\limfunc{size}}^{\alpha ,A}\left( \mathcal{P}\right) ^{2}\equiv
\sup_{K\in \Pi ^{\limfunc{goodbelow}}\mathcal{P}}\frac{1}{\left\vert
K\right\vert _{\sigma }}\left( \frac{\mathrm{P}^{\alpha }\left( K,\mathbf{1}%
_{A\setminus K}\sigma \right) }{\left\vert K\right\vert ^{\frac{1}{n}}}%
\right) ^{2}\left\Vert \mathsf{P}_{\Pi _{2}^{K,\mathbf{\tau }-\limfunc{deep}}%
\mathcal{P}}^{\omega }\mathbf{x}\right\Vert _{L^{2}\left( \omega \right)
}^{2}.  \label{def P stop energy 3}
\end{equation}
\end{definition}

We should remark that that the quasicubes $K$ in $\Pi ^{\limfunc{goodbelow}}%
\mathcal{P}$ that fail to contain any $\mathbf{\tau }$-parents of quasicubes
from $\Pi _{2}\mathcal{P}$ will not contribute to the size functional since $%
\Pi _{2}^{K,\mathbf{\tau }-\limfunc{deep}}\mathcal{P}$ is empty in this
case. We note three essential properties of this definition of size
functional:

\begin{enumerate}
\item \textbf{Monotonicity} of size: $\mathcal{S}_{\limfunc{size}}^{\alpha
,A}\left( \mathcal{P}\right) \leq \mathcal{S}_{\limfunc{size}}^{\alpha
,A}\left( \mathcal{Q}\right) $ if $\mathcal{P}\subset \mathcal{Q}$,

\item \textbf{Goodness} of testing quasicubes: $\Pi ^{\limfunc{goodbelow}}%
\mathcal{P}\subset \Omega \mathcal{D}_{\limfunc{good}}$,

\item \textbf{Control} by deep quasienergy condition: $\mathcal{S}_{\limfunc{%
size}}^{\alpha ,A}\left( \mathcal{P}\right) \lesssim \mathcal{E}_{\alpha }^{%
\limfunc{deep}}+\sqrt{A_{2}^{\alpha }}$.
\end{enumerate}

The monotonicity property follows from $\Pi ^{\limfunc{goodbelow}}\mathcal{P}%
\subset \Pi ^{\limfunc{goodbelow}}\mathcal{Q}$ and $\Pi _{2}^{K,\mathbf{\tau 
}-\limfunc{deep}}\mathcal{P}\subset \Pi _{2}^{K,\mathbf{\tau }-\limfunc{deep}%
}\mathcal{Q}$, and the goodness property follows from the definition of $\Pi
^{\limfunc{goodbelow}}\mathcal{P}$. The control property is contained in the
next lemma, which uses the stopping quasienergy control for the form $%
\mathsf{B}_{stop}^{A}\left( f,g\right) $ associated with $A$.

\begin{lemma}
\label{energy control}If $\mathcal{P}^{A}$ is as in (\ref{initial P}) and $%
\mathcal{P}\subset \mathcal{P}^{A}$, then 
\begin{equation*}
\mathcal{S}_{\limfunc{size}}^{\alpha ,A}\left( \mathcal{P}\right) \leq 
\mathbf{X}_{\alpha }\left( \mathcal{C}_{A}\right) \lesssim \mathcal{E}%
_{\alpha }^{\limfunc{deep}}+\sqrt{A_{2}^{\alpha }}+\sqrt{A_{2}^{\alpha ,%
\limfunc{punct}}}\ .
\end{equation*}
\end{lemma}

\begin{proof}
Suppose that $K\in \Pi ^{\limfunc{goodbelow}}\mathcal{P}$. To prove the
first inequality in the statement we note that%
\begin{eqnarray*}
&&\frac{1}{\left\vert K\right\vert _{\sigma }}\left( \frac{\mathrm{P}%
^{\alpha }\left( K,\mathbf{1}_{A\setminus K}\sigma \right) }{\left\vert
K\right\vert ^{\frac{1}{n}}}\right) ^{2}\left\Vert \mathsf{P}_{\left( \Pi
_{2}^{K,\mathbf{\tau }-\limfunc{deep}}\mathcal{P}\right) ^{\ast }}^{\omega }%
\mathbf{x}\right\Vert _{L^{2}\left( \omega \right) }^{2} \\
&\leq &\frac{1}{\left\vert K\right\vert _{\sigma }}\left( \frac{\mathrm{P}%
^{\alpha }\left( K,\mathbf{1}_{A\setminus K}\sigma \right) }{\left\vert
K\right\vert ^{\frac{1}{n}}}\right) ^{2}\sum_{J\in \mathcal{M}_{\left( 
\mathbf{r},\varepsilon \right) -\limfunc{deep}}\left( K\right) }\left\Vert 
\mathsf{P}_{J}^{\limfunc{subgood},\omega }\mathbf{x}\right\Vert
_{L^{2}\left( \omega \right) }^{2} \\
&\lesssim &\frac{1}{\left\vert K\right\vert _{\sigma }}\sum_{J\in \mathcal{M}%
_{\left( \mathbf{r},\varepsilon \right) -\limfunc{deep}}\left( K\right)
}\left( \frac{\mathrm{P}^{\alpha }\left( J,\mathbf{1}_{A\setminus K}\sigma
\right) }{\left\vert J\right\vert ^{\frac{1}{n}}}\right) ^{2}\left\Vert 
\mathsf{P}_{J}^{\limfunc{subgood},\omega }\mathbf{x}\right\Vert
_{L^{2}\left( \omega \right) }^{2} \\
&\lesssim &\frac{1}{\left\vert K\right\vert _{\sigma }}\sum_{J\in \mathcal{M}%
_{\left( \mathbf{r},\varepsilon \right) -\limfunc{deep}}\left( K\right)
}\left( \frac{\mathrm{P}^{\alpha }\left( J,\mathbf{1}_{A\setminus \gamma
J}\sigma \right) }{\left\vert J\right\vert ^{\frac{1}{n}}}\right)
^{2}\left\Vert \mathsf{P}_{J}^{\limfunc{subgood},\omega }\mathbf{x}%
\right\Vert _{L^{2}\left( \omega \right) }^{2}\leq \mathbf{X}_{\alpha
}\left( \mathcal{C}_{A}\right) ^{2},
\end{eqnarray*}%
where the first inequality above follows since every $J^{\prime }\in \Pi
_{2}^{K,\mathbf{\tau }-\limfunc{deep}}\mathcal{P}$ is contained in some $%
J\in \mathcal{M}_{\left( \mathbf{r},\varepsilon \right) -\limfunc{deep}%
}\left( I\right) $, the second inequality follows from (\ref{BAI}) with $%
J\subset K\subset A$, and then the third inequality follows since $J\Subset
_{\mathbf{r},\varepsilon }I$ implies $\gamma J\subset I$ by (\ref{bounded
overlap in K}), and finally since $\Pi _{2}^{K,\mathbf{\tau }-\limfunc{deep}}%
\mathcal{P}=\emptyset $ if $K\subset A$ and $K\notin \mathcal{C}_{A}$ by (%
\ref{later use}) below. The second inequality in the statement of the lemma
follows from (\ref{def stopping bounds 3}).
\end{proof}

The following useful fact is needed above and will be used later as well:%
\begin{equation}
K\subset A\text{ and }K\notin \mathcal{C}_{A}\Longrightarrow \Pi _{2}^{K,%
\mathbf{\tau }-\limfunc{deep}}\mathcal{P}=\emptyset \ .  \label{later use}
\end{equation}%
To see this, suppose that $K\in \mathcal{C}_{A}^{\mathbf{\tau }-\limfunc{%
shift}}\setminus \mathcal{C}_{A}$. Then $K\subset A^{\prime }$ for some $%
A^{\prime }\in \mathfrak{C}_{\mathcal{A}}\left( A\right) $, and so if there
is $J\in \Pi _{2}^{K,\mathbf{\tau }-\limfunc{deep}}\mathcal{P}$, then $\ell
\left( J\right) \leq 2^{-\mathbf{\tau }}\ell \left( K\right) \leq 2^{-%
\mathbf{\tau }}\ell \left( A^{\prime }\right) $, which implies that $J\notin 
\mathcal{C}_{A}^{\mathbf{\tau }-\limfunc{shift}}$, which contradicts $\Pi
_{2}^{K,\mathbf{\tau }-\limfunc{deep}}\mathcal{P}\subset \mathcal{C}_{A}^{%
\mathbf{\tau }-\limfunc{shift}}$.

We remind the reader again that $c\left\vert J\right\vert ^{\frac{1}{n}}\leq
\ell \left( J\right) \leq C\left\vert J\right\vert ^{\frac{1}{n}}$ for any
quasicube $J$, and that we will generally use $\left\vert J\right\vert ^{%
\frac{1}{n}}$ in the Poisson integrals and estimates, but will usually use $%
\ell \left( J\right) $ when defining collections of quasicubes. Now define
an atomic measure $\omega _{\mathcal{P}}$ in the upper half space $\mathbb{R}%
_{+}^{n+1}$ by%
\begin{equation*}
\omega _{\mathcal{P}}\equiv \sum_{J\in \Pi _{2}\mathcal{P}}\left\Vert
\bigtriangleup _{J}^{\omega }\mathbf{x}\right\Vert _{L^{2}\left( \omega
\right) }^{2}\ \delta _{\left( c_{J},\ell \left( J\right) \right) }.
\end{equation*}%
Define the tent $\mathbf{T}\left( K\right) $ over a quasicube $K=\Omega L$
to be $\Omega \left( \mathbf{T}\left( L\right) \right) $ where $\mathbf{T}%
\left( L\right) $ is the convex hull of the $n$-cube $L\times \left\{
0\right\} $ and the point $\left( c_{L},\ell \left( L\right) \right) \in 
\mathbb{R}_{+}^{n+1}$. Define the $\mathbf{\tau }$-$\limfunc{deep}$ tent $%
\mathbf{T}^{\mathbf{\tau }-\limfunc{deep}}\left( K\right) $ over a quasicube 
$K$ to be the restriction of the tent $\mathbf{T}\left( K\right) $ to those
points at depth $\tau $ or more below $K$:%
\begin{equation*}
\mathbf{T}^{\mathbf{\tau }-\limfunc{deep}}\left( K\right) \equiv \left\{
\left( y,t\right) \in \mathbf{T}\left( K\right) :t\leq 2^{-\tau }\ell \left(
K\right) \right\} .
\end{equation*}%
We can now rewrite the size functional (\ref{def P stop energy 3}) of $%
\mathcal{P}$ as%
\begin{equation}
\mathcal{S}_{\limfunc{size}}^{\alpha ,A}\left( \mathcal{P}\right) ^{2}\equiv
\sup_{K\in \Pi ^{\limfunc{goodbelow}}\mathcal{P}}\frac{1}{\left\vert
K\right\vert _{\sigma }}\left( \frac{\mathrm{P}^{\alpha }\left( K,\mathbf{1}%
_{A\setminus K}\sigma \right) }{\left\vert K\right\vert ^{\frac{1}{n}}}%
\right) ^{2}\omega _{\mathcal{P}}\left( \mathbf{T}^{\mathbf{\tau }-\limfunc{%
deep}}\left( K\right) \right) .  \label{def P stop energy' 3}
\end{equation}%
It will be convenient to write%
\begin{equation*}
\Psi ^{\alpha }\left( K;\mathcal{P}\right) ^{2}\equiv \left( \frac{\mathrm{P}%
^{\alpha }\left( K,\mathbf{1}_{A\setminus K}\sigma \right) }{\left\vert
K\right\vert ^{\frac{1}{n}}}\right) ^{2}\omega _{\mathcal{P}}\left( \mathbf{T%
}^{\mathbf{\tau }-\limfunc{deep}}\left( K\right) \right) ,
\end{equation*}%
so that we have simply%
\begin{equation*}
\mathcal{S}_{\limfunc{size}}^{\alpha ,A}\left( \mathcal{P}\right)
^{2}=\sup_{K\in \Pi ^{\limfunc{goodbelow}}\mathcal{P}}\frac{\Psi ^{\alpha
}\left( K;\mathcal{P}\right) ^{2}}{\left\vert K\right\vert _{\sigma }}.
\end{equation*}

\begin{remark}
The functional $\omega _{\mathcal{P}}\left( \mathbf{T}^{\mathbf{\tau }-%
\limfunc{deep}}\left( K\right) \right) $ is increasing in $K$, while the
functional $\frac{\mathrm{P}^{\alpha }\left( K,\mathbf{1}_{A\setminus
K}\sigma \right) }{\left\vert K\right\vert ^{\frac{1}{n}}}$ is `almost
decreasing' in $K$: if $K_{0}\subset K$ then%
\begin{eqnarray*}
\frac{\mathrm{P}^{\alpha }\left( K,\mathbf{1}_{A\setminus K}\sigma \right) }{%
\left\vert K\right\vert ^{\frac{1}{n}}} &=&\int_{A\setminus K}\frac{d\sigma
\left( y\right) }{\left( \left\vert K\right\vert ^{\frac{1}{n}}+\left\vert
y-c_{K}\right\vert \right) ^{n+1-\alpha }} \\
&\lesssim &\int_{A\setminus K}\frac{\left( \sqrt{n}\right) ^{n+1-\alpha
}d\sigma \left( y\right) }{\left( \left\vert K_{0}\right\vert ^{\frac{1}{n}%
}+\left\vert y-c_{K_{0}}\right\vert \right) ^{n+1-\alpha }} \\
&\leq &C_{n,\alpha }\int_{A\setminus K_{0}}\frac{d\sigma \left( y\right) }{%
\left( \left\vert K_{0}\right\vert ^{\frac{1}{n}}+\left\vert
y-c_{K_{0}}\right\vert \right) ^{n+1-\alpha }}=C_{n,\alpha }\frac{\mathrm{P}%
^{\alpha }\left( K_{0},\mathbf{1}_{A\setminus K_{0}}\sigma \right) }{%
\left\vert K_{0}\right\vert ^{\frac{1}{n}}},
\end{eqnarray*}%
since $\left\vert K_{0}\right\vert ^{\frac{1}{n}}+\left\vert
y-c_{K_{0}}\right\vert \leq \left\vert K\right\vert ^{\frac{1}{n}%
}+\left\vert y-c_{K}\right\vert +\frac{1}{2}\limfunc{diam}\left( K\right) $
for $y\in A\setminus K$.
\end{remark}

\subsection{The recursion}

Recall that if $\mathcal{P}$ is an admissible collection for a dyadic
quasicube $A$, the corresponding sublinear form in (\ref{First inequality})
is given in (\ref{def split}) by%
\begin{eqnarray*}
\left\vert \mathsf{B}\right\vert _{\limfunc{stop},1,\bigtriangleup ^{\omega
}}^{A,\mathcal{P}}\left( f,g\right) &\equiv &\sum_{J\in \Pi _{2}\mathcal{P}}%
\frac{\mathrm{P}^{\alpha }\left( J,\left\vert \varphi _{J}^{\mathcal{P}%
}\right\vert \mathbf{1}_{A\setminus I_{\mathcal{P}}\left( J\right) }\sigma
\right) }{\left\vert J\right\vert ^{\frac{1}{n}}}\left\Vert \bigtriangleup
_{J}^{\omega }\mathbf{x}\right\Vert _{L^{2}\left( \omega \right) }\left\Vert
\bigtriangleup _{J}^{\omega }g\right\Vert _{L^{2}\left( \omega \right) }; \\
\text{where }\varphi _{J}^{\mathcal{P}} &\equiv &\sum_{I\in \mathcal{C}%
_{A}^{\prime }:\ \left( I,J\right) \in \mathcal{P}}\mathbb{E}_{I}^{\sigma
}\left( \bigtriangleup _{\pi I}^{\sigma }f\right) \ \mathbf{1}_{A\setminus
I}\ .
\end{eqnarray*}%
In the notation for $\left\vert \mathsf{B}\right\vert _{\limfunc{stop}%
,1,\bigtriangleup ^{\omega }}^{A,\mathcal{P}}$, we are omitting dependence
on the parameter $\alpha $, and to avoid clutter, we will often do so from
now on when the dependence on $\alpha $ is inconsequential. Following Lacey 
\cite{Lac}, we now claim the following proposition, from which we obtain (%
\ref{First inequality}) as a corollary below. Motivated by the conclusion of
Proposition \ref{stopping bound}, we define the \emph{restricted} norm $%
\mathfrak{N}_{\limfunc{stop},1,\bigtriangleup }^{A,\mathcal{P}}$ of the
sublinear form $\left\vert \mathsf{B}\right\vert _{\limfunc{stop}%
,1,\bigtriangleup ^{\omega }}^{A,\mathcal{P}}$ to be the best constant $%
\mathfrak{N}_{\limfunc{stop},1,\bigtriangleup }^{A,\mathcal{P}}$ in the
inequality%
\begin{equation*}
\left\vert \mathsf{B}\right\vert _{\limfunc{stop},1,\bigtriangleup ^{\omega
}}^{A,\mathcal{P}}\left( f,g\right) \leq \mathfrak{N}_{\limfunc{stop}%
,1,\bigtriangleup }^{A,\mathcal{P}}\left( \alpha _{\mathcal{A}}\left(
A\right) \sqrt{\left\vert A\right\vert _{\sigma }}+\left\Vert f\right\Vert
_{L^{2}\left( \sigma \right) }\right) \left\Vert g\right\Vert _{L^{2}\left(
\omega \right) },
\end{equation*}%
where $f\in L^{2}\left( \sigma \right) $ satisfies $\mathbb{E}_{I}^{\sigma
}\left\vert f\right\vert \leq \alpha _{\mathcal{A}}\left( A\right) $ for all 
$I\in \mathcal{C}_{A}^{\limfunc{good}}$.

\begin{proposition}
\label{bottom up 3}(This is a variant for sublinear forms of the Size Lemma
in Lacey \cite{Lac}) Suppose $\varepsilon >0$. Let $\mathcal{P}$ be an \emph{%
admissible} collection of pairs for a dyadic quasicube $A$. Then we can
decompose $\mathcal{P}$ into two disjoint collections $\mathcal{P}=\mathcal{P%
}^{big}\dot{\cup}\mathcal{P}^{small}$, and further decompose $\mathcal{P}%
^{small}$ into pairwise disjoint collections $\mathcal{P}_{1}^{small},%
\mathcal{P}_{2}^{small}...\mathcal{P}_{\ell }^{small}...$ i.e.%
\begin{equation*}
\mathcal{P}=\mathcal{P}^{big}\dot{\cup}\left( \overset{\cdot }{\dbigcup }%
_{\ell =1}^{\infty }\mathcal{P}_{\ell }^{small}\right) \ ,
\end{equation*}%
such that the collections $\mathcal{P}^{big}$ and $\mathcal{P}_{\ell
}^{small}$ are admissible and satisfy 
\begin{equation}
\sup_{\ell \geq 1}\mathcal{S}_{\limfunc{size}}^{\alpha ,A}\left( \mathcal{P}%
_{\ell }^{small}\right) ^{2}\leq \varepsilon \mathcal{S}_{\limfunc{size}%
}^{\alpha ,A}\left( \mathcal{P}\right) ^{2},  \label{small 3}
\end{equation}%
and 
\begin{equation}
\mathfrak{N}_{\limfunc{stop},1,\bigtriangleup }^{A,\mathcal{P}}\leq
C_{\varepsilon }\mathcal{S}_{\limfunc{size}}^{\alpha ,A}\left( \mathcal{P}%
\right) +\sqrt{n\mathbf{\tau }}\sup_{\ell \geq 1}\mathfrak{N}_{\limfunc{stop}%
,1,\bigtriangleup }^{A,\mathcal{P}_{\ell }^{small}}\ .  \label{big 3}
\end{equation}
\end{proposition}

\begin{corollary}
The sublinear stopping form inequality (\ref{First inequality}) holds.
\end{corollary}

\begin{proof}[Proof of the Corollary]
Set $\mathcal{Q}^{0}=\mathcal{P}^{A}$. Apply Proposition \ref{bottom up 3}
to obtain a subdecomposition $\left\{ \mathcal{Q}_{\ell }^{1}\right\} _{\ell
=1}^{\infty }$ of $\mathcal{Q}^{0}$ such that%
\begin{eqnarray*}
&&\mathfrak{N}_{\limfunc{stop},1,\bigtriangleup }^{A,\mathcal{Q}^{0}}\leq
C_{\varepsilon }\mathcal{S}_{\limfunc{size}}^{\alpha ,A}\left( \mathcal{Q}%
^{0}\right) +\sqrt{n\mathbf{\tau }}\sup_{\ell \geq 1}\mathfrak{N}_{\limfunc{%
stop},1,\bigtriangleup }^{A,\mathcal{Q}_{\ell }^{1}}\ , \\
&&\sup_{\ell \geq 1}\mathcal{S}_{\limfunc{size}}^{\alpha ,A}\left( \mathcal{Q%
}_{\ell }^{1}\right) \leq \varepsilon \mathcal{S}_{\limfunc{size}}^{\alpha
,A}\left( \mathcal{Q}^{0}\right) .
\end{eqnarray*}%
Now apply Proposition \ref{bottom up 3} to each $\mathcal{Q}_{\ell }^{1}$ to
obtain a subdecomposition $\left\{ \mathcal{Q}_{\ell ,k}^{2}\right\}
_{k=1}^{\infty }$ of $\mathcal{Q}_{\ell }^{1}$ such that%
\begin{eqnarray*}
\mathfrak{N}_{\limfunc{stop},1,\bigtriangleup }^{A,\mathcal{Q}_{\ell }^{1}}
&\leq &C_{\varepsilon }\mathcal{S}_{\limfunc{size}}^{\alpha ,A}\left( 
\mathcal{Q}_{\ell }^{1}\right) +\sqrt{n\mathbf{\tau }}\sup_{k\geq 1}%
\mathfrak{N}_{\limfunc{stop},1,\bigtriangleup }^{A,\mathcal{Q}_{\ell
,k}^{2}}\ , \\
&&\sup_{k\geq 1}\mathcal{S}_{\limfunc{size}}^{\alpha ,A}\left( \mathcal{Q}%
_{\ell ,k}^{2}\right) \leq \varepsilon \mathcal{S}_{\limfunc{size}}^{\alpha
,A}\left( \mathcal{Q}_{\ell }^{1}\right) .
\end{eqnarray*}%
Altogether we have%
\begin{eqnarray*}
\mathfrak{N}_{\limfunc{stop},1,\bigtriangleup }^{A,\mathcal{Q}^{0}} &\leq
&C_{\varepsilon }\mathcal{S}_{\limfunc{size}}^{\alpha ,A}\left( \mathcal{Q}%
^{0}\right) +\sqrt{n\mathbf{\tau }}\sup_{\ell \geq 1}\left\{ C_{\varepsilon }%
\mathcal{S}_{\limfunc{size}}^{\alpha ,A}\left( \mathcal{Q}_{\ell
}^{1}\right) +\sqrt{n\mathbf{\tau }}\sup_{k\geq 1}\mathfrak{N}_{\limfunc{stop%
},1,\bigtriangleup }^{A,\mathcal{Q}_{\ell ,k}^{2}}\right\} \\
&=&C_{\varepsilon }\left\{ \mathcal{S}_{\limfunc{size}}^{\alpha ,A}\left( 
\mathcal{Q}^{0}\right) +\sqrt{n\mathbf{\tau }}\varepsilon \mathcal{S}_{%
\limfunc{size}}^{\alpha ,A}\left( \mathcal{Q}^{0}\right) \right\} +\left( n%
\mathbf{\tau }\right) \sup_{\ell ,k\geq 1}\mathfrak{N}_{\limfunc{stop}%
,1,\bigtriangleup }^{A,\mathcal{Q}_{\ell ,k}^{2}}\ .
\end{eqnarray*}%
Then with $\zeta \equiv \sqrt{n\mathbf{\tau }}$, we obtain by induction for
every $N\in \mathbb{N}$,%
\begin{eqnarray*}
\mathfrak{N}_{\limfunc{stop},1,\bigtriangleup }^{A,\mathcal{Q}^{0}} &\leq
&C_{\varepsilon }\left\{ \mathcal{S}_{\limfunc{size}}^{\alpha ,A}\left( 
\mathcal{Q}^{0}\right) +\zeta \varepsilon \mathcal{S}_{\limfunc{size}%
}^{\alpha ,A}\left( \mathcal{Q}^{0}\right) +...\zeta ^{N}\varepsilon ^{N}%
\mathcal{S}_{\limfunc{size}}^{\alpha ,A}\left( \mathcal{Q}^{0}\right)
\right\} \\
&&+\zeta ^{N+1}\sup_{m\in \mathbb{N}^{N+1}}\mathfrak{N}_{\limfunc{stop}%
,1,\bigtriangleup }^{A,\mathcal{Q}_{m}^{N+1}}\ .
\end{eqnarray*}%
Now we may assume the collection $\mathcal{Q}^{0}=\mathcal{P}^{A}$ of pairs
is finite (simply truncate the corona $\mathcal{C}_{A}$ and obtain bounds
independent of the truncation) and so $\sup_{m\in \mathbb{N}^{N+1}}\mathfrak{%
N}_{\limfunc{stop},1,\bigtriangleup }^{A,\mathcal{Q}_{m}^{N+1}}=0$ for $N$
large enough. Then we obtain (\ref{First inequality}) if we choose $%
0<\varepsilon <\frac{1}{1+\zeta }$ and apply Lemma \ref{energy control}.
\end{proof}

\begin{proof}[Proof of Proposition \protect\ref{bottom up 3}]
Recall that the `size testing collection' of quasicubes $\Pi ^{\limfunc{%
goodbelow}}\mathcal{P}$ is the collection of all \emph{good} subquasicubes
of a quasicube in $\Pi \mathcal{P}$. We may assume that $\mathcal{P}$ is a
finite collection. Begin by defining the collection $\mathcal{L}_{0}$ to
consist of the \emph{minimal} dyadic quasicubes $K$ in $\Pi ^{\limfunc{%
goodbelow}}\mathcal{P}$ such that%
\begin{equation*}
\frac{\Psi ^{\alpha }\left( K;\mathcal{P}\right) ^{2}}{\left\vert
K\right\vert _{\sigma }}\geq \varepsilon \mathcal{S}_{\limfunc{size}%
}^{\alpha ,A}\left( \mathcal{P}\right) ^{2}.
\end{equation*}%
where we recall that%
\begin{equation*}
\Psi ^{\alpha }\left( K;\mathcal{P}\right) ^{2}\equiv \left( \frac{\mathrm{P}%
^{\alpha }\left( K,\mathbf{1}_{A\setminus K}\sigma \right) }{\left\vert
K\right\vert ^{\frac{1}{n}}}\right) ^{2}\omega _{\mathcal{P}}\left( \mathbf{T%
}^{\mathbf{\tau }-\limfunc{deep}}\left( K\right) \right) .
\end{equation*}%
Note that such minimal quasicubes exist when $0<\varepsilon <1$ because $%
\mathcal{S}_{\limfunc{size}}^{\alpha ,A}\left( \mathcal{P}\right) ^{2}$ is
the supremum over $K\in \Pi ^{\limfunc{goodbelow}}\mathcal{P}$ of $\frac{%
\Psi ^{\alpha }\left( K;\mathcal{P}\right) ^{2}}{\left\vert K\right\vert
_{\sigma }}$. A key property of the the minimality requirement is that%
\begin{equation}
\frac{\Psi ^{\alpha }\left( K^{\prime };\mathcal{P}\right) ^{2}}{\left\vert
K^{\prime }\right\vert _{\sigma }}<\varepsilon \mathcal{S}_{\limfunc{size}%
}^{\alpha ,A}\left( \mathcal{P}\right) ^{2},  \label{key property 3}
\end{equation}%
for all $K^{\prime }\in \Pi ^{\limfunc{goodbelow}}\mathcal{P}$ with $%
K^{\prime }\varsubsetneqq K$ and $K\in \mathcal{L}_{0}$.

We now perform a stopping time argument `from the bottom up' with respect to
the atomic measure $\omega _{\mathcal{P}}$ in the upper half space. This
construction of a stopping time `from the bottom up' is one of two key
innovations in Lacey's argument \cite{Lac}, the other being the recursion
described in Proposition \ref{bottom up 3}.

We refer to $\mathcal{L}_{0}$ as the initial or level $0$ generation of
stopping times. Choose $\rho =1+\varepsilon $. We then recursively define a
sequence of generations $\left\{ \mathcal{L}_{m}\right\} _{m=0}^{\infty }$
by letting $\mathcal{L}_{m}$ consist of the \emph{minimal} dyadic quasicubes 
$L$ in $\Pi ^{\limfunc{goodbelow}}\mathcal{P}$ that contain a quasicube from
some previous level $\mathcal{L}_{\ell }$, $\ell <m$, such that%
\begin{equation}
\omega _{\mathcal{P}}\left( \mathbf{T}^{\mathbf{\tau }-\limfunc{deep}}\left(
L\right) \right) \geq \rho \omega _{\mathcal{P}}\left(
\dbigcup\limits_{L^{\prime }\in \dbigcup\limits_{\ell =0}^{m-1}\mathcal{L}%
_{\ell }:\ L^{\prime }\subset L}\mathbf{T}^{\mathbf{\tau }-\limfunc{deep}%
}\left( L^{\prime }\right) \right) .  \label{up stopping condition}
\end{equation}%
Since $\mathcal{P}$ is finite this recursion stops at some level $M$. We
then let $\mathcal{L}_{M+1}$ consist of all the maximal quasicubes in $\Pi ^{%
\limfunc{goodbelow}}\mathcal{P}$ that are not already in some $\mathcal{L}%
_{m}$. Thus $\mathcal{L}_{M+1}$ will contain either none, some, or all of
the maximal quasicubes in $\Pi ^{\limfunc{goodbelow}}\mathcal{P}$. We do not
of course have (\ref{up stopping condition}) for $A^{\prime }\in \mathcal{L}%
_{M+1}$ in this case, but we do have that (\ref{up stopping condition})
fails for subquasicubes $K$ of $A^{\prime }\in \mathcal{L}_{M+1}$ that are
not contained in any other $L\in \mathcal{L}_{m}$, and this is sufficient
for the arguments below.

We now define the collections $\mathcal{P}^{small}$ and $\mathcal{P}^{big}$.
The collection $\mathcal{P}^{big}$ will consist of those pairs $\left(
I,J\right) \in \mathcal{P}$ for which there is $L\in
\dbigcup\limits_{m=0}^{M+1}\mathcal{L}_{m}$ with $J\Subset _{\tau }L\subset
I $, and $\mathcal{P}^{small}$ will consist of the remaining pairs. But a
considerable amount of further analysis is required to prove the conclusion
of the proposition. First, let $\mathcal{L}\equiv \dbigcup\limits_{m=0}^{M+1}%
\mathcal{L}_{m}$ be the tree of stopping quasienergy quasicubes defined
above. By our construction above, the maximal elements in $\mathcal{L}$ are
the maximal quasicubes in $\Pi ^{\limfunc{goodbelow}}\mathcal{P}$. For $L\in 
\mathcal{L}$, denote by $\mathcal{C}_{L}$ the \emph{corona} associated with $%
L$ in the tree $\mathcal{L}$,%
\begin{equation*}
\mathcal{C}_{L}\equiv \left\{ K\in \Omega \mathcal{D}:K\subset L\text{ and
there is no }L^{\prime }\in \mathcal{L}\text{ with }K\subset L^{\prime
}\subsetneqq L\right\} ,
\end{equation*}%
and define the \emph{shifted} corona by%
\begin{equation*}
\mathcal{C}_{L}^{\mathbf{\tau }-\limfunc{shift}}\equiv \left\{ K\in \mathcal{%
C}_{L}:K\Subset _{\mathbf{\tau },\varepsilon }L\right\} \cup
\dbigcup\limits_{L^{\prime }\in \mathfrak{C}_{\mathcal{L}}\left( L\right)
}\left\{ K\in \Omega \mathcal{D}:K\Subset _{\mathbf{\tau },\varepsilon }L%
\text{ and }K\text{ is }\mathbf{\tau }\text{-nearby in }L^{\prime }\right\} .
\end{equation*}%
Now the parameter $m$ in $\mathcal{L}_{m}$ refers to the level at which the
stopping construction was performed, but for \thinspace $L\in \mathcal{L}%
_{m} $, the corona children $L^{\prime }$ of $L$ are \emph{not} all
necessarily in $\mathcal{L}_{m-1}$, but may be in $\mathcal{L}_{m-t}$ for $t$
large. Thus we need to introduce the notion of geometric depth $d$ in the
tree $\mathcal{L}$ by defining%
\begin{eqnarray*}
\mathcal{G}_{0} &\equiv &\left\{ L\in \mathcal{L}:L\text{ is maximal}%
\right\} , \\
\mathcal{G}_{1} &\equiv &\left\{ L\in \mathcal{L}:L\text{ is maximal wrt }%
L\subsetneqq L_{0}\text{ for some }L_{0}\in \mathcal{G}_{0}\right\} , \\
&&\vdots \\
\mathcal{G}_{d+1} &\equiv &\left\{ L\in \mathcal{L}:L\text{ is maximal wrt }%
L\subsetneqq L_{d}\text{ for some }L_{d}\in \mathcal{G}_{d}\right\} , \\
&&\vdots
\end{eqnarray*}%
We refer to $\mathcal{G}_{d}$ as the $d^{th}$ generation of quasicubes in
the tree $\mathcal{L}$, and say that the quasicubes in $\mathcal{G}_{d}$ are
at depth $d$ in the tree $\mathcal{L}$. Thus the quasicubes in $\mathcal{G}%
_{d}$ are the stopping quasicubes in $\mathcal{L}$ that are $d$ levels in
the \emph{geometric} sense below the top level.

Then for $L\in \mathcal{G}_{d}$ and $t\geq 0$ define%
\begin{equation*}
\mathcal{P}_{L,t}\equiv \left\{ \left( I,J\right) \in \mathcal{P}:I\in 
\mathcal{C}_{L}\text{ and }J\in \mathcal{C}_{L^{\prime }}^{\mathbf{\tau }-%
\limfunc{shift}}\text{ for some }L^{\prime }\in \mathcal{G}_{d+t}\text{ with 
}L^{\prime }\subset L\right\} .
\end{equation*}%
In particular, $\left( I,J\right) \in \mathcal{P}_{L,t}$ implies that $I$ is
in the corona $\mathcal{C}_{L}$, and that $J$ is in a shifted corona $%
\mathcal{C}_{L^{\prime }}^{\mathbf{\tau }-\limfunc{shift}}$ that is $t$
levels of generation \emph{below} $\mathcal{C}_{L}$. We emphasize the
distinction `generation' as this refers to the depth rather than the level
of stopping construction. For $t=0$ we further decompose $\mathcal{P}_{L,0}$
as%
\begin{eqnarray*}
\mathcal{P}_{L,0} &=&\mathcal{P}_{L,0}^{small}\dot{\cup}\mathcal{P}%
_{L,0}^{big}; \\
\mathcal{P}_{L,0}^{small} &\equiv &\left\{ \left( I,J\right) \in \mathcal{P}%
_{L,0}:I\neq L\right\} , \\
\mathcal{P}_{L,0}^{big} &\equiv &\left\{ \left( I,J\right) \in \mathcal{P}%
_{L,0}:I=L\right\} ,
\end{eqnarray*}%
with one exception: if $L\in \mathcal{L}_{M+1}$ we set $\mathcal{P}%
_{L,0}^{small}\equiv \mathcal{P}_{L,0}$ since in this case $L$ fails to
satisfy (\ref{up stopping condition}) as pointed out above. Then we set%
\begin{eqnarray*}
\mathcal{P}^{big} &\equiv &\left\{ \dbigcup\limits_{L\in \mathcal{L}}%
\mathcal{P}_{L,0}^{big}\right\} \dbigcup \left\{ \dbigcup\limits_{t\geq
1}\dbigcup\limits_{L\in \mathcal{L}}\mathcal{P}_{L,t}\right\} ; \\
\left\{ \mathcal{P}_{\ell }^{small}\right\} _{\ell =0}^{\infty } &\equiv
&\left\{ \mathcal{P}_{L,0}^{small}\right\} _{L\in \mathcal{L}},\ \ \ \ \ 
\text{after relabelling}.
\end{eqnarray*}%
It is important to note that by (\ref{inclusive}), every pair $\left(
I,J\right) \in \mathcal{P}$ will be included in either $\mathcal{P}^{small}$
or $\mathcal{P}^{big}$. Now we turn to proving the inequalities (\ref{small
3}) and (\ref{big 3}).

To prove the inequality (\ref{small 3}), it suffices with the above
relabelling to prove the following claim:%
\begin{equation}
\mathcal{S}_{\limfunc{size}}^{\alpha ,A}\left( \mathcal{P}%
_{L,0}^{small}\right) ^{2}\leq \left( \rho -1\right) \mathcal{S}_{\limfunc{%
size}}^{\alpha ,A}\left( \mathcal{P}\right) ^{2},\ \ \ \ \ L\in \mathcal{L}.
\label{small claim' 3}
\end{equation}%
To see (\ref{small claim' 3}), suppose first that $L\notin \mathcal{L}_{M+1}$%
. In the case that $L\in \mathcal{L}_{0}$ is an initial generation
quasicube, then from (\ref{key property 3}) we obtain that%
\begin{eqnarray*}
&&\mathcal{S}_{\limfunc{size}}^{\alpha ,A}\left( \mathcal{P}%
_{L,0}^{small}\right) ^{2} \\
&\leq &\sup_{K^{\prime }\in \Pi ^{\limfunc{goodbelow}}\mathcal{P}:\
K^{\prime }\varsubsetneqq L}\frac{1}{\left\vert K^{\prime }\right\vert
_{\sigma }}\left( \frac{\mathrm{P}^{\alpha }\left( K^{\prime },\mathbf{1}%
_{A\setminus K^{\prime }}\sigma \right) }{\left\vert K^{\prime }\right\vert
^{\frac{1}{n}}}\right) ^{2}\omega _{\mathcal{P}}\left( \mathbf{T}^{\mathbf{%
\tau }-\limfunc{deep}}\left( K^{\prime }\right) \right) \leq \varepsilon 
\mathcal{S}_{\limfunc{size}}^{\alpha ,A}\left( \mathcal{P}\right) ^{2}.
\end{eqnarray*}%
Now suppose that $L\not\in \mathcal{L}_{0}$ and also that $L\notin \mathcal{L%
}_{M+1}$. Pick a pair $\left( I,J\right) \in \mathcal{P}_{L,0}^{small}$.
Then $I$ is in the restricted corona $\mathcal{C}_{L}^{\prime }$ and $J$ is
in the $\mathbf{\tau }$\emph{-shifted} corona $\mathcal{C}_{L}^{\mathbf{\tau 
}-\limfunc{shift}}$. Since $\mathcal{P}_{L,0}^{small}$ is a finite
collection, the definition of $\mathcal{S}_{\limfunc{size}}^{\alpha
,A}\left( \mathcal{P}_{L,0}^{small}\right) $ shows that there is a quasicube 
$K\in \Pi ^{\limfunc{goodbelow}}\mathcal{P}_{L,0}^{small}$ so that%
\begin{equation*}
\mathcal{S}_{\limfunc{size}}^{\alpha ,A}\left( \mathcal{P}%
_{L,0}^{small}\right) ^{2}=\frac{1}{\left\vert K\right\vert _{\sigma }}%
\left( \frac{\mathrm{P}^{\alpha }\left( K,\mathbf{1}_{A\setminus K}\sigma
\right) }{\left\vert K\right\vert ^{\frac{1}{n}}}\right) ^{2}\omega _{%
\mathcal{P}}\left( \mathbf{T}^{\mathbf{\tau }-\limfunc{deep}}\left( K\right)
\right) .
\end{equation*}%
Now define 
\begin{equation*}
t^{\prime }=t^{\prime }\left( K\right) \equiv \max \left\{ s:\text{there is }%
L^{\prime }\in \mathcal{L}_{s}\text{ with }L^{\prime }\subset K\right\} .
\end{equation*}%
First, suppose that $t^{\prime }=0$ so that $K$ does not contain any $%
L^{\prime }\in \mathcal{L}$. Then it follows from our construction at level $%
\ell =0$ that%
\begin{equation*}
\frac{1}{\left\vert K\right\vert _{\sigma }}\left( \frac{\mathrm{P}^{\alpha
}\left( K,\mathbf{1}_{A\setminus K}\sigma \right) }{\left\vert K\right\vert
^{\frac{1}{n}}}\right) ^{2}\omega _{\mathcal{P}}\left( \mathbf{T}^{\mathbf{%
\tau }-\limfunc{deep}}\left( K\right) \right) <\varepsilon \mathcal{S}_{%
\limfunc{size}}^{\alpha ,A}\left( \mathcal{P}\right) ^{2},
\end{equation*}%
and hence from $\rho =1+\varepsilon $ we obtain 
\begin{equation*}
\mathcal{S}_{\limfunc{size}}^{\alpha ,A}\left( \mathcal{P}%
_{L,0}^{small}\right) ^{2}<\varepsilon \mathcal{S}_{\limfunc{size}}^{\alpha
,A}\left( \mathcal{P}\right) ^{2}=\left( \rho -1\right) \mathcal{S}_{%
\limfunc{size}}^{\alpha ,A}\left( \mathcal{P}\right) ^{2}.
\end{equation*}%
Now suppose that $t^{\prime }\geq 1$. Then $K$ fails the stopping condition (%
\ref{up stopping condition}) with $m=t^{\prime }+1$, and so%
\begin{equation*}
\omega _{\mathcal{P}}\left( \mathbf{T}^{\mathbf{\tau }-\limfunc{deep}}\left(
K\right) \right) <\rho \omega _{\mathcal{P}}\left(
\dbigcup\limits_{L^{\prime }\in \dbigcup\limits_{\ell =0}^{t^{\prime }}%
\mathcal{L}_{\ell }:\ L^{\prime }\subset K}\mathbf{T}^{\mathbf{\tau }-%
\limfunc{deep}}\left( L^{\prime }\right) \right) .
\end{equation*}%
Now we use the crucial fact that $\omega _{\mathcal{P}}$ is \emph{additive}
and finite to obtain from this that%
\begin{eqnarray}
&&\omega _{\mathcal{P}}\left( \mathbf{T}^{\mathbf{\tau }-\limfunc{deep}%
}\left( K\right) \setminus \dbigcup\limits_{L^{\prime }\in
\dbigcup\limits_{\ell =0}^{t^{\prime }}\mathcal{L}_{\ell }:\ L^{\prime
}\subset K}\mathbf{T}^{\mathbf{\tau }-\limfunc{deep}}\left( L^{\prime
}\right) \right)  \label{additive} \\
&=&\omega _{\mathcal{P}}\left( \mathbf{T}^{\mathbf{\tau }-\limfunc{deep}%
}\left( K\right) \right) -\omega _{\mathcal{P}}\left(
\dbigcup\limits_{L^{\prime }\in \dbigcup\limits_{\ell =0}^{t^{\prime }}%
\mathcal{L}_{\ell }:\ L^{\prime }\subset K}\mathbf{T}^{\mathbf{\tau }-%
\limfunc{deep}}\left( L^{\prime }\right) \right)  \notag \\
&\leq &\left( \rho -1\right) \omega _{\mathcal{P}}\left(
\dbigcup\limits_{L^{\prime }\in \dbigcup\limits_{\ell =0}^{t^{\prime }}%
\mathcal{L}_{\ell }:\ L^{\prime }\subset K}\mathbf{T}^{\mathbf{\tau }-%
\limfunc{deep}}\left( L^{\prime }\right) \right) .  \notag
\end{eqnarray}%
Thus using 
\begin{equation*}
\omega _{\mathcal{P}_{L,0}^{small}}\left( \mathbf{T}^{\mathbf{\tau }-%
\limfunc{deep}}\left( K\right) \right) \leq \omega _{\mathcal{P}}\left( 
\mathbf{T}^{\mathbf{\tau }-\limfunc{deep}}\left( K\right) \setminus
\dbigcup\limits_{L^{\prime }\in \dbigcup\limits_{\ell =0}^{t^{\prime }}%
\mathcal{L}_{\ell }:\ L^{\prime }\subset K}\mathbf{T}^{\mathbf{\tau }-%
\limfunc{deep}}\left( L^{\prime }\right) \right) ,
\end{equation*}%
and (\ref{additive}) we have%
\begin{eqnarray*}
&&\mathcal{S}_{\limfunc{size}}^{\alpha ,A}\left( \mathcal{P}%
_{L,0}^{small}\right) ^{2} \\
&\leq &\sup_{K\in \Pi ^{\limfunc{goodbelow}}\mathcal{P}_{L,0}^{small}}\frac{1%
}{\left\vert K\right\vert _{\sigma }}\left( \frac{\mathrm{P}^{\alpha }\left(
K,\mathbf{1}_{A\setminus K}\sigma \right) }{\left\vert K\right\vert ^{\frac{1%
}{n}}}\right) ^{2}\omega _{\mathcal{P}}\left( \mathbf{T}^{\mathbf{\tau }-%
\limfunc{deep}}\left( K\right) \setminus \dbigcup\limits_{L^{\prime }\in
\dbigcup\limits_{\ell =0}^{t^{\prime }}\mathcal{L}_{\ell }:\ L^{\prime
}\subset K}\mathbf{T}^{\mathbf{\tau }-\limfunc{deep}}\left( L^{\prime
}\right) \right) \\
&\leq &\left( \rho -1\right) \sup_{K\in \Pi ^{\limfunc{goodbelow}}\mathcal{P}%
_{L,0}^{small}}\frac{1}{\left\vert K\right\vert _{\sigma }}\left( \frac{%
\mathrm{P}^{\alpha }\left( K,\mathbf{1}_{A\setminus K}\sigma \right) }{%
\left\vert K\right\vert ^{\frac{1}{n}}}\right) ^{2}\omega _{\mathcal{P}%
}\left( \dbigcup\limits_{L^{\prime }\in \dbigcup\limits_{\ell =0}^{t^{\prime
}}\mathcal{L}_{\ell }:\ L^{\prime }\subset K}\mathbf{T}^{\mathbf{\tau }-%
\limfunc{deep}}\left( L^{\prime }\right) \right) .
\end{eqnarray*}%
and we can continue with 
\begin{eqnarray*}
&&\mathcal{S}_{\limfunc{size}}^{\alpha ,A}\left( \mathcal{P}%
_{L,0}^{small}\right) \\
&\leq &\left( \rho -1\right) \sup_{K\in \Pi ^{\limfunc{goodbelow}}\mathcal{P}%
}\frac{1}{\left\vert K\right\vert _{\sigma }}\left( \frac{\mathrm{P}^{\alpha
}\left( K,\mathbf{1}_{A\setminus K}\sigma \right) }{\left\vert K\right\vert
^{\frac{1}{n}}}\right) ^{2}\omega _{\mathcal{P}}\left( \mathbf{T}^{\mathbf{%
\tau }-\limfunc{deep}}\left( K\right) \right) \\
&\leq &\left( \rho -1\right) \mathcal{S}_{\limfunc{size}}^{\alpha ,A}\left( 
\mathcal{P}\right) ^{2}.
\end{eqnarray*}

In the remaining case where $L\in \mathcal{L}_{M+1}$ we can include $L$ as a
testing quasicube $K$ and the same reasoning applies. This completes the
proof of (\ref{small claim' 3}).

To prove the other inequality (\ref{big 3}), we need a lemma to bound the
norm of certain `straddled' stopping forms by the size functional $\mathcal{S%
}_{\limfunc{size}}^{\alpha ,A}$, and another lemma to bound sums of
`mutually orthogonal' stopping forms. We interrupt the proof to turn to
these matters.
\end{proof}

\subsubsection{The Straddling Lemma}

Given an admissible collection of pairs $\mathcal{Q}$ for $A$, and a
subpartition $\mathcal{S}\subset \Pi ^{\limfunc{goodbelow}}\mathcal{Q}$ of
pairwise disjoint quasicubes in $A$, we say that $\mathcal{Q}$ $\mathbf{\tau 
}$\emph{-straddles} $\mathcal{S}$ if for every pair $\left( I,J\right) \in 
\mathcal{Q}$ there is $S\in \mathcal{S}\cap \left[ J,I\right] $ where $\left[
J,I\right] $ denotes the geodesic in the dyadic tree $\Omega \mathcal{D}$
that connects $J$ to $I$, and moreover that $J\Subset _{\mathbf{\tau }%
,\varepsilon }S$. Denote by $\mathcal{N}_{\mathbf{\rho }-\mathbf{\tau }}^{%
\limfunc{good}}\left( S\right) $ the finite collection of quasicubes that
are both good and $\left( \mathbf{\rho }-\mathbf{\tau }\right) $-nearby in $%
S $. For any good dyadic quasicube $S\in \Omega \mathcal{D}_{\limfunc{good}}$%
, we will also need the collection $\mathcal{W}^{\limfunc{good}}\left(
S\right) $ of maximal \emph{good} subquasicubes $I$ of $S$ whose triples $3I$
are contained in $S$.

\begin{lemma}
\label{straddle 3}Let $\mathcal{S}$ be a subpartition of $A$, and suppose
that $\mathcal{Q}$ is an admissible collection of pairs for $A$ such that $%
\mathcal{S}\subset \Pi ^{\limfunc{goodbelow}}\mathcal{Q}$, and such that $%
\mathcal{Q}$ $\mathbf{\tau }$-straddles $\mathcal{S}$. Then we have the
sublinear form bound%
\begin{equation*}
\mathfrak{N}_{\limfunc{stop},1,\bigtriangleup }^{A,\mathcal{Q}}\leq C_{%
\mathbf{r},\mathbf{\tau },\mathbf{\rho }}\sup_{S\in \mathcal{S}}\mathcal{S}_{%
\limfunc{size}}^{\alpha ,A;S}\left( \mathcal{Q}\right) \leq C_{\mathbf{r},%
\mathbf{\tau },\mathbf{\rho }}\mathcal{S}_{\limfunc{size}}^{\alpha ,A}\left( 
\mathcal{Q}\right) ,
\end{equation*}%
where $\mathcal{S}_{\limfunc{size}}^{\alpha ,A;S}$ is an $S$-localized
version of $\mathcal{S}_{\limfunc{size}}^{\alpha ,A}$ with an $S$-hole given
by%
\begin{equation}
\mathcal{S}_{\limfunc{size}}^{\alpha ,A;S}\left( \mathcal{Q}\right)
^{2}\equiv \sup_{K\in \mathcal{N}_{\mathbf{\rho }-\mathbf{\tau }}^{\limfunc{%
good}}\left( S\right) \cup \mathcal{W}^{\limfunc{good}}\left( S\right) }%
\frac{1}{\left\vert K\right\vert _{\sigma }}\left( \frac{\mathrm{P}^{\alpha
}\left( K,\mathbf{1}_{A\setminus S}\sigma \right) }{\left\vert K\right\vert
^{\frac{1}{n}}}\right) ^{2}\omega _{\mathcal{Q}}\left( \mathbf{T}^{\mathbf{%
\tau }-\limfunc{deep}}\left( K\right) \right) .  \label{localized size}
\end{equation}
\end{lemma}

\begin{proof}
For $S\in S$ let $\mathcal{Q}^{S}\equiv \left\{ \left( I,J\right) \in 
\mathcal{Q}:J\Subset _{\mathbf{\tau },\varepsilon }S\subset I\right\} $. We
begin by using that $\mathcal{Q}$ $\mathbf{\tau }$-straddles $\mathcal{S}$,
together with the sublinearity property (\ref{phi sublinear}) of $\varphi
_{J}^{\mathcal{Q}}$, to write%
\begin{eqnarray*}
\left\vert \mathsf{B}\right\vert _{\limfunc{stop},1,\bigtriangleup }^{A,%
\mathcal{Q}}\left( f,g\right) &=&\sum_{J\in \Pi _{2}\mathcal{P}}\frac{%
\mathrm{P}^{\alpha }\left( J,\left\vert \varphi _{J}^{\mathcal{Q}%
}\right\vert \mathbf{1}_{A\setminus I_{\mathcal{Q}}\left( J\right) }\sigma
\right) }{\left\vert J\right\vert ^{\frac{1}{n}}}\left\Vert \bigtriangleup
_{J}^{\omega }\mathbf{x}\right\Vert _{L^{2}\left( \omega \right) }\left\Vert
\bigtriangleup _{J}^{\omega }g\right\Vert _{L^{2}\left( \omega \right) } \\
&\leq &\sum_{S\in \mathcal{S}}\sum_{J\in \Pi _{2}^{S,\mathbf{\tau }-\limfunc{%
deep}}\mathcal{Q}}\frac{\mathrm{P}^{\alpha }\left( J,\left\vert \varphi
_{J}^{\mathcal{Q}^{S}}\right\vert \mathbf{1}_{A\setminus I_{\mathcal{Q}%
}\left( J\right) }\sigma \right) }{\left\vert J\right\vert ^{\frac{1}{n}}}%
\left\Vert \bigtriangleup _{J}^{\omega }\mathbf{x}\right\Vert _{L^{2}\left(
\omega \right) }\left\Vert \bigtriangleup _{J}^{\omega }g\right\Vert
_{L^{2}\left( \omega \right) }; \\
\text{where }\varphi _{J}^{\mathcal{Q}^{S}} &\equiv &\sum_{I\in \Pi _{1}%
\mathcal{Q}^{S}:\mathcal{\ }\left( I,J\right) \in \mathcal{Q}^{S}}\mathbb{E}%
_{I}^{\sigma }\left( \bigtriangleup _{\pi I}^{\sigma }f\right) \ \mathbf{1}%
_{A\setminus I}\ .
\end{eqnarray*}%
At this point, with $S$ fixed for the moment, we consider separately the
finitely many cases $\ell \left( J\right) =2^{-s}\ell \left( S\right) $
where $s\geq \mathbf{\rho }$ and where $\mathbf{\tau }\leq s<\mathbf{\rho }$%
. More precisely, we pigeonhole the side length of $J\in \Pi _{2}\mathcal{Q}%
^{S}=\Pi _{2}^{S,\tau -\limfunc{deep}}\mathcal{Q}$ by%
\begin{eqnarray*}
\mathcal{Q}_{\ast }^{S} &\equiv &\left\{ \left( I,J\right) \in \mathcal{Q}%
^{S}:J\in \Pi _{2}\mathcal{Q}^{S}\text{ and }\ell \left( J\right) \leq 2^{-%
\mathbf{\rho }}\ell \left( S\right) \right\} , \\
\mathcal{Q}_{s}^{S} &\equiv &\left\{ \left( I,J\right) \in \mathcal{Q}%
^{S}:J\in \Pi _{2}\mathcal{Q}^{S}\text{ and }\ell \left( J\right)
=2^{-s}\ell \left( S\right) \right\} ,\ \ \ \ \ \mathbf{\tau }\leq s<\mathbf{%
\rho }.
\end{eqnarray*}%
Then we have%
\begin{eqnarray*}
\Pi _{2}\mathcal{Q}_{\ast }^{S} &\equiv &\left\{ J\in \Pi _{2}\mathcal{Q}%
^{S}:\ell \left( J\right) \leq 2^{-\mathbf{\rho }}\ell \left( S\right)
\right\} , \\
\Pi _{2}\mathcal{Q}_{s}^{S} &\equiv &\left\{ J\in \Pi _{2}\mathcal{Q}%
^{S}:\ell \left( J\right) =2^{-s}\ell \left( S\right) \right\} ,\ \ \ \ \ 
\mathbf{\tau }\leq s<\mathbf{\rho },
\end{eqnarray*}%
and we make the corresponding decomposition for the sublinear form%
\begin{eqnarray*}
\left\vert \mathsf{B}\right\vert _{\limfunc{stop},1,\bigtriangleup }^{A,%
\mathcal{Q}}\left( f,g\right) &=&\left\vert \mathsf{B}\right\vert _{\limfunc{%
stop},1,\bigtriangleup }^{A,\mathcal{Q}_{\ast }}\left( f,g\right)
+\dsum\limits_{\mathbf{\tau }\leq s<\mathbf{\rho }}\left\vert \mathsf{B}%
\right\vert _{\limfunc{stop},1,\bigtriangleup }^{A,\mathcal{Q}_{s}}\left(
f,g\right) \\
&\equiv &\sum_{S\in \mathcal{S}}\sum_{J\in \Pi _{2}\mathcal{Q}_{\ast }^{S}}%
\frac{\mathrm{P}^{\alpha }\left( J,\left\vert \varphi _{J}^{\mathcal{Q}%
_{\ast }^{S}}\right\vert \mathbf{1}_{A\setminus I_{\mathcal{Q}_{\ast
}}\left( J\right) }\sigma \right) }{\left\vert J\right\vert ^{\frac{1}{n}}}%
\left\Vert \bigtriangleup _{J}^{\omega }\mathbf{x}\right\Vert _{L^{2}\left(
\omega \right) }\left\Vert \bigtriangleup _{J}^{\omega }g\right\Vert
_{L^{2}\left( \omega \right) } \\
&&+\dsum\limits_{\mathbf{\tau }\leq s<\mathbf{\rho }}\sum_{S\in \mathcal{S}%
}\sum_{J\in \Pi _{2}\mathcal{Q}_{s}^{S}}\frac{\mathrm{P}^{\alpha }\left(
J,\left\vert \varphi _{J}^{\mathcal{Q}_{s}^{S}}\right\vert \mathbf{1}%
_{A\setminus I_{\mathcal{Q}_{S}}\left( J\right) }\sigma \right) }{\left\vert
J\right\vert ^{\frac{1}{n}}}\left\Vert \bigtriangleup _{J}^{\omega }\mathbf{x%
}\right\Vert _{L^{2}\left( \omega \right) }\left\Vert \bigtriangleup
_{J}^{\omega }g\right\Vert _{L^{2}\left( \omega \right) }\ .
\end{eqnarray*}%
By the tree-connected property of $\mathcal{Q}$, and the telescoping
property of martingale differences, together with the bound $\alpha _{%
\mathcal{A}}\left( A\right) $ on the averages of $f$ in the corona $\mathcal{%
C}_{A}$, we have%
\begin{equation}
\left\vert \varphi _{J}^{\mathcal{Q}_{\ast }^{S}}\right\vert ,\left\vert
\varphi _{J}^{\mathcal{Q}_{s}^{S}}\right\vert \lesssim \alpha _{\mathcal{A}%
}\left( A\right) 1_{A\setminus I_{\mathcal{Q}^{S}}\left( J\right) },
\label{bfi 3}
\end{equation}%
where $I_{\mathcal{Q}^{S}}\left( J\right) \equiv \dbigcap \left\{ I:\left(
I,J\right) \in \mathcal{Q}^{S}\right\} $ is the smallest quasicube $I$ for
which $\left( I,J\right) \in \mathcal{Q}^{S}$.

\bigskip

\textbf{Case} for $\left\vert \mathsf{B}\right\vert _{\limfunc{stop}%
,1,\bigtriangleup }^{A,\mathcal{Q}_{s}^{S}}\left( f,g\right) $ when $\mathbf{%
\tau }\leq s\leq \mathbf{\rho }$: Now is a crucial definition that permits
us to bound the form by the size functional with a large hole. Let 
\begin{equation*}
\mathcal{C}_{s}^{S}\equiv \pi ^{\mathbf{\tau }}\left( \Pi _{2}\mathcal{Q}%
_{s}^{S}\right)
\end{equation*}%
be the collection of $\mathbf{\tau }$-parents of quasicubes in $\Pi _{2}%
\mathcal{Q}_{s}^{S}$, and denote by $\mathcal{M}_{s}^{S}$ the set of \emph{%
maximal} quasicubes in the collection $\mathcal{C}_{s}^{S}$. We have that
the quasicubes in $\mathcal{M}_{s}^{S}$ are good by our assumption that the
quasiHaar support of $g$ is contained in the $\mathbf{\tau }$-good quasigrid 
$\Omega \mathcal{D}_{\left( \mathbf{r},\varepsilon \right) -\limfunc{good}}^{%
\mathbf{\tau }}$, and so $\mathcal{M}_{s}^{S}\subset \mathcal{N}_{\mathbf{%
\rho }-\mathbf{\tau }}\left( S\right) $. Here is the first of two key
inclusions:%
\begin{equation}
J\Subset _{\mathbf{\tau },\varepsilon }K\subset S\text{ if }K\in \mathcal{M}%
_{s}^{S}\text{ is the unique quasicube containing }J.  \label{first key}
\end{equation}

Let $I_{s}\equiv \pi ^{\mathbf{\rho }-s}S$ so that for each $J$ in $\Pi _{2}%
\mathcal{Q}_{s}^{S}$ we have the second key inclusion%
\begin{equation}
\pi ^{\mathbf{\rho }}J=I_{s}\subset I_{\mathcal{Q}^{S}}\left( J\right) .
\label{second key}
\end{equation}%
Now\ each $K\in \mathcal{M}_{s}^{S}$ is also $\left( \mathbf{\rho }-\mathbf{%
\tau }\right) $-deeply embedded in $I_{s}$ if $\mathbf{\rho }\geq \mathbf{r}+%
\mathbf{\tau }$, so that in particular, $3K\subset I_{s}$. This and (\ref%
{second key}) have the consequence that the following Poisson inequalities
hold:%
\begin{equation*}
\frac{\mathrm{P}^{\alpha }\left( J,\mathbf{1}_{A\setminus I_{\mathcal{Q}%
^{S}}\left( J\right) }\sigma \right) }{\left\vert J\right\vert ^{\frac{m}{n}}%
}\lesssim \frac{\mathrm{P}^{\alpha }\left( J,\mathbf{1}_{A\setminus
I_{s}}\sigma \right) }{\left\vert J\right\vert ^{\frac{m}{n}}}\lesssim \frac{%
\mathrm{P}^{\alpha }\left( K,\mathbf{1}_{A\setminus I_{s}}\sigma \right) }{%
\left\vert K\right\vert ^{\frac{m}{n}}}\lesssim \frac{\mathrm{P}^{\alpha
}\left( K,\mathbf{1}_{A\setminus S}\sigma \right) }{\left\vert K\right\vert
^{\frac{m}{n}}}.
\end{equation*}%
Let $\Pi _{2}\mathcal{Q}_{s}^{S}\left( K\right) \equiv \left\{ J\in \Pi _{2}%
\mathcal{Q}_{s}^{S}:J\subset K\right\} $. Let%
\begin{eqnarray*}
\left[ \Pi _{2}\mathcal{Q}_{s}^{S}\right] _{\ell } &\equiv &\left\{ J\in \Pi
_{2}\mathcal{Q}_{s}^{S}:\ell \left( J^{\prime }\right) =2^{-\ell }\ell
\left( K\right) \right\} , \\
\left[ \Pi _{2}\mathcal{Q}_{s}^{S}\right] _{\ell }^{\ast } &\equiv &\left\{
J^{\prime }:J^{\prime }\subset J\in \Pi _{2}\mathcal{Q}_{s}^{S}:\ell \left(
J^{\prime }\right) =2^{-\ell }\ell \left( K\right) \right\} .
\end{eqnarray*}%
Now set $\mathcal{Q}_{s}\equiv \dbigcup\limits_{S\in \mathcal{S}}\mathcal{Q}%
_{s}^{S}$. We apply (\ref{bfi 3}) and Cauchy-Schwarz in $J$ to bound $%
\left\vert \mathsf{B}\right\vert _{\limfunc{stop},1,\bigtriangleup }^{A,%
\mathcal{Q}_{s}}\left( f,g\right) $ by 
\begin{equation*}
\alpha _{\mathcal{A}}\left( A\right) \sum_{S\in \mathcal{S}}\sum_{K\in 
\mathcal{M}_{s}^{S}}\left( \frac{\mathrm{P}^{\alpha }\left( K,\mathbf{1}%
_{A\setminus S}\sigma \right) }{\left\vert K\right\vert ^{\frac{1}{n}}}%
\right) \left\Vert \mathsf{P}_{\Pi _{2}^{S,\mathbf{\tau }-\limfunc{deep}}%
\mathcal{Q}_{s};K}^{\omega }\mathbf{x}\right\Vert _{L^{2}\left( \omega
\right) }\left\Vert \mathsf{P}_{\Pi _{2}^{S,\mathbf{\tau }-\limfunc{deep}}%
\mathcal{Q}_{s};K}^{\omega }g\right\Vert _{L^{2}\left( \omega \right) },
\end{equation*}%
where the localized projections $\mathsf{P}_{\Pi _{2}^{S,\mathbf{\tau }-%
\limfunc{deep}}\mathcal{Q}_{s};K}^{\omega }$ are defined in (\ref{def
localization}) above.

Thus using Cauchy-Schwarz in $K$ we have that $\left\vert \mathsf{B}_{%
\limfunc{stop},1,\bigtriangleup }^{A,\mathcal{Q}_{s}}\left( f,g\right)
\right\vert $ is bounded by%
\begin{eqnarray*}
&&\alpha _{\mathcal{A}}\left( A\right) \sum_{S\in \mathcal{S}}\sum_{K\in 
\mathcal{M}_{s}^{S}}\sqrt{\left\vert K\right\vert _{\sigma }} \\
&&\times \frac{1}{\sqrt{\left\vert K\right\vert _{\sigma }}}\left( \frac{%
\mathrm{P}^{\alpha }\left( K,\mathbf{1}_{A\setminus S}\sigma \right) }{%
\left\vert K\right\vert ^{\frac{1}{n}}}\right) \left\Vert \mathsf{P}_{\Pi
_{2}\mathcal{Q}_{s}^{S}\left( K\right) }^{\omega }\mathbf{x}\right\Vert
_{L^{2}\left( \omega \right) }\left\Vert \mathsf{P}_{\Pi _{2}\mathcal{Q}%
_{s}^{S}\left( K\right) }^{\omega }g\right\Vert _{L^{2}\left( \omega \right)
} \\
&\leq &\alpha _{\mathcal{A}}\left( A\right) \sup_{S\in \mathcal{S}}\mathcal{S%
}_{\limfunc{size}}^{\alpha ,A;S}\left( \mathcal{Q}\right) \left( \sum_{S\in 
\mathcal{S}}\sum_{K\in \mathcal{N}_{\mathbf{\rho }-\mathbf{\tau }}\left(
S\right) }\left\vert K\right\vert _{\sigma }\right) ^{\frac{1}{2}}\left\Vert
g\right\Vert _{L^{2}\left( \omega \right) } \\
&\leq &\sup_{S\in \mathcal{S}}\mathcal{S}_{\limfunc{size}}^{\alpha
,A;S}\left( \mathcal{Q}\right) \alpha _{\mathcal{A}}\left( A\right) \sqrt{%
\left\vert A\right\vert _{\sigma }}\left\Vert g\right\Vert _{L^{2}\left(
\omega \right) },
\end{eqnarray*}%
since $J\Subset _{\mathbf{\tau },\varepsilon }M\subset K$ by (\ref{first key}%
), since $\mathcal{M}_{s}^{S}\subset \mathcal{N}_{\mathbf{\rho }-\mathbf{%
\tau }}\left( S\right) $, and since the collection of quasicubes $%
\dbigcup\limits_{S\in \mathcal{S}}\mathcal{M}_{s}^{S}$ is pairwise disjoint
in $A$.

\textbf{Case} for $\left\vert \mathsf{B}\right\vert _{\limfunc{stop}%
,1,\bigtriangleup }^{A,\mathcal{Q}_{\ast }}\left( f,g\right) $: This time we
let $\mathcal{C}_{\ast }^{S}\equiv \pi ^{\mathbf{\tau }}\left( \Pi _{2}%
\mathcal{Q}_{\ast }^{S}\right) $ and denote by $\mathcal{M}_{\ast }^{S}$ the
set of \emph{maximal} quasicubes in the collection $\mathcal{C}_{\ast }^{S}$%
. We have the two key inclusions,%
\begin{equation*}
J\Subset _{\mathbf{\tau },\varepsilon }M\Subset _{\mathbf{\rho }-\mathbf{%
\tau },\varepsilon }S\text{ if }M\in \mathcal{M}_{\ast }^{S}\text{ is the
unique quasicube containing }J,
\end{equation*}%
and%
\begin{equation*}
\pi ^{\mathbf{\rho }}J\subset S\subset I_{\mathcal{Q}}\left( J\right) .
\end{equation*}%
Moreover there is $K\in \mathcal{W}^{\limfunc{good}}\left( S\right) $ that
contains $M$. Thus $3K\subset S$ and we have 
\begin{equation*}
\frac{\mathrm{P}^{\alpha }\left( J,\mathbf{1}_{A\setminus S}\sigma \right) }{%
\left\vert J\right\vert ^{\frac{1}{n}}}\lesssim \frac{\mathrm{P}^{\alpha
}\left( K,\mathbf{1}_{A\setminus S}\sigma \right) }{\left\vert K\right\vert
^{\frac{1}{n}}},
\end{equation*}%
and $\left\vert \varphi _{J}\right\vert \lesssim \alpha _{\mathcal{A}}\left(
A\right) 1_{A\setminus S}$. Now set $\mathcal{Q}_{\ast }\equiv
\dbigcup\limits_{S\in \mathcal{S}}\mathcal{Q}_{\ast }^{S}$. Arguing as
above, but with $\mathcal{W}^{\limfunc{good}}\left( S\right) $ in place of $%
\mathcal{N}_{\mathbf{\rho }-\mathbf{\tau }}\left( S\right) $, and using $%
J\Subset _{\mathbf{\rho },\varepsilon }I_{\mathcal{Q}}\left( J\right) $, we
can bound $\left\vert \mathsf{B}\right\vert _{\limfunc{stop}%
,1,\bigtriangleup }^{A,\mathcal{Q}_{\ast }}\left( f,g\right) $ by%
\begin{eqnarray*}
&&\alpha _{\mathcal{A}}\left( A\right) \sum_{S\in \mathcal{S}}\sum_{K\in 
\mathcal{W}^{\limfunc{good}}\left( S\right) }\sqrt{\left\vert K\right\vert
_{\sigma }} \\
&&\times \frac{1}{\sqrt{\left\vert K\right\vert _{\sigma }}}\left( \frac{%
\mathrm{P}^{\alpha }\left( K,\mathbf{1}_{A\setminus S}\sigma \right) }{%
\left\vert K\right\vert ^{\frac{1}{n}}}\right) \left\Vert \mathsf{P}_{\Pi
_{2}\mathcal{Q}_{s}^{S}\left( K\right) }^{\omega }\mathbf{x}\right\Vert
_{L^{2}\left( \omega \right) }\left\Vert \mathsf{P}_{\Pi _{2}\mathcal{Q}%
_{\ast }^{S}\left( K\right) }^{\omega }g\right\Vert _{L^{2}\left( \omega
\right) } \\
&\leq &\alpha _{\mathcal{A}}\left( A\right) \sup_{S\in \mathcal{S}}\mathcal{S%
}_{\limfunc{size}}^{\alpha ,A;S}\left( \mathcal{Q}\right) \left( \sum_{S\in 
\mathcal{S}}\sum_{K\in \mathcal{W}^{\limfunc{good}}\left( S\right)
}\left\vert K\right\vert _{\sigma }\right) ^{\frac{1}{2}}\left\Vert
g\right\Vert _{L^{2}\left( \omega \right) } \\
&\leq &\sup_{S\in \mathcal{S}}\mathcal{S}_{\limfunc{size}}^{\alpha
,A;S}\left( \mathcal{Q}\right) \alpha _{\mathcal{A}}\left( A\right) \sqrt{%
\left\vert A\right\vert _{\sigma }}\left\Vert g\right\Vert _{L^{2}\left(
\omega \right) }.
\end{eqnarray*}%
We now sum these bounds in $s$ and $\ast $ and use $\sup_{S\in \mathcal{S}}%
\mathcal{S}_{\limfunc{size}}^{\alpha ,A;S}\left( \mathcal{Q}\right) \leq 
\mathcal{S}_{\limfunc{size}}^{\alpha ,A}\left( \mathcal{Q}\right) $ to
complete the proof of Lemma \ref{straddle 3}.
\end{proof}

\subsubsection{The Orthogonality Lemma}

Given a set $\left\{ \mathcal{Q}_{m}\right\} _{m=0}^{\infty }$ of admissible
collections for $A$, we say that the collections $\mathcal{Q}_{m}$ are \emph{%
mutually orthogonal}, if each collection $\mathcal{Q}_{m}$ satisfies%
\begin{equation*}
\mathcal{Q}_{m}\subset \dbigcup\limits_{j=0}^{\infty }\left\{ \mathcal{A}%
_{m,j}\times \mathcal{B}_{m,j}\right\} \ ,
\end{equation*}%
where the sets $\left\{ \mathcal{A}_{m,j}\right\} _{m,j}$ and $\left\{ 
\mathcal{B}_{m,j}\right\} _{m,j}$ each have bounded overlap on the dyadic
quasigrid $\Omega \mathcal{D}$: 
\begin{equation*}
\sum_{m,j=0}^{\infty }\mathbf{1}_{\mathcal{A}_{m,j}}\leq A\mathbf{1}_{\Omega 
\mathcal{D}}\text{ and }\sum_{m,j=0}^{\infty }\mathbf{1}_{\mathcal{B}%
_{m,j}}\leq B\mathbf{1}_{\Omega \mathcal{D}}.
\end{equation*}

\begin{lemma}
\label{mut orth}Suppose that $\left\{ \mathcal{Q}_{m}\right\} _{m=0}^{\infty
}$ is a set of admissible collections for $A$ that are \emph{mutually
orthogonal}. Then if $\mathcal{Q}\equiv \dbigcup\limits_{m=0}^{\infty }%
\mathcal{Q}_{m}$, the sublinear stopping form $\left\vert \mathsf{B}%
\right\vert _{\limfunc{stop},1,\bigtriangleup }^{A,\mathcal{Q}}\left(
f,g\right) $ has its restricted norm $\mathfrak{N}_{\limfunc{stop}%
,1,\bigtriangleup }^{A,\mathcal{Q}}$ controlled by the \emph{supremum} of
the restricted norms $\mathfrak{N}_{\limfunc{stop},1,\bigtriangleup }^{A,%
\mathcal{Q}_{m}}$: 
\begin{equation*}
\mathfrak{N}_{\limfunc{stop},1,\bigtriangleup }^{A,\mathcal{Q}}\leq \sqrt{nAB%
}\sup_{m\geq 0}\mathfrak{N}_{\limfunc{stop},1,\bigtriangleup }^{A,\mathcal{Q}%
_{m}}.
\end{equation*}
\end{lemma}

\begin{proof}
If $\mathsf{P}_{m}^{\sigma }=\dsum\limits_{j\geq 0}\dsum\limits_{I\in 
\mathcal{A}_{m,j}}\bigtriangleup _{\pi I}^{\sigma }$ (note the parent $\pi I$
in the projection $\bigtriangleup _{\pi I}^{\sigma }$ because of our `change
of dummy variable' in (\ref{dummy})) and $\mathsf{P}_{m}^{\omega
}=\dsum\limits_{j\geq 0}\dsum\limits_{J\in \mathcal{B}_{m,j}}\bigtriangleup
_{J}^{\omega }$, then we have%
\begin{equation*}
\mathsf{B}_{\limfunc{stop}}^{A,\mathcal{Q}_{m}}\left( f,g\right) =\mathsf{B}%
_{\limfunc{stop}}^{A,\mathcal{Q}_{m}}\left( \mathsf{P}_{m}^{\sigma }f,%
\mathsf{P}_{m}^{\omega }g\right) ,
\end{equation*}%
and%
\begin{eqnarray*}
\sum_{m\geq 0}\left\Vert \mathsf{P}_{m}^{\sigma }f\right\Vert _{L^{2}\left(
\sigma \right) }^{2} &\leq &\sum_{m\geq 0}\sum_{j\geq 0}\left\Vert \mathsf{P}%
_{\mathcal{A}_{m,j}}^{\sigma }f\right\Vert _{L^{2}\left( \sigma \right)
}^{2}\leq An\left\Vert f\right\Vert _{L^{2}\left( \sigma \right) }^{2}, \\
\sum_{m\geq 0}\left\Vert \mathsf{P}_{m}^{\omega }g\right\Vert _{L^{2}\left(
\sigma \right) }^{2} &\leq &\sum_{m\geq 0}\sum_{j\geq 0}\left\Vert \mathsf{P}%
_{\mathcal{B}_{m,j}}^{\omega }g\right\Vert _{L^{2}\left( \omega \right)
}^{2}\leq B\left\Vert g\right\Vert _{L^{2}\left( \omega \right) }^{2}\ .
\end{eqnarray*}%
The sublinear inequality (\ref{phi sublinear}) and Cauchy-Schwarz now give%
\begin{eqnarray*}
\left\vert \mathsf{B}\right\vert _{\limfunc{stop},1,\bigtriangleup }^{A,%
\mathcal{Q}}\left( f,g\right) &\leq &\sum_{m\geq 0}\left\vert \mathsf{B}%
\right\vert _{\limfunc{stop},1,\bigtriangleup }^{A,\mathcal{Q}_{m}}\left(
f,g\right) \leq \sum_{m\geq 0}\mathfrak{N}_{stop}^{A,\mathcal{Q}%
_{m}}\left\Vert \mathsf{P}_{m}^{\sigma }f\right\Vert _{L^{2}\left( \sigma
\right) }\left\Vert \mathsf{P}_{m}^{\omega }g\right\Vert _{L^{2}\left(
\sigma \right) } \\
&\leq &\left( \sup_{m\geq 0}\mathfrak{N}_{\limfunc{stop},1,\bigtriangleup
}^{A,\mathcal{Q}_{m}}\right) \sqrt{\sum_{m\geq 0}\left\Vert \mathsf{P}%
_{m}^{\sigma }f\right\Vert _{L^{2}\left( \sigma \right) }^{2}}\sqrt{%
\sum_{m\geq 0}\left\Vert \mathsf{P}_{m}^{\omega }g\right\Vert _{L^{2}\left(
\sigma \right) }^{2}} \\
&\leq &\left( \sup_{m\geq 0}\mathfrak{N}_{\limfunc{stop},1,\bigtriangleup
}^{A,\mathcal{Q}_{m}}\right) \sqrt{nAB}\sqrt{n}\left\Vert f\right\Vert
_{L^{2}\left( \sigma \right) }\left\Vert g\right\Vert _{L^{2}\left( \omega
\right) }.
\end{eqnarray*}
\end{proof}

\subsubsection{Completion of the proof}

Now we return to the proof of inequality (\ref{big 3}) in Proposition \ref%
{bottom up 3}.

\begin{proof}[Proof of (\protect\ref{big 3})]
Recall that 
\begin{eqnarray*}
\mathcal{P}^{big} &=&\left\{ \dbigcup\limits_{L\in \mathcal{L}}\mathcal{P}%
_{L,0}^{big}\right\} \dbigcup \left\{ \dbigcup\limits_{t\geq
1}\dbigcup\limits_{L\in \mathcal{L}}\mathcal{P}_{L,t}\right\} \equiv 
\mathcal{Q}_{0}^{big}\dbigcup \mathcal{Q}_{1}^{big}; \\
\mathcal{Q}_{0}^{big} &\equiv &\dbigcup\limits_{L\in \mathcal{L}}\mathcal{P}%
_{L,0}^{big}\ ,\ \ \ \ \ \mathcal{Q}_{1}^{big}\equiv \dbigcup\limits_{t\geq
1}\mathcal{P}_{t}^{big},\ \ \ \ \ \mathcal{P}_{t}^{big}\equiv
\dbigcup\limits_{L\in \mathcal{L}}\mathcal{P}_{L,t}.
\end{eqnarray*}%
We first consider the collection $\mathcal{Q}_{0}^{big}=\dbigcup\limits_{L%
\in \mathcal{L}}\mathcal{P}_{L,0}^{big}$, and claim that%
\begin{equation}
\mathfrak{N}_{\limfunc{stop},1,\bigtriangleup }^{A,\mathcal{P}%
_{L,0}^{big}}\leq C\mathcal{S}_{\limfunc{size}}^{\alpha ,A}\left( \mathcal{P}%
_{L,0}^{big}\right) \leq C\mathcal{S}_{\limfunc{size}}^{\alpha ,A}\left( 
\mathcal{P}\right) ,\ \ \ \ \ L\in \mathcal{L}.  \label{big t 3}
\end{equation}%
To see this we note that $\mathcal{P}_{L,0}^{big}$ $\mathbf{\tau }$%
-straddles the trivial collection $\left\{ L\right\} $ consisting of a
single quasicube, since the pairs $\left( I,J\right) $ that arise in $%
\mathcal{P}_{L,0}^{big}$ have $I=L$ and $J$ in the shifted corona $\mathcal{C%
}_{I}^{\mathbf{\tau }-\limfunc{shift}}$. Thus we can apply Lemma \ref%
{straddle 3} with $\mathcal{Q}=\mathcal{P}_{L,0}^{big}$ and $\mathcal{S}%
=\left\{ L\right\} $ to obtain (\ref{big t 3}).

Next, we observe that the collections $\mathcal{P}_{L,0}^{big}$ are \emph{%
mutually orthogonal}, namely 
\begin{eqnarray*}
\mathcal{P}_{L,0}^{big} &\subset &\mathcal{C}_{L}\times \mathcal{C}_{L}^{%
\mathbf{\tau }-\limfunc{shift}}\ , \\
\dsum\limits_{L\in \mathcal{L}}\mathbf{1}_{\mathcal{C}_{L}} &\leq &1\text{
and }\dsum\limits_{L\in \mathcal{L}}\mathbf{1}_{\mathcal{C}_{L}^{\mathbf{%
\tau }-\limfunc{shift}}}\leq \mathbf{\tau }.
\end{eqnarray*}%
Thus the Orthogonality Lemma \ref{mut orth} shows that%
\begin{equation*}
\mathfrak{N}_{\limfunc{stop},1,\bigtriangleup }^{A,\mathcal{Q}%
_{0}^{big}}\leq \sqrt{n\mathbf{\tau }}\sup_{L\in \mathcal{L}}\mathfrak{N}_{%
\limfunc{stop},1,\bigtriangleup }^{A,\mathcal{P}_{L,0}^{big}}\leq \sqrt{n%
\mathbf{\tau }}C\mathcal{S}_{\limfunc{size}}^{\alpha ,A}\left( \mathcal{P}%
\right) .
\end{equation*}

Now we turn to the collection%
\begin{eqnarray*}
\mathcal{Q}_{1}^{big} &=&\dbigcup\limits_{t\geq 1}\dbigcup\limits_{L\in 
\mathcal{L}}\mathcal{P}_{L,t}=\dbigcup\limits_{t\geq 1}\mathcal{P}_{t}^{big};
\\
\mathcal{P}_{t}^{big} &\equiv &\dbigcup\limits_{L\in \mathcal{L}}\mathcal{P}%
_{L,t}\ ,\ \ \ \ \ t\geq 0.
\end{eqnarray*}%
We claim that%
\begin{equation}
\mathfrak{N}_{\limfunc{stop},1,\bigtriangleup }^{A,\mathcal{P}%
_{t}^{big}}\leq C\rho ^{-\frac{t}{2}}\mathcal{S}_{\limfunc{size}}^{\alpha
,A}\left( \mathcal{P}\right) ,\ \ \ \ \ t\geq 1.  \label{S big t 3}
\end{equation}%
Note that with this claim established, we have%
\begin{equation*}
\mathfrak{N}_{\limfunc{stop},1,\bigtriangleup }^{A,\mathcal{P}^{big}}\leq 
\mathfrak{N}_{\limfunc{stop},1,\bigtriangleup }^{A,\mathcal{Q}_{0}^{big}}+%
\mathfrak{N}_{\limfunc{stop},1,\bigtriangleup }^{A,\mathcal{Q}%
_{1}^{big}}\leq \mathfrak{N}_{\limfunc{stop},1,\bigtriangleup }^{A,\mathcal{Q%
}_{0}^{big}}+\sum_{t=1}^{\infty }\mathfrak{N}_{\limfunc{stop}%
,1,\bigtriangleup }^{A,\mathcal{P}_{t}^{big}}\leq C_{\mathbf{\rho }}\mathcal{%
S}_{\limfunc{size}}^{\alpha ,A}\left( \mathcal{P}\right) ,
\end{equation*}%
which proves (\ref{big 3}) if we apply the Orthogonal Lemma \ref{mut orth}
to the set of collections $\left\{ \mathcal{P}_{L,0}^{small}\right\} _{L\in 
\mathcal{L}}$, which is mutually orthogonal since $\mathcal{P}%
_{L,0}^{small}\subset \mathcal{C}_{L}^{\prime }\times \mathcal{C}_{L}^{%
\mathbf{\tau }-\limfunc{shift}}$ . With this the proof of Proposition \ref%
{bottom up 3} is now complete since $\rho =1+\varepsilon $. Thus it remains
only to show that (\ref{S big t 3}) holds.

The cases $1\leq t\leq \mathbf{r}+1$ can be handled with relative ease since
decay in $t$ is not needed there. Indeed, $\mathcal{P}_{L,t}$ $\mathbf{\tau }
$-straddles the collection $\mathfrak{C}_{\mathcal{L}}\left( L\right) $ of $%
\mathcal{L}$-children of $L$, and so the Straddling Lemma applies to give%
\begin{equation*}
\mathfrak{N}_{\limfunc{stop},1,\bigtriangleup }^{A,\mathcal{P}_{L,t}}\leq C%
\mathcal{S}_{\limfunc{size}}^{\alpha ,A}\left( \mathcal{P}_{L,t}\right) \leq
C\mathcal{S}_{\limfunc{size}}^{\alpha ,A}\left( \mathcal{P}\right) ,
\end{equation*}%
and then the Orthogonality Lemma \ref{mut orth} applies to give%
\begin{equation*}
\mathfrak{N}_{\limfunc{stop},1,\bigtriangleup }^{A,\mathcal{P}%
_{t}^{big}}\leq \sqrt{n\mathbf{\tau }}\sup_{L\in \mathcal{L}}\mathfrak{N}_{%
\limfunc{stop},1,\bigtriangleup }^{A,\mathcal{P}_{L,t}}\leq C\sqrt{n\mathbf{%
\tau }}\mathcal{S}_{\limfunc{size}}^{\alpha ,A}\left( \mathcal{P}\right) ,
\end{equation*}%
since $\left\{ \mathcal{P}_{L,t}\right\} _{L\in \mathcal{L}}$ is mutually
orthogonal as $\mathcal{P}_{L,t}\subset \mathcal{C}_{L}\times \mathcal{C}%
_{L^{\prime }}^{\mathbf{\tau }-\limfunc{shift}}$ with $L\in \mathcal{G}_{d}$
and $L^{\prime }\in \mathcal{G}_{d+t}$ for depth $d=d\left( L\right) $.

Now we consider the case $t\geq \mathbf{r}+2$, where it is essential to
obtain decay in $t$. We again apply Lemma \ref{straddle 3} to $\mathcal{P}%
_{L,t}$ with $\mathcal{S}=\mathfrak{C}_{\mathcal{L}}\left( L\right) $, but
this time we must use the stronger localized bounds $\mathcal{S}_{\limfunc{%
size}}^{\alpha ,A;S}$ with an $S$-hole, that give%
\begin{equation}
\mathfrak{N}_{\limfunc{stop},1,\bigtriangleup }^{A,\mathcal{P}_{L,t}}\leq
C\sup_{S\in \mathfrak{C}_{\mathcal{L}}\left( L\right) }\mathcal{S}_{\limfunc{%
size}}^{\alpha ,A;S}\left( \mathcal{P}_{L,t}\right) ,\ \ \ \ \ t\geq 0.
\label{t,n 3}
\end{equation}%
Fix $L\in \mathcal{G}_{d}$. Now we note that if $J\in \Pi _{2}^{L,\mathbf{%
\tau }-\limfunc{deep}}\mathcal{P}_{L,t}$ then $J$ belongs to the $\mathbf{%
\tau }$-shifted corona $\mathcal{C}_{L^{d+t}}^{\mathbf{\tau }-\limfunc{shift}%
}$ for some quasicube $L^{d+t}\in \mathcal{G}_{d+t}$. Then $\pi ^{\mathbf{%
\tau }}J$ is $\mathbf{\tau }$ levels above $J$, hence in the corona $%
\mathcal{C}_{L^{d+t}}$. This quasicube $L^{d+t}$ lies in some child $S\in 
\mathcal{S}=\mathfrak{C}_{\mathcal{L}}\left( L\right) $. So fix $S\in 
\mathcal{S}$ and a quasicube $L^{d+t}\in \mathcal{G}_{d+t}$ that is
contained in $S$ with $t\geq \mathbf{r}+2$. Now the quasicubes $K$ that
arise in the supremum defining $\mathcal{S}_{\limfunc{size}}^{\alpha
,A;S}\left( \mathcal{P}_{L,t}\right) $ in (\ref{localized size}) belong to
either $\mathcal{N}_{\mathbf{\rho }-\mathbf{\tau }}\left( S\right) $ or $%
\mathcal{W}^{\limfunc{good}}\left( S\right) $. We will consider these two
cases separately.

So first suppose that $K\in \mathcal{N}_{\mathbf{\rho }-\mathbf{\tau }%
}\left( S\right) $. A simple induction on levels yields%
\begin{eqnarray*}
\omega _{\mathcal{P}_{L,t}}\left( \mathbf{T}^{\mathbf{\tau }-\limfunc{deep}%
}\left( K\right) \right) &=&\sum_{\substack{ J\in \Pi _{2}^{S,\mathbf{\tau }-%
\limfunc{deep}}\mathcal{P}_{L,t}  \\ J\subset K}}\left\Vert \bigtriangleup
_{J}^{\omega }\mathbf{x}\right\Vert _{L^{2}\left( \omega \right) }^{2} \\
&\leq &\omega _{\mathcal{P}}\left( \dbigcup\limits_{L^{d+t}\in \mathcal{G}%
_{d+t}:\ L^{d+t}\subset K}\mathbf{T}^{\mathbf{\tau }-\limfunc{deep}}\left(
L^{d+t}\right) \right) \\
&\leq &\frac{1}{\rho }\omega _{\mathcal{P}}\left(
\dbigcup\limits_{L^{d+t-1}\in \mathcal{G}_{d+t-1}:\ L^{d+t-1}\subset K}%
\mathbf{T}^{\mathbf{\tau }-\limfunc{deep}}\left( L^{d+t-1}\right) \right) \\
&&\vdots \\
&\lesssim &\rho ^{-\left( t-\mathbf{\rho }-\mathbf{\tau }\right) }\omega _{%
\mathcal{P}}\left( \mathbf{T}^{\mathbf{\tau }-\limfunc{deep}}\left( K\right)
\right) ,\ \ \ \ \ t\geq \mathbf{\rho }-\mathbf{\tau +}2.
\end{eqnarray*}%
Thus we have%
\begin{eqnarray*}
&&\frac{1}{\left\vert K\right\vert _{\sigma }}\left( \frac{\mathrm{P}%
^{\alpha }\left( K,\mathbf{1}_{A\setminus S}\sigma \right) }{\left\vert
K\right\vert ^{\frac{1}{n}}}\right) ^{2}\omega _{\mathcal{P}_{L,t}}\left( 
\mathbf{T}^{\mathbf{\tau }-\limfunc{deep}}\left( K\right) \right) \\
&\lesssim &\rho ^{-t}\frac{1}{\left\vert K\right\vert _{\sigma }}\left( 
\frac{\mathrm{P}^{\alpha }\left( K,\mathbf{1}_{A\setminus S}\sigma \right) }{%
\left\vert K\right\vert ^{\frac{1}{n}}}\right) ^{2}\omega _{\mathcal{P}%
}\left( \mathbf{T}^{\mathbf{\tau }-\limfunc{deep}}\left( K\right) \right)
\lesssim \rho ^{-t}\mathcal{S}_{\limfunc{size}}^{\alpha ,A}\left( \mathcal{P}%
\right) ^{2}.
\end{eqnarray*}

Now suppose that $K\in \mathcal{W}^{\limfunc{good}}\left( S\right) $ and
that $J\in \Pi _{2}^{S,\mathbf{\tau }-\limfunc{deep}}\mathcal{P}_{L,t}$ and $%
J\subset K$. There is a unique quasicube $L^{d+\mathbf{r}+1}\in \mathcal{G}%
_{d+\mathbf{r}+1}$ such that $J\subset L^{d+\mathbf{r}+1}\subset S$. Now $%
L^{d+\mathbf{r}+1}$ is good so $L^{d+\mathbf{r}+1}\Subset _{\mathbf{r}%
,\varepsilon }S$. Thus in particular $3L^{d+\mathbf{r}+1}\subset S$ so that $%
L^{d+\mathbf{r}+1}\subset K$. The above simple induction applies here to give%
\begin{eqnarray*}
\sum_{\substack{ J\in \Pi _{2}^{S,\mathbf{\tau }-\limfunc{deep}}\mathcal{P}%
_{L,t}  \\ J\subset L^{d+\mathbf{r}+1}}}\left\Vert \bigtriangleup
_{J}^{\omega }\mathbf{x}\right\Vert _{L^{2}\left( \omega \right) }^{2} &\leq
&\omega _{\mathcal{P}}\left( \dbigcup\limits_{L^{d+t}\in \mathcal{G}_{d+t}:\
L^{m-t}\subset L^{d+\mathbf{r}+1}}\mathbf{T}^{\mathbf{\tau }-\limfunc{deep}%
}\left( L^{d+t}\right) \right) \\
&\lesssim &\rho ^{-\left( t-1-\mathbf{r}\right) }\omega _{\mathcal{P}}\left( 
\mathbf{T}^{\mathbf{\tau }-\limfunc{deep}}\left( L^{d+\mathbf{r}+1}\right)
\right) ,\ \ \ \ \ t\geq \mathbf{r}+2.
\end{eqnarray*}%
Thus we have,%
\begin{eqnarray*}
&&\left( \frac{\mathrm{P}^{\alpha }\left( K,\mathbf{1}_{A\setminus S}\sigma
\right) }{\left\vert K\right\vert ^{\frac{1}{n}}}\right) ^{2}\sum_{\substack{
J\in \Pi _{2}^{K,\mathbf{\tau }-\limfunc{deep}}\mathcal{P}_{L,t}  \\ %
J\subset K}}\left\Vert \bigtriangleup _{J}^{\omega }\mathbf{x}\right\Vert
_{L^{2}\left( \omega \right) }^{2} \\
&\leq &C\left( \frac{\mathrm{P}^{\alpha }\left( K,\mathbf{1}_{A\setminus
S}\sigma \right) }{\left\vert K\right\vert ^{\frac{1}{n}}}\right) ^{2}\rho
^{-\left( t-1-\mathbf{r}\right) }\sum_{\substack{ L^{d+\mathbf{r}+1}\in 
\mathcal{G}_{d+\mathbf{r}+1}  \\ L^{d+\mathbf{r}+1}\subset K}}\omega _{%
\mathcal{P}}\left( \mathbf{T}^{\mathbf{\tau }-\limfunc{deep}}\left( L^{d+%
\mathbf{r}+1}\right) \right) \\
&\leq &C\rho ^{-\left( t-1-\mathbf{r}\right) }\left( \frac{\mathrm{P}%
^{\alpha }\left( K,\mathbf{1}_{A\setminus S}\sigma \right) }{\left\vert
K\right\vert ^{\frac{1}{n}}}\right) ^{2}\omega _{\mathcal{P}}\left( \mathbf{T%
}^{\mathbf{\tau }-\limfunc{deep}}\left( K\right) \right) \leq C\rho
^{-\left( t-1-\mathbf{r}\right) }\mathcal{S}_{\limfunc{size}}^{\alpha
,A}\left( \mathcal{P}\right) ^{2}.
\end{eqnarray*}

So altogether we conclude that%
\begin{eqnarray*}
&&\sup_{S\in \mathfrak{C}_{\mathcal{L}}\left( L\right) }\mathcal{S}_{%
\limfunc{size}}^{\alpha ,A;S}\left( \mathcal{P}_{L,t}\right) ^{2} \\
&=&\sup_{S\in \mathfrak{C}_{\mathcal{L}}\left( L\right) }\sup_{K\in \mathcal{%
N}_{\mathbf{\rho }-\mathbf{\tau }}\left( S\right) \cup \mathcal{W}^{\limfunc{%
good}}\left( S\right) }\frac{1}{\left\vert K\right\vert _{\sigma }}\left( 
\frac{\mathrm{P}^{\alpha }\left( K,\mathbf{1}_{A\setminus K}\sigma \right) }{%
\left\vert K\right\vert ^{\frac{1}{n}}}\right) ^{2}\sum_{\substack{ J\in \Pi
_{2}^{K,\mathbf{\tau }-\limfunc{deep}}\mathcal{P}_{L,t}  \\ J\subset K}}%
\left\Vert \mathsf{P}_{J}^{\omega }\mathbf{x}\right\Vert _{L^{2}\left(
\omega \right) }^{2} \\
&\leq &C_{\mathbf{r},\mathbf{\tau },\mathbf{\rho }}\rho ^{-t}\mathcal{S}_{%
\limfunc{size}}^{\alpha ,A}\left( \mathcal{P}\right) ^{2},
\end{eqnarray*}%
and combined with (\ref{t,n 3}) this gives (\ref{S big t 3}). As we pointed
out above, this completes the proof of Proposition \ref{bottom up 3}, hence
of Proposition \ref{stopping bound}, and finally of Theorem \ref{T1 theorem}.
\end{proof}

\end{document}